\documentclass[openany]{memo-l}

\usepackage{graphicx}
\usepackage{enumitem}
\usepackage{amsmath,amsfonts,amssymb,latexsym}
\usepackage{color}
\usepackage{chngcntr}
\counterwithin{figure}{chapter}

\makeatletter
 \renewcommand*\l@figure{\@tocline{0}{3pt plus2pt}{0pt}{2.0pc}{}}
\makeatother

\newtheorem{theorem}{Theorem}[chapter]
\newtheorem{lemma}[theorem]{Lemma}
\newtheorem{corollary}[theorem]{Corollary}
\newtheorem{proposition}[theorem]{Proposition}

\newtheorem{observation}[theorem]{Observation}

\theoremstyle{definition}
\newtheorem{definition}[theorem]{Definition}

\theoremstyle{remark}
\newtheorem{remark}[theorem]{Remark}

\newtheorem{claim}{Claim}
\newtheorem{subclaim}{Subclaim}

\numberwithin{section}{chapter}
\numberwithin{equation}{chapter}

\makeindex

\newcommand{\DEF}[1]{{\em #1\/}}

{\noindent{\bf Proof.}\ }%
{\hfill\eopf\par\bigskip}%

\newenvironment{proofof}[1]%
{\noindent{\bf Proof of #1.}\ }%
{\hfill\eopf\par\bigskip}%

{\noindent{\bf Sketch of a proof.}\ }%
{\hfill\eopf\par\bigskip}%
{\noindent{\bf Ideas for a proof.}\ }%
{\hfill\eopf\par\bigskip}%
{\noindent{\bf Proof in progress.}\ }%
{\hfill\eopf\par\bigskip}%

\def\eop{\hfill{{\rule[0ex]{7pt}{7pt}}}}

\newenvironment{cproofof}[1]
{\bigskip\noindent{\bf Proof of #1.}\startClaims\ }
{\hfill{\eop}\par\bigskip}

\newenvironment{cproof}
{\noindent{\bf Proof.}\startClaims\ }
{\hfill{\eop}\par\bigskip}

\newcommand{\printFullDetails}[1]{#1}
\newcommand{\major}[1]{#1}
\newcommand{\majorrem}[1]{}
\newcommand{\minor}[1]{#1}
\newcommand{\minorrem}[1]{}
\newcommand{\wording}[1]{#1}
\newcommand{\wordingrem}[1]{}
\newcommand{\dragominor}[1]{{#1}}
\newcommand{\dragominorrem}[1]{}
\newcommand{\bogdan}[1]{{#1}}
\newcommand{\crn}{\operatorname{cr}}
\newcommand\tcrn{{\operatorname{tcr}}}
\newcommand{\cT}{{\mathcal T}}
\newcommand{\cS}{{\mathcal S}}
\newcommand{\ignore}[1]{}
\newcommand{\hvng}{ V_{2n}\topol H\subseteq G}
\newcommand{\hvfg}{ V_{10}\topol H\subseteq G}
\newcommand{\hveg}{V_8\topol H\subseteq G}
\newcommand{\Nuc}{\operatorname{Nuc}}
\newcommand{\startClaims}{\setcounter{claim}{0}}
\newcommand{\oo}[1]{\left\langle #1 \right\rangle}
\newcommand{\oc}[1]{\left\langle #1 \right]}
\newcommand{\co}[1]{\left[ #1 \right\rangle}
\newcommand{\cc}[1]{\left[ #1 \right]}
\newcommand{\bc}[1]{\left. #1 \right]}
\newcommand{\cb}[1]{\left[ #1 \right.}
\newcommand{\ob}[1]{\left\langle #1 \right.}
\newcommand{\bo}[1]{\left. #1 \right\rangle}
\newcommand{\Loc}{\operatorname{Loc}}
\newcommand{\bQ}{{\overline Q}}
\newcommand{\eopf}{\raisebox{0.8ex}{\framebox{}}}
\newcommand{\comp}[1]{#1^\#}
\newcommand{\cl}{\operatorname{cl}}
\newcommand{\startSubclaims}{\setcounter{subclaim}{0}}
\newcommand{\mc}[1]{\cM_2^{#1}}
\newcommand{\cM}{{\mathcal M}}

\def\lbsp{{\hbox{\hglue -1.5pt}}}
\def\rbsp{{\hbox{\hglue -1pt}}}
\def\i4c{{inter\-nally-4-con\-nec\-ted}}
\def\p4c{{peri\-phe\-rally-4-con\-nec\-ted}}
\def\ei4c{\operatorname{p4c}}
\def\eifc{\operatorname{i4c}}
\def\2cc{{2-cros\-sing-cri\-tical}}
\def\topol{{\,\cong\,}}
\def\m2{{{\mathcal M}_2^3}}
\def\att{{\hbox{\rm att}}}

\def\updown{{\mathstrut}^{\updownarrow}}
\def\cyclization{cyclization}
\def\tileS{\mathcal T(\cS)}
\def\kr{{\hbox{\rm cr}}}
\def\pp{\mathbb RP^2}
\def\rep{\operatorname{rep}}
\def\Disc{\mathfrak D}
\def\disc{\mathfrak D}
\def\mob{\mathfrak M}
\def\Mob{\mathfrak M}
\def\iso#1{\|#1\|}
\def\isouv{\iso{\{u,v\}}}
\def\ucr{{\hbox{\rm cr}}}
\def\bbb{{\mathcal B}}
\def\mD{{\mathcal D}}
\def\eye{\mathfrak n}
\def\fsq{$(\hskip-2pt(G,H,\Pi,\gamma)\hskip-2pt)$}
\def\topleft#1{\overset{\leftarrow}{#1}}
\def\topright#1{\overset{\rightarrow}{#1}}
\def\botleft#1{\underset{\leftarrow}{#1}{}}
\def\botright#1{\underset{\rightarrow}{#1}{}}
\def\leftspine{\mathbf\sqsupset}
\def\rightspine #1{{}_{#1}\kern-3pt\sqsubset}
\def\cupspace{\textrm{\hglue.5pt}}
\def\pUl{{\overset{\leftarrow}{P_1}}}
\def\pUr{{\overset{\rightarrow}{P_1}}}
\def\pDl{{\underset{\leftarrow}{P_1}}}
\def\pDr{{\underset{\rightarrow}{P_1}}}
\def\TRi{{{{}_{{}_1}}{}{\sqsubset}}}
\def\TLi{{\sqsupset_{{}_1}}}
\def\mfK{\mathfrak K}
\def\hug{hug}
\def\triang#1{#1-triangle}
\def\arm{arm}
\def\bearhug{bear hug}
\def\head{head}

\begin{document}

\frontmatter

\title{Characterizing 2-crossing-critical graphs}


\author{Drago Bokal}
\address{University of Maribor\\Maribor SLOVENIA}
\curraddr{}
\email{drago.bokal@uni-mb.si}
\thanks{Bokal acknowledges the support of NSERC and U. Waterloo for 2006-2007, Slovenian Research Agency basic research projects L7-5459, J6-3600, J1-2043, L1-9338, J1-6150,  research programme P1-0297, and an international research grant GReGAS}

\author{Bogdan Oporowski}
\address{Louisiana State University\\ Baton Rouge U.S.A.}
\curraddr{}
\email{bogdan@math.lsu.edu}
\thanks{}

\author{R. Bruce Richter}
\address{University of Waterloo\\ Waterloo CANADA}
\curraddr{}
\email{brichter@uwaterloo.ca}
\thanks{Richter acknowledges the support of NSERC}

\author{Gelasio Salazar}
\address{UA de San Luis Potosi\\ San Luis Potosi MEXICO}
\curraddr{}
\email{gelasio.salazar@gmail.com}
\thanks{Salazar acknowledges the support of CONACYT Grant 106432.}





\begin{abstract}It is very well-known that there are precisely two minimal non-planar graphs:  $K_5$ and $K_{3,3}$ (degree 2 vertices being irrelevant in this context).  In the language of crossing numbers, these are the only 1-crossing-critical graphs:  they each have crossing number at least one, and every proper subgraph has crossing number less than one.  In 1987, Kochol exhibited an infinite family of 3-connected, simple 2-crossing-critical graphs.   In this work, we: (i) determine all the 3-connected 2-crossing-critical graphs that contain a subdivision of the M\"obius Ladder $V_{10}$; (ii) show how to obtain all the not 3-connected 2-crossing-critical graphs from the 3-connected ones; (iii) show that there are only finitely many 3-connected 2-crossing-critical graphs not containing a subdivision of  $V_{10}$; and (iv) determine all the 3-connected 2-crossing-critical graphs that do not contain a subdivision of $V_{8}$. 
\end{abstract}

\maketitle

\tableofcontents

\listoffigures


\mainmatter
\chapter{Introduction}\label{intro}

 For a positive integer $k$, a graph $G$ is {\em $k$-crossing-critical\/}\index{crossing-critical} if the crossing number $\crn(G)$ is at least $k$, but every proper subgraph $H$ of $G$ has $\crn(H)<k$.  In general, it is not true that a $k$-crossing-critical graph has crossing number exactly $k$. For example, any edge-transitive non-planar graph $G$ satisfies $\crn(G-e)<\crn(G)$, for any edge $e$ of $G$, so every such graph is $k$-crossing-critical for any $k$ satisfying $\crn(G-e)<k\le \crn(G)$.  If $G$ is the complete graph $K_n$, then $\crn(K_n)-\crn(K_n-e)$ is of order $n^2$, so $K_n$ is $k$-crossing-critical for many different values of $k$.

Insertion and suppression of vertices of degree 2 do not affect the crossing number of a graph, and a $k$-crossing-critical graph has no vertices of degree 1 and no component that is a cycle.  Thus, if $G$ is a $k$-crossing-critical graph, the graph $G'$ whose vertex set consists of the {\em nodes\/}\index{node} of $G$ (i.e., the vertices of degree different from 2) and whose edges are the {\em branches\/} of $G$ (i.e., the maximal paths all of whose internal vertices have degree 2 in $G$) is also $k$-crossing-critical.   Our interest is, therefore, in $k$-crossing-critical graphs with minimum degree at least 3.

By Kuratowski's Theorem, the only 1-crossing-critical graphs are $K_{3,3}$ and $K_5$.  The classification of 2-crossing-critical graphs is currently not known.  The earliest published remarks on this classification of which we are aware is by Bloom, Kennedy, and Quintas \cite{bkq}, where they exhibit 21 such graphs.  Kochol \cite{kochol} gives an infinite family of 3-connected, simple 2-crossing-critical graphs, answering a question of \v Sir\'a\v n \cite{siran2} who gave, for each $n\ge 3$, an infinite family of 3-connected $n$-crossing-critical graphs.   Richter \cite{rbr} shows there are just eight cubic 2-crossing-critical graphs.

About 15 years ago, Oporowski gave several conference talks about showing that every large {peripherally}-4-connected, 2-crossing-critical graph has a very particular structure which was later denoted as `being composed of tiles'.  The method suggested was to show that if a {peripherally}-4-connected, 2-crossing-critical graph has a {subdivision of a} particular $V_{2k}$ (that is, $k$ is fixed), then it has the desired structure and that only finitely many {peripherally}-4-connected, 2-crossing-critical graphs do not have a {subdivision of} $V_{2k}$.   (The graph $V_{2n}$\index{$V_{2n}$} is obtained from a $2n$-cycle by adding the $n$ diagonals.   Note that $V_4$ is $K_4$ and $V_6$ is $K_{3,3}$.)  

Approximately 10 years ago, it was proved by Ding, Oporowski, Thomas, and Vertigan \cite{DOTV} that, for any $k$, a large (as a function of $k$) 3-connected, 2-crossing-critical graph necessarily has a {subdivision of} $V_{2k}$.  It remains to show that having the $V_{2k}$-{subdivision} implies having the desired global structure.  Their proof involves first showing a statement about non-planar graphs that is of significant independent interest:  for every $k$, any large (as a function of $k$) ``almost 4-connected" non-planar graph contains a subdivision of one of four non-planar graphs whose sizes grow with $k$.  One of the {four} graphs is $V_{2k}$.
This theorem is then used for the crossing-critical application mentioned above.

Tiles have come to be a very fruitful tool in the study of crossing-critical graphs. Their fundamentals were laid out by Pinontoan and Richter \cite{tiles1}, and later they turned out to be a key in Bokal's solution of Salazar's question regarding average degrees in crossing-critical graphs \cite{avgdeg2,tiles2,avgdeg}. These results all rely on the ease of establishing the crossing number of a sufficiently large tiled graph, and they generated considerable interest in the reverse question: what is the true structure of crossing-critical graphs? How far from a tiled graph can a large crossing-critical graph be? Hlin\v en\'y's result about bounded path-width of $k$-crossing-critical graphs \cite{petr} establishes a rough structure, but is it possible that, for small values of $k$, tiles would describe the structure completely? It turns out that, for $k=2$, the answer is positive.   A more detailed discussion of these and other matters relating to crossing numbers can be found in the survey by Richter and Salazar \cite{survey}.

Our goal in this work, not quite achieved, is to classify all 2-crossing-critical graphs.   The bulk of our effort is devoted to showing that if $G$ is a 3-connected 2-crossing-critical graph that contains a subdivision of $V_{10}$, then $G$ is one of a completely described infinite family of 3-connected 2-crossing-critical graphs.    These graphs are all composed from 42 tiles.  This takes up Chapters \ref{sec:projPlane} -- \ref{sec:nextRed+Tiles}.    This combines with \cite{DOTV} to prove that a ``large" 3-connected 2-crossing-critical graph is a member of this infinite family.

The remainder of the classification would involve determining all 2-crossing-critical graphs that either are not 3-connected or are 3-connected and do not have a subdivision of $V_{10}$.  {In Chapter \ref{sec:not2conn},  we deal with the 2-crossing-critical graphs that are not 3-connected:  they are either one of a small number of known particular examples, or they are 2-connected and easily obtained from 3-connected examples. }

There remains the problem of determining the 3-connected 2-crossing-critical graphs that do not contain a subdivision of $V_{10}$. {In the first five sections of Chapter \ref{sec:3conNotI4c}, we explain how to completely determine all the 3-connected 2-crossing-critical graphs from {peripherally}-4-connected graphs that either have crossing number 1 or are themselves 2-crossing-critical.  In the sixth and final subsection, we determine which {peripherally}-4-connected graphs do not contain a subdivision of $V_8$ and  either have crossing number 1 or are themselves 2-crossing-critical.  Combining the two parts yields a definite (and practical) procedure for finding all the 3-connected 2-crossing-critical graphs that do not contain a subdivision of $V_8$.}    This leaves open the problem of classifying those that contain a subdivision of $V_8$ but do not have a subdivision of $V_{10}$.     In Sections \ref{sec:bridgesSmall} and \ref{sec:fewBridges}, we show that there are only finitely many.  (Although this follows from [11], the approach is different and it keeps our work self-contained.) 

There is hope for a complete description.  
In her master's essay, Urrutia-Schroeder \cite{isabel} begins the determination of precisely these graphs and finds 326 of them.   Oporowski (personal communication) had previously determined 531 3-connected 2-crossing-critical graphs, of which 201 contain a subdivision of $V_8$ but not of $V_{10}$.  {Austin \cite{bethann} improves on Urrutia-Schroeder's work, correcting a minor error (only 214 of Urrutia-Schroeder's graphs are actually 2-crossing-critical) and finding several others, for a total of 312 examples.  Only 8 of Oporowski's examples are not among the 312.  A few} have been determined by us as stepping stones in our classification of those that have a subdivision of $V_{10}$.    We have hopes of completing the classification.  

The principal facts that we prove in this work are summarized in the following statement.

\begin{theorem}[Classification of 2-crossing-critical graphs]\label{th:main}  Let $G$ be a 2-cros\-sing-critical graph with minimum degree at least 3.  Then either:
\begin{itemize}
\item if $G$ is 3-connected, then either $G$ has a subdivision of $V_{10}$ and a very particular tile structure or has at most 3 million vertices; or
\item $G$ is not 3-connected and is one of $49$ particular examples; or
\item $G$ is 2- but not 3-connected and is obtained from a 3-connected example by replacing digons by digonal paths. 
\end{itemize}
\end{theorem}

We remark again that vertices of degree 2 are uninteresting in the context of crossing-criticality, so we assume all graphs have minimum degree at least 3.  

 Chapters 2--13 of this work contain the proof of the following, which is the main contribution of this work.  (The formal definitions required for the statement given below are presented in Chapter 2.)

\begin{theorem}[2-crossing-critical graphs with $V_{10}$]  Let $G$ be a 3-connected, \2cc\ graph containing a subdivision of $V_{10}$.  Then $G$ is a twisted circular sequence $(T_1,T_2,\dots,T_n)$ of tiles, with each $T_i$ coming from a set of 42 possibilities.\end{theorem}

This is part of the first item in the statement of Theorem \ref{th:main}.

 Chapter 14 is devoted to \2cc\ graphs that are not 3-connected.  (We remind the reader of Tutte's theory of cleavage units and introduce digonal paths in Chapter 14.)  The results there are summarized in the following.
 
 \begin{theorem}[\2cc\ graphs with small cutsets]  Let $G$ be a  \2cc\ graph with minimum degree at least 3 that is not 3-connected.
 \begin{enumerate}  \item If $G$ is not 2-connected, then $G$ is one of 13 graphs.  (See Figure \ref{fg:1connected}.) 
 \item If $G$ is 2-connected and has two nonplanar cleavage units, then $G$ is one of 36 graphs.  (See Figures \ref{2units} and \ref{3units}.)
 \item If $G$ is 2-connected with at most one nonplanar cleavage unit, then $G$ has precisely one nonplanar cleavage unit and is obtained from a 3-connected, \2cc\ graph by replacing pairs of parallel edges by digonal paths.
 \end{enumerate}\end{theorem}
 
Chapter 15 shows how to reduce the determination of 3-connected 2-crossing-critical graphs to ``\p4c" 2-crossing-critical graphs.  A graph $G$ is {\em \p4c\/} if $G$ is 3-connected and, for every 3-cut $X$ in $G$, any partition of the components into nonnull subgraphs $H$ and $J$ has one of $H$ and $J$ being a single vertex. The main result here is the following.
 
 \begin{theorem}  Every 3-connected, \2cc\ graph is obtained from a \p4c, \2cc graph by replacing each degree 3 vertex with one of at most 20 different graphs, each having at most 6 vertices. \end{theorem}
 
  We combine this with Robertson's characterization of $V_8$-free graphs to explain how to determine all the 3-connected 2-crossing-critical graphs that do not have a subdivision of $V_8$.  This requires a further reduction to ``internally 4-connected" graphs.

   Chapter 16 shows that a 3-connected, 2-crossing-critical graph with a subdivision of $V_8$ but no subdivision of $V_{10}$ has at most three million vertices.  The general result we prove there is the following.
   
   \begin{theorem} Suppose $G$ is a 3-connected, \2cc\ graph. Let $n\ge3$ be such that $G$ has a subdivision of $V_{2n}$  but not of $V_{2(n+1)}$.
Then $|V(G)|  = O(n^3)$. 
\end{theorem}

\chapter{Description of 2-crossing-critical graphs with $V_{10}$}\label{descr}

In this section, we describe the structure of the $2$-crossing-critical graphs that contain $V_{10}$. As mentioned in the introduction, they are composed of tiles. This concept was first formalized by Pinontoan and Richter \cite {tiles1,tiles2} who studied large sequences of equal tiles. Bokal \cite {avgdeg2} extended their results to sequences of arbitrary tiles, which are required in this section. In those results,  ``perfect'' tiles were introduced to establish the crossing number of the constructed graphs. However, this property required a lower bound on the number of the tiles that is just slightly too restrictive to include all our graphs. As we are able to establish the lower bound on the crossing number of all these graphs in a different way (Theorem \ref{th:tiledAre2cc}), we summarize the concepts of \cite{avgdeg2} without reference to ``perfect'' tiles. Where the reader feels we are imprecise, please refer to \cite{avgdeg2} for details.

\begin{definition}\label{df:tile}  \begin{enumerate}\item A {\em tile\/}\index{tile} is a triple $T=(G,\lambda,\rho)$, consisting of a graph $G$ and two sequences $\lambda$ and $\rho$ of distinct vertices of $G$, with no vertex of $G$ appearing in both $\lambda$ and $\rho$. 
\item A {\em tile drawing\/}\index{tile drawing} is a drawing $D$ of $G$ in the unit square $[0,1]\times[0,1]$ for which the intersection of the boundary of the square with $D[G]$ contains precisely 
the images of the vertices of the {\em left wall\/} $\lambda$ and the {\em right wall\/} $\rho$, and these are drawn in $\{0\}\times [0,1]$ and $\{1\}\times [0,1]$, respectively, such that the $y$-coordinates of the vertices are increasing with respect to their orders in the sequences $\lambda$ and $\rho$. 
\item The {\em tile crossing number\/}\index{tile!crossing number} $\tcrn(T)$\index{$\tcrn$} of a tile $T$ is the smallest number of crossings in a tile drawing of $T$. 
\item The tile $T$ is {\em planar\/} if $\tcrn(T)=0$.
\item {A {\em $k$-drawing\/}\index{1-drawing}\index{$k$-drawing} of a graph or a {\em $k$-tile-drawing\/} of a tile is a drawing or tile-drawing, respectively, with at most $k$ crossings.}
\end{enumerate}
\end{definition}

It is a central point for us that tiles may be ``glued together" to form larger tiles.   We formalize this as follows.

\begin{definition}\label{df:join}  \begin{enumerate}\item The tiles $T=(G,\lambda,\rho)$ and $T'=(G',\lambda',\rho')$  are {\em compatible\/}\index{tile!compatible} if $|\rho|=|\lambda'|$. 
\item A sequence $(T_0,T_1,\dots,T_m)$ of tiles is {\em compatible\/} if, for each $i=1,2,\dots,m$, $T_{i-1}$ is compatible with $T_i$.
\item The {\em join\/}\index{tile!join} of compatible tiles  $(G,\lambda,\rho)$ and $(G',\lambda',\rho')$ is the tile $(G,\lambda,\rho)\otimes(G',\lambda',\rho')$ whose graph is obtained from $G$ and $G'$ by identifying the sequence $\rho$ term by term with the sequence $\lambda'$;  left wall is $\lambda$; and  right wall is $\rho'$.
\item As $\otimes$\index{$\otimes$} is associative,  the join $\otimes \cT$ of a compatible sequence $\cT=(T_0,T_1,\dots,$ $T_m)$ of tiles is well-defined as $T_0\otimes T_1\otimes \cdots \otimes T_m$.
\end{enumerate}
\end{definition}{

Note that identifying wall vertices in a join may introduce either multiple edges or vertices of degree two. If we are interested in 3-connected graphs, we may suppress vertices of degree two, but we keep the multiple edges. 

We have the following simple observation.

}\begin{observation}\label{obs:otimes}  Let $(T_0,T_1,\dots,T_m)$ be a compatible sequence $\cT$ of tiles.  Then

\hskip 2truein${\displaystyle\tcrn(\otimes \cT)\le \sum_{i=0}^m\tcrn(T_i)}$. \hfill{\eop}
\end{observation}{

\def\cyclization{cyclization}

An important operation on tiles that we need converts a tile into a graph.

}\begin{definition}\label{df:cyclization}\begin{enumerate}\item A tile $T$ is {\em cyclically compatible\/} if $T$ is compatible with itself.
\item   For a cyclically-compatible tile $T$, the \DEF{cyclization}\index{tile!cyclization} of $T$ is the graph $\circ T$ obtained by identifying the respective vertices of the left wall with the right wall. A cyclization of a cyclically-compatible sequence of tiles is defined as $\circ \cT=\circ (\otimes \cT)$. 
\end{enumerate}
\end{definition}{

The following useful observation is easy to prove.  Typically, we will apply this to the tile $\otimes \cT$ obtained from a compatible sequence $\cT$ of tiles.

}\begin{lemma}[\cite{avgdeg2,tiles2}]\label{lm:join}
Let $T$ be a cyclically compatible tile. Then $\crn(\circ T) \le \tcrn(T)$. \hfill{\eop} \end{lemma}{ 

\def\updown{{\mathstrut}^{\updownarrow}}

We now describe various operations that turn one tile into another.

}\begin{definition}\label{df:inverting}
\begin{enumerate}
\item
For a sequence $\omega$, $\bar\omega$ denotes the reversed sequence. 
\item 
\begin{itemize} 
\item 
The \DEF{right-inverted} tile of a tile $T=(G,\lambda,\rho)$ is the tile $T{\updown}=(G,\lambda,\bar\rho)$; \item the  \DEF{left-inverted} tile is ${\updown}T=(G,\bar\lambda,\rho)$; \item the \DEF{inverted} tile is ${\updown}T{\updown}=(G,\bar\lambda,\bar\rho)$;  and 
\item the \DEF{reversed} tile is $T^{\leftrightarrow}=(G,\rho,\lambda)$. {($T^{\leftrightarrow}$ made an item.)}
\end{itemize}
\item A tile $T$ is \DEF{$k$-degenerate}\index{tile!$k$-degenerate} if $T$ is planar and, for every edge $e$ of $T$, \newline $\tcrn(T{\updown}-e)< k$. 
\end{enumerate}
\end{definition}{

Note that our $k$-degenerate tiles are not {necessarily} perfect{, as opposed to the definition in \cite{avgdeg2}}. However, the following analogue of \cite[Cor.~8]{avgdeg2} is still true.

}\begin{lemma} \label{lm:ksequence}
Let $\cT=(T_0,\ldots,T_m)$, $m\ge 0$, be a cyclically-compatible sequence of $k$-degenerate tiles. Then $\otimes (\cT)$ is a $k$-degenerate tile. 
\end{lemma}{
\begin{proof}
By Lemma \ref{lm:join}, $\otimes\cT$ is planar.   Let  $e$ be any edge of $\otimes\cT$. Let $T_i$ be the tile of $\cT$ containing $e$.  \wording{Let}  $\cT'=(T_0,\ldots,T_{i-1},T_i{\updown}-e,{\mathstrut}{\updown}T_{i+1}{\updown},\ldots,{\mathstrut}{\updown}T_m{\updown})${, so $\otimes \cT'=\otimes\cT{\updown}-e$; in particular, they} have the same tile crossing number. As $T_i{\updown}$ is $k$-degenerate, $\tcrn(T_i{\updown}-e)<k$. Since all other tiles of $\cT'$ are planar, Lemma \ref{lm:join} implies $\tcrn(\otimes\cT{\updown}-e)\le\tcrn(T_i{\updown}-e)<k$.
\end{proof}

The following is an obvious corollary. 

}\begin{corollary}\label{co:ktiles}
Let $T$ be a $k$-degenerate tile so that $\crn(\circ (T{\updown}))\ge k$. Then $\circ (T{\updown})$ is a $k$-crossing-critical graph. \hfill\eop
\end{corollary}{

}\begin{definition}\label{df:reverse}
\begin{enumerate}
\item  $\cT$ is a compatible sequence $(T_0,T_1,\ldots,T_m)$,
~then:
\begin{itemize}
\item  the \DEF{reversed sequence} $\cT^{\leftrightarrow}$ is the sequence {$(T_m^{\leftrightarrow},T_{m-1}^{\leftrightarrow},\ldots,T_0^{\leftrightarrow})$};
\item the \DEF{$i$-flip} $\cT^{i}$ is the sequence $(T_0,\ldots,T_i{\updown},{}{\updown}T_{i+1},T_{i+2}\ldots,T_m)$; and \item the \DEF{$i$-shift} $\cT_{i}$ is the sequence $(T_i,\ldots,T_m,T_0,\ldots,T_{i+1})$.
\end{itemize}
\item Two sequences of tiles are {\em equivalent\/} if one can be obtained from the other by a series of shifts, flips, and reversals. 
\end{enumerate}
\end{definition}{

Note that the \cyclization s of two equivalent sequences of tiles are the same graph.

}\begin{definition}\label{df:theTiles}  The set $\cS$\index{$\cS$} of tiles consists of those tiles obtained as combinations of two \DEF{frames}, illustrated in  Figure \ref{fg:frames}, and 13 \DEF{pictures}, shown in Figure \ref{fg:pictures}, in such a way, that a picture is inserted into a frame by identifying the two squares.  A given picture may be inserted into a frame either with the given orientation or with a $180^\circ$ rotation \wording{(some examples are given in Figure \ref{fg:cStiles})}.  
\end{definition}{

\begin{figure}[!ht]
\begin{center}
\includegraphics[scale=.8]{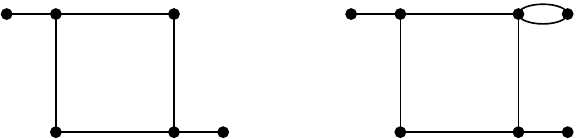}
\caption{The two frames.}\label{fg:frames}
\end{center}
\end{figure}

\begin{figure}[!ht]
\begin{center}
\includegraphics[scale=.8]{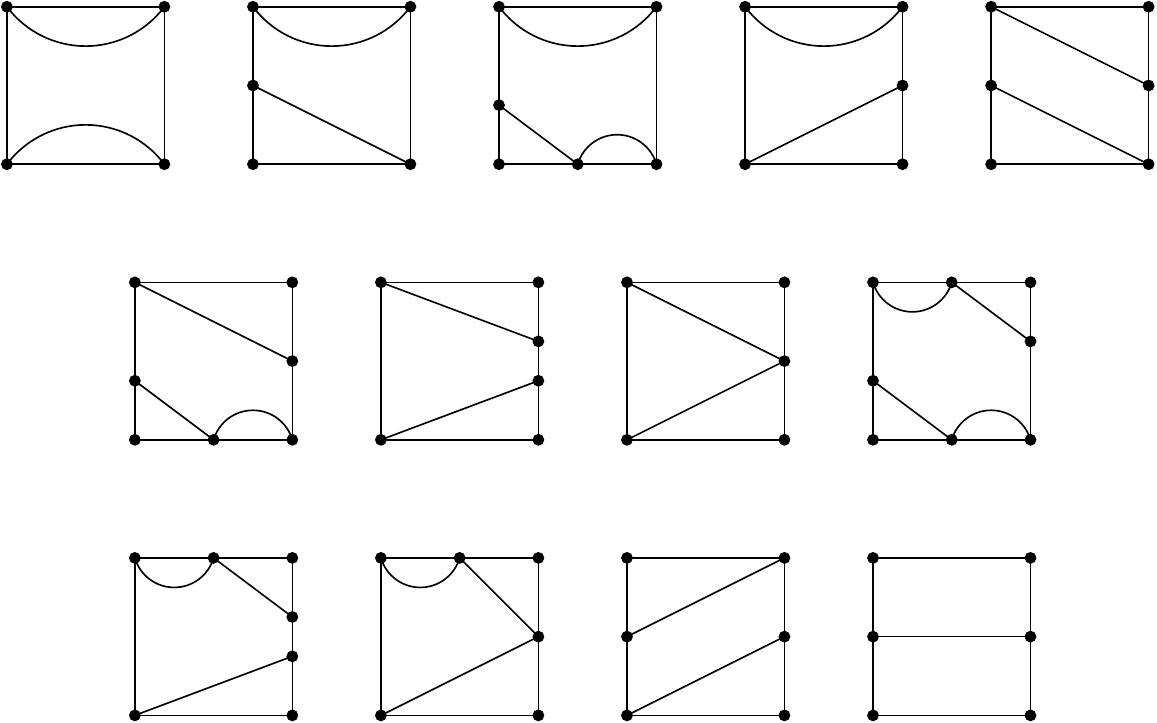}
\caption{The thirteen pictures.}\label{fg:pictures}
\end{center}
\end{figure}

We remark that each picture produces either two or four tiles in $\cS$; see Figure \ref{fg:cStiles}

\begin{figure}[!ht]
\begin{center}
\includegraphics[scale=.8]{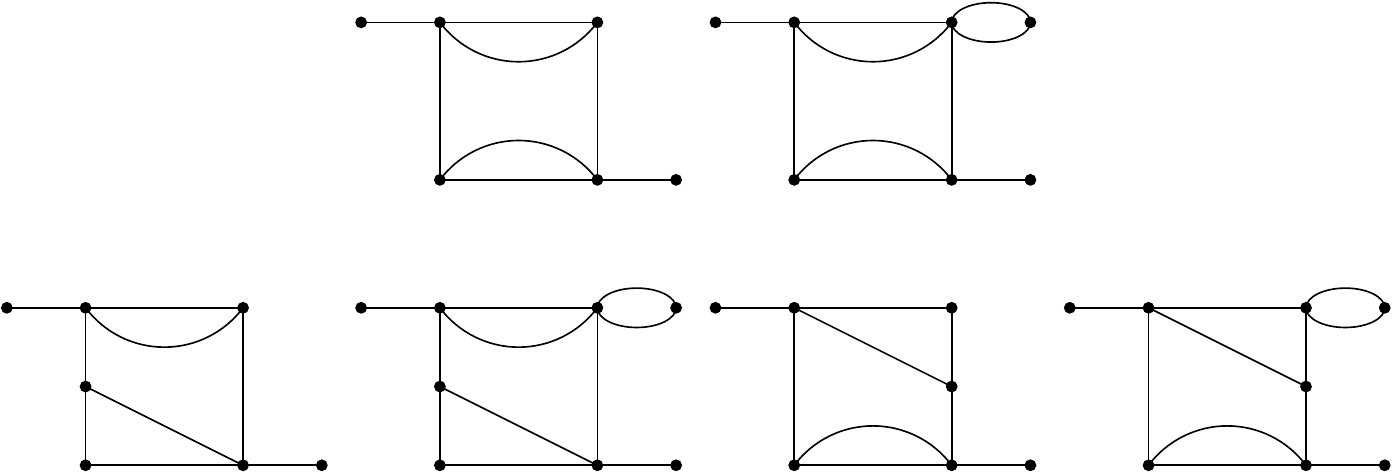}
\caption{Each picture produces either two or four tiles.}\label{fg:cStiles}
\end{center}
\end{figure}
}\begin{lemma}\label{lm:2degenerate}
Let $T$ be a tile in the set $\cS$. Then both $T$ and ${\updown}T_{i}{\updown}$ are $2$-degenerate.
\end{lemma}{
\begin{proof}
Figure \ref{fg:tilesInS} shows that all the tiles are planar. The claim for $T$ implies the result for ${\updown}T_{i}{\updown}$, so it is enough to prove the result for an arbitrary $T\in \cS$. Let $e$ be an arbitrary edge of $T$. We consider cases, depending on whether $e$ is either dotted, thin solid, thick solid, thin dashed, or thick dashed in Figure \ref{fg:tilesInS}. Using this classification, we argue that $\tcrn(T-e)<2$.

\begin{figure}[!ht]
\begin{center}
\includegraphics[scale=.8]{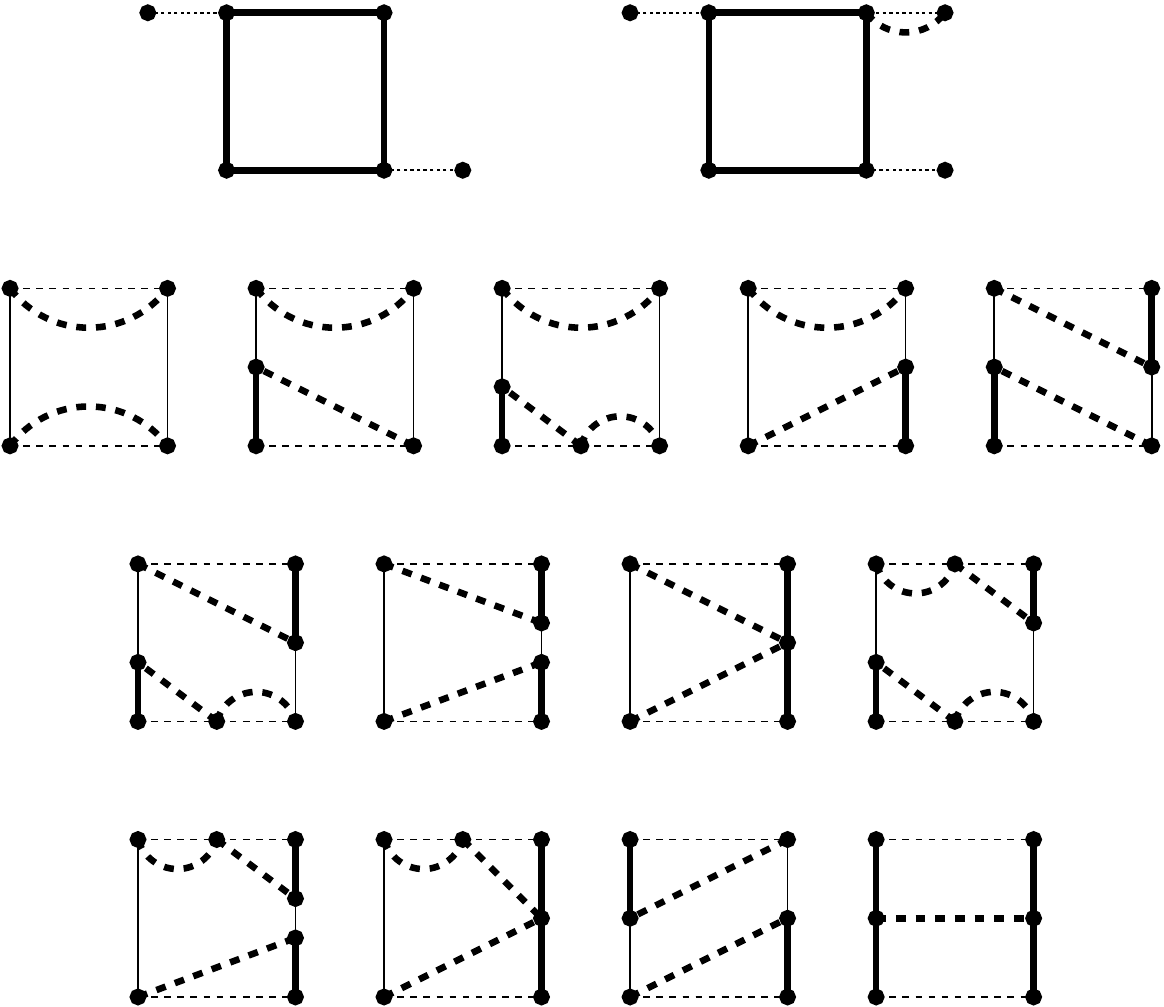}
\caption{The different kinds of edges in the pictures.}\label{fg:tilesInS}
\end{center}
\end{figure}

If $e$ is a dotted edge, then $T-e$ has a wall with a single vertex and $\tcrn(T{\updown}-e)=0$. 

If $e$ is a thin solid edge, then there is a $1$-tile-drawing of $T{\updown}$ with two dotted edges of $T$ crossing each other.

If $e$ is a thick solid edge, then there is a unique thin dashed edge $f$ adjacent to $e$, and there exists a 1-tile-drawing of $T{\updown}-e$ with $f$ crossing the dotted edge not on the same horizontal side of $T$ as $f$.

If $e$ is a thin dashed edge, then there is a unique thick dashed edge $e'$ such that $e$ and $e'$ are in the same face of the exhibited planar drawing of $T$, as well as a unique dotted edge $f$, that is not in the same horizontal side of $T$ as $e$. For such $e$ and $e'$, there exists a $1$-tile-drawing of $T{\updown}-e$ with $e'$ crossing $f$, as well as a $1$-tile-drawing of $T{\updown}-e'$ with $e$ crossing $f$. As each thick dashed edge corresponds to at least one thin dashed edge, this concludes the proof.
 \end{proof}

\def\tileS{\mathcal T(\cS)}

We now define the set of graphs that is central to this work.  

}\begin{definition} \label{df:t(s)}  The set $\tileS$\index{$\tileS$}\index{$\tileS$} consists of all graphs of the form {$\circ((\otimes \mathcal T)\updown)$}, where $\mathcal T$ is a sequence $(T_0,{\updown}T\updown_1,T_2,$ $\dots,{\updown}T\updown_{2m-1},T_{2m})$ so that $m\ge 1$ and, for each $i=0,1,2,\dots,2m$, $T_i\in \cS$.

\wording{The {\em rim}\index{rim} of an element of $\tileS$ is the cycle $R$ that consists of the top and bottom horizontal path in each frame (including the part that sticks out to either side) and, if there is a parallel pair in the frame, one of the two edges of the parallel pair.}
\end{definition}{

The following is an immediate consequence of Lemmas \ref{lm:ksequence} and \ref{lm:2degenerate}.  

}\begin{corollary}\label{co:t(s)2degen}  Let $G\in\tileS$.  For every edge $e$ of $G$, $\kr(G-e)<2$.  \hfill\eop\end{corollary}{
 
In Theorem \ref{th:tiledAre2cc}, we complete the proof that each graph $G$ in $\tileS$ is 2-crossing-critical by proving there that $\kr(G)\ge 2$.

We are now able to state the central result of this work.

}\begin{theorem}\label{th:classification}
 If $G$ is a $3$-connected $2$-crossing-critical graph containing a subdivision of $V_{10}$, then $G\in\tileS$.
\end{theorem}{

This theorem is proved in the course of Chapters \ref{sec:projPlane} -- \ref{sec:nextRed+Tiles}.
}
{We remark that not every graph in $\mathcal T(\mathcal S)$ contains a subdivision of $V_{10}$.}
\chapter{Moving into the projective plane}\printFullDetails{\label{sec:projPlane}

It turns out that considering the relation of a 2-crossing-critical graph to its embeddability in the projective plane is useful.  This perspective was employed by Richter to determine all eight cubic 2-crossing-critical graphs \cite{rbr}.  It is a triviality that, if $G$ has a 1-drawing, then $G$ embeds in the projective plane (put the crosscap on the crossing).  Therefore, any graph $G$ that does not embed in the projective plane has crossing number at least 2.  Moreover, Archdeacon \cite{danThesis,danJGT} proved that it contains one of the 103 graphs that do not embed in the projective plane but every proper subgraph does. Each obstruction for projective planar embedding has crossing number at least 2.  Of these, only the ones in Figure \ref{fig:2cc-nonPP} are 3-connected and 2-crossing-critical.  (The non-projective planar graphs that are not 3-connected are found by different \wording{means in Section \ref{sec:not2conn}}.)  \bogdan{These are the ones labelled --- left to right, top to bottom --- D17,  E20,  E22,  E23,
E26,  F4,   F5,   F10,
 F12,  F13,  and G1 in Glover, Huneke, and Wang \cite{ghw}.}

\begin{figure}[!ht]
\begin{center}
\includegraphics[scale=.8]{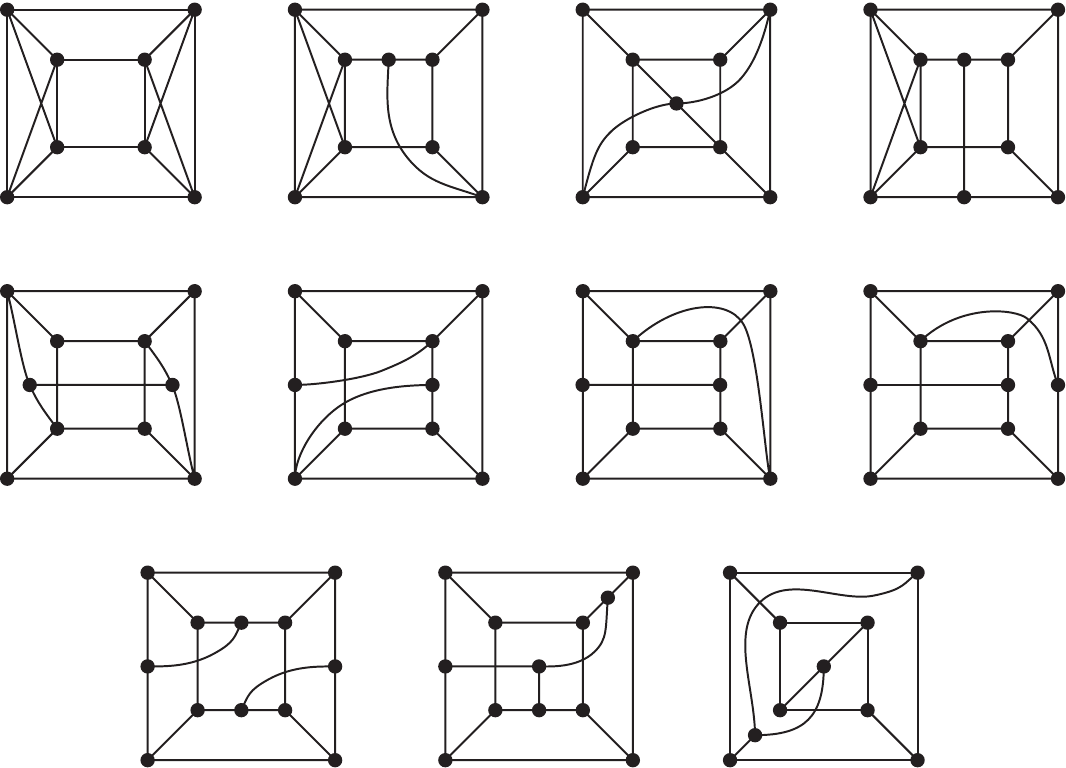}
\caption{\bogdan{The 3-connected, 2-crossing-critical graphs that do not embed in $\pp$.}}\label{fig:2cc-nonPP}
\end{center}
\end{figure}

}\begin{definition} Let $G$ be a graph embedded in a (compact, connected) surface $\Sigma$.  Then:
\begin{enumerate}
\item  the {\em representativity\/}\index{representativity} $\rep(G)$ of $G$ is the largest integer $n$ so that every non-contractible, simple, closed curve in $\Sigma$ intersects $G$ in at least $n$ points (this parameter is undefined when $\Sigma$ is the sphere);  
\item $G$ is {\em $n$-representative\/} if $n\ge r(G)$;                     
\item $G$ is {\em embedded with representativity $n$\/} if $\rep(G)=n$.
\end{enumerate}
\end{definition}\printFullDetails{

Representativity is also known as {\em face-width\/} and gained notoriety in the Graph Minors project of Robertson and Seymour.  We only require very elementary aspects of this parameter; the reader is invited to consult \cite{diestel} or \cite{mt} for further information on representativity and Graph Minors. 

Barnette \cite{barnette} and  Vitray \cite{vitray} independently proved that every 3-representative embedding in the projective plane topologically contains one of the 15 graphs \bogdan{(\cite[Figure 2.2]{vitray})}.   Vitray pointed out in a conference talk \cite{OSUtalk} that each of these 15 graphs has crossing number at least 2.  Therefore, any graph that has a 3-representative embedding in the projective plane has crossing number at least 2.  One immediate conclusion is that there are only finitely many \2cc graphs that embed in $\pp$ and do not have a representativity at most 2 embedding in $\pp$, and, not only are there only finitely many of these, but they are all known and \wording{are shown in Figure \ref{fig:3repPP}}.    \wording{ Vitray went on to show} that the only \2cc graph whose crossing number is not equal to 2 is $C_3\Box C_3$, whose crossing number is 3.

\ignore{\begin{figure}[!ht]
\begin{center}
\includegraphics[scale=.4]{Vitray}
\caption{The \dragominor{minimal} 3-representative embeddings in $\pp$.}\label{fig:3repPP}
\end{center}
\end{figure}}

\begin{figure}[!ht]
\begin{center}
\includegraphics[scale=.8]{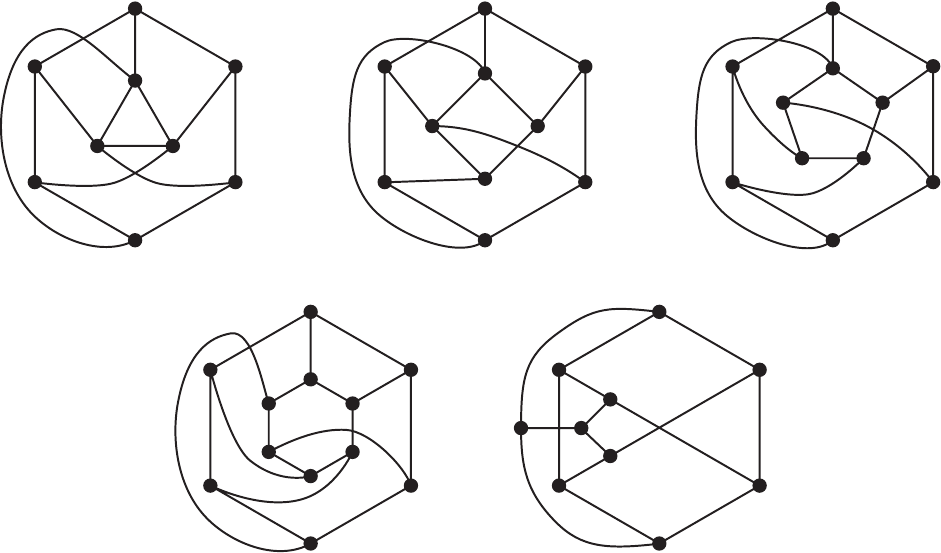}
\caption{The 2-crossing-critical 3-representative embeddings in $\pp$.}\label{fig:3repPP}
\end{center}
\end{figure}

Since every graph that has an embedding in the projective plane with representativity at most 1 is planar, 
it remains to explore those \2cc graphs that have an embedding in $\pp$ with representativity precisely 2.  To cement some terminology and notation, we have the following.

}\begin{definition}  Let $n\ge 3$ be an integer.  The graph $V_{2n}$\index{$V_{2n}$} is the {\em M\"obius ladder\/}\index{M\"obius ladder} consisting of:
\begin{itemize}\item the {\em rim\/}\index{rim}\index{M\"obius ladder!rim} $R$ of $V_{2n}$, which is a $2n$-cycle $(v_0,v_1,v_2,\dots,v_{2n-1},v_0)$; and,
\item for $i=0,1,2,\dots,n-1$, {\em the spoke\/}\index{spoke}\index{M\"obius ladder!spoke} $v_iv_{n+i}$.
\end{itemize}
Suppose $\hvng$.  (The notation $L\topol H$\index{$L\topol H$} means that $H$ is a subdivision of $L$.  Thus, $\hvng$ means $H$ is a subgraph of $G$ \wording{and is also} a subdivision of $V_{2n}$.)
\begin{itemize}\item The {\em $H$-nodes\/}\index{node}\index{$H$-node} are the vertices of $H$ corresponding to $v_0$, $v_1$, \dots, $v_{2n-1}$ in $V_{2n}$; the $H$-nodes are also labelled $v_0$, $v_1$, \dots, $v_{2n-1}$.
\item For $i=0,1,2,\dots,2n-1$, the {\em $H$-rim branch\/}\index{rim branch}\index{M\"obius ladder!rim branch} $r_i$ is the path in $H$ corresponding to the edge $v_iv_{i+1}$ of $V_{2n}$.
\item For $i=0,1,2,\dots,n-1$, the {\em $H$-spoke\/}\index{M\"obius ladder!$H$-spoke} is the path $s_i$ in $H$ corresponding to the edge $v_iv_{n+i}$ in $V_{2n}$.
\item We also use {\em $H$-rim\/}\index{rim}\index{$H$-rim}\index{M\"obius ladder!$H$-rim} and $R$\index{$R$} for the cycle in $H$ corresponding to the rim of $V_{2n}$.
\end{itemize}
\end{definition}\printFullDetails{

Whenever we discuss elements of a subdivision $H$ of \wording{the} M\"obius ladder $V_{2n}$, we presume the indices are read appropriately.  For the $H$-nodes $v_k$ and the $H$-rim branches $r_k$, the index $k$ is to be read modulo $2n$.  For the $H$-spokes $s_\ell$, the index $\ell$ is to be read modulo $n$.  Thus, for example, $s_{5+n}=s_5$ and $v_{8+2n}=v_8$, while $r_{8+n}\ne r_8$.

Let $G$ be a \2cc\ graph embedded in $\pp$ with representativity 2.  Let $\gamma$\index{$\gamma$} be a simple closed curve in $\pp$ meeting $G$ in precisely the two points $a$\index{$a$} and $b$\index{$b$}.  We further assume $\hvng$, with $n\ge 3$.  
Because $G-a$ and $G-b$ have 1-representative embeddings in the projective plane, they are both planar.  We note that, for $n\ge 3$,  $V_{2n}$ is not planar; therefore, $a,b\in H$. 

\begin{remark} \dragominor{Throughout this work, we abuse notation slightly.  If $K$ is any graph and $x$ is either a vertex or an edge of $K$, then we write $x\in K$, rather than the technically correct $x\in V(K)$ or $x\in E(K)$.  We have taken care so that, in any instance, the reader will never be in doubt about whether $x$ is a vertex or an edge.}\end{remark} 

If $n\ge 4$, the deletion of a spoke of $V_{2n}$ leaves a non-planar subgraph; thus, when $n\ge 4$, we conclude $a,b\in R$.  If $\gamma$ does not cross $R$ at $a$, say, then deleting the $H$-spoke incident with $a$ (if there is one), and shifting $\gamma$ away from $a$ leaves a subdivision of $K_{3,3}$ in $\pp$ that meets the adjusted $\gamma$ only at $b$.  But then this $K_{3,3}$ has a 1-representative embedding in $\pp$, showing $K_{3,3}$ is planar, a contradiction.  Therefore,  $\gamma$ must cross $R$ at $a$ and $b$.  
As any two non-contractible curves cross an odd number of times, $R$ is contractible and so  bounds a closed disc $\Disc$\index{$\Disc$} and a closed M\"obius strip $\Mob$\index{$\Mob$}.

Let $P$ and $Q$ be the two $ab$-subpaths of $R$, let $\alpha=\gamma\cap \Disc$\index{$\alpha$} and $\beta=\gamma\cap \Mob$\index{$\beta$}.  (We alert the reader that the notations $\Disc$, $\Mob$, $\alpha$, $\beta$, and $\gamma$ will be reserved for these objects.)
Since each spoke is internally disjoint from $\gamma$, the spoke is either contained in $\Disc$ or contained in $\Mob$.  Since the spokes interlace on $R$, at most one can be embedded in $\Disc$.  

Moreover, observe that $\alpha$ divides $\Disc$ into two regions, one bounded by $P\cup \alpha$ and the other bounded by $Q\cup \alpha$.  Thus, if a spoke --- label it $s_0$ --- is embedded in $\Disc$, then $s_0$ has both attachments in just one of $P$ and $Q$, say $P$.  In this case, $P$ contains either all the $H$-nodes $v_0,v_1,\dots,v_n$ or all the $H$-nodes $v_n,v_{n+1},\dots,v_{2n-1},v_0$.   It follows that, for $n\ge 4$, there are only two (up to relabelling) representativity 2 embeddings of $V_{2n}$ in the projective plane.  See Figure \ref{fig:Rep2Vten}\index{$V_{2n}$!embeddings}.  We remark that it is possible that one or both of $a$ and $b$ might be an $H$-node.

\begin{figure}[htb!]
\begin{center}
	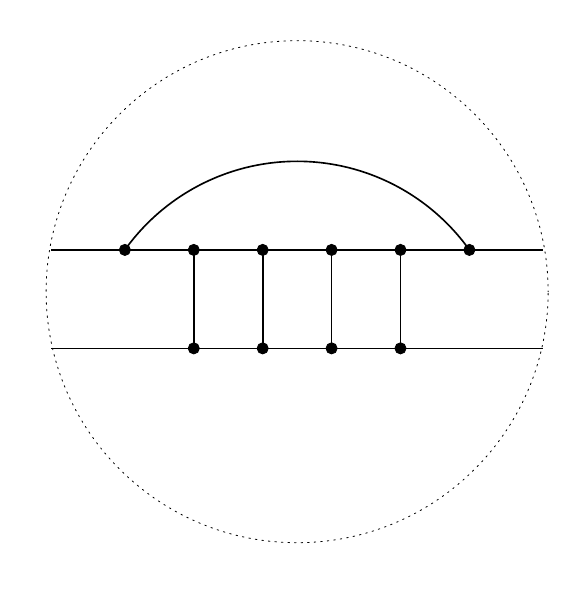\hglue 1 cm 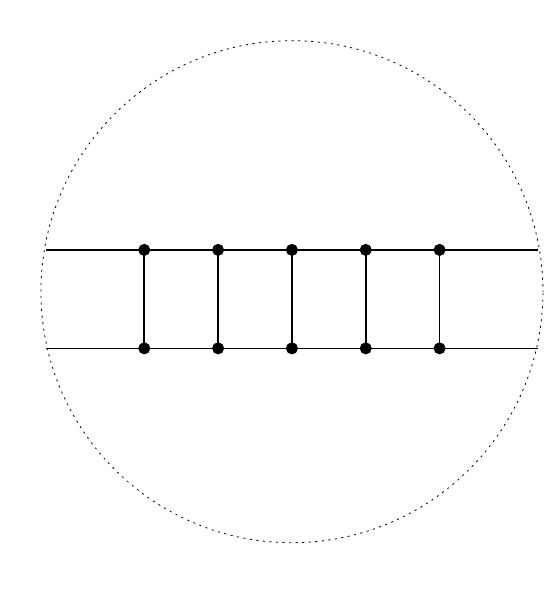
	\caption{\dragominor{Standard labellings of the r}epresentativity $2$ embeddings of $V_{10}$.}
\label{fig:Rep2Vten}
\end{center}
\end{figure}

\ignore{\begin{figure}[!ht]
\begin{center}
\includegraphics[scale=.75]{example1}
\caption{Representativity 2 embeddings of $V_{10}$.}\label{fig:Rep2Vten}
\end{center}
\end{figure}}

We introduce a notation that will be used extensively in this work.

}\begin{definition}  The set of 3-connected, 2-crossing-critical graphs is denoted $\m2$\index{$\m2$}.
\end{definition}\printFullDetails{

It is a tedious \wording{(and unimportant)} exercise to check the observation that none of the  graphs in $\m2$ found among the obstructions to having a representativity 2 embedding in $\pp$ has a subdivision of $V_{10}$.   We record it in the following assertion.\wordingrem{(Text removed here.)}

}\begin{theorem}\label{th:v10rep2}  Let $G\in \m2$ and $\hvfg$.  Then $G$ has a representativity 2 embedding in $\pp$. \hfill\eop \end{theorem}\printFullDetails{

We will also need information about 1-drawings of $V_{2n}$, for $n\ge 4$.  These are similarly straightforward facts that can be proved by considering $K_{3,3}$'s in $V_{2n}$.

}\begin{lemma}\label{lm:1drawingsV2n}  Let $n\ge 4$ and let $D$ be a 1-drawing of $V_{2n}$.  Then there is an $i$ so that $r_i$ crosses one of $r_{i+n-1}$, $r_{i+n}$, and $r_{i+n+1}$.  \hfill\eop\end{lemma}\printFullDetails{
}

\chapter{Bridges}\label{bridges}\printFullDetails{

The notion of a bridge of a subgraph of a graph is a valuable tool that allows us to organize many aspects of this work.  This section is devoted to their definition and an elucidation of their properties that are relevant to us.  Bridges are discussed at length in \cite{TutEncy} and, under the name $J$-components, in \cite{TutCinG}.  

}\begin{definition}\label{df:bridge} Let $G$ be a graph and let $H$ be a subgraph of $G$. 
 \begin{enumerate}
 \item\label{it:isolated} For a set $W$ of vertices of $G$, $\iso W$\index{$\iso W$} consists of the subgraph of $G$ with vertex set $W$ and no edges.
\item An {\em $H$-bridge in $G$\/}\index{bridge}\index{$H$-bridge} is a subgraph $B$ of $G$ such that either $B$ is an edge not in $H$, together with its ends, both of which are in $H$, or $B$ is obtained from a component $K$ of $G-V(H)$ by adding to $K$ all the edges from vertices in $K$ to vertices in $H$, along with their ends in $H$.  
\item For an $H$-bridge $B$ in $G$, a vertex $u$ of $B$ is an {\em attachment of $B$\/}\index{attachment}\index{bridge!attachment} \wording{if $u\in V(H)$; $\att(B)$\index{$att(B)$} denotes} the set of attachments of $B$. 
\item If $B$ is an $H$-bridge, then the {\em nucleus $\Nuc(B)$\index{$\Nuc(B)$} of $B$\/}\index{nucleus}\index{bridge!nucleus} is $B-\att(B)$.
\item For $u,v\in V(G)$, a $uv$-path $P$ in $G$ is {\em $H$-avoiding\/}\index{avoiding}\index{$H$-avoiding} if $P\cap H\subseteq \isouv$. 
 \item  \minor{Let $A$ and $B$ be either subsets of $V(G)$ or subgraphs of $G$.  An {\em $AB$-path\/}\index{$AB$-path}\index{path}\index{path!$AB$-path} is a path with an end in each of $A$ and $B$ but otherwise disjoint from $A\cup B$.  If, for example, $A$ is the single vertex $u$, we write $uB$-path for $\{u\}B$-path.}
 
\end{enumerate}
\end{definition}\printFullDetails{

We will be especially interested in the bridges of a cycle.

}\begin{definition}  Let $C$ be a cycle in a graph $G$ and let $B$ and $B'$ be distinct $C$-bridges.
\begin{enumerate}
\item The {\em residual arcs\/}\index{residual arc}\index{bridge!residual arc} of $B$ in $C$ are the $B$-bridges in $C\cup B$; if $B$ has at least two attachments, then these are \wording{the maximal $B$-avoiding subpaths of $C$.}
\item  The $C$-bridges $B$ and $B'$ {\em do not overlap\/}\index{overlap}\index{bridge!overlap} if all the attachments of $B$ are in the same residual arc of $B'$; otherwise, they {\em overlap\/}.
\item The {\em overlap diagram\/}\index{overlap diagram}\index{bridge!overlap diagram} $OD(C)$\index{$OD(C)$} of $C$ has as its vertices the $C$-bridges; two $C$-bridges are adjacent in $OD(C)$ precisely when they overlap.
\item The cycle $C$ has \wording{{\em bipartite overlap diagram}\index{bipartite overlap diagram}\index{overlap diagram!bipartite}\index{bridge!bipartite overlap diagram}, denoted} {\em BOD\/}\index{BOD}\index{$BOD$}\wording{,} if $OD(C)$ is bipartite; otherwise, $C$ has \wording{{\em non-bipartite overlap diagram}, denoted} {\em NBOD\/}\index{NBOD}\index{$NBOD$}.
\end{enumerate}
\end{definition}\printFullDetails{

The following is easy to see and well-known.

}\begin{lemma}\label{lm:overlapClaw}  Let $C$ be a cycle in a graph $G$.
The distinct $C$-bridges $B$ and $B'$ overlap if and only if either:
\begin{enumerate}\item\label{it:skew} there are attachments $u,v$ of $B$ and $u',v'$ of $B'$ so that the vertices $u,u',v,v'$ are distinct and occur in this order in $C$ (in which case $B$ and $B'$ are {\em skew $C$-bridges\/}\index{skew bridges}\index{bridge!skew}); or
\item\label{it:equivalent} $\att(B)=\att(B')$ and $|\att(B)|=3$ (in which case $B$ and $B'$ are {\em 3-equivalent\/}\index{equivalent}\index{3-equivalent bridges}\index{bridge!equivalent}).\hfill\eop
\end{enumerate}
\end{lemma}\printFullDetails{

The following concept plays a central role through the next few sections of this work.

}\begin{definition}  Let $C$ be a cycle in a graph $G$ and let $B$ be a $C$-bridge.  Then $B$ is a {\em planar $C$-bridge\/}\index{planar $C$-bridge}\index{bridge!planar $C$-bridge} if $C\cup B$ is planar. Otherwise, $B$ is a {\em non-planar $C$-bridge\/}\index{non-planar $C$-bridge}.  \end{definition}\printFullDetails{

\dragominor{Note} that there is a difference between $C\cup B$ being planar and, in some embedding of $G$ in $\pp$, $C\cup B$ being {\em plane\/}, that is, embedded in some closed disc in $\pp$.   If $C\cup B$ is plane, then $B$ is planar, but the converse need not hold.

\wording{We now present the} major embedding and drawing results that we shall use.  The \dragominor{theorem is due to Tutte}, while the corollary is the form that we shall frequently use.

}\begin{theorem}\label{th:TutteOne}\cite[Theorems XI.48 and XI.49]{TutEncy}  Let $G$ be a graph.
\begin{enumerate}
\item\label{it:TutteOneA} $G$ is planar if and only if \minor{either $G$ is a forest or} there is a  cycle $C$ of $G$ having BOD and all $C$-bridges planar.
\item\label{it:TutteOneB} $G$ is planar if and only if, for every cycle $C$ of $G$, $C$ has BOD. \hfill\eop
\end{enumerate}
\end{theorem}\printFullDetails{

For the corollary, we need the following important notion.

}\begin{definition}  Let $H$ be a subgraph of a graph $G$ and let $D$ be a drawing of $G$ in the plane. Then {\em 
$H$ is clean\index{clean} in $D$\/} if no edge of $H$ is crossed in $D$.\minorrem{(text removed)}
\end{definition}\printFullDetails{

}\begin{corollary}\label{co:TutteTwo}  Let $G$ be a graph and let $C$ be a cycle with BOD.  If there is a $C$-bridge $B$ so that every other $C$-bridge is planar and there is a 1-drawing of $C\cup B$ in which $C$ is clean, then $\ucr(G)\le 1$. \end{corollary}\printFullDetails{

\begin{cproof}  Let $\times$ denote the crossing in a 1-drawing $D$ of $C\cup B$ in which $C$ is clean.  As $C$ is not crossed in $D$, $\times$ is a crossing of two edges of $B$.  Let $G^\times$ denote the graph obtained from $G$ by deleting those two edges and adding a new vertex adjacent to the four ends of the deleted edges.  Then $C$ has BOD in $G^\times$ and every $C$-bridge in $G^\times$ is planar.  By \wording{Theorem \ref{th:TutteOne} (\ref{it:TutteOneB})}, $G^\times$ is planar.  Any planar embedding of $G^\times$ easily converts to a 1-drawing of $G$.  \end{cproof}

We will also need the following result.

}\begin{lemma}[Ordering Lemma]\label{lm:orderingLemma}
Let $G$ be a graph, $C$ a cycle in $G$, $\bbb$ a set of
non-overlapping $C$-bridges. 
Let $P$ and $Q$ be disjoint paths in $C$, with $V(C) = V(P\cup
Q)$. Suppose that each 
$B\in\bbb$ has at least one attachment in each of $P$ and
$Q$. Let $P_B$ and $Q_B$ be the minimal subpaths of $P$ and $Q$,
respectively, containing $P\cap B$ and $Q\cap B$, respectively. Then:
\begin{enumerate}
\item the $\{P_B\}$ and $\{Q_B\}$ are pairwise internally disjoint and there
is an ordering $$(B_1,\ldots,B_k)$$ of $\bbb$ so that $$P=P_{B_1}\ldots P_{B_2}\ldots P_{B_i}
\ldots P_{B_k}\qquad \textrm{and}\qquad Q=Q_{B_1}\ldots Q_{B_2}\ldots Q_{B_i} \ldots Q_{B_k}\,;$$ and
\item if,
for each $B, B' \in \bbb$, 
$\att(B) \neq \att(B')$, 
the order is
unique up to inversion.
\end{enumerate}
\end{lemma}\printFullDetails{

\begin{cproof}
Suppose $B, B' \in \bbb$ are such that $P_B$ and $P_{B'}$ have a
common edge $e$. Then $B$ and $ B'$ have attachments $x_1, x_2, x_1', x_2'$
in both components of $P-e$ and attachments $x,x'$ in $Q$. If 
$|\{ x_1,x_1',x_2,x_2',x,x'\}|=3$, then they have $3$ common
attachments and so overlap, a contradiction. Otherwise, some $y\in\{x_1', x_2',
x'\}$ is not in $ \{x_1, x_2, x\}$.  Then $y$ is in one
residual arc $A$ of $x_1, x_2, x$ in $C$ and not both of the other two
of $\{ x_1', x_2', x'\}$ are in $A$. So again $B, B'$ overlap, a contradiction from which we conclude $P_B$ and $P_{B'}$ are internally disjoint.

Let $C = P^{-1} R_1 Q R_2$. Suppose $B, B' \in\bbb$ are such that $P =
\ldots P_B \ldots P_{B'} \ldots$ and $Q = \ldots Q_{B'} \ldots$ $ Q_B
\ldots$. We claim that either $P_B = P_{B'}$ or $Q_B = Q_{B'}$.
If not, then there is an attachment $u_P$ of one of $B$ and $B'$ in $P$ that is not an
attachment of the other and likewise an attachment
$u_Q$ of one of $B$ and $B'$ in  $Q$ that is not an attachment of the other.
Note that $u_P$ and $u_Q$ are not attachments of the same one of $B$ and $B'$, as otherwise the orderings in $P$ and $Q$ imply $B$ and $B'$ overlap.  

For the sake of definiteness, we assume $u_P\in \att(B)$, so that $u_Q\in \att(B')$.  \dragominorrem{(text removed)\ }
Let $w_P\in\att(B')\cap P$ and let $w_Q\in \att(B)\cap Q$.  The ordering of $B$ and $B'$ in $P$ and $Q$ imply that, in $C$, these vertices appear in the cyclic order $w_P,u_P,u_Q,w_Q$.  Since $u_P,u_Q,w_P,w_Q$ are all different,   we conclude that $B$ and $B'$ overlap on $C$, a contradiction.

It follows that, by symmetry, we may assume $P_B = P_{B'}$. As $P_B$ and $P_{B'}$ are internally disjoint, they are just
a vertex. So if $P=\ldots P_B \ldots P_{B'} \ldots$ and $Q = \ldots
Q_{B'} \ldots Q_B \ldots$, we may exchange $P_B$ and $P_B'$, to see that 
$P=\ldots P_{B'} \ldots P_{B} \ldots$ 
and $Q = \ldots
Q_{B'} \ldots Q_B \ldots$\ . We conclude there is an ordering of $\bbb$ as
claimed.

Let $(B_1,\ldots,B_k)$ and 
$(B_{\pi(1)},\ldots,B_{\pi(k)})$ be distinct orderings so that $P=P_{B_1}, \ldots, $ $ P_{B_k}$,
$P = P_{B_{\pi(1)}},\ldots,P_{B_{\pi(k)}}$, 
$Q = Q_{B_1} \ldots Q_{B_k}$ and 
$Q = Q_{B_{\pi(1)}},\ldots,Q_{B_{\pi(k)}}$. There exist $i<j$ so
that
$\pi(i) > \pi(j)$.  We may choose the labelling ($P$ versus $Q$) so that the preceding argument implies that $P_{B_i} = P_{B_j} = u$. If $Q_{B_i} = Q_{B_j}$,
then $Q_{B_i}= Q_{B_j} = w$ and $\att(B_i) = \att(B_j)$, which is (2).  Therefore, we may assume there is an attachment $y$ of one of $B_i$
and $B_j$ that is not an attachment of the other. Let $z$ be an
attachment of the other. Since $Q$ is either
$(Q_1,y,Q_2,z,Q_3)$ or $(Q_3^{-1},z,Q_2^{-1},y,Q_1^{-1})$, the only
possibility is that $\pi$ is the inversion $(k,k-1,\ldots,1)$. 
\end{cproof}

}

\chapter{Quads have BOD}\printFullDetails{\label{sec:prebox}

There are two main results in this section.  One is to show that each graph in the set $\mathcal T(\mathcal S)$ is 2-crossing-critical and the other, rather more challenging and central to the characterization of 3-connected 2-crossing-critical graphs with a subdivision of $V_{10}$,  is to show that all $H$-quads and some $H$-hyperquads have BOD.   We start with the definition of quads and hyperquads.  

}\begin{definition}  Let $G$ be a graph and $\hvfg$.  
\begin{enumerate}
\item \minor{For a path $P$ and distinct vertices $u$ and $v$ in $P$, $\cc{uPv}$\index{$[{uPv}]$} denotes the $uv$-subpath of $P$, while $\co{uPv}$\index{$[{uPv}\!>$} denotes $\cc{uPv}-v$, $\oc{uPv}$\index{$<\!{uPv}]$} is $\cc{uPv}-u$, and $\oo{uPv}$\index{$<\!{uPv}\!>$} is $\oc{uPv}-v$.}
\item \minor{When concatenating a $uv$-path $P$ with a $vw$-path $Q$, we may write either $PQ$\index{$PQ$} or $\cc{uPvQw}$\index{$\cc{uPvQw}$}. If $u=w$ and $P$ and $Q$ are internally disjoint, then both $PQ$ and $\cc{uPvQu}$ are cycles.  The reader may have to choose the appropriate direction of traversal of either $P$ or $Q$ in order to make the concatenation meaningful.}  
\item \minor{If $L$ is a subgraph of $G$ and $P$ is a path in $G$, then $L-\oo{P}$\index{$L-{\oo{P}}$} is obtained from $L$ by deleting all the edges and interior vertices of $P$. (In particular, this includes the case $P$ has length 1, in which case $L-\oo{P}$ is just $L$ less one edge.)}
\item  For $i=0,1,2,3,4$, the \wording{$H$-quad $Q_i$\index{quad}\index{$H$-quad} is the cycle $r_i\,s_{i+1}\,r_{i+5}\,s_i$}. 
\item For $i=0,1,2,3,4$, the {\em $H$-hyperquad\/}\index{hyperquad}\index{$H$-hyperquad} $\bQ_i$ is the cycle $(Q_{i-1}\cup Q_i)-\oo{s_i}$.
\item The {\em M\"obius bridge of $Q_i$\/}\index{M\"obius bridge}\index{bridge!M\"obius} is the $Q_i$-bridge $M_{Q_i}$ in $G$  such that $H\subseteq Q_i\cup M_{Q_i}$.
\item The {\em M\"obius bridge of $\bQ_i$\/} is the $\bQ_i$-bridge $M_{\bQ_i}$ in $G$ for which $(H-\oo{s_i})\subseteq \bQ_i\cup M_{\bQ_i}$.
\end{enumerate}
\end{definition}\printFullDetails{

The following notions will help our analysis.

}\begin{definition}\label{df:Hclose}  Let $G$ be a graph, $\hvng$, $n\ge 3$, and let $K$ be a subgraph of $G$.  Then:
\begin{enumerate}\item a {\em claw\/}\index{claw} is a subdivision of $K_{1,3}$ with {\em centre\/}\index{centre}\index{claw!centre} the vertex of degree 3 and {\em talons\/}\index{talon}\index{claw!talon} the vertices of degree 1;
\item an {\em $\{x,y,z\}$-claw\/}\index{$\{x,y,z\}$-claw} is a claw with talons $x$, $y$, and $z$;
\item an {\em open $H$-claw\/}\index{open $H$-claw} is the subgraph of $H$ obtained from a claw in $H$ consisting of the three $H$-branches incident with an $H$-node, which is the {\em centre\/} of the open $H$-claw, but with the three talons deleted;
\item $K$ is {\em $H$-close\/}\index{close}\index{$H$-close} if $K\cap H$ is contained either in a closed $H$-branch or in a open $H$-claw.
\item  A cycle $C$ in $K$ is a {\em $K$-prebox\/}\index{prebox}\index{$K$-prebox} if, for each edge $e$ of $C$, $K-e$ is not planar.
\end{enumerate}
\end{definition}\printFullDetails{

The following is elementary but not trivial.

}\begin{lemma}\label{lm:closeIsPrebox}  Let $C$ be an $H$-close cycle, for some $H\topol V_6$.  Then $C$ is a $(C\cup H)$-prebox.\end{lemma}\printFullDetails{

\begin{cproof}  For $e\in E(C)$, if $e\notin H$, then evidently $(C\cup H)-e$ contains $H$, which is a $V_6$; therefore $(C\cup H)-e$ is not planar.  So suppose $e\in H$.   Since $C$ is $H$-close, $C\cap H$ is contained in either a closed $H$-branch $b$ or an open $H$-claw $Y$.  There is an $H$-avoiding path $P$ in $C-e$ having ends in both components of either $b-e$ or $Y-e$. In the former case,  $(H-e)\cup P$, and hence $(C\cup H)-e$, contains a $V_6$.  In the latter case,  $(Y-e)\cup P$ contains a different claw that has the same talons as $Y$, so again $(H-e)\cup P$, and $(C\cup H)-e$, contains a $V_6$. \end{cproof}

}\begin{lemma}\label{lm:preboxClean} Let $K$ be a graph  and  $C$ a cycle of $K$.  If $C$ is a $K$-prebox, then, in any 1-drawing of $K$, $C$ is clean.\end{lemma}\printFullDetails{

\begin{cproof}  Let $D$ be a 1-drawing of $K$ and let $e$ be any edge of $C$.  Since $K-e$ is not planar, $D(K-e)$ has a crossing.  It must be the only crossing of $D(K)$ and, therefore, $e$ is not crossed in $D(K)$.  \end{cproof}


We can now show that any of the tiled graphs described in Section \ref{sec:nextRed+Tiles} in fact have crossing number 2, thereby completing the proof that they are all 2-crossing-critical.

}\begin{theorem}\label{th:tiledAre2cc}  If $G\in \tileS$, then $G\in\m2$.  \end{theorem}\printFullDetails{

\begin{cproof}  By Lemmas \ref{lm:ksequence} and \ref{lm:2degenerate} and Corollary \ref{co:ktiles}, we know that if $K$ is a proper subgraph of $G$, then $\crn(K)\le 1$.  Thus, it suffices to prove that $\crn(G)\ge 2$. 

\wording{There are two edges in a tile that are not in the corresponding picture and are not part of a parallel pair.  An edge of  $G$ is \dragominor{a {\em $\Delta$-base\/}} if it is one of these edges. A {\em $\Delta$-cycle\/} is a face-bounding cycle in the natural projective planar embedding of $G$ containing precisely one \dragominor{$\Delta$-base}.  Recall that the rim $R$ of $G$ is described in Definition \ref{df:t(s)}.}  

There are at least three $\Delta$-cycles contained in $G$ and any two are totally disjoint.  From each $\Delta$-cycle we choose either of its $RR$-paths \dragominor{(by definition, these are $R$-avoiding)} as a ``spoke", and, with $R$ as the rim, we find 8 different subdivisions of $V_6$.  There are two of these that are edge-disjoint on the spokes, so if $D$ is a 1-drawing of $G$, the crossing must involve two \wording{edges} of $R$.

\minorrem{(The following is now recast as a claim.) }

\begin{claim}\label{cl:pictureClose} If $e$ is a rim edge in one of the 13 \wording{pictures}, then $e$ is in an $H'$-close cycle $C_e$, for some $H'\topol V_6$ in $G$. \end{claim}

 The \wording{point} of this is that Lemmas \ref{lm:closeIsPrebox} and \ref{lm:preboxClean} imply that $C_e$ is clean in $D$.  This is also obviously true for the other edges of the rim that are in digons.  The conclusion is that we know the two crossing edges must be from among the \wording{$\Delta$-bases}.  We shall show below that no two of these can cross in a 1-drawing of $G$, the desired contradiction.

\bigskip \begin{proofof}{Claim \ref{cl:pictureClose}} \wordingrem{(Text that is now useless has been removed.)}\wording{Let $e$ be in edge in the rim $R$ of $G$ that is in the picture $T$, let $r$ be the component of $T\cap R$ containing $e$,} and let $r'$ be the other \wording{component of $T\cap R$}.  There is a unique cycle in $T-r'$ containing $e$; this is the cycle $C_e$.  \dragominor{Let $e'$ be the one of the two \dragominor{$\Delta$-bases} incident with $T$ that has an end in $r$.}  \wording{Choose the $RR$-subpath  of the $e'$-containing $\Delta$-cycle} that is disjoint from $r$.  For any other two of the $\Delta$-cycles, choose arbitrarily one of the $RR$-subpaths.  These three ``spokes", together with $R$, constitute a subdivision $H'$ of $V_6$ for which $C_e$ is $H'$-close, as required. \end{proofof}

The proof is completed by showing that no two \dragominor{$\Delta$-bases} can cross \dragominor{in a 1-drawing of $G$}.  If there are at least five tiles, then it is easy to find a subdivision of $V_8$ \wording{so that the two \dragominor{$\Delta$-bases}} are on disjoint $H$-quads and therefore cannot be crossed in a 1-drawing of $G$.  \wording{Thus,} we may assume there are precisely three tiles and the crossing \dragominor{$\Delta$-bases} $e_1$ and $e_2$ are\wording{, therefore,} in consecutive $\Delta$-cycles.

Let $T$ be the \wording{picture} incident with both $e_1$ and $e_2$.  Choose a subdivision $H'$ of $V_6$ containing $R$ but so that $T\cap H'=T\cap R$.  There is a unique 1-drawing $D$ of $H'$ with $e_1$ and $e_2$ being the crossing pair.  For $i=1,2$, let the $H'$-branch containing $e_i$ be $b_i$.  The end $u_i$ of $e_i$ that is in $T$ is in the interior of $b_i$.  

\wording{The vertices $u_1$ and $u_2$ are two of the four attachments of $T$ in $G$.  Let $w_1$ and $w_2$ be the other two, labelled so that $w_1$ is in the same component of $T\cap R$ as $u_2$.  It follows that $w_2$ is in the same component of $T\cap R$ as $u_1$.  In $T$, there}
is a unique \wording{pair of totally disjoint} $R$-avoiding $u_1w_1$- and $u_2w_2$-paths $P_1$ and $P_2$, respectively\wordingrem{(text deleted)}.  \minor{The crossing in $D$ is of $e_1$ with $e_2$, so $\cc{u_1b_1w_2}$ and $\cc{u_2b_2w_1}$ are both not crossed in $D$.  Therefore, $D[P_1]$ and $D[P_2]$ are both in the same face $F$ of $D$.}

Since the two paths $P_1$ and $P_2$ are totally disjoint\wording{(text deleted)},  $D[P_1]$ and $D[P_2]$ are disjoint arcs in $F$; the contradiction arises from the fact that their ends alternate in the boundary of $F$, showing there must be a second crossing.
\end{cproof}

One important by-product of cleanliness is that it frequently shows a cycle has BOD.

}\begin{lemma}\label{lm:cleanBOD}  Let $C$ be a cycle in a graph $G$.  Let $D$ be a 1-drawing of $G$ in which $C$ is clean.  If there is a non-planar $C$-bridge, then $C$ has BOD and exactly one non-planar bridge.\end{lemma}\printFullDetails{

\begin{cproof}  Let $B$ be a non-planar $C$-bridge.  Then $D[C\cup B]$ has a crossing, and, since $C$ is clean in $D$, the crossing does not involve an edge of $C$.  Therefore, it involves two edges of $B$.  This is the only crossing of $D$, so inserting a vertex at this crossing turns $D$ into a planar embedding of a graph $G^\times$.  As $C$ is still a cycle of $G^\times$, $C$ has BOD in $G^\times$ and all $C$-bridges in $G^\times$ are planar.  But $OD_{G^\times}(C)$ is the same as $OD_{G}(C)$ and all $C$-bridges other than $B$ are the same in $G$ and $G^\times$. \end{cproof}

We shall routinely make use of the following notions.

}\begin{definition}  Let $G$ be a connected graph and let $H$ be a subgraph of $G$.  Then:
\begin{enumerate}\item  $\comp H$\index{$\comp H$} is the subgraph of $G$ induced by $E(G)\setminus E(H)$; and
\item if $G$ is embedded in $\pp$, then an {\em $H$-face\/}\index{face}\index{$H$-face} is a face of the induced embedding of $H$ in $\pp$.
\end{enumerate}
\end{definition}\printFullDetails{

 We will often use this when $B$ is a $C$-bridge, for some cycle $C$ in a graph $G$, in which case $\comp B$ is the union of $C$ and all $C$-bridges other than $B$.  The following two lemmas are useful examples.

}\begin{lemma}\label{lm:nucMeetGamma}  Let $G$ be a graph embedded in $\pp$ with representativity 2 and let $\gamma$ be a non-contractible curve in $\pp$ so that $G\cap \gamma=\{a,b\}$.  Let $C$ be a contractible cycle in $G$ and let $B$ be a $C$-bridge so that $\Nuc(B)\cap \{a,b\}\ne\varnothing$.  Then $\comp{B}$ is planar. \end{lemma}\printFullDetails{

\begin{cproof}  This is straightforward:  $\comp{B}=G-\Nuc(B)\subseteq G-( \{a,b\}\cap \Nuc{B})$ and the latter has a representativity at most 1 embedding in $\pp$.  Therefore it is planar.  \end{cproof}

The following result, when combined with the (not yet proved) fact that $H$-quads and some $H$-hyperquads have BOD, yields the fact, often used in the sections to follow, that deleting some edge results in a 1-drawing in which a particular $H$-quad or $H$-hyperquad must be crossed.

}\begin{lemma}\label{lm:BODcrossed}  Let $G$ be a graph with $\crn(G)\ge 2$ and let $C$ be a cycle in $G$.  If $C$ has BOD in $G$, then, for any planar $C$-bridge $B$, $C$ is crossed in any 1-drawing of $\comp B$. \end{lemma}\printFullDetails{

\begin{cproof}  Suppose there is a 1-drawing $D$ of $\comp B$ with $C$ clean.  Since $C$ has BOD and $G$ is not planar,  there is a non-planar $C$-bridge $B'$.   Because $C$ is clean, any crossing in $D[C\cup B']$ involves two edges of $B'$.  The only crossing in $D$ involves two edges of $B'$, so every other $C$-bridge in $\comp{B}$ is planar.  Since $B$ is planar, it follows from Corollary \ref{co:TutteTwo} that $\crn(G)\le 1$, a contradiction. \end{cproof}


We remark that $M_Q$ is a non-planar $Q$-bridge whenever $Q$ is an $H$-quad or $H$-hyperquad.

}\begin{corollary}\label{co:hyperBODquadBOD}  Let $G\in \m2$ and $\hvfg$.  If the $H$-quad $Q_i$ and $H$-hyperquad $\bQ_j$ are disjoint,  $\bQ_j$ has BOD, and there is a planar $\bQ_j$-bridge $B$,  then $Q_i$ has BOD and there is precisely one non-planar $Q_i$-bridge. \end{corollary}\printFullDetails{

\begin{cproof}  Let $B$ be a planar $\bQ_j$-bridge.  Because $G$ is 2-crossing-critical, there is a 1-drawing $D$ of $\comp B$.  By Lemma \ref{lm:BODcrossed}, $\bQ_j$ is crossed in $D$.  Note that $H-\oo{s_j}\subseteq \comp B$.   In any 1-drawing of $H-\oo{s_j}$ in which $\bQ_j$ is crossed, the crossing is between $r_{j-2}\cup r_{j-1}\cup r_j\cup r_{j+1}$ and  $r_{n+j-2}\cup r_{n+j-1}\cup r_{n+j}\cup r_{n+j+1}$.  Since $Q_i$ is edge-disjoint from these crossing rim segments,  $Q_i$ is clean in $D$.  

The two graphs $OD_G(Q_i)$ and $OD_{\comp B}(Q_i)$ are isomorphic:  the $Q_i$-bridges in both $G$ and $\comp B$ are the same, except $M_{Q_i}$ in $G$ becomes $M_{Q_i}-\Nuc(B)$ in $\comp B$ and they have the same attachments. Since $Q_i$ is clean in $D$, $OD_{\comp B}(Q_i)$ is bipartite.  Furthermore, the crossing in $D$ is between two edges of $\bQ_j$, so $D$ shows that every $Q_i$-bridge other than $M_{Q_i}$ is planar. \end{cproof}

We next introduce boxes, which are cycles that, it turns out, cannot exist in a 2-crossing-critical graph $G$.  On several occasions in the subsequent sections, we prove a result by showing that otherwise $G$ has a box.

}\begin{definition}  Let $C$ be a cycle in a graph $G$.  Then $C$ is a {\em box\/}\index{box} in $G$ if $C$ has BOD in $G$ and there is a planar $C$-bridge $B$ so that $C$ is a $\comp B$-prebox.\end{definition}\printFullDetails{

}\begin{lemma}\label{lm:noBox}  Let $G\in \m2$.  Then $G$ has no box.\end{lemma}\printFullDetails{

\begin{cproof}  Suppose $C$ is a box in $G$.  Then $C$ has BOD and there is a planar $C$-bridge $B$ so that $C$ is a $\comp B$-prebox.  As $\comp B$ is a proper subgraph of $G$, there is a 1-drawing $D$ of $\comp B$.  By Lemma \ref{lm:preboxClean}, $D[C]$ is clean.  This contradicts Lemma \ref{lm:BODcrossed}.  \end{cproof}

\majorrem{(The following lemma and its two corollaries were moved from the old Section 7, which otherwise seems irrelevant.)}\major{We can now determine the complete structure of a 2-connected $H$-close subgraph.}

}\major{\begin{lemma}\label{lm:closeIsCycle}  Let $G\in\m2$ and $\hvng$ with $n\ge 4$.  If $K$ is a 2-connected $H$-close subgraph of $G$, then $K$ is a cycle.  \end{lemma}}\printFullDetails{

\wordingrem{(The text of the proof has also substantially changed.)}

\major{\begin{cproof}  If $K\cap H$ consists of at least two vertices, then we include in $K$ the minimal connected subgraph of the $H$-branch or open $H$-claw containing $K\cap H$.  
\dragominor{Since $K$ is $H$-close, there is} a $K$-bridge $M_K$ in $G$ so that $H\subseteq K\cup M_K$.  Let $e$ be an edge of any $H$-spoke totally disjoint from $K$.  Note that $M_K-e$ is a $K$-bridge in $G-e$ 
and that $M_K$ has the same attachments in $G$ as $M_K-e$ has in $G-e$.
 \newline\indent
  Since $K$ is 2-connected, every edge of $K$ is in an $H$-close cycle contained in $K$.  Thus, for any 1-drawing $D$ of $G-e$, Lemmas \ref{lm:closeIsPrebox} and \ref{lm:preboxClean} imply that $D[K]$ is clean.  There 
 is a face $F$ of $D[K]$ containing 
$D[M_K-e]$. As $D[K]$ is clean and $K$ is 2-connected, $F$ is bounded by a cycle $C$ of $K$. 
 \newline\indent
Lemma \ref{lm:closeIsPrebox} implies the cycle $C$ is a $(C\cup H)$-prebox.  If $K$ is not just $C$, then there is a $C$-bridge $B$ contained on the side of $D[C]$ disjoint from $M_K$.  Evidently $B$ is a planar $C$-bridge.
\newline\indent
Lemma \ref{lm:cleanBOD} implies $C$ has BOD.  Since $C$ is a $(C\cup H)$-prebox, $C$ is a $\comp B$-prebox.  We conclude that $C$ is a box, contradicting Lemma \ref{lm:noBox}.  This shows that $K=C$.
\end{cproof}}

\major{The second of the following two corollaries is used several times later in this work.
We recall from Definition \ref{df:bridge} that, for a set $W$ of vertices, $\iso W$ is the subgraph with vertex set $W$ and no edges.}

}\major{\begin{corollary}\label{co:closeAtts}  Let $G\in\m2$, let $\hvng$ with $n\ge 4$, let $B$ be an $H$-bridge.
\begin{enumerate} \item If $x,y\in\att( B )$ are such that $\iso{\{x,y\}}$ is $H$-close, then there is a unique $H$-avoiding $xy$-path in $G$.
\item There do not exist vertices $x,y,z\in\att( B )$ so that $\iso{\{x,y,z\}}$ is $H$-close.
\end{enumerate}
 \end{corollary}}\printFullDetails{

\major{\begin{cproof}  Suppose $P_1$ and $P_2$ are distinct $H$-avoiding $xy$-paths.    There is either a closed $H$-branch or an open $H$-claw containing an $xy$-path; this subgraph of $H$ contains a unique $xy$-path $P$.  Then $P\cup P_1\cup P_2$ is a 2-connected $H$-close subgraph of $G$ and so, by Lemma \ref{lm:closeIsCycle}, is a cycle.  But it contains three distinct $xy$-paths, a contradiction.
\newline\indent
For the second point, suppose by way of contradiction that such $x,y,z$ exist.  Let $Y$ be an $\{x,y,z\}$-claw in $B$.  There is a minimal connected subgraph $Z$ of $H$ contained either in a closed $H$-branch or in an open $H$-claw and containing $x$, $y$, and $z$.  We note that $Z$ is either a path or an $\{x,y,z\}$-claw.  Thus, $Y\cup Z$ is 2-connected and is $H$-close.  It is a cycle by Lemma \ref{lm:closeIsCycle}, but the centre of $Y$ has degree 3 in $Y\cup Z$, a contradiction. 
\end{cproof}}

}\major{\begin{corollary}\label{co:attsMissBranch}  Let $G\in \m2$, let $\hvfg$, and let $B$ be a $Q$-local $H$-bridge, for some $H$-quad $Q$.  If $s$ is an $H$-spoke and $r$ is an $H$-rim branch, both contained in $Q$, then $|\att( B )\cap s|\le 2$ and $|\att( B )\cap (Q-\cc{r})|\le 2$. \end{corollary}}\printFullDetails{

\major{\begin{cproof}  The first claim follows immediately from Corollary \ref{co:closeAtts}.  For the second, suppose there are three such attachments $x$, $y$, and $z$.  Corollary \ref{co:closeAtts} implies they are not all in the other $H$-rim branch $r'$ of $Q$, so at least one of $x$, $y$, and $z$ is in the interior of some $H$-spoke of $Q$.  
\newline\indent
Suppose first that some $H$-spoke $s$ in $Q$ is such that $\oo{s}\cap \{x,y,z\}=\varnothing$.  Then let $H'=H-\oo{s}$, let $B'$ be the $H'$-bridge containing $B$, and let $r'$ and $s'$ be the two $H$-branches in $Q$ other than $r$ and $s$.  Then $x$, $y$, and $z$ are all attachments of $B'$ and they are all in the same open $H'$-claw containing $(r'\cup s')-r$, contradicting Corollary \ref{co:closeAtts}.
\newline\indent
Otherwise, we may suppose both $H$-spokes $s$ and $s'$ in $Q$ have one of $x$, $y$, and $z$ in their interiors.  We may suppose $s$ has no other one of $x$, $y$ and $z$.   Choose the labelling so that $x\in\oo{s}$.   Let $r'$ \dragominor{be the $H$-rim branch in $Q$ other than $r$}  and again let $H'=H-\oo{s}$ and  $B'$ be the $H'$-bridge containing $B$.  \dragominor{Then} $y$ and $z$ are attachments of $B'$, as is the $H$-node in $s\cap r'$.  But now these three attachments of $B'$ contradict Corollary \ref{co:closeAtts}. \end{cproof}}

We want to find cycles having BOD in our $G\in \m2$ that is embedded with representativity 2 in the projective plane.  The following will be helpful.

}\begin{lemma}\label{lm:planeNotOverlap} Let $G$ be a graph embedded  in $\pp$ and let $C$ be a contractible cycle in $G$.  Suppose $B$ is a $C$-bridge so that $C\cup B$ has no non-contractible cycles and let $F$ be the $C$-face containing $B$. If $B'$ is another $C$-bridge \dragominor{embedded} in $F$, then $B$ and $B'$ do not overlap on $C$. \end{lemma}\printFullDetails{

\begin{cproof}  Let $x$ and $y$ be any distinct attachments of $B$ and let $P$ be a $C$-avoiding $xy$-path in $B$.  Then $C\cup P$ has three cycles, all contractible by hypothesis.  We claim that one bounds a closed disc $\Delta$ so that $C\cup P\subseteq \Delta$.    If $P$ is contained in the disc $\Delta$ bounded by $C$, then we are done.
In the remaining case, let $C'$ be one of these cycles containing $P$.  If the closed disc $\Delta'$ bounded by $C'$ contains $C$, then we are done.  Otherwise, $\Delta\cap \Delta'$ is a path in $C$ and then $\Delta\cup \Delta'$ is the desired closed disc.  

\dragominor{As} no other $C$-bridge in  $F$ can have attachments in the interiors of  both the two $xy$-subpaths of $C$ and, therefore, there is no $C$-bridge embedded in $F$ that is skew \wording{(see Lemma \ref{lm:overlapClaw} (\ref{it:skew}))} to $B$.

Likewise, if $x,y,z$ are three distinct attachments of $B$, then there is a disc $\Delta'$ containing the union of $C$ with a $C$-avoiding $\{x,y,z\}$-claw in $B$.  This disc shows that no other $C$-bridge embedded in $F$ can have all of $x,y,z$ as attachments and, therefore, no $C$-bridge embedded in $F$ is 3-equivalent \wording{(see Lemma \ref{lm:overlapClaw} (\ref{it:equivalent}))} to $B$. \end{cproof}

\wording{The following is an immediate consequence of Lemma \ref{lm:planeNotOverlap}\dragominor{~and the fact that $C$ has only two faces}.}

}\begin{corollary}\label{co:contractibleBOD}  Let $G$ be a graph embedded in $\pp$ and let $C$ be a cycle of $G$ bounding a closed disc in $\pp$.  If at most one $C$-bridge $B$ is such that $C\cup B$ contains a non-contractible cycle, then $C$ has BOD and, for every other $C$-bridge $B'$, $C\cup B'$ is planar. \hfill{{\rule[0.8ex]{7pt}{7pt}}}\end{corollary}\printFullDetails{

The following result is surprisingly useful in later sections.

}\begin{lemma}\label{lm:noBdy}  Let $G\in \m2$ and suppose $G$ is embedded with representativity 2 in the projective plane.  Let $\gamma$ be a non-contractible curve in the projective plane so that $|\gamma\cap G|=2$ and let $C$ be a cycle of $G$ so that $\gamma\cap C=\varnothing$.  If there is a non-planar $C$-bridge $B$, then $\gamma\cap G\subseteq B$, $C$ has BOD, and, for every other $C$-bridge $B'$, $C\cup B'$ is planar. \end{lemma}\printFullDetails{

\begin{cproof}  Let $a$ and $b$ be the two points in $\gamma\cap G$.  We note that $G-a$ and $G-b$ are planar, as they have representativity 1 embeddings in $\pp$.  Thus, if, for example, $a\notin B$, then $C\cup B\subseteq G-a$ and so $C\cup B$ is planar, a contradiction.  

If $B'$ is any other $C$-bridge, then $a,b\notin C\cup B'$ and, therefore, $C\cup B'$ is disjoint from $\gamma$.  Since any non-contractible cycle must intersect $\gamma$,  $C\cup B'$ has no non-contractible cycles.   The result is now an immediate consequence of Corollary \ref{co:contractibleBOD}. 
\end{cproof}

Here is a simple result that we occasionally use.

}\begin{lemma}\label{lm:threeAtts}  Suppose $G\in\m2$ and $V_{2n}\topol H\subseteq G$, with $n\ge 4$.  Let $B$ be an $H$-bridge.  
\begin{enumerate} \item Then $|\att( B )|\ge 2$.
\item If $|\att( B )|=2$, then $B$ is isomorphic to $K_2$.
\item If $|\att( B )|=3$, then $B$ is isomorphic to $K_{1,3}$.
\end{enumerate}
\end{lemma}\printFullDetails{

\begin{cproof} Note that $\att( B )=B\cap \comp B$ and $G=B\cup \comp B$.  If $|\att( B )|\le 1$, then $G$ is not 2-connected.  If $|\att( B )|=2$ and $\Nuc( B )$ has a vertex, then $G$ is not 3-connected.  

Now suppose $|\att( B )|=3$ and $B$ is not isomorphic to $K_{1,3}$.  Let $Y$ be an $\att( B )$-claw contained in $B$.  As $\comp B\cup Y$ is a proper subgraph of $G$, it has a 1-drawing $D_1$; $Y$ is clean in $D_1$, as $H$ must be self-crossed.  On the other hand, if $s$ is an $H$-spoke disjoint from $B$, there is a 1-drawing $D_2$ of $G-\oo{s}$.  Again, the crossing in $D_2$ involves two edges of $H-\oo{s}$, so $B$ is clean.  We can substitute $D_2[B]$ for $D_1[Y]$ to convert $D_1$ into a 1-drawing of $G$, a contradiction.\end{cproof}

The following lemma is the last \wording{substantial} one we need before proving that every $H$-quad has BOD.

}\begin{lemma}\label{lm:CdisjointNCcycle}  Let $G$ be a graph that is embedded in $\pp$ and let $C$ be a cycle of $G$.  Let $B$ be a $C$-bridge so that $\Nuc(B )$ contains a non-contractible cycle.  Then $C$ is contractible,
 $C$ has BOD, and
every $C$-bridge other than $B$ is planar.
\end{lemma}\printFullDetails{

\begin{cproof} \wording{Let $N$ be a non-contractible cycle in $\Nuc(B)$ and let $B'$ be a $C$-bridge different from $B$.  Then $C\cup B'$ is disjoint from $N$.}  Since any two non-contractible cycles in $\pp$ intersect,  $C\cup B'$ does not contain a non-contractible cycle.  Clearly this implies $C$ is contractible and the remaining items are an immediate consequence of Corollary \ref{co:contractibleBOD}.\end{cproof}

We prove below that every \minor{$H$-}quad has BOD and that at least two hyperquads have BOD.  A standard labelling of the embedded $V_{10}$ will help make the details of the statement comprehensible. We have seen that, up to relabelling,  there are two representativity 2 embeddings of $V_{10}$ in $\pp$.   There is a simple non-contractible curve $\gamma$ in $\pp$ meeting $G$ in two points $a$ and $b$.  These are both in the rim $R$ of $H$ and either none or one of the $H$-spokes is outside the M\"obius band $\Mob$ bounded by $R$.   Let $\alpha$ and $\beta$ be the two $ab$-subarcs of $\gamma$, labelled so that $\beta\subseteq \Mob$. 

}\minor{\begin{definition}\label{df:embedExposed}  Let $G$ be a graph and let $\hvfg$.  If $G$ is embedded in $\pp$ so that one $H$-spoke is not in $\Mob$, then $H$ has an {\em exposed spoke\/}\index{exposed}\index{spoke!exposed} and the exposed spoke is the $H$-spoke not in $\Mob$. 
\newline\indent
In this case, the {\em standard labelling\/}\index{standard labelling} is chosen so that the exposed spoke is $s_0$ and so that $v_0,v_1,v_2,v_3,v_4,v_5$ are all incident with one of the two faces of $H\cup \gamma$ incident with $s_0$.\end{definition}}
\printFullDetails{

   \wordingrem{(Text removed.)}
The faces of $H\cup \gamma$ are bounded by the cycles:
\begin{enumerate}
\item $\cc{a,r_{9},v_0}s_0\cc{v_5,r_5,b,\alpha,a}$;
\item $r_0\,r_1\,r_3\,r_3\,r_{4}\,s_0$;
\item $\cc{a,r_{9},v_0}r_0\,s_1\cc{v_{6},r_5,b,\beta,a}$;
\item $Q_1$, $Q_2$, $Q_{3}$;
\item $r_{4}\cc{v_5,r_5,b,\beta,a,r_{9},v_{9}}s_{4}$;
and
\item $\cc{b,r_5,v_6} r_6\,r_7\,r_8\cc{v_{9},r_{9},a,\alpha,b}$.
\end{enumerate}

This case is illustrated in the \wording{diagram to the left} in Figure \ref{fig:Rep2Vten}.

In the case all the $H$-spokes are in $\Mob$, the labelling of $H$ may be chosen so that the faces of $H\cup \gamma$ are bounded by:
\begin{enumerate}
\item $\cc{a,r_{9},v_0}\rbsp r_0\,r_1\,r_2\,r_{3}\lbsp\cc{v_{4},r_{4},b,\alpha,a}$;
\item $\cc{a,r_{9},v_0,s_0,v_5,r_{4},b,\beta,a}$;
\item $Q_0$, $Q_1$, $Q_2$, $Q_{3}$;
\item $\cc{v_{4},r_{4},b,\beta,a,r_{9},v_{9},s_{4},v_{4}}$; and
\item $\cc{b,r_{4},v_{5}}r_{5}\, r_6\,r_7\,r_8\cc{v_{9},r_{9},a,\alpha,b}$.\end{enumerate}

This case is illustrated in the \wording{diagram to the right} in Figure \ref{fig:Rep2Vten}.

We need one more technical lemma before the main result of this section.

}\begin{lemma}\label{lm:nonconsecHyper}   Let $G\in \m2$, let $\hvfg$, and let $i,j\in\{0,1,2,3,4\}$ be such that $\bQ_i$ and $\bQ_j$ have precisely one $H$-spoke in common.  If $\bQ_i$ has BOD and $s_i$ is in a planar $\bQ_i$-bridge, then $\comp{(M_{\bQ_j})}$ is planar. \end{lemma}\printFullDetails{

\begin{cproof}  Let $e$ be any edge of $s_i$ and let $D$ be a 1-drawing of $G-e$.  By Lemma \ref{lm:BODcrossed}, $\bQ_i$ is crossed in $D$.  Thus, the crossing of $D$ involves an edge of $M_{\bQ_j}$,  showing that $\comp{(M_{\bQ_j})}$ is planar.  
\end{cproof}

The following is the main result of this section.

}\begin{theorem}\label{th:BODquads}  Let $G\in \m2$ and $\hvfg$.  Let $G$ be embedded with representativity 2 in the projective plane, with the standard labelling.  Then:
\begin{enumerate}\item\label{it:quadBOD} each $H$-quad $Q$ of $G$ has BOD and exactly one non-planar bridge;
\item\label{it:bQ2} $\bQ_2$ has BOD;
\item\label{it:nearlyBOD} for each $i\in \{0,1,3,4\}$, $\comp{(M_{\bQ_i})}$ is planar; 
\item\label{it:exposedBQ3} if there is an exposed spoke, then $\bQ_3$ has BOD;
\item\label{it:noExposedBQ1bQ3} if there is no exposed spoke, then at least one of $\bQ_1$ and $\bQ_3$ has BOD.  
\item\label{it:cornerInGamma}if there is no exposed spoke and $\bQ_1$ does not have BOD, then there is a $\bQ_1$-bridge $B$ different from $M_{\bQ_1}$ so that \minor{$B\subseteq \disc$ and} either:
\begin{enumerate}\item\label{it:aIsV0} $a=v_0$ and $B$ has an attachment at $a$, an attachment in $r_5\,r_6$, and $\att(B)\subseteq \{a\}\cup r_5\,r_6$;  or 
\item\label{it:bIsV5} $b=v_5$ and  $B$ has an attachment at $b$, an attachment in  $r_0\,r_1$, and $\att(B)\subseteq\{b\}\cup r_0\,r_1$.  (The analogous statement holds for $\bQ_3$ in place of $\bQ_1$.)
\end{enumerate}
\end{enumerate}
\end{theorem}\printFullDetails{

The following definitions will be useful throughout the remainder of this work.

}\begin{definition}  Let $G$ be a graph embedded in $\pp$ and let $C$ be a cycle of $G$ bounding a closed disc $\Delta$ in $\pp$.  A $C$-bridge $B$ is {\em 
$C$-interior\/}\index{interior}\index{$C$-interior} if $B$ is contained in $\Delta$ and {\em $C$-exterior\/}\index{exterior}\index{$C$-exterior} otherwise.\end{definition}\printFullDetails{

\begin{cproofof}{Theorem \ref{th:BODquads}}  \dragominor{We distinguish two cases.}

\medskip\noindent\dragominor{{\bf Case 1:}  {\em $H$ has an exposed spoke.}}

\medskip We adopt the standard labelling, so $s_0$ is the exposed spoke.   We note that $Q_2$ is disjoint from $G\cap \gamma$ and, therefore, Lemma \ref{lm:noBdy} implies $Q_2$ has BOD and precisely one non-planar bridge\dragominor{, which is part of (\ref{it:quadBOD})}.

The arguments for $Q_1,Q_3,\bQ_2,\bQ_{3}$ are all analogous and so we do $\bQ_2$.  Since $s_0$ is exposed, the cycle $\cc{a,r_{9},v_0}\rbsp s_0\,r_{4}\,s_{4}\lbsp\cc{v_{9},r_{9},a}$ is not contractible and is disjoint from $\bQ_2$.  Lemma \ref{lm:CdisjointNCcycle} shows $\bQ_2$ has BOD and precisely one non-planar bridge, proving (\ref{it:bQ2})\dragominor{\ and (\ref{it:exposedBQ3}).   We have also proved (\ref{it:nearlyBOD}) for $j=3$ and (\ref{it:quadBOD}) for $Q_1$ and $Q_3$}.

To complete the proof of (\ref{it:quadBOD}) \dragominor{in Case 1,} it remains to \dragominor{deal with} $Q_0$ and $Q_{4}$.  These two cases are symmetric and so it suffices to prove $Q_0$ has BOD \dragominor{and only one non-planar bridge}.   We note that $\bQ_3$ is completely disjoint from $Q_0$ and we have shown that $\bQ_3$ has BOD.  Let $B$ be the $\bQ_3$-bridge containing $s_3$.  As $\bQ_3$ is contractible and $B$ is $\bQ_3$-interior, we conclude that $B$ is planar.  Therefore, Corollary \ref{co:hyperBODquadBOD} implies $Q_0$ has BOD, and each $Q_0$-bridge except $M_{Q_0}$ is planar, as required  \dragominor{for (\ref{it:quadBOD})}. 

\dragominor{For (\ref{it:nearlyBOD}), i}t remains to prove that, for $j\in\{0,1,4\}$, $\comp{(M_{\bQ_j})}$ is planar.  We apply Lemma \ref{lm:nonconsecHyper}: for $j=0$ or 4, we take $i=2$; for $j=1$, we take $i=3$.  In all cases, the result follows.

\medskip\noindent{\bf Case 2:}  {\em $H$ has no exposed spoke.}

\medskip Lemma \ref{lm:noBdy} shows $Q_1$, $Q_2$, and $\bQ_2$ all have BOD and just one non-planar bridge.  \dragominor{This proves (\ref{it:bQ2}) and part of (\ref{it:quadBOD})}.  We use this in Corollary \ref{co:hyperBODquadBOD} to see that $Q_{4}$  has BOD and just one non-planar bridge, \dragominor{another part of (\ref{it:quadBOD})}.  Also,  taking $i=2$ and $j\in\{0,4\}$ in Lemma \ref{lm:nonconsecHyper}, we see that $\comp{(M_{\bQ_j})}$ is planar\dragominor{, part of (\ref{it:nearlyBOD})}.

If $\bQ_3$ has BOD, then Corollary \ref{co:hyperBODquadBOD} implies $Q_0$ has BOD, so in order to show $Q_0$ has BOD, we may assume $\bQ_3$ has NBOD.   There is an analogous situation for $Q_3$ and $\bQ_1$.  
We first prove (\ref{it:cornerInGamma}) for $\bQ_3$; we will use this to prove both $Q_0$ has BOD and (\ref{it:noExposedBQ1bQ3}).

If $v_4\ne b$ and $v_9\ne a$, then Lemma \ref{lm:noBdy} shows that  $\bQ_3$ has BOD and exactly one non-planar bridge. So suppose either (or both) $v_4=b$ or $v_9=a$.  If every $\bQ_3$-bridge other than $M_{\bQ_3}$ has only contractible cycles, then $\bQ_3$ has BOD by Corollary \ref{co:contractibleBOD}.  Thus, some $\bQ_3$-bridge $B$ other than $M_{\bQ_3}$ is such that $\bQ_3\cup B$ contains a non-contractible cycle.   Evidently, $B$ is $\bQ_3$-exterior.
If  $B\subseteq\Mob$, then again $\bQ_3\cup B$ has only contractible cycles.  Thus, $B\subseteq \Disc$.   

Any $\bQ_3$-exterior bridge $B$ contained in the face of $H\cup \gamma$ bounded by $$\cc{a,r_{9},v_0}\rbsp r_1\,r_2\,r_3\lbsp\cc{v_4,r_4,b,\alpha,a}$$ has all its attachments in $\{a\}\cup r_2\,r_3$. Note that $B$ is planar; moreover, if $a$ is not an attachment, then $\bQ_3\cup B$ has no non-contractible cycle and, therefore, does not overlap any other $\bQ_3$-exterior bridge.  We have the analogous conclusions if $B$ is contained in the face of $H\cup \gamma$ bounded by $\cc{b,r_5,v_6}\rbsp r_6\,r_7\,r_8\lbsp\cc{v_9,r_9,a,\alpha,b}$.

We conclude that either $B$ has $a$ as an attachment and also has an attachment in $r_2\,r_3$ or, symmetrically, $B$ has $b$ as an attachment and also has an attachment in $r_7\,r_8$.   This proves (\ref{it:cornerInGamma}). 

We now prove (\ref{it:noExposedBQ1bQ3}).  If $\{v_0,v_5\}\cap \{a,b\}=\varnothing$, then $\bQ_1$ has BOD and just one non-planar bridge; likewise if $\{v_4,v_9\}\cap \{a,b\}=\varnothing$, then $\bQ_3$ has BOD and just one non-planar bridge. Up to symmetry,  the only other possibility is that $v_0=a$ and $v_4=b$.

Now suppose that $\bQ_1$ also has NBOD.  Then (\ref{it:cornerInGamma}) implies that there must be, up to symmetry, a $\bQ_1$-bridge $B_1$ different from $M_{\bQ_1}$ having attachments at $a$ and in $r_5\,r_6$.  Likewise, there is an $H$-bridge $B_3$ different from $M_{\bQ_3}$ having attachments at $b$ and in $r_7\,r_8$.  As $B_1$ cannot have an attachment at $b$, $B_1\ne B_3$.  Considering
 the embedding of $G$ in $\pp$, we see that both $B_1$ and $B_3$ must be embedded in the face of $H\cup \gamma$ incident with $\cc{b,r_4,v_5}\rbsp r_5\,r_6\,r_7\,r_8\lbsp\cc{v_9,r_9,a,\alpha,b}$.  If $B_1$, say, has an attachment other than $a$ and $v_7$, then the $H$-avoiding path in $B_3$ from $b$ to any attachment in $r_7\,r_8$ crosses $B_1$, a contradiction.  So $\att(B_1)=\{a,v_7\}$, $\att(B_3)=\{b,v_7\}$, and, by Lemma \ref{lm:threeAtts}, both $B_1$ and $B_2$ are just edges.

Now recall that $\bQ_2$ has BOD and,  letting $B_2$ be the $\bQ_2$-bridge containing $s_2$, Lemma \ref{lm:BODcrossed} implies $\bQ_2$ is crossed in a 1-drawing $D$ of $\comp B_2$.   The crossing must be between the paths $r_0\,r_1\,r_2\,r_3$ and $r_5\,r_6\,r_7\,r_8$.

 There are two maximal uncrossed subpaths of $R$ in $D$ and we know that $v_0$ and $v_9$ are on one uncrossed segment, say $S_1$, of $R$, while  $v_4$ and $v_5$ are on $S_2$.   Suppose first that $v_7$ is on $S_1$.  Then the cycle $\cc{v_0,B_1,v_7}\rbsp r_6\,r_5\,r_4\,s_4\,r_0$ separates $v_8$ from $v_3$ in $D$, yielding the contradiction that $s_3$ is crossed in $D$.   On the other hand, if $v_7$ is on $S_2$, then the same cycle separates $v_6$ from $v_1$, yielding the contradiction that $s_1$ is crossed in $D$.    

 We conclude that not both $\bQ_1$ and $\bQ_3$ can have NBOD\dragominor{\ which is (\ref{it:noExposedBQ1bQ3})}.   By symmetry, we may assume $\bQ_1$ has BOD.  Then Lemma \ref{lm:nonconsecHyper} shows $\comp{(M_{\bQ_3})}$ is planar.  Furthermore, Corollary \ref{co:hyperBODquadBOD} implies $Q_3$ has BOD and precisely one non-planar bridge.  
 
 What remains is to prove that $Q_0$ has BOD and precisely one non-planar bridge and that there is precisely one non-planar $\bQ_1$-bridge.  \dragominor{Recall that symmetry implies this will show the same things for $Q_3$ and $\bQ_3$, completing the proofs of (\ref{it:quadBOD}) and (\ref{it:nearlyBOD}).}
 
From (\ref{it:cornerInGamma}), we may assume that $v_9=a$ and that there is a $\bQ_3$-bridge $B_3$ attaching at $a$ and in $r_2\,r_3$.   Let $w$ be any attachment of $B_3$ in $r_2\,r_3$,  let $P$ be an $H$-avoiding $v_9w$-path in $B_3$, and let $Q$ be the subpath of $r_2\, r_3$ joining $w$ to $v_4$.  Then the cycle $\cc{v_9,P,w,Q,v_4,s_4,v_9}$ is non-contractible in $\pp$ and is disjoint from $Q_0$.  By Lemma \ref{lm:CdisjointNCcycle}, $Q_0$ has BOD and has just one non-planar bridge.

As for $\bQ_1$, we consider two cases.  If $\bQ_3$ has BOD, then Lemma \ref{lm:nonconsecHyper} implies $\comp{(M_{\bQ_1})}$ is planar.  If $\bQ_3$ has NBOD, then (\ref{it:cornerInGamma}) implies either $v_9=a$ or $v_4=b$.  In both cases, $\Nuc(M_{\bQ_1})\cap\{a,b\}\ne\varnothing$, so Lemma \ref{lm:nucMeetGamma} implies $\comp{(M_{\bQ_1})}$ is planar, as required.
\end{cproofof}

The following technical corollary of Theorem \ref{th:BODquads} and Lemmas \ref{lm:cleanBOD} and \ref{lm:BODcrossed}  will be used in a few different places later.

}\begin{corollary}\label{co:QibarBOD}  Let $G\in \m2$ and $\hvfg$.  With indices read modulo $5$, suppose, $i\in\{0,1,2,3,4\}$ is such that $\bQ_i$ has BOD and, where $\{j,k\}=\{i+2,i+3\}$, suppose further that $\bQ_{j}$ has NBOD.   Then $s_i$ is in a planar $\bQ_i$-bridge $B_i$ and $\bQ_k$ has BOD.   Moreover, if $e_i$ is any edge of $B_i$ and $D_i$ is a 1-drawing of $G-e_i$, then either $r_{i-1}\,r_i$ crosses whichever of $r_{i+3}$ and $r_{i+6}$ is in $\bQ_j$ or $r_{i+4}\,r_{i+5}$ crosses whichever of $r_{i-2}$ and $r_{i+1}$ is in $\bQ_j$.   \end{corollary}\printFullDetails{

The two possibilities for $D_i$ in the case $j=i+2$ are illustrated in Figure \ref{Di}.

\medskip 
\begin{figure}[h!]
\begin{center}
\scalebox{1.0}{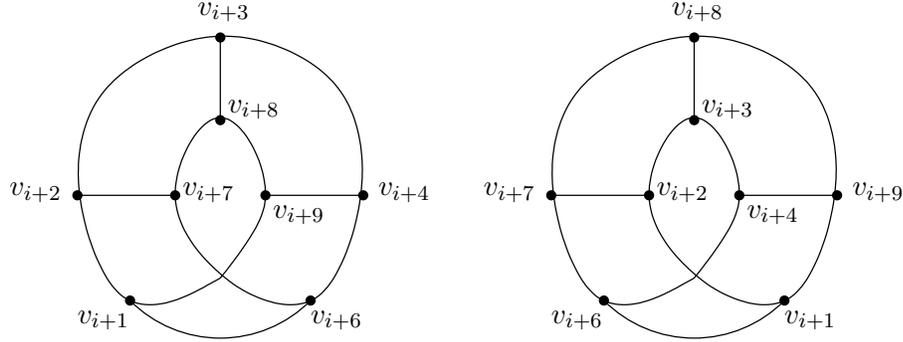}
\end{center}
\caption{The two possibilities for $D_i$ when $j=i+2$.}\label{Di}
\end{figure}

\begin{cproofof}{Corollary \ref{co:QibarBOD}} By way of contradiction, suppose $s_i$ is not in a planar $\bQ_i$-bridge.   \wording{We observe that $s_0$ must be exposed, as otherwise we have the contradiction that,} for every $\ell\in\{0,1,2,3,4\}$, $s_\ell$ is in a planar $\bQ_\ell$-bridge.  \wording{It follows that,} for \dragominor{$\ell\in\{2,3\}$, $s_\ell$} is in a planar \dragominor{$\bQ_\ell$-bridge}.  Thus, $i\notin\{2,3\}$.  

\dragominor{Let $\ell\in\{2,3\}$} be such that $i$ and \dragominor{$\ell$ are} not consecutive in the cyclic order $(0,1,2,3,4)$.  \dragominor{Let $e_\ell$} be the edge \dragominor{of $s_\ell$} incident with \dragominor{$v_\ell$} and let \dragominor{$D_\ell$ be} a 1-drawing \dragominor{of $G-e_\ell$}.  By Lemma \ref{lm:BODcrossed}, \dragominor{$\bQ_\ell$ is} crossed \dragominor{in $D_\ell$}.  

\dragominor{If $\bQ_\ell$} is self-crossed \dragominor{in $D_\ell$}, then \dragominor{$D_\ell$ shows} that the $\bQ_i$-bridge containing $s_i$ is planar.  Thus, we have \dragominor{that $\bQ_\ell$} is not self-crossed \dragominor{in $D_\ell$}.  One of \dragominor{$s_{\ell-1}$ and $s_{\ell+1}$} is exposed \dragominor{in $D_\ell$}.  If this exposed spoke is not also in $\bQ_i$, then again $s_i$ is in a planar $\bQ_i$-bridge; therefore, we must have that the exposed spoke is in $\bQ_i$.  For \wording{the} sake of definiteness, we assume that \dragominor{$s_{\ell-1}$ is} exposed, which implies \dragominor{that $\ell=i+2$}.  

As the only non-planar $\bQ_i$-bridge is $M_{\bQ_i}$, we must have an $H$-avoiding path $P$ from the interior of $s_i$ to the interior of one of \dragominor{$r_{\ell-1}\,r_\ell\,r_{\ell+1}$ and $r_{\ell+4}\,r_{\ell+5}\,r_{\ell+6}$}.    The drawing \dragominor{$D_\ell$ restricts} the possibility to the interior of one of \dragominor{$r_{\ell-1}\,r_\ell$ and 
$r_{\ell+4}\,r_{\ell+5}$}.  But now the embedding in $\pp$ implies $i=0$.  \major{T}his implies $j\in\{2,3\}$; however, neither $\bQ_2$ nor $\bQ_3$ has NBOD.  Therefore, $s_i$ is in a planar $\bQ_i$-bridge.

Because $M_{\bQ_j}-e_i$ and $M_{\bQ_j}$ have the same attachments, $OD_{G-e_i}(\bQ_j)$ and $OD_G(\bQ_j)$ are isomorphic.  As the latter is not bipartite, neither is the former.  By Lemma \ref{lm:cleanBOD}, $\bQ_j$ is not clean in $D_i$.  Thus, either $r_{j-1}\,r_j$ or $r_{j+4}\,r_{j+5}$ is crossed in $D_2$.  These are edge-disjoint from $\bQ_i$.  

Lemma \ref{lm:BODcrossed} implies that $\bQ_i$ is also crossed in $D_i$.  Since $\bQ_i$ is crossed and, from the preceding paragraph, something outside of $\bQ_i$ is crossed, either 
$$r_{i-1}\,r_i\textrm{\quad crosses\quad} r_{i+3}\cup r_{i+6}$$
or 
$$r_{i+4}\,r_{i+5}\textrm{\quad crosses\quad} r_{i-2}\cup r_{i+1}\,,$$as required. \end{cproofof}

Since $\bQ_2$ always has BOD, Corollary \ref{co:QibarBOD} implies at least one of $\bQ_0$ and $\bQ_4$ has BOD.  Together with the fact that, in all cases, at least one of $\bQ_1$ and $\bQ_3$ has BOD, we conclude that at least three of the $H$-hyperquads have BOD.

\minorrem{(Text and definition removed.)}


The last result in this section will be useful early in the next section.  

}\begin{corollary}\label{lm:diagonal}   Let $G\in \m2$ and let $\hvfg$ and suppose $G$ has a representativity 2 embedding in the projective plane, with the standard labelling.   Suppose, for some $i$,  $B$ is an $H$-bridge having an attachment in both $\oo{r_{i-1}\,s_{i-1}}$ and $\oo{r_{n+i}\,s_{i+1}}$.  
\begin{enumerate}\item\label{it:iNe0} if $i\ne 0$, then $B\subseteq \Disc$.
\item\label{it:iEq0} If  $i=0$, then either $\bQ_3$ has NBOD or $B$ consists only of the edge $v_6v_9$.   
\end{enumerate}
\end{corollary}\printFullDetails{

\begin{cproof}   For (\ref{it:iNe0}), we may assume $B\subseteq\Mob$.  The two representativity 2 embeddings of $V_{10}$ in $\pp$ show that $B$ can only be embedded in a face bounded by either $\cc{a,r_9,v_0}\rbsp r_1\,s_1\lbsp\cc{v_6,b,\beta,a}$ or $\cc{b,\beta,a,r_9,v_9}\rbsp s_4\,r_4\lbsp\cc{v_5,r_5,b}$ and that $s_0$ is necessarily exposed in $\pp$.    Notice that $i=0$ in both cases, proving (\ref{it:iNe0}).  

Now assume $i=0$ and suppose $\bQ_3$ has BOD.   
From Theorem \ref{th:BODquads}, we know that $\bQ_2$ also has BOD.  For $j\in\{2,3\}$, let $e_j$ be the edge of $s_j$ incident with $v_j$ and let $D_j$ be a 1-drawing of $G-e_j$.  Because $s_j$ is in a $\bQ_j$-interior bridge, from Lemma \ref{lm:BODcrossed}, we know that $\bQ_j$ is crossed in $D_j$.  

If $\bQ_0$ is clean in $D_j$, then no face of $D_j$ is incident with vertices in both $\oo{r_9\,s_4}$ and $\oo{s_1\,r_5}$.  Therefore, $D_j[B]$ cannot be crossing-free in $D_j$, a contradiction.  Thus, $\bQ_0$ is crossed in $D_j$.  The two possibilities for $D_2$ are shown in Figure \ref{D2}, while the two possibilities for $D_3$ are shown in Figure \ref{D3}.

\begin{figure}[!ht]
\begin{center}
\scalebox{1.2}{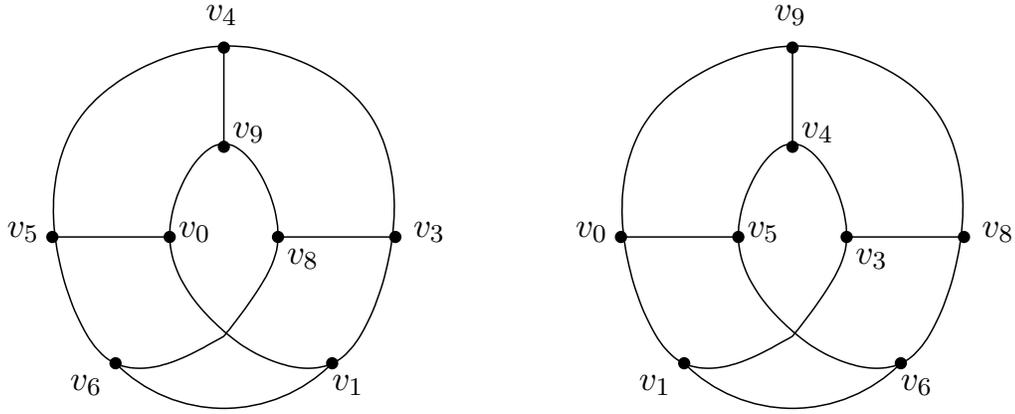}
\end{center}
\caption{The two possibilities for $D_2$.}\label{D2}
\end{figure}

\begin{figure}[!ht]
\begin{center}
\scalebox{1.2}{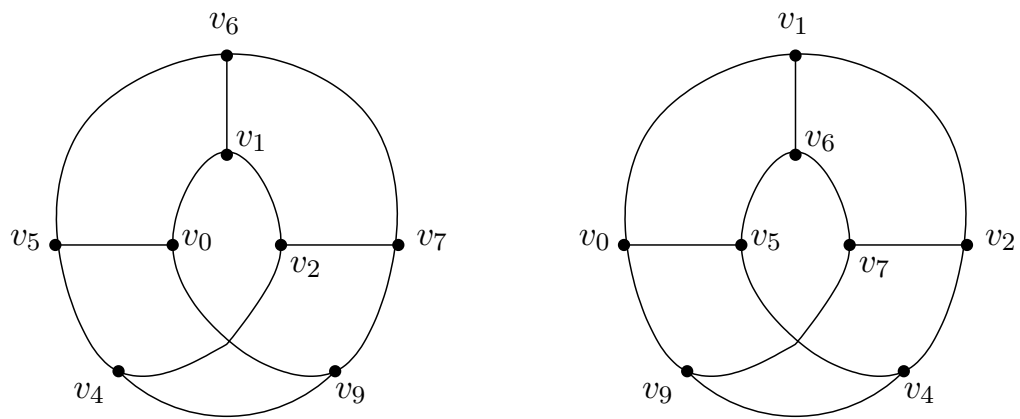}
\end{center}
\caption{The two possibilities for $D_3$.}\label{D3}
\end{figure}

Let $P$ be an $H$-avoiding path in $G$ joining a vertex in each of $\oo{r_9\,s_4}$ and $\oo{s_1\,r_5}$.   The left-hand version of $D_2$ has no face incident with both these paths, and so we must have the right-hand version of $D_2$.  Thus, $D_2$ implies $P$ has one end in $\oc{v_0,r_9,v_9}$ and one end in $\oc{v_1,s_1,v_6}$.   The right-hand version of $D_3$ has no face incident with these paths, so it must be the left-hand version of $D_3$.  The only possibility there for the ends of $P$ are $v_6$ and $v_9$, as claimed.  \end{cproof}
}

\chapter{Green cycles}\printFullDetails{

In this \wording{section,} we begin our study of the rim edges of $H$.  Ultimately, we will partition them into three types:  ``green", ``yellow", and ``red", and it will be the red ones that we focus on to find the desired tile structure.    In this section, however, we begin with the \dragominor{study of} green edges.  We shall show that the cycles $C$ we label green and yellow cannot be crossed in any 1-drawing of $H\cup C$.  

}\begin{definition}\label{df:red}  \dragominor{An edge $e$ of a non-planar graph $G$ is {\em red\/}\index{red} in $G$ if $G-e$ is planar.}\end{definition}\printFullDetails{ 

We will eventually prove that every edge of $R$ is either in a green cycle, or in a yellow cycle, or red.   The main result in this section, one of the three main steps of the entire proof, is that no edge of $R$ is in two green cycles.

}\begin{definition}\label{df:green}  Suppose $G$ is a graph and $\hvfg$.  Suppose further that $G$ is embedded in $\pp$ with  representativity 2 and \wording{that $\Mob$ is} the M\"obius band bounded by \wording{the $H$-rim} $R$.  
\begin{enumerate}\item A cycle $C$ in $G$ is {\em $H$-green\/}\index{green}\index{$H$-green} if $C$ is the composition $P_1P_2P_3P_4$ of four paths, such that:
\begin{enumerate}\item $P_1\subseteq R$ and $P_1$ has length at least 1;
\item $P_2 P_3P_4$ is $R$-avoiding;
\item  $P_2\cup P_4\subseteq H$; 
\item $P_3$ is $H$-avoiding (and, therefore, is either trivial or contained in an $H$-bridge); and
\item\label{it:P1options}  either 
\begin{enumerate} \item $P_1$ contains at most $3$ $H$-nodes or \item\label{it:exceptional} \dragominor{$P_1$ is {\em exceptional\/}}\index{exceptional}, that is, for some $i\in\{0,1,2,\dots,9\}$ and indices read modulo 10, $$P_1=r_i\,r_{i+1}\,r_{i+2}\,.$$
\end{enumerate}
\end{enumerate}
\item An edge of $R$ is {\em $H$-green\/} if it is in an $H$-green cycle.
\item A vertex $v$ of $R$ is {\em $H$-green\/} if both edges of $R$ incident with $v$ are in the same $H$-green cycle.
\end{enumerate}
\end{definition}\printFullDetails{

There is a natural symmetry between $P_2$ and $P_4$:  if $C$ is an $H$-green cycle, consisting of the composition $P_1P_2P_3P_4$ as in Definition \ref{df:green}, then $P_1^{-1}P_4^{-1}P_3^{-1}P_2^{-1}$ is another $H$-green cycle.  Thus $P_4^{-1}$ and $P_2$ can both be considered to be $P_2$.  As the orientations of the individual $P_i$ will not be of any importance (except in as much as they are required to make $C$ a cycle), we may say $P_2$ and $P_4$ are symmetric.

\dragominor{Note that the exceptional case \ref{it:exceptional} is the only one in which $P_1$ has $4$ $H$-nodes.} 

}\begin{lemma}\label{lm:greenBasics}   Suppose $G$ is a graph and $\hvfg$.  \wording{Let $C$ be any $H$-green cycle expressed as the composition $P_1P_2P_3P_4$ as in Definition \ref{df:green}.
\begin{enumerate}
\item\label{it:p2p4spoke} If $i\in\{2,4\}$, then $P_i$ has an end in $R$ and is either trivial or contained in an $H$-spoke.
\item\label{it:p3} The path $P_3$ is not trivial.
\item\label{it:differentSpokes}  If $P_2$ and $P_4$ are both non-trivial, then they are contained in different $H$-spokes.
\end{enumerate}}
\end{lemma}\printFullDetails{

\begin{cproof} {\bf (\ref{it:p2p4spoke})} For sake of definiteness, we assume $i=2$.  If $P_2$ is not trivial, then there is an edge $e$ in $P_2$.  From the definition, $e$ is in $H$ but not in $R$.  Therefore, there is a spoke $s$ containing $e$.  If $P_2$ has a vertex $u$ not in $s$, then $P_2$ is a path contained in $H$ and containing $e$ and $u$.  This implies that one end of $s$, a vertex of $R$, is internal to $P_2$, contradicting the fact that $P_2P_3P_4$ is $R$-avoiding.  So $P_2\subseteq s$, as required.  Since $P_1\subseteq R$ and $P_2$ has an end in common with $P_1$, $P_2$ has an end in $R$.

{\bf (\ref{it:p3})}\ \  Suppose $P_3$ is trivial.   Then $P_2P_4$ is an $R$-avoiding path joining the ends of $P_1$.  Each of $P_2$ and $P_4$ is either trivial or in a spoke and, since $P_2P_4$ is $R$-avoiding, either both are trivial or $P_2P_4$ is contained in a single spoke.   If both are trivial, then $P_1$ is the cycle $P_1P_2P_3P_4$, which is impossible, since $P_1$ is properly contained in the cycle $R$.   Each of $P_2$ and $P_4$ has an end in $R$ (or is trivial) and $P_2P_4$ has both ends in common with $P_1$, so $P_2P_4$ is the entire spoke. But then $P_1$ contains six $H$-nodes, a contradiction.

{\bf(\ref{it:differentSpokes})}  For $j=2,4$, $P_j$ is  non-trivial by hypothesis.  Therefore,  (\ref{it:p2p4spoke}) shows it is contained in an $H$-spoke $s$.  As it  has a vertex in common with $P_1$, $P_j$ has a vertex in $R$.  This vertex is an $H$-node incident with $s$.  If $P_2$ and $P_4$ are contained in the same spoke $s$, then, as in the proof of (\ref{it:p3}), they contain different $H$-nodes.  But then $P_1$ contains six $H$-nodes,  contradicting Definition \ref{df:green}.
\end{cproof}

There is a small technical point that must be dealt with before we can successfully analyze the relation of an $H$-green cycle to the embedding of $G$ in $\pp$.

}\begin{definition}\label{df:friendly}  Let $\Pi$ be a representativity 2 embedding of a graph $G$ in $\pp$ and let $\hvfg$.    Then $\Pi$ is {\em $H$-friendly\/}\index{friendly}\index{$H$-friendly} if, for each $H$-green cycle $C$ of $G$ and any non-contractible simple closed curve $\gamma$ in $\pp$ meeting $\Pi(G)$ in precisely two points, $\Pi[C]$ is contained in the closure of some face of $\Pi[H]\cup \gamma$.  \end{definition}\printFullDetails{

}\begin{lemma}\label{lm:greenCyclesFriendly}  Suppose $G\in\m2$ and $\hvfg$.  Let $\Pi$ be any representativity 2 embedding of $G$ in $\pp$, let $\gamma$ be a non-contractible simple closed curve  in $\pp$ meeting $\Pi(G)$ in \wording{precisely two points}, and let $C$ be an $H$-green cycle in $G$.  Give $H$ the standard labelling relative to $\gamma$.
\begin{enumerate}
\item\label{it:nearlyFriendly}  Either $\Pi[C]$ is contained in the closure of some face of $\Pi[H]\cup \gamma$ or $v_6v_9$ is an edge of $G$ embedded in $\Mob$ and $C=r_6\,r_7\,r_8\lbsp\cc{v_9,v_6v_9,v_6}$.  In particular, if $\Pi[H]\subseteq \Mob$, then $\Pi$ is $H$-friendly.
\item\label{it:friendlyEmbedding}
If $\Pi$ is not $H$-friendly, then there is an $H$-friendly embedding of $G$ in $\pp$ obtained from $\Pi$ by reembedding only $v_6v_9$.   
\item\label{it:alwaysFriendly} \dragominor{In particular, there is an $H$-friendly embedding of $G$ in $\pp$.}
\end{enumerate}
\end{lemma}\printFullDetails{

\begin{cproof}  Suppose $\Pi[C]$ is not contained in the closure of any face of $\Pi[H]\cup \gamma$ and let $P_1P_2P_3P_4$ be the decomposition of $C$ as in Definition \ref{df:green}.   As $P_3$ is $(H\cup \gamma)$-avoiding \dragominor{and non-trivial by}  Lemma \ref{lm:greenBasics} \dragominor{(\ref{it:p3}),  there} is an $(H\cup \gamma)$-face $F_3$ containing $P_3$.  Note that, if $P_2$ is not trivial, then Lemma \ref{lm:greenBasics} (\ref{it:p2p4spoke}) \dragominor{asserts} it is contained in an $H$-spoke $s$ and it contains an end of $P_3$, so $P_2$ is contained in the boundary of $F_3$.  Likewise for $P_4$.    We assume by way of contradiction that $P_1\not\subseteq \cl(F_3)$.

\begin{claim}\label{cl:p1IsV6v9} Then:
\begin{enumerate}\item $P_1=r_6\,r_7\,r_8$; \item $s_0$ is exposed; \item either $a=v_9$ or $b=v_6$; and \item if $F_3\subseteq \Disc$, then both $v_6=b$ and $v_9=a$.  \end{enumerate}\end{claim}

\begin{proof}  We first consider the case $F_3\subseteq \Disc$.  Both ends of $P_1$ are contained in one of the $ab$-subpaths of $R$.  If $P_1$ is not contained in the boundary of $F_3$, then it must contain the other complete $ab$-subpath of $R$.  As each of these has at least $4$ $H$-nodes, the only possibility is that it is precisely $4$ $H$-nodes.  In this case, $P_1$ must be exceptional and $s_0$ must be exposed.  In particular, $P_1=r_{6}\,r_{7}\,r_{8}$ and $P_3$ has ends $v_6$ and $v_9$.   The paths $P_2$ and $P_4$ are both trivial.  Moreover, as $P_1$ is not incident with $F_3$, we must have $v_{6}=b$ and $v_{9}=a$.  

In the other case, $F_3\subseteq \Mob$.  If $F_3$ is contained in the interior of an $H$-quad, then $P_1$ joins two vertices in the same quad and is not contained in the quad.  In this case, $P_1$ must have at least $5$ $H$-nodes, which is impossible.  Therefore, $F_3$ is not contained in the interior of an $H$-quad, and so is bounded by one of $\cc{a,r_{9},v_0}\rbsp r_0\,s_1\lbsp \cc{v_{6},r_5,b,\beta,a}$ and $\cc{a,r_9,v_9}\rbsp s_4\,r_{4}\lbsp\cc{v_5,r_5,b,\beta,a}$.    (Recall $\beta=\gamma\cap\Mob$.)  Notice that $s_0$ is exposed.

These cases are symmetric; for sake of definiteness, we presume $F_3$ is bounded by $\cc{a,r_{9},v_0}\rbsp r_0\,s_1\lbsp\cb{v_{6},r_5,}$ $\bc{b,\beta,a}$.  The path $P_1$ has at most $4$ $H$-nodes and joins two vertices on $\bQ_0$.  If $P_1\subseteq \bQ_0$, then $\Pi[C]$ is contained in the closure of one of the two $(\Pi[H]\cup \gamma)$-faces whose boundary is contained in $\Pi[\bQ_0]\cup\gamma$;  thus, $P_1\not\subseteq \bQ_0$.  Therefore, $P_1$ has at least $4$ $H$-nodes; by definition it has at most 4, so $P_1$ has precisely 4 $H$-nodes.  In particular, $P_1$ can only be $r_{6}\,r_7\,r_{8}$ and $v_{9}=a$.  \end{proof}

Because $s_0$ is exposed, Theorem \ref{th:BODquads} implies that both $\bQ_2$ and $\bQ_3$ have  BOD.   Let $e$ be any edge in $s_2$ and let $D_2$ be a 1-drawing of $G-e$.  Since $\bQ_2$ has BOD, Lemma \ref{lm:BODcrossed} shows $\bQ_2$ is crossed in $D_2$, so $r_0\,r_1\,r_2\,r_3$ crosses $r_5\,r_6\,r_7\,r_8$.  This implies that neither $s_0$ nor $s_4$ is exposed in $D_2$ and, therefore, $P_3$ cannot be in the same $(H-\oo{s_0})$-bridge as $s_0$.   

Let $B_0$ and $B$ be the $(H-\oo{s_0})$-bridges containing $s_0$ and $P_3$, respectively.  These evidently overlap on $\bQ_0$ and they both overlap $M_{\bQ_0}-e$ (in $G-e$).  Therefore, $\bQ_0$ has NBOD.  Since $M_{\bQ_0}-e$ is a non-planar $\bQ_0$-bridge in $G-e$, Lemma \ref{lm:cleanBOD} implies that $\bQ_0$ is not clean in $D_2$.  

As $\bQ_0$ and $\bQ_2$ have only $s_1$ in common and both are crossed in $D_2$, $s_1$ must be exposed in $D_2$.   It follows that $D_2[P_3]$ is in the face of $D_2[H-\oo{s_2}]$ bounded by $s_1\,r_{6}\,r_7\,r_{8}\,r_{9}\,r_0$. 

The same arguments apply with $\bQ_3$ in place of $\bQ_2$, showing that $D_3[P_3]$ is in the face of $D_3[H-\oo{s_3}]$ bounded by $s_{4}\,r_{4}\,r_5\,r_6\,r_{7}\,r_{8}$.  These two drawings imply that $\att( B )\subseteq r_{6}\,r_{7}\,r_{8}$.  

If $F_3\subseteq \Disc$, then $F_3$ is bounded by $r_{9}\,s_0\,r_5\lbsp\cc{v_{6},\alpha,v_{9}}$ (recall $\alpha = \gamma\cap \Disc$).  Thus, $\att( B ) = \{v_6,v_9\}$ and Lemma \ref{lm:threeAtts} implies that $P_3$ is just the edge $v_{6}v_{9}$.  In this case, Claim \ref{cl:p1IsV6v9} implies $P_3$ can obviously be embedded in the other face of $H\cup \gamma$ contained in $\Disc$ and incident with both $v_{6}$ and $v_{9}$.

If $F_3\subseteq \Mob$, then $F_3$ is bounded by either $$\cc{a,r_9,v_0}\rbsp r_0\,s_1\lbsp\cc{v_6,r_5,b,\beta,a}\textrm{ or }\cc{a,r_9,v_9}\rbsp s_4\,r_4\cc{v_5,r_5,b,\beta,a}\,.$$  Again, this implies that $\att( B )\subseteq \{v_6,v_9\}$, so $P_3$ is just the edge $v_6v_9$.  In this case, Claim \ref{cl:p1IsV6v9} implies only that either $v_6 = b$ or $v_9 = a$.  Again these cases are symmetric, so we assume $v_9=a$.

\minorrem{(Text and definition removed.)}


We remark that if $v\in A\cap B$, then $(v)$ is an $AB$-path and this is the only path containing $v$ that is an $AB$-path.  We now return to the proof.

We wish to reembed $v_6v_9$ in the $(H\cup \gamma)$-face incident with $v_6$, $v_7$, $v_8$, and $v_9$.  We need only verify that there is no $H$-avoiding $\co{b,r_5,v_6}\oc{v_6,r_6,v_7,r_7,v_8,r_8,v_9}$-path.     But such a path would have to appear in $D_3$, where it can only also be in the face of $D_3[H-\oo{s_3}]$ bounded by $s_{4}\,r_{4}\,r_5\,r_6\,r_{7}\,r_{8}$.  But then it crosses $v_6v_9$ in $D_3$, a contradiction completing the proof. \end{cproof}

We are now prepared for our analysis of $H$-green cycles.

}\begin{lemma}\label{lm:greenCycles}  Let $G\in\m2$, $\hvfg$, and let $\Pi$ be an $H$-friendly embedding of $G$ in $\pp$.  Let $C$ be an $H$-green cycle expressed as the composition $P_1P_2P_3P_4$ as in Definition \ref{df:green}.  Then:
\begin{enumerate} 
\item\label{it:p1} $P_1$ is contained in one of the two $ab$-subpaths of $R$;
\item \label{it:Cprebox1} if $C\subseteq \Mob$ and $s$ is any $H$-spoke contained in $\Mob$ that is totally disjoint from $C$, then $C$ is a $(C\cup (H-\oo{s}))$-prebox;
\item\label{it:Cprebox2} if $C$ is not contained in $\Mob$ and $s$ is any $H$-spoke contained in $\Mob$ having one end in the interior of $P_1$, then $C$ is a $(C\cup (H-\oo{s}))$-prebox;
\item\label{it:HinCbridge} there is a $C$-bridge $M_C$ so that $H\subseteq C\cup M_C$;
\item\label{it:BODplanar} $C$ is contractible, $C$ has BOD, and all $C$-bridges other than $M_C$ are planar;
\item\label{it:prebox} $C$ is a $(C\cup H)$-prebox;
\item\label{it:oneCbridge} $M_C$ is the unique $C$-bridge (that is, there are no planar $C$-bridges); 
\item\label{it:CboundsFace}  $C$ bounds a face \dragominor{of $\Pi$}; 
\item\label{it:shortJump} there are at most two $H$-nodes in the interior of $P_1$; and
\item\label{it:notCrossed}  in any 1-drawing of $H\cup C$, $C$ is clean.
\end{enumerate}
\end{lemma}\printFullDetails{

\begin{cproof}
Because $\Pi$ is $H$-friendly, there is a face $F$ of $(H\cup \gamma)$ whose closure contains $C$.

\medskip
{\bf(\ref{it:p1})}\ \ This is an immediate consequence of Definition \ref{df:friendly}, as the boundary $\partial$ of any face of $H\cup \gamma$ has each component of $\partial\cap R$ contained in one of the $ab$-subpaths of $R$.

{\bf (\ref{it:Cprebox1}) and (\ref{it:Cprebox2})}\ \ Note that $H-\oo{s}$ contains a subdivision of $V_8$.  In particular, if $e$ is an edge of $C$ not in $R$, then $H-\oo{s}$ is a non-planar subgraph of $(C\cup (H-\oo{s}))-e$, as required.  If $e\in C$ is in $R$, then we claim the cycle $R'=(R-\oo{P_1})\cup P_2P_3P_4$ is the rim of a $V_6$.  We see this in the two cases.

\medskip\noindent{\bf Case 1:  (\ref{it:Cprebox1})}\ \   
In this case, there are three $H$-spokes $t_1,t_2,t_3$ other than $s$ contained in $\Mob$.  Each $t_i$ has an end $v_i$ in $R-\oo{P_1}$ and a maximal $R'$-avoiding subpath $t'_i$ containing $v_i$.  It is straightforward to verify that $R'\cup t'_1\cup t'_2\cup t'_3$ is a subdivision of $V_6$, as required.

\medskip\noindent{\bf Case 2: (\ref{it:Cprebox2})}\ \ In the exceptional case $P_1=r_i\,r_{i+1}\,r_{i+2}$, $s$ is different from all of $s_i$, $s_{i+3}$, and $s_{i+4}$, so $R'\cup s_i\cup s_{i+3}\cup s_{i+4}$ is the required $V_6$.  (Note that one of $s_i$ and $s_{i+3}$ can be the exposed spoke and part of that spoke might be in either $P_2$ or $P_4$, but whatever part is not in $P_2\cup P_4$ makes the third spoke.)

In the remaining case, there are two $H$-spokes $s_i$ and $s_{i+1}$ that are completely disjoint from $C$.  Any other $H$-spoke $s'$, different from $s$, $s_i$, and $s_{i+1}$, and contained in $\Mob$, will connect to $R'$ to make a third spoke, either because both its ends are in $R'$ or because one end is in $R'$ and the other end is in $P_1$ and one of the paths in $P_1-e$ joins the other end of $s$ to a vertex in $R'$.

{\bf(\ref{it:HinCbridge})}\ \  Let $M_C$ be the $C$-bridge containing the $ab$-subpath $Q$ of $R$ that is $P_1$-avoiding. We claim $H\subseteq C\cup M_C$.  Observe that the maximal $P_1$-avoiding subpath $Q'$ of $R$ containing $Q$ is contained in $M_C$ and, therefore, $R\subseteq C\cup M_C$.  Note that every $H$-spoke has at least one end in $Q'$ that is not in $P_1$ and, therefore, that end is in $\Nuc(M_C)$.  Thus, if $P_3$ is not contained in $\Mob$, it is obvious that $H\subseteq C\cup M_C$.  So suppose $P_3$ is contained in $\Mob$.  \wording{The $H$-spokes} other than those that contain $P_2$ and $P_4$ are obviously in $M_C$, and the ones containing $P_2$ and $P_4$ are in the union of $M_C$ and $C$.

{\bf (\ref{it:BODplanar})}\ \  If either $P_1$ has at most $3$ $H$-nodes, or $s_0$ is not exposed, or $P_1$ is neither $r_1\,r_2\,r_{3}$ nor $r_{6}\,r_{7}\,r_{8}$, then there is an $H$-spoke $s$ contained in $\Mob$ and totally disjoint from $C$.  
The spoke $s$ combines with the one of the two subpaths of $R$ joining the ends of $s$ that is disjoint from $P_1$ to give a non-contractible cycle disjoint from $C$.  The claim now follows immediately from Lemma \ref{lm:CdisjointNCcycle}.

We now treat the case $s_0$ is exposed and $P_1$ is either $r_1\,r_2\,r_{3}$ or $r_{6}\,r_{7}\,r_{8}$.   In this case, $F$ is a face of $H\cup \gamma$ contained in $\Disc$.  Let 
 $B'$ be a $C$-bridge other than $M_C$.  If $B'\subseteq\cl(F)$, then $C\cup B'\subseteq \cl(F)$ and $\cl(F)$ is a closed disc in $\pp$.  Therefore, $C\cup B'$ has no non-contractible cycles in $\pp$.  Otherwise, $B'$ is contained in the closure of one of the $H$-faces bounded by $Q_1$ or $Q_2$ or $Q_3$. For each $i\in \{1,2,3\}$, let $F_i$ be the $H$-face bounded by $Q_i$.  Then $\cl(F_i)\cap \cl(F)$ is a path and, therefore, $\cl(F_i)\cup \cl(F)$ is a closed disc containing $C\cup B'$ and again $C\cup B'$ has no non-contractible cycles.  The result now follows from Corollary \ref{co:contractibleBOD}.

{\bf(\ref{it:prebox})}\ \ In the case $P_3\subseteq \Mob$, at most the $H$-spokes containing $P_2$ and \wording{$P_4$} meet $C$.  There are at least two others contained in $\Mob$ that are disjoint from $C$; let $s$ be one of these.  By (\ref{it:Cprebox1}), for any edge $e$ of $C$, $(C\cup (H-\oo{s}))-e$ is not planar, so $(C\cup H)-e$ is not planar.

Now suppose $P_3\subseteq\Disc$.  If some $H$-spoke $s$ contained in $\Mob$ has an end in the interior of $P_1$, then (\ref{it:Cprebox2}) implies that, for any edge $e$ of $C$, $(C\cup (H-\oo{s}))-e$ is not planar, so $(C\cup H)-e$ is not planar.  

In the alternative, no $H$-spoke contained in $\Mob$ has an end in the interior of $P_1$.  If $e$ is not in $P_1$, then $H\cap \Mob$, which is a $V_8$ or $V_{10}$, is contained in $(C\cup H)-e$, so we may assume $e\in P_1$.   But then $(R-\oo{P_1})\cup P_2P_3P_4$ and the $H$-spokes contained in $\Mob$ make a $V_8$ or $V_{10}$, showing $(C\cup H)-e$ is not planar.




{\bf (\ref{it:oneCbridge})}\ \  Observe that (\ref{it:BODplanar}) shows any other $C$-bridge is planar and that $C$ has BOD.  If $B$ is any other $C$-bridge, then $C$ is a $\comp B$-prebox by (\ref{it:prebox}) and, therefore, is\dragominor{, by definition,} a box, contradicting Lemma \ref{lm:noBox}.

{\bf (\ref{it:CboundsFace})}\ \ This is an immediate consequence of the facts that $C$ is contractible (\ref{it:BODplanar}) and there is only one $C$-bridge (\ref{it:oneCbridge}).

{\bf (\ref{it:shortJump})}\ \  Suppose by way of contradiction that $v_{i-1},v_i,v_{i+1}$ are internal to $P_1$.  Notice that $P_1$ is not exceptional.  We claim that $\bQ_i$ is a box, contradicting Lemma \ref{lm:noBox}.   

For $s\in\{s_{i-1},s_i,s _{i+1}\}$,  $s$ is contained in one of the two faces of $R$ (i.e., the M\"obius band $\Mob$ and the disc $\Disc$).  By (\ref{it:CboundsFace}), $C$ is the boundary of some face $F$ of $G$.  Clearly $F$ and $s$ are in different $R$-faces, so one is in $\Mob$ and the other is in $\Disc$.  Therefore, all of $s_{i-1}$, $s_i$, and $s_{i+1}$ are contained in the same one of $\Mob$ and $\Disc$.  Since $\Disc$ contains at most one $H$-spoke, it must be that all three are contained in $\Mob$.  Clearly, this implies $F\subseteq \Disc$ and, therefore, $P_2P_3P_4\subseteq \Disc$.

There is another $H$-spoke $s$ contained in $\Mob$ that is totally disjoint from $\bQ_i$. As $P_2P_3P_4\subseteq \Disc$,  $R\cup P_2P_3P_4\cup s$ contains a non-contractible cycle including both $P_2P_3P_4$ and $s$ that is totally disjoint from $\bQ_i$.  Thus, Lemma \ref{lm:CdisjointNCcycle} implies $\bQ_i$ has BOD and all $\bQ_i$-bridges except $M_{\bQ_i}$ are planar.

We claim $\bQ_i$ is a $(\bQ_i\cup M_{\bQ_i})$-prebox.  Note that $\bQ_i\cup M_{\bQ_i}$ contains $H-\oo{s_i}$ and so the deletion of any edge in $s_{i-1}\cup s_{i+1}$ leaves a $V_6$.  
By (\ref{it:Cprebox2}), $C$ is a $C\cup (H-\oo{s_i})$-prebox, so the deletion of any edge $e$ in $r_{i-1}\cup r_i$ leaves a non-planar subgraph in $(C-e)\cup (H-\oo{s_i})$, which is contained in  $(\bQ_i-e)\cup M_{\bQ_i}$.  That is, if $e\in r_{i-1}\cup r_i$, then $(\bQ_i-e)\cup M_{\bQ_i}$ is not planar.  


We must also consider an edge in $r_{i+4}\cup r_{i+5}$ (these indices are read modulo 10).  Let $R'$ be the cycle made up of the following four parts:  the two paths in $R-\oo{P_1}-\oo{r_{i+4}\, r_{i+5}}$, $P_2P_3P_4$, and $s_{i-1}\,r_{i-1}\,r_i\,s_{i+1}$.  To get the $V_6$, add to $R'$  both $H$-spokes totally disjoint from $P_1$ and either of the two $R'$-avoiding subpaths of $P_1$ whose ends are in $R'$. Thus, if $e\in r_{i+4}\, r_{i+5}$, then $(\bQ_i-e)\cup M_{\bQ_i}$ is not planar, completing the proof that $\bQ_i$ is a $(\bQ_i\cup M_{\bQ_i})$-prebox.   \wording{(See Figure \ref{fg:qibarPrebox}.)}

\begin{figure}[!ht]
\begin{center}
\scalebox{1.0}{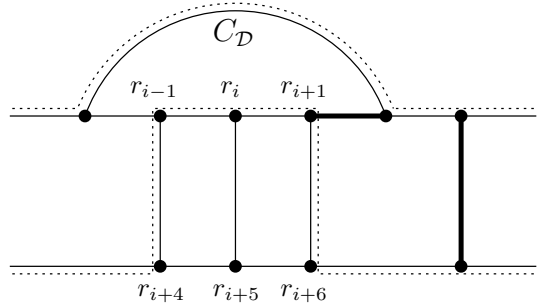}
\end{center}
\caption{\minor{The case $e\in r_{i+4}\,r_{i+5}$ for $\bar Q_i$ being a $(\bar Q_i\cup M_{\bar Q_i})$-prebox. Only two of the three spokes are shown.}}\label{fg:qibarPrebox}
\end{figure}

Since the $\bQ_i$-bridge $B$ containing $s_i$ is contained in the closed disc in $\pp$ bounded by $\bQ_i$, $B$ is planar and, therefore, $\bQ_i$ is a box, the desired contradiction. 

{\bf (\ref{it:notCrossed})} Let $D$ be a 1-drawing of $H\cup C$.  Let $P_1P_2P_3P_4$ be the decomposition of $C$ into paths as in Definition \ref{df:green}, so $P_1\subseteq R$ and $P_3$ is $H$-avoiding.  If $C$ is crossed in $D$, then it is $P_1$ that is crossed, while $P_2P_3P_4$, being $R$-avoiding, is not crossed in $D$.  We claim that  there is an $H$-spoke $v_iv_{i+5}$ disjoint from $C$ that is not exposed in $D$.    The existence of $s$ \dragominor{and the fact that $C$ is crossed in $D$} shows that no face of $R\cup s$ is incident with both ends of $P_1$ and, therefore, $P_2P_3P_4$ must cross $R\cup s$ in $D$, the desired contradiction.

To prove the claim, we consider two cases.  If $P_1$ has at most $3$ $H$-nodes, then this is obvious, since only one $H$-spoke can be exposed.  In the alternative, $P_1$ is exceptional, say $P_1=r_i\,r_{i+1}\,r_{i+2}$.   As \wording{the spoke exposed in $D$} is incident with an end of the $H$-rim branch that is crossed, we see that $s_{i+4}$ is not the exposed spoke and is disjoint from $P_1$, as required.
\end{cproof}
}

\printFullDetails{
The next result is the main result of this section and the first of three main steps along the way to obtaining the classification of 3-connected, 2-crossing-critical graphs having a subdivision of $V_{10}$.   The other two major steps are, for $G\in \m2$ containing a subdivision $H$ of $V_{10}$:  (i) $G$  has a representativity 2 embedding in $\pp$ so that $H\subseteq \Mob$; and (ii) $G$ contains a subdivision of $V_{10}$ with additional properties (that we call ``tidiness").  It is this tidy $V_{10}$ for which the partition of the edges of the rim into the red, yellow, and green edges \wording{that allows us to} find the decomposition into tiles.

}\begin{theorem}\label{th:twoGreenCycles}  If $G\in\m2$ and $\hvfg$, then no two $H$-green cycles have an edge of $R$ in common.  \end{theorem}\printFullDetails{

\def\Disc{\mathfrak D}
\def\disc{\mathfrak D}
\def\mob{\mathfrak M}
\def\Mob{\mathfrak M}

\begin{cproof}  \dragominor{Suppose $e_0\in R$ is in distinct $H$-green cycles.}
By Lemma \ref{lm:greenCyclesFriendly} (\ref{it:alwaysFriendly}), there is an $H$-friendly embedding $\Pi$ of $G$ in $\pp$.   By Lemma \ref{lm:greenCycles} (\ref{it:CboundsFace}), any $H$-green cycle bounds a face of $\Pi[G]$. 
 \dragominor{As $e_0$ is} in $R$ and $R$ is the boundary of both the (closed) M\"obius band $\Mob$ and the (closed) disc $\Disc$, one of these faces, call it $F_{\mob}$, is contained in $\Mob$, while the other, call it $F_{\disc}$, is contained in $\Disc$.  For \dragominor{$\eye\in\{\mob,\disc\}$}, let  \dragominor{$C_\eye$} be the green cycle bounding \dragominor{$F_\eye$} and let \dragominor{$P_1^\eye P_2^\eye P_3^\eye P_4^\eye$} be the path decomposition of \dragominor{$C_\eye$} as in Definition \ref{df:green}; in particular, \dragominor{$P_1^\eye\subseteq R$ and $P_3^\eye$ is} $H$-avoiding.  

Note $P_2^\disc P_3^\disc P_4^\disc$ is disjoint from $\Mob$ (except for its ends) and $P_2^\mob P_3^\mob P_4^\mob$ is contained in $\Mob$.
Thus, $C_\Disc\cap C_\Mob= P_1^\disc\cap P_1^\mob$.  Lemma \ref{lm:greenCycles} (\ref{it:shortJump}) implies that, for \dragominor{$\eye\in\{\mob,\disc\}$, $P_1^\eye$} has at most $4$ $H$-nodes.  We conclude that $P_1^\disc\cup P_1^\mob$ is not all of $R$, and so $C_\Disc\cap C_\Mob$  is a path.  Therefore, there is a unique cycle $C$ in $C_\Disc\cup C_\Mob$ not containing $e_0$ and, furthermore, $C$ bounds a closed disc in $\pp$ having $e_0$ in its interior.

On the other hand, Lemma \ref{lm:greenCycles} (\ref{it:p1}) shows there is an $ab$-subpath $A_1$ of $R$ that contains $P_1^\disc$.  Since $e_0\in P_1^\disc\cap P_1^\mob$, it is also the case that $P_1^\mob\subseteq A_1$.  Let $A$ be the other $ab$-subpath of $R$, so that $A$ is $(C_\disc\cup C_\mob)$-avoiding.  In particular, there is a $C$-bridge $M_C$ containing $A$. By Lemma \ref{lm:greenCycles} (\ref{it:oneCbridge}),  for \dragominor{$\eye\in\{\Mob,\Disc\}$},  $A$ is in the unique \dragominor{$C_\eye$-bridge $M_{C_\eye}$.  Since $M_{C_\eye}$ (and therefore $A$) is not contained in the face of $G$ bounded by $C_\eye$}, we conclude that $A$ is not in the disc bounded by $C$.  Therefore, $M_C$ is different from the $C$-bridge $B_C$ containing $e_0$. 

\begin{claim}\label{cl:spokeEndNotInUnion}  For each $H$-spoke $s$, some $H$-node incident with $s$ is not in $C_\mob\cup C_\disc$.  \end{claim}

\begin{proof}  By Lemma \ref{lm:greenCycles} (\ref{it:shortJump}), there exists an $i$ so that $P_1^\disc\subseteq r_i\,r_{i+1}\,r_{i+2}$.  In particular, $e_0$ is in $r_i\cup r_{i+1}\cup r_{i+2}$.  Thus, $P_1^\mob$ has an edge in at least one of $r_i$, $r_{i+1}$, and $r_{i+2}$.

Lemma \ref{lm:greenCycles} (\ref{it:CboundsFace}) implies that $C_\mob$ bounds a face of $G$.  Therefore, $C_\mob$ is contained in the closure $\cl(F)$ of a face $F$ of $\Pi[H]$ and $F\subseteq\Mob$.  Thus, $P_1^\mob$ is contained in one of the two components of $\cl(F)\cap R$.  Since such a component is contained in consecutive $H$-rim branches, if $P_1^\mob$ contains an edge in $r_j$, then $P_1^\mob$ is contained in either $r_{j-1}\,r_j$ or $r_j\,r_{j+1}$.    \minor{From the preceding paragraph, $P_1^\mob$ is contained in one of $r_{i-1}r_i$, $r_ir_{i+1}$, $r_{i+1}r_{i+2}$, and $r_{i+2}r_{i+3}$.}

We conclude that $P_1^\disc\cup P_1^\mob$ is contained in either $r_{i-1}\,r_i\,r_{i+1}\,r_{i+2}$ or $r_i\,r_{i+1}\,r_{i+2}$ $r_{i+3}$ showing that no $H$-spoke has both ends in $P_1^\disc\cup P_1^\mob$.
\end{proof}

\begin{claim}\label{cl:HinCMB} (1) $H\subseteq C\cup M_C\cup B_C$. 

(2) If $s$ is an $H$-spoke contained in $\Mob$ disjoint from $C_{\mob}$, then $(C\cup M_C)-\oo{s}$ is not planar.\end{claim}

\begin{proof} For (1), we note that it is clear that $R\subseteq C\cup M_C\cup B_C$.  Now let $s$ be an $H$-spoke.  Suppose first that $s\subseteq \Mob$.  By Claim \ref{cl:spokeEndNotInUnion}, there is an $H$-node $v$ incident with $s$ and not in $C_\mob\cup C_\disc$.  If $s\cap C_\mob$ is at most an end of $s$, then it is evident that $s\subseteq M_C$.  If $s\cap C_\mob$ is more than just an end of $s$, then $s$ consists of a $C_\mob$-avoiding subpath $s'$ joining $v$ to a vertex in $C_\mob$, together with the path $C_\mob\cap s$ (which is by Lemma \ref{lm:greenBasics} (\ref{it:p2p4spoke})) either $P_2^\mob$ or $P_4^\mob$).   But then it is again evident that $s\subseteq C\cup M_C$.  

Otherwise, $s$ is exposed, in which case we have the same argument, but replacing $C_\mob$ with $C_\disc$, completing the proof of (1).

For (2), a $V_6$ is found whose rim is $(R-\oo{P_1^\mob})\cup P_2^\mob P_3^\mob P_4^\mob$.  The spokes are \dragominor{contained in} the three other spokes in $\Mob$, \dragominor{namely they are the} parts that are not in $P_2^\mob\cup P_4^\mob$.   \end{proof}

\begin{claim}\label{cl:ChasBOD} $C$ has BOD.    \end{claim}

\begin{proof}  Let $S$ be the set of $H$-spokes contained in $\Mob$ and disjoint from $C_\mob$.  As $C_\mob$ meets at most two $H$-spokes in $\Mob$, $|S|\ge 2$.     If some $s\in S$ is also disjoint from $C_\disc$, then $R\cup s$ contains a non-contractible cycle disjoint from $C$, in which case Lemma \ref{lm:CdisjointNCcycle} shows $C$ has BOD, as claimed.

So we may assume that no element of $S$ is also disjoint from $C_\disc$.  
%
%
 Let $s$ be any element of $S$; then $s\cap C_\disc$ is a vertex $v$ of $P_1^\disc$.   Let $e$ be the edge of $s$ incident with $v$.  
In order to show that $C$ has BOD, we will show that\dragominor{: (i) the overlap diagrams $OD_{G-e}(C)$ and $OD_G(C)$ are the same; and (ii)} $OD_{G-e}(C)$ is bipartite.   For \dragominor{(i)}, note that $C_\disc$ bounds a face in $\pp$ and that $\oo{s}$ is in the boundary of two $(H\cup \gamma)$-faces.  Thus, there can be no $C$-bridge that overlaps $M_C$ in $G$ because of its attachment at $v$.   That is, $OD_{G-e}(C)$ and $OD_G(C)$ are the same.

\dragominor{For (ii)},  Lemma \ref{lm:greenCycles} (\ref{it:Cprebox1}) applied to $C_\mob$ and (\ref{it:Cprebox2}) applied to $C_\disc$, combined with Lemma \ref{lm:preboxClean}, shows $C_\disc$ and $C_\mob$ are both clean in $D_e$.  Therefore, $C$ is clean in $D_e$.   By Claim \ref{cl:HinCMB} (2), $(C\cup M_C)-e$ is not planar, so Lemma \ref{lm:cleanBOD} shows $C$ has BOD in $G-e$.  Therefore,  $C$ has BOD in $G$.  \end{proof}

\begin{claim}\label{cl:CisHprebox}  $C$ is a $C\cup H$-prebox. \end{claim}

\begin{proof}  Note that $C_\disc\cup C_\mob\subseteq C\cup H$.  If $e\in C$, then let $i\in\{\Mob,\Disc\}$ be such that $e\in C_i$.  Lemma \ref{lm:greenCycles} (\ref{it:prebox})  says that $C_i$ is a $(C_i\cup H)$-prebox and, therefore, $(C_i\cup H)-e$ is not planar.  Since $(C_i\cup H)-e\subseteq (C\cup H)-e$, we conclude that $C$ is a $(C\cup H)$-prebox.  \end{proof}

\begin{claim}\label{cl:GisCMB}  $G=C\cup M_C\cup B_C$.  \end{claim}

\begin{proof} By way of contradiction, suppose there is another $C$-bridge $B'$.  Let $F$ be the $(H\cup \gamma)$-face containing $B'$.  Then $C\cup B'$ is contained in the closed disc that is the union of the closure of $F$ and the disc bounded by $C$, showing $B'$ is planar.  By Claim \ref{cl:CisHprebox} and the fact that $C\cup H\subseteq \comp{B'}$, Lemma \ref{lm:preboxClean} says that $C$ is clean in a 1-drawing of $\comp{B'}$, of which there is at least one, since $G$ is 2-crossing-critical.  This yields a 1-drawing of $C\cup M_C$ with $C$ clean. By Claim \ref{cl:ChasBOD}, $C$ has BOD, $B_C$ is planar because it is contained in the closed disc bounded by $C$, and above we showed that every other $C$-bridge is planar;  Corollary \ref{co:TutteTwo} implies $cr(G)\le 1$, a contradiction.  \end{proof}

We are now on the look-out for a box in $G$; it is not true that $C$ is necessarily one. Our next claim gives a sufficient condition under which we can find some box and the following two claims show that, in all other cases, $C$ is a box.

\begin{claim}\label{cl:otherBox}  Suppose \dragominor{all of the following}:
\begin{enumerate}
\item\label{it:p3mob} there is an $i$ so that $P_3^\mob$ is in a $Q_i$-local $H$-bridge;
\item\label{it:p2mob}  $P_2^\mob$ contains $v_i$ and is a non-trivial subpath of $s_i$; and 
\item\label{it:v2p1}  $v_{i+2}$ is in the interior of $P_1^\disc$.  
\end{enumerate}
Then $G$ has a box.  \end{claim}

\begin{proof}  We note that (\ref{it:p2mob}) implies $s_i\subseteq \Mob$.  

\begin{subclaim} \dragominor{Both $s_{i+1}$ and $s_{i+2}$ are contained in $\Mob$. } \end{subclaim}

\begin{proof} Suppose first that $s_{i+2}$ is exposed.  Then (\ref{it:v2p1}) implies $P_2^\disc$ and $P_4^\disc$ are both trivial.  That is, $C_\disc= P_1^\disc P_3^\disc$.  But $P_3^\disc$ is $H$-avoiding and overlaps $s_{i+2}$ on $R$ (because $P_1^\disc$ has at most four $H$-nodes, only two of which can be in the interior of $P_1^\disc$).  Thus, $P_3^\disc$ and $s_{i+2}$ cross in $\pp$, a contradiction.  Therefore, $s_{i+2}\subseteq \Mob$.

Next, suppose $s_{i+1}$ is exposed.  Then, by symmetry, we may assume $i=4$ or $i=9$.  In either case, $P_1^\mob$ and $P_1^\disc$ are in different $ab$-subpaths of $R$ and so do not have an edge in common, a contradiction.  Hence $s_{i+1}$ is also contained in $\Mob$. \end{proof}

Let $u$ \dragominor{be the common end of $P_2^\mob$ and $P_3^\mob$ and let $w$ be the common end  of $P_4^\mob$ and $P_1^\mob$.  By (\ref{it:p2mob}), $u\in \oo{s_i}$ and, by (\ref{it:p3mob}) and (\ref{it:p2mob}), $w\in r_i$.}  Observe that the edge $e_0$ common to $C_\mob$ and $C_\disc$ is in $\cc{v_i,r_i,w}$.

Let $C'$ be the cycle  
$
\cc{v_{i+5},s_{i},u,P_3^\mob P_4^\mob,w,r_i,v_{i+1}}\rbsp r_{i+1}\,s_{i+2}\,r_{i+6}\,r_{i+5}
$.
We note that there are two obvious $C'$-bridges: the $C'$-interior bridge $B_{C'}$ containing the edge of $s_{i+1}$ incident with $v_{i+6}$; and the $C'$-exterior bridge $M_{C'}$ for which $H-\oo{s_{i+1}}\subseteq C'\cup M_{C'}$.  To show $C'$ is a box, it suffices to show that $C'$ has BOD and $C'$ is a $(C'\cup M_{C'})$-prebox.

Notice that $v_{i+2}$ is in the interior of $P_1^\disc$ by hypothesis and $v_{i+1}$ is in the interior of $P_1^\disc$ because $e_0\in r_i$.   Lemma \ref{lm:greenCycles} (\ref{it:shortJump}) implies that the only $H$-nodes in the interior of $P_1^\disc$ are $v_{i+1}$ and $v_{i+2}$.   In particular, $v_i$ and $v_{i+3}$ are in $R-\oo{P_1^\disc}$, as are all the ends of $s_{i+3}$ and $s_{i+4}$.

To see that $C'$ has BOD, we produce a non-contractible cycle in $\Nuc(M_{C'})$. \wording{Lemma \ref{lm:CdisjointNCcycle} then implies $C'$ has BOD and precisely one non-planar bridge.}  We start with the two paths $P_2^\disc P_3^\disc P_4^\disc$ and $s_{i+4}$, and \wording{easily complete the required cycle} using two paths in $R$, one containing $r_{i+3}$ and the other containing $r_{i+9}$. 

It remains to show that $C'$ is a $(C'\cup M_{C'})$-prebox.   Since $V_8\topol H-\oo{s_{i+1}}\subseteq C'\cup M_{C'}$, it is obvious that, if $e\in C'$ and $e\notin R$, then $(C'\cup M_{C'})-e$ contains a $V_6$ and so is not planar.  So suppose $e\in C'$ and $e\in R$.  There are two cases.

If $e\in r_i\, r_{i+1}$, then take $(R-\oo{P_1^\disc})\cup P_2^\disc P_3^\disc P_4^\disc$ as the rim.  We choose as spokes $s_i$, $s_{i+3}$,  and $s_{i+4}$.  

If $e\in r_{i+5}\, r_{i+6}$, then the rim consists of the two paths $P_2^\disc P_3^\disc P_4^\disc$ and $C'-\oo{r_{i+5}\,r_{i+6}}$, together with the two subpaths of $R$ joining them, one containing $v_{i+3}$, $v_{i+4}$, and $v_{i+5}$, and the other containing $v_{i+7}$, $v_{i+8}$, $v_{i+9}$, and $v_i$.  In this case, the spokes are $s_{i+3}$, $s_{i+4}$, and $P_2^\mob$.
\end{proof}

In the remaining case, we show that $C$ is a box.  The following simple observations get us started, the first being the essential ingredient.

\begin{claim}\label{cl:twoCases}  Either:
\begin{enumerate}
\item\label{it:right} there is an $i$ so that \begin{itemize}
\item  $P_3^\mob$ is in a $Q_i$-local $H$-bridge;
\item  $s_i$ contains an edge of $C_\mob$; and
\item  $v_{i+2}$ is in the interior of $P_1^\disc$;  
\end{itemize}
or (symmetrically)
\item\label{it:left} there is an $i$ so that \begin{itemize}
\item $P_3^\mob$ is in a $Q_i$-local $H$-bridge;
\item  $s_{i+1}$ contains an edge of $C_\mob$; and
\item  $v_{i-1}$ is in the interior of $P_1^\disc$;  
\end{itemize}
or
\item\label{it:other}
there are three $H$-spokes not having an edge in $C_\mob$ and not having an incident vertex in the interior of $P_1^\disc$.
\end{enumerate}
\end{claim}

\begin{proof}   Lemma \ref{lm:greenCycles} (\ref{it:shortJump}) implies there are at most two $H$-nodes in the interior of $P_1^\disc$.  Therefore, if no $H$-spoke contains an edge of $C_\mob$, then (\ref{it:other}) holds.  So we may suppose $C_\mob$ has an edge in some $H$-spoke. 

Suppose first that $s_0$ is exposed, $C_\mob$ has an edge in $s_1$ and $e_0$ is in either $\cc{a,r_9,v_0,r_0,v_1}$ or $\cc{b,r_5,v_6}$.   Therefore, $P_1^\disc$ has one end in either $\co{a,r_9,v_0,r_0,v_1}$ or $\co{b,r_5,v_6}$.     Lemma \ref{lm:greenCycles} (\ref{it:shortJump}) implies at most two $H$-nodes can be in the interior of $P_1^\disc$, so no end of $s_3$ can be in the interior of $P_1^\disc$.  We conclude that $s_0$, $s_3$ and $s_4$ are the required three spokes yielding (\ref{it:other}).

\minor{Symmetry treats the same case on the other side.}

In the remaining case, $P_3^\mob$ is contained in a $Q_i$-local $H$-bridge and both $s_i$ and $s_{i+1}$ are contained in $\Mob$.   The edge $e_0$ is in either $r_i$ or $r_{i+5}$.   If the only $H$-nodes in the interior of $P_1^\disc$ are incident with either $s_i$ or $s_{i+1}$, then the other three $H$-spokes suffice for (\ref{it:other}).  

Thus, by symmetry we may assume an end of $s_{i+2}$ is in the interior of $P_1^\disc$.  This implies that an end of $s_{i+1}$ is also in the interior of $P_1^\disc$.  Lemma \ref{lm:greenCycles} (\ref{it:shortJump}) shows these are the only $H$-nodes in the interior of $P_1^\disc$.  If $s_i$ does not contain an edge of $C_\mob$, then the three spokes other than $s_{i+1}$ and $s_{i+2}$ suffice for (\ref{it:other}), while if $s_i$ does contain an edge of $C_\mob$, then we have (\ref{it:right}). \end{proof}

Claims \ref{cl:otherBox} and \ref{cl:twoCases} show we need only consider the third possibility in Claim \ref{cl:twoCases} to find a box.

\begin{claim}\label{cl:CMprebox} If there are three $H$-spokes not having any edge in $C_\mob$ and not having an incident $H$-node in $\oo{P_1^\disc}$,  
then $C$ is a box.
\end{claim}

\begin{proof}  By Claim \ref{cl:ChasBOD} and the fact that $B_C$ is a planar $C$-bridge, it suffices to show $C$ is a $(C\cup M_C)$-prebox.  For each $e\in C$, we\wordingrem{(text removed)} show that $(C\cup M_C)-e$ contains a $V_6$. 

We note that 3-connection \wording{and the fact that $C_\mob$ and $C_\disc$ both bound faces} implies $C_\mob\cap C_\disc$ is just $e_0$ and its ends.   That is, $B_C$ consists of just $e_0$ and its ends.  Thus, Claim \ref{cl:GisCMB} implies that $G-e_0= C\cup M_C$.  In particular, every spoke is in $C\cup M_C$. 

Let $w$ be any $H$-node that is not in $C$.  There are two $wC$-paths in $R-e_0$;  let them be $R_x$ with end $x\in C$ and $R_y$ with end $y\in C$.  Thus, $R$ consists of the $C$-avoiding path $R_x\cup R_y$, a subpath of $C$, the edge $e_0$, and another subpath of $C$.   The cycle $C$ consists of two $xy$-paths; let them be $N^\disc$ containing $P_2^\disc P_3^\disc P_4^\disc$ and  $N^\mob$  containing $P_2^\mob P_3^\mob P_4^\mob$.   We note that $N^\disc\subseteq \Disc$ and $N^\mob\subseteq \Mob$.

\begin{subclaim}\label{sc:specialSpokes}  Let $s$ be an $H$-spoke with no edge in $C_\mob$ and not having an incident $H$-node in $\oo{P_1^\disc}$.
\begin{enumerate}  \item\label{it:sInMob} If $s\subseteq \Mob$, then $s\cap C$ is either empty, $x$, or $y$.
\item\label{it:sInDisc} If $s\subseteq \Disc$ (that is, $s=s_0$ is exposed), then  $s\cap C$ \major{contains at most one of $v_0$ and $v_5$}. 
\end{enumerate}
\end{subclaim}

\begin{proof} For (\ref{it:sInMob}), the alternative is that $s$ contains a vertex $u$ in $\oo{N^\mob}$.  By hypothesis, $s$ has no edge in $C_\mob$ and, therefore, $s$ has no edge in $C$.  Being in $N^\mob$, the vertex $u$ is either in $R$ or in $P_2^\mob P_3^\mob P_4^\mob$.  

Suppose that $u$ is in $P_2^\mob P_3^\mob P_4^\mob$.  If $u$ is in $P_3^\mob$, then, since $P_3^\mob$ is $H$-avoiding, $u$ is an end of $P_3^\mob$, and so is in $P_2^\mob\cup P_4^\mob$.  Thus, if $u$ is in $P_2^\mob P_3^\mob P_4^\mob$, then $u$ is in $P_2^\mob\cup P_4^\mob$.  Since both $P_2^\mob$ and $P_4^\mob$ are contained in $H$, are $R$-avoiding, and neither has an edge of $s$, the one containing $u$ is trivial and $u$ is in $R$.  

Thus, in every case $u$ is in $R$ and so is an $H$-node.  It follows that one of $\cc{x,N^\mob,u}$ and $\cc{u,N^\mob,y}$ contains $P_2^\mob P_3^\mob P_4^\mob$ and the other is contained in $R$.  We choose the labelling so that $\cc{x,N^\mob,u}\subseteq R$.  

As we follow $R-e_0$ from $w$ to $x$ and continue to $u$ along $N^\mob$, we see there is an edge  of $C$ incident with $x$ and not in $R$.  \wording{That it is in $N^\disc$} implies it is in $P_2^\disc P_3^\disc P_4^\disc$.   All the vertices in $\co{x,N^\mob,u}$ are incident with two rim edges in what we have just traversed.  In particular, $e_0$ is not incident with any of these vertices and, therefore,  $\cc{x,N^\mob,u}$ is contained in $C^\disc$. More precisely, $\cc{x,N^\mob,u}$ is contained in $P_1^\disc$.  As we continue along $R$ past $u$, we either find $e_0$ is incident with $u$ or the other edge of $C$ incident with $u$ is in $R$.  In either case, $u$ is in $\oo{P_1^\disc}$, a contradiction.

For (\ref{it:sInDisc}), suppose $v_0$ and $v_5$ are both in $C$.  Then $P_1^\mob\cup P_1^\disc$ contains both $v_0$ and $v_5$.   By Definition \ref{df:green} (\ref{it:P1options}), $v_0$ and $v_5$ are not both in the same one of $P_1^\mob$ and $P_1^\disc$, so one is in $P_1^\mob$ and the other is in $P_1^\disc$.  By \wording{symmetry,} we may assume $v_0$ is in $P_1^\mob$.  Because $\Pi$ is $H$-friendly, $P_1^\mob$ is contained in either $\cc{a,r_9,v_0,r_0,v_1}$ or, if $a=v_0$, $r_9$ (these being the only two faces of $\Pi[H]\cup \gamma$ in $\Mob$ that can be incident with $v_0$).  

Recall that $e_0$ is in both $P_1^\mob$ and $P_1^\Disc$. If $P_1^\mob\subseteq \cc{a,r_0,v_0,r_0,v_1}$, then $e_0$ is in either $r_9$ or $r_0$ and $P_1^\disc$ is, by Definition \ref{df:green} (\ref{it:P1options}), contained in either $\cc{a,r_9, v_0}\rbsp r_0\,r_1\lbsp\cb{v_2,}$ $\bo{r_2,v_3}$ or \minor{$r_0\,r_1\,r_2\co{v_3,r_4,v_4}$}, and $v_5$ is not in $C$.  If $P_1^\mob\subseteq r_9$, then  $e_0$ is in $r_9$, so $P_1^\disc$ is contained in \minor{$r_9\,r_8\,r_7\co{v_7,r_6,v_6}$}, and again $v_5$ is not in $C$.
\end{proof}

The case $e\in N^\disc$ is easy: the rim of the $V_6$ is $(R-\oo{P_1^\mob})\cup P_2^\mob P_3^\mob P_4^\mob$ and we choose as spokes any three of the $H$-spokes that are contained in $\Mob$.  (If one intersects $C_\mob$, then only the part of the spoke that is $C_\mob$-avoiding will be the actual spoke of the $V_6$.) 

If $e\in N^\mob$, then the rim $R'$ of the $V_6$ is $(R-\oo{P_1^\disc})\cup P_2^\disc P_3^\disc P_4^\disc$ and the spokes are the three $H$-spokes from the hypothesis.  If all three hypothesized $H$-spokes are contained in $\Mob$, then it is evident from Subclaim \ref{sc:specialSpokes} (\ref{it:sInMob}) that we have indeed described a $V_6$ in $(C\cup M_C)-e$.

So suppose that one of the $H$-spokes in the hypothesis is the exposed spoke $s_0$.  From Subclaim \ref{sc:specialSpokes} (\ref{it:sInDisc}), either $s_0$ is disjoint from $C$ or precisely one $H$-node incident with $s_0$ is in $C$.  We may choose the labelling so that $v_0$ is not in $C$.  

If $v_5$ is not in $C$, then $s_0$ is disjoint from $C$.   Subclaim \ref{sc:specialSpokes} (\ref{it:sInMob}) shows the other two hypothesized $H$-spokes meet $C$ in at most $x$ or $y$; it is now obvious that the three hypothesized $H$-spokes combine with $R'$ to make a $V_6$. 

Finally, suppose $v_5$ is in $C$.  \major{Because $C_\disc$ is $H$-green, $P_1^\disc\subseteq r_2\,r_3\,r_4\cc{v_5,r_5,b}$.  In particular, $s_1$ is disjoint from $C_\mob$.  If $s_2$ has no edge in $C_\mob$, then $R'\cup s_1\cup s_2$, together with the portion of $s_0$ from $v_0$ to $C_\disc$ is a $V_6$ avoiding $N_\mob$.  If $s_2$ has an edge in $C_\mob$, then $C_\mob$ is in the $\Pi[H]$-face bounded by $Q_2$.  In this case, we may replace $s_2$ with $s_4\,r_4$ to obtained the desired $V_6$.%
}%
\minorrem{(Text removed.)} 
 \end{proof}







Evidently, Claims \ref{cl:otherBox}, \ref{cl:twoCases}, and \ref{cl:CMprebox} show that $G$ has a box, contradicting Lemma \ref{lm:noBox}.  \end{cproof}
}

\majorrem{(Section 7 on Local $H$-bridges is removed.  Lemma 7.3 is now Lemma 5.5; it is the only thing that seems to have been used.)}

\chapter{Exposed spoke with\\ additional attachment not in $\bQ_0$}\printFullDetails{

The main result of this section is the proof of the following \wording{technical \dragominor{theorem}}, which limits possibilities for the $V_{10}$-bridges.  This will be used in the next section when \wording{we get our} second major step by showing that there is a representativity 2 embedding of $G$ in $\pp$ for which all the $H$-spokes are contained in the M\"obius band.

}\begin{theorem}\label{th:expSpokeNoAtt}
Suppose $G\in\m2$ and $\hvfg$.  Let $\Pi$ be an $H$-friendly embedding of $G$ in $\pp$, with the standard labelling.  Then there is no $H$-bridge having attachments in both $\oo{s_0}$ and $\oo{r_1\,r_2\,r_3}$.  \end{theorem}\printFullDetails{

At one point in the proof of this theorem, we need the following lemma.  Most of it is used again several times.

}\begin{lemma}\label{lm:technicalV8colour}  Let $G$ be a graph and let $\hveg$. Let $P$ be an $H$-avoiding path in $G$ joining distinct vertices $x$ and $y$ of $R$ and let $P'$ be one of the two $xy$-subpaths of $R$.  Let $D$ be a 1-drawing of $H\cup P$. 
\begin{enumerate}
\item\label{it:atMostTwo}  If $P'$ has at most two $H$-nodes or, for some $i$, $P'=r_i\,r_{i+1}$, then $P'$ is not crossed in $D$.
\item\label{it:V8yellow} \wordingrem{(Order changed.)}\wording{If  there are only the two $H$-nodes  $v_i$, $v_{i+1}$ in the interior of $P'$ and $P'$ has at most one other $H$-node, then $r_{i+4}$} is not crossed in $D$.
\item\label{it:2halfJump} Suppose $r_i\,r_{i+1}\subseteq P'$, $P'\not\subseteq r_i\,r_{i+1}$, but $P'\subseteq r_i\,r_{i+1}\lbsp \co{v_{i+2},r_{i+2},v_{i+3}}$.  
\begin{enumerate}
\item\label{it:2halfJump1} Then $r_i\,r_{i+1}$ is not crossed in $D$.  
\item\label{it:2halfJump3/2} If $P'$ is crossed in $D$, then \minor{$s_{i+3}$} is exposed in $D$ and $P'\cap r_{i+2}$ crosses $r_{i-1}$.
\wordingrem{\item (Old item removed as not used anywhere.)}\ignore{\item\label{it:2halfJump2} 
 If, in addition, there is an $(H\cup P)$-avoiding path $P''$ joining an internal vertex of $P$ to a vertex of $\oo{r_{i+4}\,r_{i+5}\,r_{i+6}}$, then $P'$ is not crossed in $D$.}
\end{enumerate}
\end{enumerate}
\end{lemma}\printFullDetails{

\begin{cproof}  Let $x$ and $y$ be the ends of $P$ and let $R'=(R-\oo{P'})\cup P$.  For (1) and (2), we find three spokes to add to $R'$ to find a subdivision of  $V_6$ disjoint from $P'$ --- or at least some part of $P'$.   The part of $P'$ disjoint from the $V_6$ cannot be crossed in any 1-drawing of $H$.

For (\ref{it:atMostTwo}), if $P'$ contains at most one $H$-node, then this is easy:  any three $H$-spokes not having an end in $P'$ will suffice.  If $P'=r_i\,r_{i+1}$, then the three $H$-spokes $s_i$, $s_{i+2}$, and $s_{i+3}$ suffice.  

In the remaining case, $P'$ has precisely two $H$-nodes.  We may express $P'$ in the form $$P'=\cc{x,r_{j-1},v_j}\rbsp r_j\lbsp \cc{v_{j+1},r_{j+1},y},$$ where either of $\cc{x,r_{j-1},v_j}$ and $\cc{v_{j+1},r_{j+1},y}$ might be a single vertex.  In this case, the spokes are $s_{j+2}$, $s_{j+3}$ and $s_{j+1}\lbsp\cc{v_{j+1},r_{j+1},y}$, showing that $\cc{x,r_{j-1},v_j}\rbsp r_j$ is not crossed in $D$, while replacing $s_{j+1}\lbsp \cc{v_{j+1},r_{j+1},y}$ with $\cc{x,r_{j-1},v_j}\rbsp s_j$ shows $\cc{v_{j+1},r_{j+1},y}$ is not crossed in $D$.  This completes the proof of (\ref{it:atMostTwo}).

\wordingrem{(Order changed.)}  For (\ref{it:V8yellow}), replace $R'$ with $(R'-\oo{r_{i+4}})\cup (s_i\,r_i\,s_{i+1})$.  We now need three spokes.  \minor{If there is a third $H$-node in $P'$, then symmetry allows us to assume it is $v_{i-1}$.  In either case, we choose $s_{i-1}$, $\cc{v_{i+1},r_{i+1},y}$, and $s_{i+2}$ as the three spokes for the $V_6$.  This $V_6$ avoids $r_{i+4}$, showing it is not crossed in $D$.}



For (\ref{it:2halfJump}), $x=v_i$ and the hypotheses imply that $y\in \oo{r_{i+2}}$.  For (\ref{it:2halfJump1}), we may use the spokes $s_i$,  $s_{i+2}\lbsp\cc{v_{i+2},r_{i+2},y}$, and $s_{i+3}$ to see that $r_i\,r_{i+1}$ is not crossed in $D$, as required.  


For (\ref{it:2halfJump3/2}), suppose $P'$ is crossed in $D$.  \wording{Part (\ref{it:2halfJump1}) shows} that it must be $P'\cap r_{i+2}$ that is crossed and (\ref{it:V8yellow}) \wordingrem{(phrase removed)}shows that \wording{$r_{i+5}=r_{i-3}$} is not crossed in $D$.  We need only show that $r_{i-2}$ is also not crossed in $D$.  If it were, then \wordingrem{(text removed)}
\wording{$\cc{v_{i+2},r_{i+2},y}$} crosses $r_{i-2}$.  But then the cycle $r_{i+3}\,r_{i+4}\,r_{i-3}\,r_{i-2}\,s_{i-1}$ separates \wording{$v_i=x$ from $y$} in $D$, showing that $P$ is also crossed in $D$, a contradiction.
\ignore{
For (\ref{it:2halfJump2}), we have the additional path $P''$.  In this case, we \wording{make our $V_6$ from $R'$ and} the spokes $s_i$, $s_{i+3}$ and $P''$ to see that $\cc{v_{i+2},r_{i+2},y}$ is not crossed in $D$.}
 \end{cproof}

\begin{cproofof}{Theorem \ref{th:expSpokeNoAtt}} This is obvious if no spoke is exposed in $\Pi$, so we may suppose $s_0$ is exposed. 

\major{\begin{claim}\label{cl:noPaths} There is no $H$-avoiding $\oo{s_0}\oc{v_1,r_1,v_2}$- or $\oo{s_0}\co{v_3,r_3,v_4}$-path.  \end{claim}}

\begin{proof}  \wording{By symmetry, it suffices to prove only one.}  By way of contradiction, we suppose that there is an $H$-avoiding path $P$ from $x\in \oo{s_0}$ to $y\in \oc{v_1,r_1,v_2}$.   

Let $e\in s_3$ and consider a 1-drawing $D$ of $G-e$.  By Lemma \ref{lm:BODcrossed} and Theorem \ref{th:BODquads} (\ref{it:exposedBQ3}), we know that $\bQ_3$ is crossed in $D$.  This implies that $r_1\,r_2\,r_3\,r_4$   crosses $r_6\,r_7\,r_8\,r_9$.   This already implies neither $s_0$ nor $s_1$ is exposed in $D$.   Furthermore, the crossing is of two edges in $R$ and, since $P$ is $H$-avoiding, we conclude that $D[P]$ is not crossed in $D$.  Therefore, the end of $P$ in $\oc{v_1,r_1,v_2}$ must occur in the interval of $r_1\,r_2\,r_3\,r_4$ between the crossing and $v_5$; that is, the crossing must involve an edge of $r_1$. In particular, $r_2\,r_3\,r_4\,r_5$ is not crossed in $D$.  

\minor{Since $\bQ_3$ is crossed in $D$ and $r_1$ is crossed in $D$, the other crossing edge is in $r_7\,r_8$.  Thus it is in $r_6\,r_7\,r_8$.   It follows that $s_2$ is exposed in $D$.  Thus, the cycle $r_4\,r_5\,s_1\,r_0\,r_9\,s_4$  separates $x$ from $y$ in $D$, showing $P$ is crossed in $D$, a contradiction.}  \end{proof}

It follows from Claim \ref{cl:noPaths} that, if there is an $H$-avoiding path $P_0$ joining $x\in\oo{s_0}$ to $y\in \oo{r_1\,r_2\,r_3}$, then $y\in \oo{r_2}$.  Let $K=H\cup P_0$.   See Figure \ref{KinPP}.  

\begin{figure}[!ht]
\begin{center}
\scalebox{1.0}{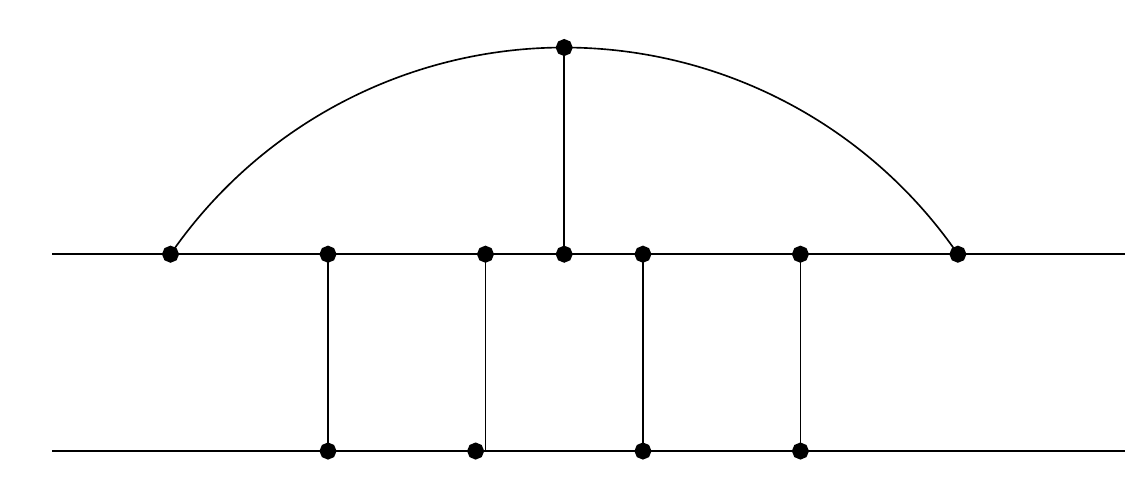}
\end{center}
\caption{The subgraph $K$ of $G$ in $\pp$.}\label{KinPP}
\end{figure}

Let $J_1$ and $J_2$ be the two cycles $r_0\,r_1\,\cc{v_2,r_2,y,P_0,x,s_0,v_0}
$  and $r_4\,r_3\,\cb{v_3,r_2,y,P_0,}$ $\bc{x,s_0,v_5}$, respectively.

\begin{claim}\label{cl:twoFaces} The cycles $J_1$ and $J_2$ 
both bound faces of $G$ in $\pp$. \end{claim}

\begin{proof}  These cycles are both $H$-green, so this is just Lemma \ref{lm:greenCycles} (\ref{it:CboundsFace}).\end{proof}

The following claim completes the determination of the $(H\cap \Mob)$-bridge containing $s_0$.

\begin{claim}\label{cl:s0inK13}  The \wording{$(H-\oo{s_0})$-bridge} containing $s_0$ is $s_0\cup P_0$. \end{claim}

\begin{proof} Suppose not and let $B$ be the \wording{$(H-\oo{s_0})$-bridge} containing $s_0$.  Then Lemma \ref{lm:threeAtts} implies that \wording{$B$ has an attachment $z$ other than $v_0$, $y$, and $v_5$}.  By Claim \ref{cl:twoFaces}, \wording{$z\in\co{a,r_9,v_0}\cup \oc{v_5,r_5,b}$}; by symmetry we may assume the former.  Let $P$ be a $K$-avoiding \minor{$z\oo{s_0}$-path}.  

Suppose $z=v_9$.
 Let $e$ be the edge of $s_0$ incident with $v_0$.  We show that $\crn((K\cup P) -e)\ge 2$.  As this is a proper subgraph of $G$, we contradict the fact that $G$ is \2cc.  In $P\cup (s_0-e)\cup P_0$, there is a \wording{claw $Y$} with talons $z=v_9$, $y$ and $v_5$.  We show \wording{$\crn((H-\oo{s_0})\cup Y)\ge 2$}. 
 
 By way of contradiction, we suppose $D$ is a 1-drawing of \wording{$(H-\oo{s_0})\cup Y$}.  As $H-\oo{s_0}\topol V_8$, Lemma \ref{lm:technicalV8colour} (\ref{it:atMostTwo}) \minor{implies that (using the labelling from $H$) $\cc{y,r_2,v_3}r_3\,r_4$ is not crossed in $D$, while (\ref{it:V8yellow}) of the same lemma implies neither $r_6$ nor $r_8$ is crossed in $D$.  Part  (\ref{it:2halfJump1}) implies $r_9\,r_0\,r_1$ is not crossed, while (\ref{it:2halfJump3/2}) implies (since $r_9$ is not crossed) that $\cc{v_2,r_2,y}$ is not crossed.}  The only remaining possibilities for crossed $(H-\oo{s_0})$-rim branches are $r_5$ and $r_7$.  But no 1-drawing of $H-\oo{s_0}$ has these two rim-branches crossed, the desired contradiction.

So $z\ne v_9$.  But then we may replace $s_0$ with the $zv_5$-path $s'_0$ in \wording{$P\cup s_0$} and replace $P_0$ with the $ys'_0$-path in $P_0\cup s_0$ to get a new subdivision $H'$ of $V_{10}$.  We notice that Lemma \ref{lm:greenCyclesFriendly} (\ref{it:nearlyFriendly}) implies that $\Pi$ is $H'$-friendly.  However, the analogue $J'_1$ of $J_1$ does not bound a face, contradicting Claim \ref{cl:twoFaces}. \end{proof}

\begin{claim}\label{cl:oneK}  There is a unique 1-drawing of $K$.  In this 1-drawing, $s_0$ is exposed. \end{claim}

The 1-drawing of $K$ is illustrated in Figure \ref{onlyK}.

\bigskip

\begin{figure}[!ht]
\begin{center}
\scalebox{1.2}{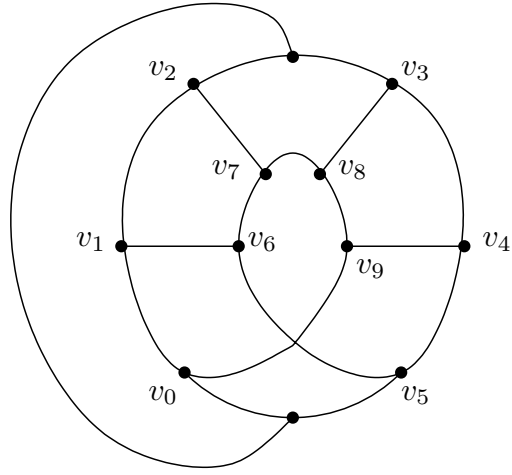}
\end{center}
\caption{The 1-drawing of $K$.}\label{onlyK}
\end{figure}

\begin{proof}  If $D$ is a 1-drawing of $K$, then Claim \ref{cl:twoFaces} and Lemma \ref{lm:greenCycles} (\ref{it:notCrossed}) imply neither $J_1$ nor $J_2$ is crossed in $D$.  It follows that none of $r_0$, $r_1$, $r_2$, $r_3$, and $r_4$ is crossed in $D$.  Lemma \ref{lm:1drawingsV2n} implies $r_7$ cannot be crossed in $D$, so $Q_2$ is clean in $D$.  Therefore, $s_0$ must be in a face of $D[R\cup Q_2]$ incident with $r_2$.  This is only possible if $s_0$ is exposed, which determines $D$.  \end{proof}

\wording{For $j\in\{2,3\}$, let $D_j$ be a 1-drawing of $G-\oo{s_j}$.}

\begin{claim}\label{cl:D2}    The crossing in $D_2[(H-s_2)\cup P_0]$ is of $r_5$ with $\cc{y,r_2,v_3}$.  Likewise, the crossing in $D_3[(H-s_3)\cup P_0]$ is of $r_9$ with $\cc{v_2,r_2,y}$.\end{claim}

The 1-drawings of Claim \ref{cl:D2} are illustrated in Figure \ref{D2plusP}.

\begin{figure}[!ht]
\begin{center}
\scalebox{1.2}{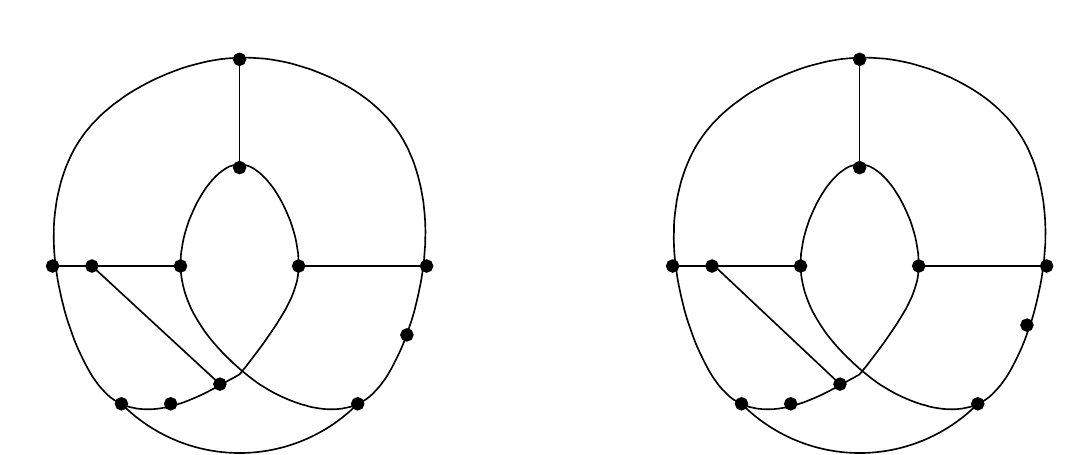}
\end{center}
\caption{The 1-drawings $D_2[(K-\oo{s_2})\cup P_0]$ and $D_3[(K-\oo{s_3})\cup P_0]$.}\label{D2plusP}
\end{figure}

\bigskip
\begin{proof}  We treat the \wording{case $j=2$; the} case $j=3$ is very similar.  By Theorem \ref{th:BODquads} (\ref{it:bQ2}), $\bQ_2$ has BOD, so Lemma \ref{lm:BODcrossed} implies $\bQ_2$ is crossed in $D_2$.  This implies that $s_0$ is not exposed in $D_2$.  The $H$-avoiding path $P_0$ joins $x\in \oo{s_0}$ to $y\in\oo{r_2}$, so $y$ must be on a face incident with $s_0$.  It follows that $Q_0$ must be crossed in $D_2$.  This implies that $s_1$ is exposed.  We deduce that either $r_5$ crosses $r_1\cup r_2$ or $r_0$ crosses $r_6\cup r_7$.  In the latter case, $D_2[P_0]$ must cross $D_2[H-s_2]$, a contradiction, so it must be the former.

As $D_2[P_0]$ is not crossed, $y$ occurs between $v_1$ and the crossing in $r_1\cup r_2$, as required.  \end{proof}

The following claims help us obtain the structure of $\comp{(M_{\bQ_0})}$; we will use this to find a 1-drawing of $G$, which is the final contradiction.



\begin{claim}\label{cl:v0v6}  Suppose $B$ is a $\bQ_0$-bridge having an attachment in each of $r_9$ and $r_5$.  Then $B$ is one of $M_{\bQ_0}$,   $v_6v_9$, $v_0v_6$, and $v_5v_9$. \end{claim}

\begin{proof}  We note that $s_0\cup P_0\subseteq M_{\bQ_0}$.  Either $B=M_{\bQ_0}$, or, in the drawing $D_2$, $B$ is in a face of $D_2[(H-s_2)\cup P_0]$ incident with both $r_9$ and $r_5$.  There are only two such faces, namely $F$, bounded by $Q_4$, and $F'$, the other face incident with $r_9$.  Whichever face $B$ is in, its attachments are in the intersection of $\bQ_0$ with the boundary of the containing face.  Thus, if $B$ is in $F$, then $\att(B)\subseteq  r_4\,s_4\, r_9$.  In this case, the only possibility for an attachment in $r_5$ is $v_5$, so $v_5\in\att(B)$.  If, on the other hand, $B$ is in $F'$, then $\att(B)\subseteq r_9\, r_0\, s_1$.  In this case, $v_6\in\att(B)$.  Similarly,  $D_3$ shows either $B=M_{\bQ_0}$, or  $\att(B)\subseteq r_0\, s_1\, r_5$ and $v_0\in \att(B)$, or $\att(B)\subseteq  s_4\, r_4\, r_5$ and $v_9\in\att(B)$.  Comparing these possibilities, we conclude that one of the following four cases holds for $\att(B)$:  $\att(B)=\{v_0,v_5\}$; $\att(B)=\{v_6,v_9\}$; $v_5,v_9\in\att(B)$ and $\att(B)\subseteq r_4\cup s_4$; and $v_0,v_6\in\att(B)$ and $\att(B)\subseteq r_0\cup s_1$. 

We claim $v_0v_5$ is not an $H$-bridge.  For if it were, let $D$ be a 1-drawing of $G-v_0v_5$.  Then $s_0\cup P_0$ is not crossed in $D$ and Claim \ref{cl:s0inK13} says the $(H-\oo{s_0})$-bridge containing $s_0$ is $s_0\cup P_0$.  In particular, $s_0$ consists of the two edges $v_0x$ and $xv_5$, and $x$ has degree 3 in $G$.  Thus, we can draw $v_0v_5$ alongside $s_0$, yielding a 1-drawing of $G$, a contradiction.

We must show that, if $v_0,v_6\in \att(B)$ and $\att(B)\subseteq r_0\cup s_1$, then $B=v_0v_6$.  Likewise, if $v_5,v_9\in \att(B)$ and $\att(B)\subseteq r_4\cup s_4$, then $B=v_5v_9$.
We consider the former case, the latter being completely analogous.  Corollary \ref{co:attsMissBranch} shows that $B$ can have at most one other attachment.  Lemma \ref{lm:threeAtts} shows that either $B=v_0v_6$ or $B$ is a claw with talons $v_0$, $v_6$, and $z\in \oo{v_0,r_0,v_1,s_1,v_6}$.   Since we are trying to show $B=v_0v_6$, we assume the latter.  Let $e$ be the edge of $B$ incident with $z$ and let $D$ be a 1-drawing of $G-e$.  Since $K\subseteq G-e$, $D$ extends the 1-drawing illustrated in Figure \ref{onlyK}.  We modify $D$ to obtain a 1-drawing of $G$, which is impossible.

Observe that $B-z$ is an $H$-avoiding $v_0v_6$-path $P$ (having length 2); there is only one place $D[P]$ can occur in Figure \ref{onlyK}.    Notice that $B$ is a $Q_0$-local $H$-bridge and, furthermore,  $P$ overlaps $M_{Q_0}$.

Theorem \ref{th:BODquads} shows $Q_0$ has BOD in $G$; let $(\mathcal B,\mathcal M)$ be the bipartition of $OD(Q_0)$, with $B\in \mathcal B$.  Then $M_{Q_0}\in \mathcal M$.  Every $Q_0$-bridge is drawn in $D$, with the exception that we have $B-e$ in place of $B$.  

Because we cannot add $e$ back into $D$ to get a 1-drawing of $G$, there must be an $H$-avoiding path $P'$ in $G-e$ joining the two components of $[v_0,r_0,v_1,s_1,v_6]-z$ so that $D[P']$ is on the same side --- henceforth, the {\em inside\/} --- of $D[Q_0]$ as $P$.   Let $B'$ be the $Q_0$-bridge containing $P'$.  If $B'$ has just $v_0$ and $v_6$ as attachments, then let $D$ be a 1-drawing of $G-v_0v_6$.  As we did above for $v_0v_5$, we can add $v_0v_6$ alongside $P$ to recover a 1-drawing of $G$.  Therefore, $B'$ does not have just $v_0$ and $v_6$ as attachments. 

It follows that $B'$ overlaps $B$, so it is in $\mathcal M$.  Therefore, it does not overlap $M_{Q_0}$; in particular, it cannot have an attachment in both $\co{v_6,s_1,v_1}$ and $\co{v_0,r_0,v_1}$.  We conclude that, for \dragominor{some $q\in \{r_0,s_1\}$} \dragominor{; and (ii) $\att(B')\subseteq q$.  Let $q'$ be such that $\{q,q'\}=\{r_0,s_1\}$}.

Let $B_1,B_2,\dots,B_k$ be a path in $OD(Q_0)-\{M_{Q_0},B\}$ so that $B'=B_1$.

\startSubclaims
\begin{subclaim} For $i=1,2,\dots,k$, \dragominor{$\att(B_i)\subseteq q$}.  \end{subclaim}

\begin{proof} Above, we \dragominor{chose $q$} to contain $\att(B')$, which is the case $i=1$.   Notice that $B_1$, $B_3$, \dots\ are all on the same side of $D[Q_0]$ as $B'$ and $P$, while $B_2$, $B_4$, \dots are all on the other side of $D[Q_0]$.  The former are all in $\mathcal M$, while the latter are in $\mathcal B$.   Let $i$ be least so that $B_i$ has an attachment \dragominor{outside $q$}.  Then it also has an \dragominor{attachment in $\oo{q}$} (in order to overlap $B_{i-1}$).  

If $B_i$ is inside $D[Q_0]$, then $B_i$ does not overlap $M_{Q_0}$, so it has no \dragominor{attachment in $q'-q$}.  As $B_i$ cannot cross $P$ in $D$, \dragominor{$\att(B_i)\subseteq q$, a contradiction}.  

If $B_i$ is outside $D[Q_0]$, then either $\att(B_i)\subseteq s_1$, \dragominor{so $q=s_1$} and we are done, or $\att(B_i)\subseteq r_0\cup\cc{v_0, s_0,x}$, so, \dragominor{in particular, $q=r_0$}. Furthermore, $B_i$ does not overlap $B$.   Therefore, $B_i$ \wording{has} no attachment in $\oc{v_0,s_0,x}$, so $\att(B_i)\subseteq r_0$. \end{proof}

Let $L$ be the component of $OD(Q_0)-\{M_{Q_0},B\}$ containing $B'$.  We can flip the $Q_0$-bridges in $L$ so that they exchange sides of  $D[Q_0]$, yielding a new 1-drawing of $G-e$ with fewer $Q_0$-bridges in $\mathcal M$ on the same side of $D[Q_0]$ as $P$.  Inductively, this shows there is a 1-drawing $D'$ of $G-e$ in which all $Q_0$-bridges in the face of $D'[K\cup P]$ bounded by $r_0\,s_1\,P$ are in $\mathcal B$.  As none of these overlaps $B$, we may add $e$ into $D'$ to obtain a 1-drawing of $G$, a contradiction.
\end{proof} 

Let $e_5$ be the edge in $r_5$ that is crossed in $D_2$ and let $e_9$ be the edge in $r_9$ that is crossed in $D_3$.  For $i=5,9$, let $u_i$ be the end of $e_i$ nearer to $v_i$ in $r_i$ and let $w_i$ be the other end of $e_i$.  See Figure \ref{D2plusPlabelled}. 
We highlight some relevant ``cut" properties of these edges in the next three claims.

\begin{figure}[!ht]
\begin{center}
\scalebox{1.2}{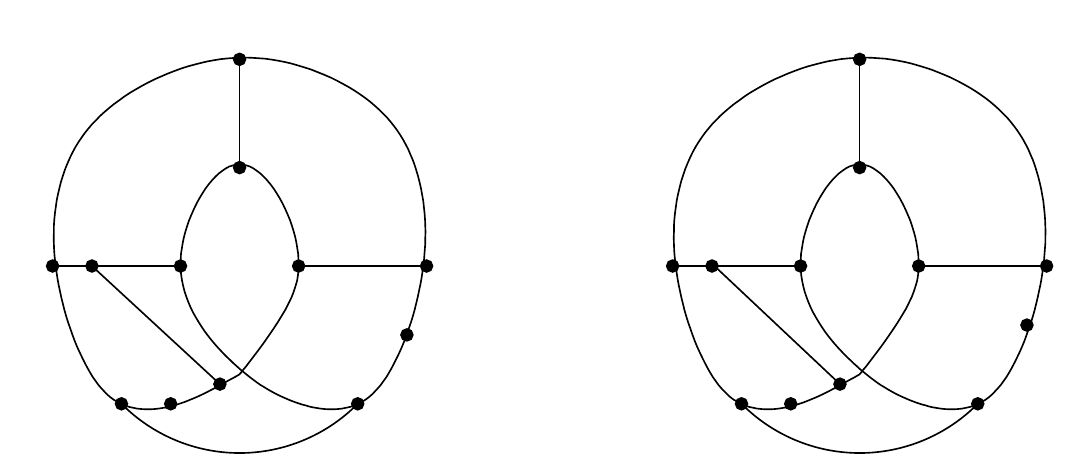}
\end{center}
\caption{The 1-drawings $D_2[(K-\oo{s_2})\cup P_0]$ and $D_3[(K-\oo{s_3})\cup P_0]$.}\label{D2plusPlabelled}
\end{figure}

\begin{claim}\label{cl:no<nine,five]<zero,six]}  Any $r_9$-avoiding $\oc{s_4\,r_4}\lbsp\oc{r_0\,s_1}$-path in $\comp{(M_{\bQ_0})}$ contains $e_5$.  In particular, there are not two edge-disjoint $r_9$-avoiding  $\oc{s_4\,r_4}\lbsp\oc{r_0\,s_1}$-paths in $\comp{(M_{\bQ_0})}$.  \end{claim}

\begin{proof}  Suppose $P$ is a $r_9$-avoiding $\oc{s_4\,r_4}\lbsp\oc{r_0\,s_1}$-path.   Let $e$ be any edge of $s_2$ and let $D$ be any 1-drawing of $G-e$.  By Claim \ref{cl:D2}, \wording{$D_2[(H-\oo{s_2})\cup P_0]$} is illustrated in Figure \ref{D2plusP}.  But here we see that the cycle $C=\cc{v_0,s_0,x}\rbsp P_0\rbsp\cc{y,r_2,v_3}\lbsp s_3\,r_8\,r_9$ separates $\oc{s_4\,r_4}$  and $\oc{r_0\,s_1}$.  Note that $C$ consists of $r_9$ and a $\bQ_0$-avoiding $v_0v_9$-path in $M_{\bQ_0}$.  Therefore,  $P$ is disjoint from $C$, and so it must cross $C$ in $D_2$.  As this can only happen at the crossing in $D_2$, it must be that the edge of $r_5$ crossed in $D_2$ is in $P$.  \end{proof}

\dragominor{Analogously, d}eleting $e\in s_3$  provides a proof of the \dragominor{following} claim.

\begin{claim}\label{cl:no[nine,five>[zero,six>}  Any $r_5$-avoiding  $\co{s_4\,r_4}\,\co{r_0\,s_1}$-path in $\comp{(M_{\bQ_0})}$ contains $e_9$.  In particular, there are not two edge-disjoint $r_5$-avoiding  $\co{s_4\,r_4}\,\co{r_0\,s_1}$-paths in $\comp{(M_{\bQ_0})}$. \hfill$\Box$ \end{claim}

The final claim is a central point about $M_{\bQ_0}$.  \wordingrem{(Text removed.)}

\begin{claim}\label{cl:bQ0cut} Let $P_1$ and $P_2$ be the two paths  \dragominor{of}~$\bQ_0-\{e_5,e_9\}$.  Then there is no $P_1P_2$-path in \wording{$$\comp{(M_{\bQ_0})}-\{e_5,e_9,v_6v_9\}\,.$$} \end{claim}

\begin{proof} \dragominor{Assume that there is a $P_1P_2$-path $P$ in $\comp{(M_{\bQ_0})}-\{e_5,e_9\}$.}  For $i=1,2$, let $z_i$ be the end of $P$ in $P_i$.

\dragominor{Suppose first that} $z_1$ is in $\oo{s_4\,r_4}$.    If $z_2$ is in $\cc{v_6,r_5,w_5}$, then  $P\rbsp \cc{z_2,r_5,v_6}$ is an $r_9$-avoiding $\oc{s_4\,r_4}\lbsp\oc{r_0\,s_1}$-path in $\comp{(M_{\bQ_0})}$ that also avoids $e_5$, contradicting Claim \ref{cl:no<nine,five]<zero,six]}.  If $z_2$ is not in $\cc{v_6,r_5,w_5}$, then there is an $r_5$-avoiding $\co{s_4\,r_4}\,\co{r_0\,s_1}$-path in $\comp{(M_{\bQ_0})}$ that also avoids $e_9$, contradicting Claim \ref{cl:no[nine,five>[zero,six>}.   Therefore, $z_1$ is in $P_1-\oo{s_4\,r_4}$; that is $z_1$ is in $\cc{v_9,r_9,u_9}\cup \cc{v_5,r_5,u_5}$.  Symmetrically, $z_2$ is in $\cc{w_9,r_9,v_0}\cup \cc{w_5,r_5,v_6}$.

\dragominor{If $z_1$ is in $\cc{v_5,r_5,u_5}$, then}   Claim \ref{cl:no<nine,five]<zero,six]} implies $z_2$ is not in $\cc{w_5,r_5,v_5}$.  Therefore, $z_2$ is in $\cc{w_9,r_9,v_0}$.  By Claim \ref{cl:v0v6}, $P$ is one of  $v_6v_9$, $v_0v_6$, and $v_5v_9$.  Clearly, neither $z_1$ nor $z_2$ is $v_6$ and neither is $v_9$, so none of these outcomes is possible.

Therefore, $z_1$ is in $\cc{v_9,r_9,u_9}$.  \wording{ Claim \ref{cl:no[nine,five>[zero,six>} implies $z_2$} is not in $\cc{w_9,r_9,v_0}$.   By Claim \ref{cl:v0v6}, the only possibility is that $z_1=v_9$ and $z_2=v_6$ and $P$ is just the edge $v_6v_9$\dragominor{, as required}. 
\end{proof} 

\dragominorrem{(Text removed.)}We will show that there is an embedding $\Pi'$ of $G$ in $\pp$ and a non-contractible simple closed curve $\gamma'$ in $\pp$ so that $\gamma'\cap G$ consists of one point in each of the interiors of $\Pi'[e_5]$ and $\Pi'[e_9]$.   Standard surgery then implies that $\crn(G)\le 1$ (see, for example, \cite{rbr}).

Consider the two faces of $\Pi[K]$ incident with both $e_5$ and $e_9$.  Let $F_{\bQ_0}$ be the one bounded by $\bQ_0$.  Let $F'$ be the other; it is bounded by the cycle  $s_0\,r_5\,r_6\,r_7\,r_8\,r_9$, which we call $C'$.  Both $\bQ_0$ and $C'$ contain both $e_5$ and $e_9$.  What we would like to prove\dragominorrem{(text removed)} is that, for \dragominor{each such face $F$} with boundary $C$, there is no $K$-avoiding path contained in $F$ and having an end in each of the two components of $C-\{e_5,e_9\}$.   \dragominor{Although not necessarily true for $\Pi$, it is} true for an embedding obtained from $\Pi$ by possibly re-embedding the edges $v_0v_6$ and $v_5v_9$.

Let us begin with the possible re-embeddings.  We deal with $v_0v_6$; the argument for $v_5v_9$ is completely analogous.  If $v_0v_6$ is not embedded in $F'$, then do nothing with it.  Otherwise, it is embedded in $F'$ and we claim we can re-embed it in $F_{\bQ_0}$.  

The embedding $\Pi$ shows that $v_0v_6$ is contained in one of the two faces of $K\cup \gamma$ into which $F'$ is split.  Therefore, $v_0$ and $v_6$ must be on the same $ab$-subpath of $R$.  This implies that either $v_0=a$ or $v_6=b$, or both.  In order not to be able to embed $v_0v_6$ in $F_{\bQ_0}$, there must be a $\bQ_0$-avoiding path $P$ contained in $F_{\bQ_0}$ joining $\oo{r_0\,s_1}$ to $\oo{r_5\,r_4\,s_4\,r_9}$.  

We first consider where $D_2[P]$ can \wording{be.   There are only two possibilities:  it is} either in the face of $D_2[K-\oo{s_2}]$ bounded by $\cc{v_2,r_2,\times,r_5,v_6}\rbsp s_1\,r_1$\wording{; or in} the face incident with both $r_0$ and $s_1$.  The latter cannot occur, as $v_0v_6$ is also in that face and they \wording{overlap on the} boundary of this face.  So it must be the former.

However, in this case, both $v_0v_6$ and $P$ are in the face of $D_3[K-\oo{s_3}]$ bounded by $Q_0$, and they overlap on $Q_0$, the final contradiction that shows that $P$ does not exist, so we can re-embed $v_0v_6$ in $F_{\bQ_0}$.  Let $\Pi'$ be the embedding of $G$ obtained by any such re-embeddings of $v_0v_6$ and $v_5v_9$.

The faces $F_{\bQ_0}$ and $F'$ of $\Pi[K]$ are also faces of $\Pi'[K]$ with the same boundaries; we will continue to use these names for them, while $\bQ_0$ and $C'$ are still their boundaries.  

We now show that there is no $K$-avoiding path in $F_{\bQ_0}$ joining the two paths $P_1$ and $P_2$ of $\bQ_0-\{e_5,e_9\}$.  Such a path is necessarily in $\comp{(M_{\bQ_0})}$.  By Claim \ref{cl:bQ0cut}, such a path is necessarily $v_6v_9$.  But $\Pi$ is $H$-friendly, so $v_6v_9$ is not embedded in $\Mob$ and so, in particular, is not embedded in $F_{\bQ_0}$.  Thus, $v_6v_9$ is also not in this face of $\Pi'$, whence there is no $P_1P_2$-path in $F_{\bQ_0}$, as required.

Now consider the possibility of a $K$-avoiding path in $F'$ having its ends in each of the two paths in $C'-\{e_5,e_9\}$.  Such a path is in a $C'$-bridge $B$ embedded in $F'$.  By Claim \ref{cl:s0inK13}, $B$ has no attachment in $\oo{s_0}$.  Thus, $B$ has an attachment 
either in $\cc{v_0,r_9,w_9}$ or in $\cc{v_5,r_5,u_5}$.  

We claim it must also have an attachment in $\oo{r_6\,r_7\,r_8}$.  If not, then all its attachments are in $$\cc{v_0,r_9,w_9}\cup \cc{v_5,r_5,u_5} \cup \cc{w_5,r_5,v_6}\cup \cc{v_9,r_9,u_9}.$$  But then $B$ is a $\bQ_0$-bridge.  If it has an attachment in both $r_5$ and $r_9$, then Claim \ref{cl:v0v6} implies $B$ is one of $v_0v_6$, $v_5v_9$, and $v_6v_9$.  The first two are not embedded in the $\Pi'$-face $F'$ and the last does not have attachments in both components of $C'-\{e_5,e_9\}$.  In the alternative, either $\att(B)\subseteq r_5$ or $\att(B)\subseteq r_9$, and then we contradict either Claim \ref{cl:no<nine,five]<zero,six]} or Claim \ref{cl:no[nine,five>[zero,six>}.

So $B$ has an attachment in $\oo{r_6\,r_7\,r_8}$.    If $B$ has an attachment in $\cc{v_0,r_9,w_9}$, then $D_3[B]$ must have a crossing, which is not possible.  If $B$ has an attachment in $\cc{v_5,r_5,u_5}$, then $D_2[B]$ must have a crossing, which is not possible.  Therefore, there is no such $B$, as claimed.

For each of the faces $F_{\bQ_0}$ and $F'$ of $\Pi'$ and any points $x$ and $y$ in the interiors of $\Pi'[e_5]$ and $\Pi'[e_9]$, the preceding paragraphs show that there is a $G$-avoiding simple $xy$-arc in the face.  The union of these two arcs is a simple closed curve $\gamma'$ in $G$ that meets $\Pi'[G]$ in just the two points $x$ and $y$.

\dragominor{I}n a neighbourhood of $x$, there are points of $e_5$ on both sides of $\gamma'$.  If $\gamma'$ were contractible in $\pp$, then $\{e_5,e_9\}$ would be an edge-cut of size 2 in the 3-connected graph $G$, which is impossible.  So $\gamma'$ is non-contractible.  But this is also impossible, as it meets $G$ precisely in $x$ and $y$, showing that $G$ has a 1-drawing, the final contradiction.
\end{cproofof}

}

\chapter{$G$ embeds with all spokes in $\Mob$}\printFullDetails{

In this section, we prove that if $G\in\m2$ and $\hvfg$, then $G$ has a representativity 2 embedding in $\pp$ with $H\subseteq \Mob$.  This is an important step as it provides the embedding structure we need to find the tiles.

It turns out that we need something stronger than $H\subseteq \Mob$.  We must also show that, in addition to $H\subseteq \Mob$, the representativity 2 embedding of $G$ is such that  $M_{Q_4}$ is the only $Q_4$-local $H$-bridge $B$ for which $Q_4\cup B$ contains a non-contractible cycle.    (We remind the reader that $Q_4$ is special.  Each $H$-quad bounds a face of $\Pi[H]$. In the standard labelling, the only one of these five faces that contains an arc of $\gamma$ is the one bounded by $Q_4$.)

}\begin{theorem}\label{th:allSpokesMob} Suppose $G\in\m2$ and $\hvfg$.   Then $G$ has a representativity 2 embedding $\Pi$ in $\pp$ so \wording{that, with the standard labelling}: 
\begin{enumerate} 
\item\label{it:HinMob} $s_0$ is not exposed in $\Pi$, that is, $\Pi[H]\subseteq \Mob$; and,
\item\label{it:niceQ4}  \wordingrem{(text moved)}if $B$ is a $Q_4$-local $H$-bridge other than $M_{Q_4}$, then $\Pi[Q_4\cup B]$ has no non-contractible cycle.
\end{enumerate}
\end{theorem}\printFullDetails{

In principle, these two arguments are consecutive:  we first show we can arrange $H\subseteq \Mob$, and then deal with the $Q_4$-bridges.  However, the arguments are essentially the same.
Therefore, we shall have parallel statements and arguments, one for getting the five $H$-spokes in $\Mob$ and one for getting such an embedding with $Q_4$ nicely behaved.   (If we knew that $G$ had an embedding with $H$ not contained in $\Mob$, then we could do both simultaneously.)

}\begin{definition}\label{df:fsq}  A {\em friendly, standard quadruple\/}\index{friendly, standard quadruple}\index{fsq}\index{\fsq}, denoted \fsq, consists of $G\in \m2$, $\hvfg$, an $H$-friendly embedding $\Pi$ of $G$, and a non-contractible, simple closed curve $\gamma$ meeting $\Pi[G]$ in precisely two points, used as the reference for giving $H$ the standard labelling relative to $\Pi$.  We abbreviate friendly, standard quadruple as {\em fsq\/}.
\end{definition}\printFullDetails{

Observe that Theorem \ref{th:v10rep2} implies $G$ has a representativity 2 embedding in $\pp$.  Lemma \ref{lm:greenCyclesFriendly} (\ref{it:alwaysFriendly}) implies $G$ has an $H$-friendly embedding $\Pi$.  Any non-contractible simple closed curve $\gamma$ in $\pp$ meeting $G$ in precisely two points yields a standard labelling of $H$ relative to $\Pi$ and $\gamma$.  Summarizing, we have the following observation.

}\begin{lemma}\label{lm:fsq}  If $G\in \m2$ and $\hvfg$, then there is an fsq \fsq. \hfill\eop\end{lemma}\printFullDetails{


Let $Q^*$ be $\bQ_0$ if $s_0$ is exposed in $\Pi$ and let $Q^*$ be $Q_4$ if $s_0$ is not exposed in $\Pi$, that is, if $\Pi[H]\subseteq \Mob$.     Our first step is to show that $OD(Q^*)$ is (nearly) bipartite. Theorem \ref{th:BODquads} (\ref{it:quadBOD}) implies $OD(Q_4)$ is bipartite.  For $Q^*=\bQ_0$, this is more involved.  In the following statement, $v_1v_4$ and $v_6v_9$ are meant to be possible $\bQ_0$-bridges consisting of a single edge joining the two indicated vertices.  They need not exist in $G$.

}\begin{lemma}\label{lm:bQ0bipartite} Let \fsq\ be an fsq.  
 If $s_0$ is exposed in $\Pi$, then $OD(\bQ_0)-\{v_1v_4,v_6v_9\}$ is bipartite.
\end{lemma}\printFullDetails{

The following observations will be needed throughout the proof of Theorem \ref{th:allSpokesMob} and, in particular, the proof of Lemma \ref{lm:bQ0bipartite}.

}\begin{definition}\label{df:scriptN} Let \fsq\ be an fsq and let $Q^*$ be either $\bQ_0$ (if $s_0$ is exposed) or $Q_4$ (otherwise).  Then $\mathcal N$\index{$\mathcal N$} --- a function of \fsq\ --- denotes the set of $Q^*$-bridges $B$ other than $M_{Q^*}$ for which $\Pi[Q^*\cup B]$ has a non-contractible cycle.   In the case $Q^*=\bQ_0$, any of $v_1v_4$ and $v_6v_9$ that occurs in $G$ is a $\bQ_0$-bridge\dragominorrem{(text removed)} $B$ for which $\Pi[\bQ_0\cup B]$ has a non-contractible cycle, and \wording{we do not} include these in $\mathcal N$. \end{definition}\printFullDetails{

 We remark that, if $s_0$ is exposed in $\Pi$, then Theorem \ref{th:expSpokeNoAtt} implies the $(H\cap \Mob)$-bridge $B^0$ containing $s_0$ is distinct from $M_{\bQ_0}$.  In this case, $B^0\in \mathcal N$.  If $s_0$ is not exposed in $\Pi$, then $Q^*=Q_4$.  If $\mathcal N=\varnothing$, then $\Pi$ satisfies the conclusions of Theorem \ref{th:allSpokesMob}.  Therefore, in this case, we may assume $\mathcal N\ne \varnothing$.

 Before we can prove Lemma \ref{lm:bQ0bipartite}, we need some results common to both cases.

An easy corollary of the following lemma will be used to deal with the main case in the proof of Lemma \ref{lm:bQ0bipartite}.

}\begin{lemma}\label{lm:D2D3}  Let $D$ be a 1-drawing of $V_8$ (with the usual labelling) in which $Q_1$ is crossed.  Then:
\begin{enumerate} 
\item $Q_3$ bounds a face of $D$; and
\item if $\bQ_0$ is crossed in $D$, then 
 either $r_1\textrm{ crosses }r_4$  or  $r_5 \textrm{ crosses } r_0$.
\end{enumerate}
\end{lemma}\printFullDetails{

\begin{cproof}  As $Q_1$ is crossed in $D$, either $r_1$ crosses $r_4\,r_5\,r_6$ in $D$ or $r_5$ crosses $r_0\,r_1\,r_2$ in $D$.   This already shows that $Q_3$ bounds a face of $D$.

As $\bQ_0$ is crossed in $D$,  either $r_7\,r_0$ or $r_3\,r_4$ is crossed in $D$.
Compare each of these with the possible crossing of $Q_1$.  In the former case, $r_0$ crosses $r_5$, while in the latter case $r_4$ crosses $r_1$.  
\end{cproof}

The following is the simple corollary that we will use.

}\begin{corollary}\label{co:D2D3}   Let $G\in\m2$ and $\hvfg$.   Let $D_2$ be a 1-drawing of $G-\oo{s_2}$\wordingrem{(text removed)}
.  Then:
\begin{enumerate} 
\item $Q_4$ bounds a face of $D_2[H-s_2]$; and
\item if $\bQ_0$ is crossed in $D_2$, then 
 either $r_6\,r_7$ crosses $r_1$ or  $r_1\,r_2$ crosses $r_5$ (see Figure \ref{D2noAtt} for the possibilities for $D_2[H-\oo{s_2}])$.

Likewise, if $D_3$ is a 1-drawing of $G-\oo{s_3}$ in which\wordingrem{(text removed)} 
$\bQ_0$\wordingrem{(text removed)}
\wording{\ is crossed,} then the two possibilities for $D_3[H-\oo{s_3}]$ are illustrated in Figure \ref{D3noAtt}.
\end{enumerate}
\end{corollary}\printFullDetails{

\begin{cproof}  Theorem \ref{th:BODquads} implies $\bQ_2$ has BOD.  Lemma \ref{lm:BODcrossed} implies $\bQ_2$ is crossed in $D_2$.  The results now follow immediately from Lemma \ref{lm:D2D3}.\end{cproof}

\begin{figure}[!ht]
\begin{center}
\scalebox{1.0}{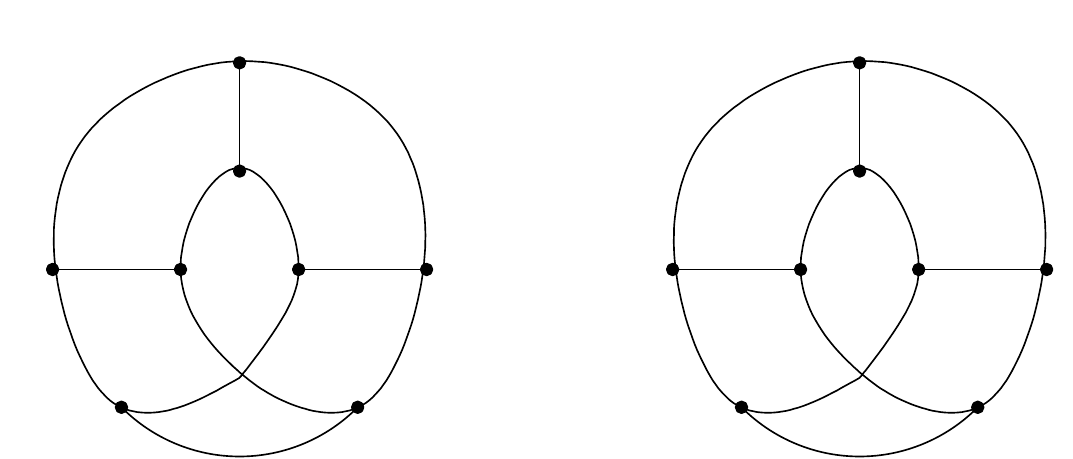}
\end{center}
\caption{The two possibilities for $D_2$.}\label{D2noAtt}
\end{figure}

\begin{figure}[!ht]
\begin{center}
\scalebox{1.0}{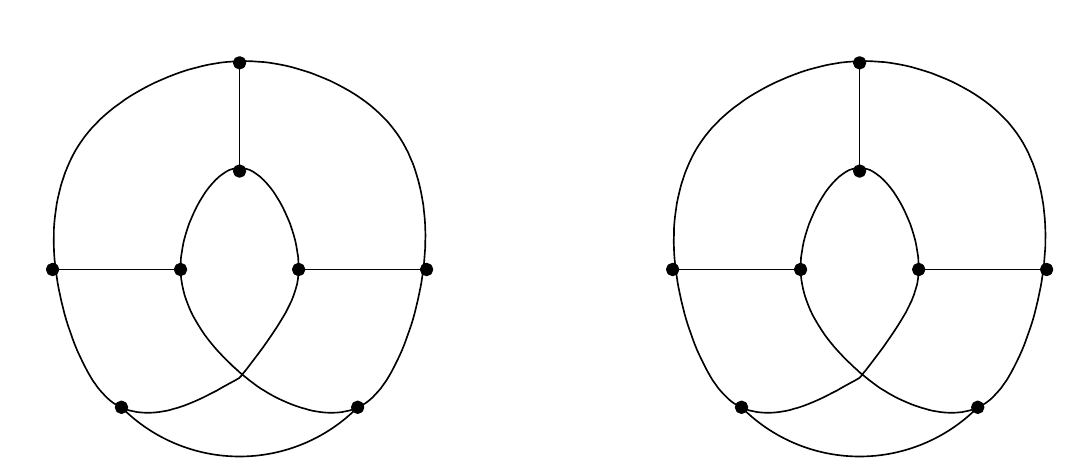}
\end{center}
\caption{The two possibilities for $D_3$.}\label{D3noAtt}
\end{figure}

\def\r5{r^*_{+5}}

Let $r^*$ denote $r_9\cup r_0$ in the case $Q^*=\bQ_0$ and $r_9$ in the case $Q^*=Q_4$.  We also let $\r5$ denote the other component of $Q^*\cap R$.

}\begin{lemma}\label{lm:attB}  Let \fsq\ be an fsq.  If $B\in\mathcal N$, then $\Pi[B]\subseteq \Disc$, $\att(B)\subseteq r^*\cup \r5$, and $B$ has an attachment in each of $r^*$ and $\r5$.  \end{lemma}\printFullDetails{

\begin{cproof} If $\Pi[B]\subseteq \Mob$, then $\Pi[Q^*\cup B]$ is contained in a closed disc and, therefore, has only contractible cycles, a contradiction.  Thus, $\Pi[B]\subseteq\Disc$.   It now follows that $\att(B)$ is contained in the intersection of $\bQ_0$ with the boundary of $\Disc$; that is, $\att(B)\subseteq r^*\cup \r5$. 

Suppose by way of contradiction that $\att(B)\subseteq r^*$.  Let $\bar r^*$ be a minimal subpath of $r^*$ containing $\att(B)$.    Then there is a non-contractible cycle $C$ contained in $B\cup\bar r^*$.  

Let $F$ be the closed $(\Pi[H]\cup \gamma)$-face containing $\Pi[B]$.  Then $F$ contains $\Pi[B\cup \bar r^*]$, so the non-contractible cycle $\Pi[C]$ is contained in the closed disc $F$, a contradiction.  So $\att(B)$ is not contained in $r^*$ and, likewise, it is not contained in $r^*_{+5}$.  \end{cproof} 



Let \fsq\ be an fsq, with $s_0$ exposed in $\Pi$.  Suppose $D_2$ is a 1-drawing of $G-\oo{s_2}$ in which $\bQ_0$ is crossed.\wordingrem{(Text removed.)} 
Corollary \ref{co:D2D3} implies that $D_2[H-\oo{s_2}]$ is one of the \wording{two drawings illustrated} in Figure \ref{D2noAtt}.  The {\em outside of $D_2[\bQ_0]$\/}\index{outside} is the face of $D_2[\bQ_0]$ containing $D_2[s_3]$.  The {\em inside\/}\index{inside} is the other face of $D_2[\bQ_0]$.  Likewise, if $D_3$ is a 1-drawing of $G-\oo{s_3}$ in which $\bQ_0$ is crossed, then the {\em outside of $D_3[\bQ_0]$\/} is the face of $D_3[\bQ_0]$ containing $D_3[s_2]$.

}\begin{lemma}\label{lm:bQ0limitA} Let \fsq\ be an fsq, with $s_0$ exposed in $\Pi$.  For $i=2,3$, let $D_i$ be a 1-drawing of $G-\oo{s_i}$ in which $\bQ_0$ is crossed.   Suppose  $B$ is a $\bQ_0$-bridge in $\mathcal N$.  
\begin{enumerate}\item\label{it:d2b} If $D_2[B]$ is outside of $D_2[\bQ_0]$, then $B\in \{v_1v_5,v_0v_6\}$.  
\item\label{it:d3b} If $D_3[B]$  is outside of $D_3[\bQ_0]$, then $B\in \{v_0v_4,v_5v_9\}$.  
\end{enumerate}
\end{lemma}\printFullDetails{

\begin{cproof}  We prove (\ref{it:d2b}); (\ref{it:d3b})	 is completely analogous.   We remark that $B\ne B^0$ as $D_2[s_0]$ is inside $D_2[\bQ_0]$.   Lemma \ref{lm:attB} shows that either: (i)  $\att(B)\subseteq \cc{b,r_5,v_6}\cup \cc{v_9,r_9,a}$ and $B$ has attachments in both $\cc{b,r_5,v_6}$ and $\cc{v_9,r_9,a}$; or  (ii) $\att(B)\subseteq \cc{a,r_9,v_0}\rbsp r_1\cup r_4\lbsp\cc{v_5,r_5,b}$ and $B$ has attachments in both $\cc{a,r_9,v_0}\rbsp r_1$ and $r_4\lbsp\cc{v_5,r_5,b}$. 

Suppose first that $D_2$ is the left-hand possibility illustrated in Figure \ref{D2noAtt}.  Considering $D_2$, we see that $v_1$ is one attachment of $B$ and the others are in $r_4\,r_5$.  

Now consider the possibilities for $D_3[B]$.  We see that $D_3[B]$ can be outside $D_3[\bQ_0]$ in only one of the two possible $D_3$'s, namely the right-hand one, and then only if $\att(B)=\{v_1,v_4\}$.  But in this case $B$ is just the edge $v_1v_4$, which is not in $\mathcal N$.  So $D_3[B]$ is inside $D_3[\bQ_0]$.   It now follows from this and the previous paragraphs that $\att(B)\subseteq \{v_1\}\cup r_5$.   

Putting this information into $\Pi$, we see that the only possibility for $B$, which is embedded in $\Disc$ and not in $\Mob$, is that $B=v_1v_5$. 

In the case $D_2$ is the right-hand possibility in Figure \ref{D2noAtt},  $D_2$ shows that $\att(B)\subseteq \{v_6\}\cup r_9\,r_0$.   Since $v_6v_9\notin \mathcal N$,  $B\ne v_6v_9$, so $D_3[B]$ is not outside $D_3[\bQ_0]$.  Therefore, $D_3$ shows $\att(B)\subseteq \{v_6\}\cup r_0$.    

Again we recall that $B$ is embedded in $\Disc$ in $\pp$.  If $B$ is embedded in the face bounded by $\cc{a,r_9,v_0,s_0,v_5,r_5,b,\alpha,a}$, then $b=v_6$ and the only other possible attachment for $B$ is $v_0$, as required.  If $B$ is embedded in the face bounded by $\cc{b,r_5,v_6}\rbsp r_6\,r_7\,r_8\lbsp\cc{v_9,r_9,a,\alpha,b}$, then $a=v_0$ and again this is the only possible attachment other than $v_6$, as required.
\end{cproof}

Let $N$ be the graph
$\displaystyle{
\bigcup_{B\in \mathcal N}B\,.}
$

}\begin{lemma}\label{lm:noncont} Let \fsq\ be an fsq. Then \dragominor{there are not disjoint $(N\cap r^*)(N\cap r^*_{+5})$-paths in $N$}.  In particular, if $Q^*=\bQ_0$ and $|\mathcal N|\ge 2$, then either every $B\in \mathcal N$ has only $v_0$ as an attachment in $r_9\,r_0$ or every $B\in \mathcal N$ has only $v_5$ as an attachment in $r_4\,r_5$. \end{lemma}\printFullDetails{  

\begin{cproof} Suppose by way of contradiction that $P_1$ and $P_2$ are disjoint $r^*\r5$-paths in $N$, with, for $j=1,2$, $P_j$ having the end $p_j$ in $r^*$ and the end $q_j$ in $\r5$.  Choose the labelling so that, in $r^*$, $p_1$ is closer to $v_9$ than $p_2$ is.  There are three possibilities for how $P_1$ and $P_2$ are embedded by $\Pi$:  both in the (closed) disc contained in $\Disc$ bounded by  $\cc{a,r_9,v_0}\rbsp r_0\,r_1\,r_2\,r_3\,r_4\rbsp\cc{v_5,r_5,b}\rbsp\alpha$ (recall that $\alpha=\gamma\cap \Disc$);  both in the disc in $\Disc$ bounded by $\cc{b,r_5,v_6}\rbsp r_6\,r_7\,r_8\lbsp\cc{v_9,r_9,a}\rbsp\alpha$; or one in each of these discs.  In all cases, we conclude that $q_1$ is closer in $\r5$ to $v_6$ than $q_2$ is.  Summarizing, we have the following.

\medskip\noindent{\bf Fact 1} {\em Any two disjoint $r^*\r5$-paths in $N$ overlap on $Q^*$.}

\medskip
For $Q^*=Q_4$ we are done:  Corollary \ref{co:D2D3} implies $D_2[Q_4]$ bounds a face of $D_2[H-\oo{s_2}]$.  Both $P_1$ and $P_2$ have ends in both $r^*$ and $\r5$, so both must be inside $D_2[Q_4]$, yielding the contradiction that they cross in $D_2[Q_4]$.  

Now suppose $Q^*=\bQ_0$.  For $i=2,3$, $D_i[\bQ_0]$ is not self-crossing; thus Fact 1 implies that $D_i[P_1]$ and $D_i[P_2]$ are on different sides of $D_i[\bQ_0]$.  If $\bQ_0$ is clean in $D_i$, then we have a contradiction, as no face of $D_i[H-\oo{s_i}]$ is incident with both $r^*$ and $\r5$ except the ones bounded by $Q_4$ and $Q_0$.  

Thus,  $\bQ_0$ is crossed in $D_i$.  By Lemma \ref{lm:bQ0limitA}, the one that is outside is one of $v_0v_4$, $v_0v_6$, $v_1v_5$, and $v_5v_9$.  We treat in detail that this one is $v_0v_4$, as the other cases are completely analogous.  It is in $D_3$ that $v_0v_4$ is outside $D_3[\bQ_0]$.

Because $q_1$ is closer to $v_6$ than $q_2$ is, $q_1$ cannot be $v_4$; it follows that it is $P_2$ that is $v_0v_4$.
Lemma \ref{lm:bQ0limitA} also implies that $P_2$, that is $v_0v_4$,  is not outside $D_2[\bQ_0]$ and, therefore, it is inside $D_2[\bQ_0]$.   Thus, $P_1$ is outside $D_2[\bQ_0]$.  By Lemma \ref{lm:bQ0limitA},  $P_1$ is one of $v_0v_6$ and $v_1v_5$.  By choice of the labelling, it cannot be that $v_1$ is an end of $P_1$,  so $P_1=v_0v_6$, which is not disjoint from $P_2=v_0v_4$, a contradiction.
We conclude that there are not such disjoint paths.  

For the ``in particular",  there is a cut vertex  $u$ of $N$ separating $N\cap (r_9\,r_0)$ and $N\cap (r_4\,r_5)$ in $N$, as claimed.  \wording{As $s_0$ is a $(\cc{r_9\,r_0})\,(\cc{r_4\,r_5})$-path in $N$, we deduce $u\in s_0$}.  If $B^0$ is not the only member of $\mathcal N$, then any other element $B$ of $\mathcal N$ shares the vertex $u$ with $B^0$, so $u$ is an attachment of both.  But $u\in s_0$ implies $u\in \{v_0,v_5\}$.
\end{cproof}

As a final preparatory remark, we have the following.

}\begin{lemma}\label{lm:Noverlap} Let \fsq\ be an fsq.   Let $B$ and $B'$ be distinct elements of $\mathcal N$.  Then: \begin{enumerate}  \item\label{it:noOverlap} $B$ and $B'$ do not overlap on $Q^*$;  and \item\label{it:overlapM} either $B$ overlaps $M_{Q^*}$ on $Q^*$ or $Q^*=Q_4$ and $B$ is either $v_4v_9$ or $v_0v_5$. \end{enumerate}\end{lemma}\printFullDetails{

\begin{cproof} In the case $Q^*=Q_4$,  Corollary \ref{co:D2D3} and Lemma \ref{lm:attB} imply $B$ and $B'$ are both drawn inside the face of $D_2[H-\oo{s_2}]$ bounded by $Q_4$ and, therefore, they do not overlap, yielding (\ref{it:noOverlap}) for $Q_4$.

For $Q^*=\bQ_0$, if both $B$ and $B'$ are \dragominor{in the same face} of either $D_2[\bQ_0]$ or $D_3[\bQ_0]$, then they obviously do not overlap on $\bQ_0$.  Thus, we may assume one is outside $D_2[\bQ_0]$ and the other is inside $D_2[\bQ_0]$ and that one is outside $D_3[\bQ_0]$ and the other is inside $D_3[\bQ_0]$.

By Lemma \ref{lm:bQ0limitA}, the one outside $D_2[\bQ_0]$ is either $v_1v_5$ or $v_0v_6$, while the one outside $D_3[\bQ_0]$ is either $v_0v_4$ or $v_5v_9$.  Thus, we may assume $B\in\{v_1v_5,v_0v_6\}$ and $B'\in\{v_0v_4,v_5v_9\}$.  But none of the four possibilities is an overlapping pair, which is (\ref{it:noOverlap}) for $\bQ_0$.

As for overlapping $M_{Q^*}$, we suppose first that $B$ has an attachment $x$ in the interior of one of $r^*$ and $\r5$.  (The ``in particular" part of Lemma \ref{lm:noncont} implies this is always the case when $Q^*=\bQ_0$.)  In this case, it is a simple exercise to see that $x$, together with any attachment of $B$ in the other one of $r^*$ and $\r5$, are skew \wording{to at least one} of the pairs of diagonally opposite corners of $Q^*$ (in the case of $Q_4$ these pairs are $\{v_9,v_5\}$ and $\{v_4,v_0\}$; for $\bQ_0$, they are $\{v_9,v_6\}$ and $\{v_4,v_1\}$).  \wording{Thus, $B$ overlaps $M_{Q^*}$.}

In the remaining case, $Q^*=Q_4$ and $\att(B)\subseteq \{v_9,v_0,v_5,v_4\}$.  If both $v_9$ and $v_5$ are attachments, then $B$ is again skew to $M_{Q^*}$; the same happens if both $v_0$ and $v_4$ are attachments.  The only remaining cases are:  $\att(B)=\{v_4,v_9\}$ and $\{v_0,v_5\}$\wording{, as claimed}.  \end{cproof}

The next result contains the essence of the proof of Lemma \ref{lm:bQ0bipartite}.

}\begin{lemma}\label{lm:notConnected}  Let \fsq\ be an fsq.  Suppose $B_1\in \mathcal N$, $B_k=M_{Q^*}$, and $B_1,B_2,\dots,B_k$ is an induced cycle in $OD(Q^*)$.  Then either 
\begin{enumerate} \item\label{it:kIs3} $Q^*=\bQ_0$, $k=3$, and $B_2\in\{v_1v_4,v_6v_9\}$ or 
\item\label{it:kIsEven} $k$ is even and $B_{k-1}\in \mathcal N\cup\{v_1v_4,v_6v_9\}$. 
\end{enumerate}
\end{lemma}\printFullDetails{

\begin{cproof}  {\bf Case 1.} {\em  $k$ is odd.\/}

\medskip  Theorem \ref{th:BODquads} implies $OD(Q_4)$ is bipartite.  Therefore, $Q^*=\bQ_0$ and $s_0$ is exposed in $\Pi$.

For $i=2,3$, let $e_i$ be the edge of $s_i$ incident with $v_i$ and let $D_i$ be a 1-drawing of $G-e_i$.   Theorem \ref{th:BODquads} implies $\bQ_i$ has BOD; Lemma \ref{lm:BODcrossed} implies $\bQ_i$ is crossed in $D_i$.

If, for some $i\in \{2,3\}$, $\bQ_0$ is clean in $D_i$, then Lemma \ref{lm:cleanBOD} implies $\bQ_0$ has BOD, yielding the contradiction that $k$ is even.   Therefore,  $\bQ_0$ is crossed in both $D_2$ and $D_3$.

\begin{claim}\label{cl:v1v4v6v9Bi}  If some $B_i$ is either $v_1v_4$ or $v_6v_9$, then $i=2$ and $k=3$.  \end{claim}

\begin{proof}  Since both $v_1v_4$ and $v_6v_9$ overlap $M_{\bQ_0}$, neither is in $\mathcal N$, $B_1$ is in $\mathcal N$,  and the cycle is induced, it must be that $i=k-1$.  For sake of definiteness, we suppose $B_{k-1} = v_1v_4$; the alternative is treated completely analogously.   

Because $B_{k-1}=v_1v_4$, we deduce that $D_2$ is the left-hand one of the two drawings in Figure \ref{D2noAtt}, while $D_3$ is the right-hand drawing in Figure \ref{D3noAtt}; in both drawings, $B_{k-1}$ is outside $\bQ_0$.  

We note that $B^0$ overlaps $v_1v_4$, so if $B_1$ is $B^0$, then $k=3$, as claimed.  Otherwise, $B_1\in\mathcal N\setminus\{B^0\}$.  By Lemma \ref{lm:noncont}, either the only attachment of $B_1$ in $r_9\,r_0$ is $v_0$ or the only attachment of $B_1$ in $r_4\,r_5$ is $v_5$.  For sake of definiteness, we assume the former; the latter is completely analogous.  In order not to overlap $v_1v_4$, the only attachment for $B_1$ in $r_4\,r_5$ is $v_4$.  Therefore, either $k=3$ and we are done, or $B_1$ is just the edge $v_0v_4$.   We show that $B_1=v_0v_4$ is not possible.

Suppose that $B_1=v_0v_4$.  Because we know $D_2$, we see that $D_2[B_1]=D_2[v_0v_4]$ is inside $D_2[\bQ_0]$, while $D_2[B_{k-1}]=D_2[v_1v_4]$ is outside.  In $D_3$, both are outside.  But this is impossible, as $B_1,B_2,B_3,\dots,B_{k-2},$ $B_{k-1}$ alternate sides of $\bQ_0$ in both $D_2$ and $D_3$. 

We conclude that $B_1=v_0v_4$ is impossible and therefore $k=3$, as claimed. \end{proof}

It remains to show that no other possibility can occur with $k$ odd.  So suppose no $B_i$ is either $v_1v_4$ or $v_6v_9$. Suppose some $B_i$ other than $B_1$ is in $\mathcal N$.  As $B_i$ overlaps $M_{\bQ_0}$ and the cycle $B_1,B_2,\dots,B_k$ is induced, Lemma \ref{lm:Noverlap}  implies $i=k-1$.  The same lemma implies $k\ge 5$.   Therefore, Lemma \ref{lm:planeNotOverlap} implies $B_1,B_2,\dots,B_{k-2},B_{k-1}$ alternate sides of $\Pi[\bQ_0]$.   Since $k$ is odd, $B_1$ and $B_{k-1}$ are on different sides of $\Pi[\bQ_0]$, contradicting the fact that both are in $\mathcal N$.  Hence no other $B_i$ is in $\mathcal N$.  

By Lemma \ref{lm:bQ0limitA}, for at least one $i\in\{2,3\}$, $D_i[B_1]$ is inside $D_i[\bQ_0]$.  For the sake of definiteness, we consider the case $i=2$ and $D_2$ is the left-hand drawing of $H-\oo{s_2}$ in Figure \ref{D2noAtt}; the remaining cases are completely analogous.  Thus, either $B_1$ is $B^0$ or $B_1$ is either a $Q_0$- or a $Q_1$-bridge.

Since $k$ is odd, $B_{k-1}$ is on the other side of $D_2[\bQ_0]$ from $B_1$.   Therefore, $B_{k-1}$ is outside $D_2[\bQ_0]$.  In order to understand how $B_{k-1}$ can overlap $M_{\bQ_0}$ in $D_2$, we analyze $D_2[M_{\bQ_0}]$.

Let $e$ be the edge of $M_{\bQ_0}$ that is crossed in $D_2$.   The end $w$ of $e$ outside $D_2[\bQ_0]$ is in $\Nuc(M_{\bQ_0})$.  If the other end $u$ of $e$ is not in $\Nuc(M_{\bQ_0})$, then $u=v_6$ and $\cc{\times,r_6,v_6}$ is the only part of $M_{\bQ_0}$ inside $D_2[\bQ_0]$.  Otherwise, $\Nuc(M_{\bQ_0})-\{e_2,e\}$ is not connected.  Since $\Nuc(M_{\bQ_0})-e_2$ is connected, $\Nuc(M_{\bQ_0})-\{e_2,e\}$ consists of the component  inside $D_2[\bQ_0]$ and the component $O$ outside.   In particular, $M_{\bQ_0}-\{e_2,e\}$ consists of two $\bQ_0$-bridges in $G-\{e_2,e\}$.  Let $I$ be the one contained inside $D_2[\bQ_0]$ and let $O$ be the one outside.  All attachments of $M_{\bQ_0}$ are attachments of either $I$ or $O$, and possibly both.  In the case $u=v_6$, we take $I$ to be the portion of $e$ from $\times$ to $v_6$.

We observe that $D_2$ shows that, except for one end of $e$, all the attachments of $I$ are in $Q_0$.   On the other hand, Theorem \ref{th:expSpokeNoAtt} implies that $M_{\bQ_0}$, and, therefore $I$, has no attachment in $\oo{s_0}$.  The embedding $\Pi$ shows that $I$ has no attachment in $\oo{r_0}$: otherwise, $I$ is not just $\cc{\times,e_6,v_6}$ and $u\ne v_6$.  Thus, the simple closed curve $s_1\,r_1\,r_2\,r_3\,s_4\lbsp\cc{v_9,r_9,a}\rbsp \alpha\lbsp\cc{b,r_5,v_6}$ bounds a closed disc in $\pp$ separating $u$ from $\oo{r_0}$ and is disjoint from $\Nuc(I)\cup \oo{r_0}$.  Unless $v_0=a$, the same simple closed curve separates $u$ from $v_0$; thus, if $v_0$ is an attachment of $I$, then $a=v_0$.

Because $B_{k-1}$ is outside $D_2[\bQ_0]$ and $\att(B_{k-1})\subseteq \bQ_0$, there are four candidates for the face of $D_2[H-\oo{s_2}]$ that contains $B_{k-1}$.  The one bounded by $Q_3$ is not possible: if $B_{k-1}$ were in that face, it would not overlap $M_{\bQ_0}$, as all the $M_{\bQ_0}$ attachments there would be in $s_4$ and, therefore, all in $O$ and not in $I$; both $B_{k-1}$ and $O$ being outside $D_2[\bQ_0]$ shows they do not overlap.

The face of $D_2[H-\oo{s_2}]$ incident with $\cc{\times,r_0,v_1}$ is not a possibility for $B_{k-1}$ for exactly the same reason:  the only attachment of $I$ there can be $v_1$ and $v_1$ is not part of a pair of attachments of $M_{\bQ_0}$ that are skew to two attachments of $B_{i-1}$, which are all contained in $\cc{\times,r_0,v_1}$.

The face of $D_2[H-\oo{s_2}]$ incident with  $r_8\,r_9$ is also not a possibility for $B_{k-1}$.  To see this, $v_0$ is the only possible attachment of $I$ in the boundary of this face.  Thus, $v_0$ is an attachment of $I$ and $B_{k-1}$ must have attachments in each of $\co{v_9,r_9,v_0}$ and $\oc{v_0,r_0,\times}$.  However, in $\Pi$ we must have $a=v_0$ and then there is no way to embed $B_{k-1}$.  

Therefore, $B_{k-1}$ is in the face of $D_2[H-\oo{s_2}]$ incident with $r_5\,s_1$.

By way of contradiction, suppose $B_{k-1}$ is outside $D_3[\bQ_0]$. Identical arguments as those just above show that $B_{k-1}$ is in the face of $D_3[H-\oo{s_2}]$ incident with $r_9\,s_4$.  Because the previous paragraph shows $\att(B_{k-1})\subseteq r_4\,r_5\,s_1$, it cannot overlap $M_{\bQ_0}$ using an attachment of the portion of $M_{\bQ_0}$ that is inside $D_3[\bQ_0]$ and, therefore, it cannot overlap $M_{\bQ_0}$ at all, a contradiction.  Therefore, $B_{k-1}$ is inside $D_3[\bQ_0]$.  This implies $B_{k-1}$ is either a $Q_0$- or $Q_4$-bridge.

If $B_{k-1}$ is a $Q_4$-bridge, then $\att(B_{k-1})\subseteq r_4$ (because of $D_2$). Letting $\bar r$ denote the minimal subpath of $r_4$ containing $\att(B_{k-1})$, $D_2$ shows that no attachment of $I$ is in $\oo{\bar r}$ and, because $O$ and $B_{k-1}$ do not overlap (in $D_2$), $O$ also has no attachment in $\oo{\bar r}$.  Consequently, $B_{k-1}$ does not overlap $M_{\bQ_0}$, a contradiction.  Therefore, $B_{k-1}$ is a $Q_0$-bridge.

Because $B_{k-2}$ is inside $D_2[\bQ_0]$, has no attachments in $s_0$, and overlaps $B_{k-1}$ as $\bQ_0$-bridges, we see that $B_{k-2}$ is also a $Q_0$-bridge.  Continuing back, we see that each of $B_{k-3}$, \dots, $B_2$ is a $Q_0$-bridge and that $B_1$ is outside $D_3[\bQ_0]$.  By Lemma \ref{lm:bQ0limitA}, $B_1$ is either $v_0v_4$ or $v_5v_9$.  But neither of these overlaps $B_2$.  This contradiction 
shows that, except for the case described in Claim \ref{cl:v1v4v6v9Bi}, $k$ is even.

\medskip\noindent{\bf Case 2.}  {\em $k$ is even.}

\medskip
For each $i=2,3,\dots,k-2$, $B_i\cup Q^*$ has no non-contractible cycle in $\pp$.  Thus, Lemma \ref{lm:planeNotOverlap} implies $B_1$ and $B_{k-1}$ are on the same side  of $Q^*$ in $\pp$; since $B_1$ is $Q^*$-exterior, we have that $B_{k-1}$ is $Q^*$-exterior.  If $\Pi[Q^*\cup B_{k-1}]$ has no non-contractible cycle, then Lemma 
\ref{lm:planeNotOverlap} shows that it cannot overlap $M_{Q^*}$, a contradiction.  In the case $Q^*=Q_4$, this implies that $B_{k-1}$ is in $\mathcal N$, while if $Q^*=\bQ_0$, then $B_{k-1}$ is in $\mathcal N\cup \{v_1v_4,v_6v_9\}$.  \end{cproof}

\bigskip
\begin{cproofof}{Lemma \ref{lm:bQ0bipartite}} We show that any odd cycle $C$ in $OD(\bQ_0)$ contains either $v_1v_4$ or $v_6v_9$.   Theorem \ref{th:BODquads} (\ref{it:nearlyBOD}) implies that $OD(\bQ_0)-M_{\bQ_0}$ is bipartite.  Therefore, $C$ contains $M_{\bQ_0}$.  Lemma \ref{lm:notConnected} shows that any odd cycle in $OD(\bQ_0)$ containing $M_{\bQ_0}$ and an element of $\mathcal N$ has length 3 and contains one of $v_1v_4$ and $v_6v_9$, as required.  

Thus, we may suppose $C$ avoids $\mathcal N\cup\{v_1v_4,v_6v_9\}$; let $C=(B_1,B_2,\dots,B_{2k},M_{\bQ_0})$.  For \wording{each $i=1,2,\dots,2k$}, $\Pi[B_i\cup \bQ_0]$ has no non-contractible cycles in $\pp$.  Lemma \ref{lm:planeNotOverlap} implies \wording{$B_i$ and $B_{i+1}$} are on different sides of $\Pi[\bQ_0]$.  From this, parity implies that $B_1$ and $B_{2k}$ are on opposite sides of $\Pi[\bQ_0]$.  On the other hand, they are both on the side of $\Pi[\bQ_0]$ not containing $M_{\bQ_0}$, a contradiction.
\end{cproofof}

We are now prepared for the proof of Theorem \ref{th:allSpokesMob}.

\bigskip

\begin{cproofof}{Theorem \ref{th:allSpokesMob}} By Theorem \ref{th:v10rep2}, $G$ has a representativity 2 embedding $\Pi$ in $\pp$.   For (\ref{it:HinMob}), if no spoke is exposed in $\Pi$, then we are \wording{done; thus, with} the standard labelling, we may suppose that $s_0$ is exposed in $\Pi$.  From Theorem \ref{th:expSpokeNoAtt}, we know that the $\bQ_0$-bridge $B^0$ containing $s_0$ is different from $M_{\bQ_0}$.    From Lemma \ref{lm:bQ0bipartite}, we know that $OD(\bQ_0)-\{v_1v_4,v_6v_9\}$ is bipartite and from Theorem \ref{th:BODquads} (\ref{it:nearlyBOD}), we know that $\comp{(M_{\bQ_0})}$ is planar.

We need to modify $\Pi$ so that the set $\mathcal N$ \wording{(Definition \ref{df:scriptN})} becomes empty.  \wording{We start with terminology that} will be useful for the next claims.

}\begin{definition}  Let $L$ be a graph.  A path $(v_1,v_2,\dots,v_k)$ in $L$ is {\em chordless in $L$\/}\index{chordless} if there is no edge $v_iv_j$ of $L$ that is not in $P$ except possibly $v_1v_k$.  \end{definition}\printFullDetails{
  
The following is a simple consequence of Lemma \ref{lm:notConnected}.

\begin{claim}\label{cl:noNtoMbQ0path}  \begin{enumerate}\item \dragominor{If $Q^*=\bQ_0$, then e}very $\mathcal N M_{\bQ_0}$-path in $OD(\bQ_0)$ of length at least two contains one of $v_1v_4$ and $v_6v_9$.
\item \dragominor{If $Q^*=Q_4$, then e}very chordless $\mathcal NM_{Q_4}$-path in $OD(Q_4)$ of length at least two has length exactly two, one end is either $v_4v_9$ or $v_0v_5$, and that end does not overlap $M_{Q_4}$.
\end{enumerate}
 \end{claim}

\begin{proof}  Suppose first that $Q^*=\bQ_0$.  Let  $P$ be any $\mathcal N M_{\bQ_0}$-path in $OD(\bQ_0)$ that has length at least 2.    We may assume $P$ is chordless: otherwise there is a shorter $\mathcal N M_{\bQ_0}$-path $P'$ of length at least 2 and $V(P')\subseteq V( P)$; if $P'$ contains either $v_1v_4$ or $v_6v_9$, then so does $P$.  By Lemma \ref{lm:Noverlap} (\ref{it:overlapM}), the ends of $P$ are adjacent in $OD(\bQ_0)$.  Thus, $P$ together with this edge of $OD(\bQ_0)$ makes an induced cycle.  As this cycle has only one vertex in $\mathcal N$, Lemma \ref{lm:notConnected} implies the cycle has length 3 and contains one of $v_1v_4$ and $v_6v_9$.   

Now suppose that $Q^*=Q_4$ and $P=(B_1,B_2,\dots,B_k,M_{Q_4})$ is a chordless $\mathcal NM_{Q_4}$-path in $OD(Q_4)$ of length at least 2.  Then $B_1\in \mathcal N$.  Since $P$ is chordless and $B_k\notin \mathcal N$, Lemma \ref{lm:notConnected} (\ref{it:kIsEven}) implies $B_1$ does not overlap $M_{Q_4}$.  Now Lemma \ref{lm:Noverlap} (\ref{it:overlapM}) implies $B_1$ is either $v_4v_9$ or $v_0v_5$.
Thus, $B_2$ is skew to \wording{$B_1$. Since $\att(B_1)\subseteq \att(M_{Q_4})$, $B_2$ is also} skew to $M_{Q_4}$. Since $P$ is chordless, $k=2$, as required. 
\end{proof}

If $Q^*=\bQ_0$, then set $\mathcal M$ to be the set $\{M_{\bQ_0},v_1v_4,v_6v_9\}$, while if $Q^*=Q_4$, then set $\mathcal M$ to be the set $\{M_{Q_4},v_4v_9,v_0v_5\}$.   In either case, let $\mathcal M^-=\mathcal M\setminus\{M_{Q^*}\}$.  

Let $\mathcal N^+$ be  the set of $Q^*$-bridges $B$ so that there is an $\mathcal N B$-path in $OD(Q^*)$ that is disjoint from $\mathcal M$.   The next lemma shows that $\mathcal N^+$ consists of the members of $\mathcal N$, which have attachments in both $r^*$ and $\r5$, and other $Q^*$-bridges $B$ that simply extend out along either $r^*$ or $\r5$.  This structure is what will allow us to find natural ``breaking points" $a'$ and $b'$ in $r^*$ and $\r5$, respectively, to allow us to ``flip" the members of $\mathcal N$ into $\Mob$, yielding the embedding with $H\subseteq \Mob$ and $\mathcal N=\varnothing$.

\begin{claim}\label{cl:attNplus}  If $B\in \mathcal N^+$, then $\att(B)\subseteq r^*\cup \r5$.   Furthermore, if $B\in\mathcal N^+\setminus \mathcal N$, then either $\att(B)\subseteq r^*$ or $\att(B)\subseteq \r5$. \end{claim}

\begin{proof} Let $P$ be a shortest $\mathcal N B$-path in $OD(Q^*)$ that is disjoint from $\mathcal M$.   We proceed by induction on the length of $P$.  

If $B\in \mathcal N$, then the result follows from Lemma \ref{lm:attB}.  Otherwise,  $B\notin \mathcal N$.  The neighbour $B'$ of $B$ in $P$ is closer to $\mathcal N$ than $B$ is, so $\att(B')\subseteq r^*\cup \r5$.  

   If $B$ overlaps $M_{Q^*}$, then $P$ extends to a chordless $\mathcal N M_{Q^*}$-path in $OD(Q^*)-\mathcal M^-$ of length at least 2.  This contradicts Claim \ref{cl:noNtoMbQ0path}, showing $B$ does not overlap $M_{Q^*}$.  

Suppose by way of contradiction that $B$ has an attachment $x$ in the interior of some $H$-spoke $s$ contained in $Q^*$. As $B$ overlaps $B'$ and $\att(B')\subseteq r^*\cup \r5$, not all attachments of $B$ can be in $\cc{s}$.  But any attachment $y$ of $B$ in $Q^*-\cc{s}$ combines with $x$ to show that $B$ is skew to the ends of $s$ and, therefore, overlaps $M_{Q^*}$.  Therefore, $\att(B)\subseteq r^*\cup \r5$.  

Next suppose that $B$ has an attachment in $\oo{r^*}$. If $B$ also has an attachment in $Q^*-\cc{r^*}$, then $B$ overlaps $M_{Q^*}$ (the two identified attachments of $B$ are skew to the two ends of $r^*$).  Thus, if $B$ has an attachment in $\oo{r^*}$, then $\att(B)\subseteq r^*$.  Likewise, if $B$ has an attachment in $\oo{\r5}$, then $\att(B)\subseteq \r5$.

If $B$ has an attachment in each of $r^*$ and $\r5$, then the preceding paragraph shows that $\att(B)$ consists of some of the four $H$-nodes that comprise the ends of $r^*$ and $\r5$.  Because $B$ overlaps $B'$, $\att(B)$  cannot be just the two ends of one of the two $H$-spokes in $Q^*$.  In the remaining case, $B$ is skew to $M_{Q^*}$, a contradiction.  Thus, either $\att(B)\subseteq r^*$ or $\att(B)\subseteq \r5$.
\end{proof}

Let $OD^-(\bQ_0)=OD(\bQ_0)-\{v_1v_4,v_6v_9\}$ and let $OD^-(Q_4)=OD(Q_4)$.  By Lemma \ref{lm:bQ0bipartite} or Theorem \ref{th:BODquads} (\ref{it:quadBOD}), $OD^-(Q^*)$ is bipartite; let $(S,T)$ be a bipartition of $OD^-(Q^*)$, with $M_{Q^*}\in T$.  We briefly treat separately the cases $Q^*=\bQ_0$ and $Q^*=Q_4$.

For the former,  every element of $\mathcal N$ overlaps $M_{\bQ_0}$ and so $\mathcal N\subseteq S$.      There is an embedding $\Phi$ of $(G-\{v_1v_4,v_6v_9\})-\Nuc(M_{\bQ_0})$ in the plane so that all the $\bQ_0$-bridges in $\mathcal N$ are on the same side of $\Phi[\bQ_0$].  

In the case of $Q^*=Q_4$, $\mathcal N\setminus\{v_4v_9,v_0v_5\}\subseteq S$.  There is an embedding $\Phi$ of $G-\Nuc(M_{Q_4})$ in the plane so that all the $Q_4$-bridges in $\mathcal N\setminus\{v_4v_9,v_0v_5\}$ are on the same side of $\Phi[Q_4]$.  Any of $v_4v_9$ and $v_0v_5$ that is also in $S$ can also be embedded on that same side of $\Phi[Q_4]$.

\dragominorrem{(Text moved.)}Among the attachments of the elements of $\mathcal N^+$, let $a_9$  be the one in $r^*$ nearest $v_9$ and let $a_4$ be the one in $\r5$ nearest $v_4$.   

\begin{claim}\label{cl:cutPair} No $Q^*$-bridge not in $\mathcal M$ is skew to $\{a_4,a_9\}$.  \end{claim}

\begin{proof}  It is clear that, in the case $Q^*=Q_4$, neither $v_4v_9$ nor $v_0v_5$ is skew to $\{a_4,a_9\}$.  We show that a $Q^*$-bridge not in $\mathcal M$ that is skew to $\{a_4,a_9\}$ must overlap some $Q^*$-bridge in $\mathcal N^+$; this implies the contradiction that it is in $\mathcal N^+$.

By the Ordering Lemma \ref{lm:orderingLemma}, the elements of $\mathcal N\cap S$ occur in order on $Q^*$ in $\Phi$. Thus, there is one element $B'$ of $\mathcal N\cap S$ that has both an attachment nearest to $v_4$ (relative to $r^*$) and an attachment nearest to $v_9$ (relative to $\r5$).  Let $x'$ and $y'$ be the attachments of $B'$ nearest $v_4$ in $r^*$ and $v_9$ in $\r5$, respectively.  In the case $Q^*=\bQ_0$, $B^0$ is a candidate for $B'$, so, even in this case, we have that  $x'\in \cc{v_4,r_4,v_5}$ and $y'\in \cc{v_9,r_9,v_0}$.

Suppose by way of contradiction that some $Q^*$-bridge $B''$ not in $\mathcal M$ has attachments $x''$ and $y''$ in the two components of $Q^*-\{a_4,a_9\}$.  We note that, when $Q^*=Q_4$, $B''\ne v_4v_9$ and $B''\ne v_0v_5$.  

If one of $x''$ and $y''$ is in the component of $Q^*-\{x',y'\}$ that is disjoint from $s_4-\{x',y'\}$, then $B''$ overlaps $B'$.  Since $B'\in\mathcal N$, Lemma \ref{lm:Noverlap} implies $B''\notin\mathcal N$ and, therefore, $B''\in\mathcal N^+$.  But this contradicts the definition of either $a_4$ or $a_9$ and, therefore, both $x''$ and $y''$ are contained in the component of $Q^*-\{x',y'\}$ that contains $s_4-\{x',y'\}$.  In particular, we may assume $y''\in \oc{a_4,r_4,x'}\cup \oc{a_9,r_9,y'}$.  For the sake of definiteness, we assume \wording{$y''\in \oc{a_9,r_9,y'}$}.   

Some $Q^*$-bridge \wording{$B^+$}  in $\mathcal N^+$ has $a_9$ as an attachment; since $y''$ is in $\oc{a_9,r_9,y'}$, $y'\ne a_9$ and, therefore, \wording{$B^+$} is not in $\mathcal N$.  There is a shortest path $P=(B',B_1,\dots,B_n)$ in $OD^-(Q^*)-M_{Q^*}$ from $B'$ to some element $B_n$ of $\mathcal N^+$ so that $B_n$ has an attachment $y_n$ in $\co{a_9,r_9,y''}$; choose $y_n$ so that it is as close to $a_9$ in $\co{a_9,r_9,y''}$ as possible.  

The $Q^*$-bridge $B_{n-1}$ is in $\mathcal N^+$ and so, by minimality of $n$, does not have an attachment in $\co{a_9,r_9,y''}$.  Since $B_n$ overlaps $B_{n-1}$, there is an attachment $z_n$ of $B_n$ in $\oc{y'',r_9,x'}$.  Since $B''$ is skew to $\{a_4,a_9\}$, there is an attachment $z''$ of $B''$ in \wording{$\oc{a_9,r_9,v_9}\rbsp s_4\lbsp \co{v_4,r_4,a_4}$.}  But now $z_n$, $y''$, $y_n$, and $z''$ show $B''$ overlaps $B_n$.  Since $B''\notin \mathcal M$, $B''$ is in $\mathcal N^+$.  But this contradicts the definition of $a_4$ or $a_9$.   \end{proof}

The following is immediate from Claim \ref{cl:cutPair}.

\begin{claim}\label{cl:attsNotSkew}  Each $Q^*$-bridge not in $\mathcal M$ has all its attachments in one of the two $a_4a_9$-subpaths of $Q^*$.  $\Box$ \end{claim}

The proof now bifurcates into the two cases.  We consider first the case $Q^*=\bQ_0$ and that $s_0$ is exposed in $\Pi$.  The following is immediate from Claim \ref{cl:attsNotSkew}.

\begin{claim}\label{cl:a4a9curve} The  planar embedding $\Phi$ of $(G-\{v_1v_4,v_6v_9\})-\Nuc(M_{\bQ_0})$ has the property that there is a simple closed curve in the plane that meets $\Phi[(G-\{v_1v_4,v_6v_9\})-\Nuc(M_{Q^*})]$ precisely at $a_4$ and $a_9$.\hfill $\Box$ \end{claim} 

We are now prepared to describe a representativity 2 embedding of $G$ in $\pp$ so that all $H$-spokes are in $\Mob$. 

Let $\Psi$ be an embedding of $H$ in $\pp$ so that all $H$-spokes are contained in the M\"obius band $\Mob_\Psi$ bounded by $\Psi[R]$ and let $\gamma_\Psi$ be a non-contractible, simple, closed curve that meets $H$ in precisely the points $a_4$ and $a_9$.  The claim is that this embedding extends to an embedding of $G$ so that $\gamma_\Psi$ meets $G$ only at $a_4$ and $a_9$.  

Claim \ref{cl:attsNotSkew} implies that we can add all the $\bQ_0$-bridges other than $v_1v_4$, $v_6v_9$, and $M_{\bQ_0}$  to $\Psi$ so that there is no additional intersection with $\gamma_\Psi$.  It remains to show that we may also add the \dragominor{at most three} remaining $\bQ_0$-bridges.

\begin{claim}\label{cl:v1v4v6v9}  At most one of $v_1v_4$ and $v_6v_9$ is in $G$. \end{claim}

\begin{proof} Suppose both are in $G$.  We consider a 1-drawing $D_2$ of $G-\oo{s_2}$.  As $\bQ_2$ must be crossed in $D_2$ (it has BOD and $s_2$ is contained in a planar $\bQ_2$-bridge\wording{; apply Lemma \ref{lm:BODcrossed}}), we conclude that $r_0\,r_1\,r_2\,r_3$ crosses $r_5\,r_6\,r_7\,r_8$ in $D_2$.  In particular, $s_0$ and $s_4$ cannot be exposed. 

In order for $v_1v_4$ to be not crossed in $D_2$, we must have the crossing in $r_0$.  Likewise, $v_6v_9$ implies the crossing is in $r_5$.  But then neither $r_1\,r_2$ nor $r_6\,r_7$ is crossed, so $\bQ_2$ is not crossed in $D_2$, a contradiction.
 \end{proof}

We note that $v_1v_4 $ and $v_6v_9$ are not symmetric:  the embedding $\Pi$ of $G$ in $\pp$ distinguishes these two cases.    However, it is easy to add either of these to $\Psi$ so that the newly added edge is  in the closed disc $\Disc_\Psi$ bounded by $\Psi[R]$ in $\Psi$.

Finally, it remains to show that we may also add $M_{\bQ_0}$ to $\Psi$.   Here the argument depends slightly on which of $v_1v_4$ and $v_6v_9$ occurs in $G$.  We will assume, for the sake of definiteness, that it is $v_1v_4$ that occurs; the argument in the other case is completely analogous.    We \dragominor{shall simply import} $\Pi[M_{\bQ_0}]$ in $\pp$ as its embedding in $\Psi$.  

\dragominor{To this end, let $B$ be any $H$-bridge contained in $M_{\bQ_0}$ so that $\Pi[B]\subseteq \Disc$.  We show that either $\att(B)\subseteq r_0\,r_1\,r_2\,r_3\cc{v_4,r_4,a_4}$ or $\att(B)\subseteq r_5\,r_6\,r_7\,r_8\cc{v_9,r_9,a_9}$.}

We begin by observing that such a $B$ cannot overlap $v_1v_4$ (as $R$-bridges), as both are are embedded in $\Disc$ by $\Pi$.  \dragominor{An analogous discussion applies if $v_1v_4$ is replaced by $v_6v_9$.}

\dragominor{The embedding $\Pi$ shows $B$ cannot have an attachment in each of $\oo{r_1\,r_2\,r_3}$ and $\oo{r_5\,r_6\,r_7\,r_8\,r_9}$.  Likewise, $B$ cannot have an attachment in each of $\oo{r_6\,r_7\,r_8}$ and $r_0\,r_1\,r_2\,r_3\,r_4$}.
The next  claim treats the remaining possibilities.

\begin{claim} \dragominor{The $H$-bridge $B$ does not have an attachment in each of $\oo{r_1\,r_2\,r_3}$ and $\oc{a_4,r_4,v_5}$.  Likewise, $B$ does not have an attachment in each of $\oo{r_6\,r_7\,r_8}$ and either $r_5$ or $\oc{a_9,r_9,v_0}$.} \end{claim}

\begin{proof} \dragominor{Suppose by way of contradiction that $B$ has an attachment $x$  in $\oc{a_4,r_4,v_5}$ and an attachment $y\in \oo{r_1\,r_2\,r_3}$.  Let $P$ be an $H$-avoiding $xy$-path in $B$.}  Since $a_4$ is an attachment of some element of $\mathcal N^+$, there is a shortest \dragominor{path $S$ in} $OD(\bQ_0)-\{v_1v_4,v_6v_9,M_{\bQ_0}\}$ joining some $B_{\mathcal N}$ in $\mathcal N$ to a $\bQ_0$-bridge $B_{\mathcal N^+}$ so that $B_{\mathcal N^+}$ has an attachment in $\co{v_4,r_4,x}$.    

\dragominor{If $B_{\mathcal N^+}\in \mathcal N$, then $B_{\mathcal N^+}\subseteq \Disc$.  Lemma \ref{lm:attB} shows $B_{\mathcal N^+}$ has an attachment in each of $r^*$ and $r^*_{+5}$; therefore, $B_{\mathcal N^+}$ is not contained in the closed disc bounded by $P$ and a subpath of $r_1\,r_2\,r_3\,r_4$, $B_{\mathcal N^+}$ and $P$ must cross in $\Pi$.   Therefore, $B_{\mathcal N^+}\in \mathcal N^+\setminus \mathcal N$.  }

The neighbour $B'_{\mathcal N^+}$ of $B_{\mathcal N^+}$ \dragominor{in $S$} does not have an attachment in $\co{v_4,r_4,x}$.  Since $B_{\mathcal N^+}$ overlaps $B'_{\mathcal N^+}$, it follows that $B_{\mathcal N^+}$ has another attachment in $\ob{x,r_4,v_5,}$ $\bc{r_5,b}$.    In particular, \dragominor{the edge $e$} of $\cc{v_4,r_4,x}$ incident with $x$ is $H$-green because of $B_{\mathcal N^+}$.

\dragominor{On the other hand, if either $x\ne v_5$ or $y\notin\oo{r_1}$, then $P$ combines with the $xy$-subpath of $r_1\,r_2\,r_3\lbsp\cc{v_4,r_4,x}$ to make another $H$-green cycle containing $e$, contradicting Theorem \ref{th:twoGreenCycles}. Therefore,  $x=v_5$ and  $y\in\oo{r_1}$.  But then $\att(B)\subseteq \bQ_0$, contradicting the fact that $B\subseteq M_{\bQ_0}$.}

\dragominor{The ``likewise" statement has an analogous proof.}  \end{proof}

\dragominorrem{(Text removed.)}

\dragominor{We now see that $\Psi$ may be extended to include $\Pi[M_{\bQ_0}]$}, completing the proof when $Q^*=\bQ_0$.

The proof will be completed by now considering the case $Q^*=Q_4$.  The only difference in how we proceed is to note that the $H$-bridges $v_4v_9$ and $v_0v_5$, if they exist, may be transferred to $\Mob$ at the start.  
To see this, first observe that $v_4v_9$ and $v_0v_5$ overlap on $R$ and so cannot both be embedded in $\Disc$.  
If  $v_4v_9$ is not contained in $\Mob$, then we may consider $H'$ to be $(H-\oo{s_4})+v_4v_9$, relabel $H'$ so that $v_4v_9$ --- the exposed spoke --- is $s_0$ and proceed as above to move $v_4v_9$ into $\Mob$.       \end{cproofof}

The following notions will be helpful for the duration of the work.

}\begin{definition}\label{df:local} Let $G$ be a graph, $\hvfg$ and let $B$ be an $H$-bridge in $G$.
 \begin{enumerate}
\item If there is an $i\in\{0,1,2,3,4\}$ so that $\att( B )\subseteq Q_i$, then $B$ is both a {\em local $H$-bridge\/} and a {\em $Q_i$-local $H$-bridge\/}\index{local $H$-bridge}\index{$Q_i$-local $H$-bridge}\index{bridge!local}.
\item Otherwise, $B$ is a {\em global $H$-bridge\/}\index{global $H$-bridge}\index{bridge!global}.
\end{enumerate}
\end{definition}\printFullDetails{

}\begin{corollary}\label{co:greenCover}  Let $G\in\m2$ and $\hvfg$.   Then there is no $i$ so that $\bQ_i$ has BOD and each edge of $r_{i-2}\, r_{i-1}\, r_i\, r_{i+1}$ is in an $H$-green cycle consisting of a global $H$-bridge and a path in $R$ having at most two $H$-nodes other than $v_i$. \end{corollary}\printFullDetails{

\begin{cproof}  \wording{By way of contradiction, suppose} there is such an $i$.  By Theorem \ref{th:allSpokesMob}, $G$ has a representativity 2 embedding in $\pp$ so that $H\subseteq\Mob$.   Thus, $s_i$ is in a $\bQ_i$-bridge other than $M_{\bQ_i}$.

By Lemma \ref{lm:greenCycles} (\ref{it:notCrossed}), no edge of  $r_{i-2}\, r_{i-1}\, r_i\, r_{i+1}$ can be crossed in any 1-drawing $D$ of $G-\oo{s_i}$.  By hypothesis, $\bQ_i$ has BOD, so Lemma \ref{lm:BODcrossed} implies $\bQ_i$ is crossed in $D$, which further implies that some edge of $r_{i-2}\, r_{i-1}\, r_i\, r_{i+1}$ is crossed in $D$, a contradiction. 
 \end{cproof}

}

\chapter{\major{Parallel edges}}\printFullDetails{

\majorrem{(This section is moved to this much earlier position.)}In this very short chapter\dragominor{,} we present some observations on how parallel edges can occur in 2-crossing-critical graphs.  This will be used in \wording{later sections, especially Section \ref{sec:3conNotI4c}, where we} determine all the \wording{3-connected, 2-crossing-critical} graphs that do not have a subdivision of $V_8$.   There are easy generalizations to $k$-crossing-critical graphs.  

}\begin{definition}  For an edge $e$ of a graph $G$, $\mu(e)$\index{$\mu(e)$} denotes the number of edges parallel to $e$ (including $e$ itself).  \end{definition}\printFullDetails{

}\begin{observation}\label{obs:parallel}  Let $G$ be a 2-crossing-critical graph and let $e$ and $e'$ be parallel edges of $G$.  Then:
\begin{enumerate}
\item\label{notplanar} if $\overline G$ is the underlying simple graph, then $\overline G$ is not planar;
\item\label{e'crossed} the edge $e'$ is crossed in any 1-drawing of $G-e$;
\item\label{multtwo}  $\mu(e)\le 2$; 
\item\label{classplanar}  if $e'$ is an edge parallel to $e$, then $G-\{e,e'\}$ is planar;
\item\label{c3xc3}  if $\crn(G)>2$, then $G$ is simple; and 
\item\label{rim} if $n\ge 4$ and $\hvng$, then one of $e$ and $e'$ is in the $H$-rim.
\end{enumerate}
\end{observation}\printFullDetails{

\begin{proof}  For (\ref{notplanar}), a planar embedding of $\overline G$ allows us to introduce all the parallel edges of $G$ with no crossings, showing $G$ is planar, a contradiction.

For (\ref{e'crossed})--\dragominor{(\ref{c3xc3})}, let $D$ be a 1-drawing of $G-e$ and suppose $e'$ is not crossed in $D$  Then we may add $e$ alongside $D[e']$ to obtain a 1-drawing of $G$\dragominor{, a contradiction}.  Since $D$ has at most one crossing, it must be of $e'$, which is (\ref{e'crossed}).   Adding $e$ alongside $D[e']$ yields a 2-drawing of $G$.  Thus we have (\ref{classplanar}) and (\ref{c3xc3}).  Also, (\ref{multtwo}) follows, since any other edge $e''$ parallel to $e$ does not cross $e'$ in $D_e$.  Thus, $e''$ is not crossed in $D_e$, which contradicts the second sentence, with $e''$ in place of $e'$.

Finally, \dragominor{for (\ref{rim}), we may suppose $e$ is not in $H$.  Lemma \ref{lm:1drawingsV2n} shows that the only} edges  that are in every non-planar subgraph of \dragominor{$G-e$} are those in the $H$-rim.  \dragominor{Therefore, $e'$ is in the $H$-rim.}
 \end{proof}

}
\chapter{Tidiness and global $H$-bridges}\printFullDetails{

In this section, we show that, if $G\in\m2$ and $\hvfg$, then there is a $V_{10}\topol H'\subseteq G$ with many useful additional characteristics that we call ``tidiness".  The main result is that a tidy subdivision of $V_{10}$ has only very particular global bridges, each of which is an edge.  We start with a slightly milder version of tidiness.

}\begin{definition}\label{df:pre-tidy}  Let $\Pi$ be a representativity 2 embedding of $G$ in $\pp$ and let $\hvfg$.  Then $H$ is {\em $\Pi$-pretidy\/}\index{pretidy}\index{$\Pi$-pretidy} if:
\begin{enumerate}
\item all $H$-spokes are embedded in $\Mob$; and
\item for every $H$-quad $Q$ and for every $Q$-bridge $B$ other than $M_Q$, $Q\cup B$ has no non-contractible cycle in $\Pi$.
\end{enumerate}
\end{definition}\printFullDetails{

The first step in this section is to find an embedding with a pretidy subdivision of $V_{10}$.

}\begin{lemma}\label{lm:existsPretidy}   Let $G\in\m2$ and $\hvfg$.  Then $G$ has a representativity 2 embedding $\Pi$ in $\pp$ so that $H$ is $\Pi$-pretidy.  \end{lemma}\printFullDetails{

\begin{cproof}  By Theorem \ref{th:allSpokesMob}, $G$ has a representativity 2 embedding $\Pi$ in $\pp$ so that all the $H$-spokes are contained in $\Mob$ and so that, for any $Q_4$-bridge $B$ other than $M_{Q_4}$, $\Pi[Q_4\cup B]$ has no non-contractible cycle.   We note that every global $H$-bridge is contained in $\Disc$.  We describe a particular representativity 2 embedding $\Pi^*$ of $G$ in $\pp$ for which $H$ is $\Pi^*$-pretidy.  Let $\gamma$ be the non-contractible simple closed curve that meets $\Pi(G)$ at just the two points $a$ and $b$. 

The embedding $\Pi^*$ is obtained by adjusting the local $H$-bridges\dragominor{; we do not adjust those that are $Q_4$-local}.  We start with $\Pi^*$ being the same as $\Pi$ on $H$ and all the $Q_4$-bridges other than $M_{Q_4}$.
Let $Q$ be an $H$-quad other than $Q_4$.  By Theorem \ref{th:BODquads}, $Q$ has BOD and all $Q$-bridges other than $M_Q$ are planar.   Let \dragominor{$(S,T)$} be a bipartition of $OD(Q)$ labelled so that \dragominor{$M_Q\in T$}.  Let $\Pi_Q$ be a planar embedding of $Q$ and all the $Q$-bridges other than $M_Q$ so that all the $Q$-bridges in \dragominor{$T\setminus\{M_Q\}$} are on one side of $\Pi_Q[Q]$ and all the $Q$-bridges \dragominor{in $S$} are on the other side of $\Pi_Q[Q]$.  

Extend $\Pi^*$ to include all the $Q$-bridges other than $M_Q$ by placing the $Q$-bridges \dragominor{in $S$} into the $H$-face in $\Pi^*$ bounded by $\Pi^*[Q]$, using $\Pi_Q$.  As every $Q$-bridge in \dragominor{$T\setminus\{M_Q\}$} does not overlap $M_Q$, each of these has all its attachments on one of the four $H$-branches in $Q$ and these may be embedded in $\Pi^*$ on the other side of $\Pi^*[Q]$, and without crossing $M_Q\cup \gamma$. 

The only concern here is that a local $H$-bridge can be local for distinct $H$-quads.  Such an $H$-bridge $B$ must have all its attachments on the same $H$-spoke $s_i$.  We claim it \dragominor{is in $T$} for one of $Q_{i-1}$ and $Q_i$ and \dragominor{in $S$} for the other \dragominor{one of $Q_{i-1}$ and $Q_i$}.   

As $G$ is 3-connected, $OD(Q_i)$ is connected (see \cite[Thm.\ 1]{bc}, where this is proved for binary matroids).   There is a shortest $M_{Q_i}B$-path $P=(B_0,B_1,\dots,B_n)$ in $OD(Q_i)$ (thus, $B_0=M_{Q_i}$ and $B_n=B$).  Let $k$ be least so that $B_k$ has an attachment in $\oo{s_i}$.  

\begin{claim}  For $j>k$, $\att(B_j)\subseteq s_i$, and $k\le 1$.  \end{claim}

\begin{proof} If, for some $j> k$, $B_j$ has an attachment not in $s_i$, then $j<n$.  If $B_j$ has an attachment in $\oo{s_i}$, then $B_j$ is skew to $M_{Q_i}$ and $P$ is not a shortest $M_{Q_i}B$-path, a contradiction.  Thus, there is a least $j'> j$ so that $B_{j'}$ has an attachment in $\oo{s_j}$.  Since $B_{j'}$ overlaps $B_{j'-1}$ and $B_{j'-1}$ has no attachment in $\oo{s_i}$, $B_{j'}$ has an attachment not in $s_i$.  Again, $B_{j'}$ is skew to $M_{Q_i}$, so $P$ is not a shortest $M_{Q_i}B$-path, a contradiction.  Thus, for all $j>k$, $\att(B_j)\subseteq s_i$.

If $k=0$, then obviously $k\le 1$, so we may assume $k\ge 1$.  As $B_k$ has an attachment in $\oo{s_i}$ and $B_{k-1}$ does not, it follows that $B_k$ has an attachment not in $s_i$.  But then $B_k$ is skew to $M_{Q_i}$.  Because $P$ is a shortest $M_{Q_i}B$-path, we deduce that $k\le 1$.  \end{proof}

The claim shows that the $Q_i$-bridges $B_{k+1}$, $B_{k+2}$, \dots, $B_n$ are also $Q_{i-1}$-bridges and, therefore, $(B_{k+1},B_{k+2},\dots,B_n)$ is a path in $OD(Q_{i-1})$.    Suppose first that $k=0$.  Then $M_{Q_i}$ contains a vertex $x$ in $\oo{s_i}$ so that $x$ and $v_{i+1}$ are skew to $B_1$.  There is a shortest $Q_i$-avoiding path $P$ in $M_{Q_i}$ joining $x$ to a vertex in $\Nuc(M_{Q_i})\cap H$.  Since $P$ is not in the face of $\Pi[Q_i]$ contained in $\Mob$, we deduce that $P$ is contained in the face of $\Pi[Q_{i-1}]$ contained in $\Mob$.  But then we conclude that $P$ is contained in a $Q_{i-1}$-local $H$-bridge $B'$, showing that $B'$ is skew to both $M_{Q_{i-1}}$ and to $B_1$.  We deduce that, in $OD(Q_i)$, $M_{Q_i}$ and $B_1$ are on opposite sides of the bipartition of $OD(Q_i)$, while $M_{Q_{i-1}}$ and $B_1$ are on the same side of the bipartition of $OD(Q_{i-1})$.  Since $B_1$ and $B=B_n$ have not changed their relative positions, we see that in one of $OD(Q_i)$ and $OD(Q_{i-1})$, $B$ is on the same side of the bipartition as the corresponding M\"obius bridge, while in the other $B$ and the other corresponding M\"obius bridge are on opposite sides of the bipartition.

The argument works exactly in reverse when $k=1$.  In this case, $B_1$ is skew to $M_{Q_i}$ and $B_2$.  Since $B_1\subseteq M_{Q_{i-1}}$, we conclude that $B_2$ is skew to $M_{Q_{i-1}}$, and the result follows analogously to the argument in the preceding paragraph.

Finally, suppose $B$ is a global $H$-bridge.  Then, for each $H$-quad $Q$, $B\subseteq M_Q$, so $B$ does not overlap any of the $Q$-local $H$-bridges already embedded in $\Disc_{\Pi^*}$ and, since $\Pi[B]\subseteq \Disc$,  $B$ can also be added to $\Pi^*$.
 \end{cproof}

We are now ready to move to tidiness. 

}\begin{definition}\label{df:tidy}  Let \dragominor{$\hvfg$} and let $\Pi$ be a representativity 2 embedding of $G$.
Then $H$ is {\em $\Pi$-tidy\/}\index{tidy}\index{$\Pi$-tidy} if:
\begin{enumerate}\item\label{it:spokeMob}  $H\subseteq\Mob$; 
\item\label{it:bridgeMob} every local $H$-bridge is contained in $\Mob$; 
\item\label{it:QnoOverlap} for each $H$-quad $Q$, no two $Q$-local $H$-bridges overlap; and 
\item\label{it:Hjump} there is no $H$-avoiding path $P$ in $\Disc$ and an index $i\in\{0,1,2,\dots,9\}$ so that $P$ has both its ends in  $\oo{v_i,r_i,v_{i+1},r_{i+1},v_{i+2},r_{i+2},v_{i+3}}$.  
\end{enumerate}
If \dragominor{$\hvfg$}, then $H$ is {\em tidy\/} if there is a representativity 2 embedding $\Pi$ of $G$ so that $H$ is $\Pi$-tidy.
\end{definition}\printFullDetails{

Our aim is the following result.

}\begin{theorem}\label{th:existTidy}  Let $G\in \m2$ have a subdivision of $V_{10}$.  Then there exists a representativity 2 embedding $\Pi$ in $\pp$ of $G$ with a $\Pi$-tidy subdivision of $V_{10}$.  \end{theorem}\printFullDetails{

The following concept is central to the proof.

}\begin{definition}\label{df:Loc}  Let $\hvfg$.  Then $\Loc(H)$\index{$\Loc(H)$} denotes the union of $H$ and all the local $H$-bridges in $G$.    \end{definition}\printFullDetails{

\begin{cproofof}{Theorem \ref{th:existTidy}} For any $\hvfg$, Lemma \ref{lm:existsPretidy} implies there is a representativity 2 embedding $\Pi$ of $G$ in $\pp$ so that $H$ is $\Pi$-pretidy.  Among all $H$ for which $\Loc(H)$ is maximal and all $\Pi$ so that $H$ is $\Pi$-pretidy, we consider the pairs $(H,\Pi)$ so that $G\cap \Mob_{\Pi(H)}$ is maximal.  Among all these pairs $(H,\Pi)$, we choose one for which the number of edges of $G$ in $H$-spokes in minimized. We claim that this $H$ is $\Pi$-tidy.  We note that (\ref{it:spokeMob}) is satisfied by the fact that $H$ is $\Pi$-pretidy.

If $H$ and $\Pi$ fail to satisfy either (\ref{it:bridgeMob}) or (\ref{it:Hjump}), then either there is an $H$-quad $Q$ so that some $Q$-local $H$-bridge $B$ is not embedded in $\Mob_H$, or there is an $H$-avoiding path $P$ contained in $\Disc_H$ and an index $i\in\{0,1,2,\dots,9\}$ so that $P$ has both ends in $\oo{r_i\,r_{i+1}\,r_{i+2}}$.    In the first case, as $Q\cup B$ has no non-contractible cycles, the only possibility is that $B$ has all its attachments in one of the $H$-rim branches of $Q$.  Thus, \dragominor{the first case is a special case of the second;  we now consider the second case.}

 Let $P'$ be the subpath of  $\oo{r_i\,r_{i+1}\,r_{i+2}}$ joining the ends $u$ and $w$ of $P$, with the labelling chosen so that $u$ is nearer to $v_i$ in $P'$ than $w$ is.    Note that the cycle $P\cup P'$ is an $H$-green cycle and, therefore, bounds a face of $G$.  

We construct a new subdivision $H'$ of $V_{10}$ in $G$.  The $H'$-rim is obtained from the $H$-rim by replacing $P'$ with $P$.  The spokes $s_i$, $s_{i+3}$, and $s_{i+4}$ of $H'$ are also spokes of $H'$.  The $H$-spokes $s_{i+1}$ and $s_{i+2}$ might need extension, using the subpaths of $r_i\,r_{i+1}\,r_{i+2}$ joining $u$ and/or $w$ to either $v_{i+1}$ or $v_{i+2}$ as necessary, to become spokes of $H'$.   Evidently all $H'$-spokes are contained in $\Mob_{H'}$, so $H'\subseteq G\cap \Mob_{H'}\subseteq \Loc(H')$.  Furthermore, if $F$ is the (closed) face of $G$ bounded by $P\cup P'$, then $\Mob_{H'}=\Mob_H\cup F$.

\begin{claim}\label{cl:growLoc}  $\Loc(H)\subseteq \Loc(H')$. \end{claim}

\begin{proof}  Let $e$ be an edge of $\Loc(H)$.    If $e\in \Mob_{H'}$, then $e\in \Loc(H')$, so we may assume $e\notin\Mob_{H'}$.  Let $B$ be the local $H$-bridge containing $e$.  Since $e\notin \Mob_{H'}$ and $\Mob_H\subseteq \Mob_{H'}$, we deduce that $B\subseteq \Disc_H$, and so all attachments of $B$ are in some $H$-rim branch (recall $H$ is $\Pi$-pretidy).  Thus, Corollary \ref{co:attsMissBranch} implies $B$ has precisely two attachments and therefore is just the edge $e$.  Consequently, $B$ is disjoint from $P$ (it is not in $\Mob_{H'}$), and so $B$ is an $H'$-bridge, whence $e\in \Loc(H')$.  \end{proof}

If $P$ is not contained in a local $H$-bridge, then, since $P\subseteq \Loc(H')$, we contradict maximality of $\Loc(H)$.  Therefore, $P$ is contained in, and therefore is, a local $H$-bridge $B$.  But this implies that $H'$ is $\Pi$-pretidy and that $G$ has one more edge in $\Mob_{H'}$ than it has in $\Mob_H$, contradicting the maximality of $G\cap \Mob_H$.  Therefore, (\ref{it:bridgeMob}) and (\ref{it:Hjump}) hold for $(H,\Pi)$.  

It follows that, if $H$ is not $\Pi$-tidy, \dragominor{then (\ref{it:QnoOverlap}) is violated:  there exists an $H$-quad $Q$ and two $Q$-bridges $B$ and $B'$ in $\comp{(M_Q)}$ that overlap}.  As both $B$ and $B'$ are contained in $\Mob$, one, say $B$, is $Q$-interior in $\Pi$, while $B'$ is $Q$-exterior.  This implies that $\att(B')\subseteq s$, for some $H$-spoke $s\subseteq Q$.  Corollary \ref{co:attsMissBranch} implies that $B'$ is just an edge $uw$.  We note that $B$ has an attachment $x$ in $\oo{u,s,w}$ and an attachment $y$ not in $\cc{u,s,w}$.  

Let $H''$ be the subdivision of $V_{10}$ obtained from $H$ by replacing $s$ with $(s-\oo{u,s,w})\cup B'$.  We note that $H''$ is $\Pi$-pretidy, $\Loc(H')=\Loc(H)$, and $\Mob_{H''}=\Mob_H$, so $G\cap \Mob_{H''}$ is maximal.  However, the $H''$-spokes have in total at least one fewer edge than the $H$-spokes, contradicting the choice of $H$.
\end{cproofof}

We now turn our attention to the global $H$-bridges of a tidy $H$.  

}\begin{theorem}\label{th:globalBridges}  Let $G\in\m2$ and $\hvfg$.  If $H$ is tidy, then any global $H$-bridge is just an edge, and, in particular, has one of the forms $v_iv_{i+2}$, $v_iv_{i+3}$, or has $v_i$ as one end and the other end is in $\oo{r_{i-3}}\cup \oo{r_{i+2}}$.  \end{theorem}\printFullDetails{

\begin{cproof} \wordingrem{(Useless text removed.)}Let $\Pi$ be a representativity 2 embedding of $G$ for which $H$ is $\Pi$-tidy.  In particular, all $H$-spokes and all local $H$-bridges are in $\Mob$, and, for each $i=0,1,2\dots,9$, no global $H$-bridge has two attachments in $\oo{r_i\,r_{i+1}\,r_{i+2}}$.  

Let $B$ be a global $H$-bridge.  We note that $B\subseteq \Disc$.

\begin{claim}\label{cl:3rimPath} If there is an $i$ so that $\att(B)\subseteq r_i\,r_{i+1}\,r_{i+2}$, then either  $B=v_iv_{i+2}$ or $B=v_{i+1}v_{i+3}$ or $B=v_iv_{i+3}$ or $B$ has $v_i$ as one end and the other end is in $ \oo{r_{i+2}}$ or $B$ has $v_{i+3}$ as one end and the other end is in $\oo{r_{i}}$. \end{claim}

\begin{proof}  Because $H$ is tidy, no two attachments of $B$ are in $\oo{r_i\,r_{i+1}\,r_{i+2}}$.  Thus, at least one of $v_i$ and $v_{i+3}$ is an attachment of $B$\wording{; for the sake of definiteness, let it be} $v_i$.  Then tidiness implies no attachment of $B$ can be in $\oo{r_i\,r_{i+1}}$.  As tidiness also implies $r_{i+2}$ has at most one, and therefore exactly one, attachment of $B$, the result follows. \end{proof}

\begin{claim}\label{cl:no3rimPath}  If there is no $i$ so that $\att(B)\subseteq r_i\,r_{i+1}\,r_{i+2}$, then either $\att(B)=\{v_0,v_{5},z\}$, with $z\in \oo{r_{2}}\cup\oo{r_7}$ or $\att(B)=\{v_4,v_9,z\}$, with $z\in \oo{r_1}\cup\oo{r_6}$. \end{claim}

\begin{proof}  We may assume that $B$ is embedded in the $(H\cup \gamma)$-face contained in $\Disc$ and incident with $v_0$, $v_1$, \dots, $v_4$. \dragominor{As $H$ is tidy and $B$ is $H$-global,} there exist $i,j\in\{9,0,1,2,3,4\}$ so that (taking $9$ to be equal to $-1$)\wordingrem{(comma removed)} $i<j$, $B$ has attachments $x$ in $r_i-v_{i+1}$ and $y$ in $r_j-v_j$, and $j-i\ge 3$;  choose such $i,j$ so that $j-i$ is as small as possible.   By tidiness, there is no other attachment of $B$ in $$\co{r_{i-1}\, r_i\, r_{i+1}}\cup \oc{r_{j-1}\, r_j\, r_{j+1}}\,.$$  

\startSubclaims

\begin{subclaim}\label{sc:x9orY4}  Either $i=-1$ or $j=4$.  \end{subclaim}

\begin{proof}  In the alternative, $i\ge 0$ and $j\le 3$.  As $j-i\ge 3$, we conclude that $i=0$ and $j=3$, so the six $H$-rim branches \dragominor{$r_{i-1}$, $r_i$, $r_{i+1}$, $r_{j-1}$, $r_j$, and $r_{j+1}$} are all distinct and cover the entire $ab$-subpath in the boundary of $(H\cup \gamma)$-face containing $B$, with the possible exception of $v_2$, in which case both $x=v_0$ and $y=v_4$. 

\wording{Let $e$ be an edge in $s_2$ and let $D$ be a 1-drawing of $G-e$.  Theorem \ref{th:BODquads} implies $\bQ_2$ has BOD; now   Lemma \ref{lm:BODcrossed} implies $\bQ_2$ is crossed in $D$. In particular, $r_0\,r_1\,r_2\,r_3$ crosses $r_5\,r_6\,r_7\,r_8$ in $D$.}

\wording{In the case $v_2$ is an attachment of $B$, let $P$ and $P'$ be $H$-avoiding $v_0v_2$- and $v_2v_4$-paths in $B$, respectively.  Then the cycles $r_0\,r_1\lbsp\cc{v_2,P,v_0}$  and $r_2\, r_3\lbsp\cc{v_4,P',v_2}$ are both $H$-green.    Lemma \ref{lm:technicalV8colour} (\ref{it:atMostTwo}) implies neither is crossed in $D$, yielding the contradiction that $r_0\,r_1\,r_2\,r_3$ is not crossed in $D$.}

\wording{Thus, $B$ is the edge $xy$}.  \wordingrem{(Now-redundant text removed.)}Note that $B$ is not a local $H$-bridge and, therefore, not both $v_0$ and $v_4$ are attachments of $B$.  As $B$ is not crossed in $D$, we deduce that the $xy$-subpath of $r_0\,r_1\,r_2\,r_3$ is also not crossed in $D$.  Therefore, either $r_0$ or $r_3$ is crossed in $D$.  From this, we conclude that, since $\bQ_2$ is crossed in $D$, $r_6\,r_7$ is crossed in $D$.    Moreover, either $s_1$ or $s_3$ is exposed in $D$.  By symmetry, we may assume $s_1$ is exposed in $D$.

 If $x\ne v_0$, then the cycle $r_1\,r_2\,r_3\,s_3\,r_8\,r_9\,s_0\,r_5\,s_1$ is clean in $D$ and separates $x\in\oo{r_1}$ from $y\in \oc{v_3,r_3,v_4}$, so $B$ must be crossed in $D$, a contradiction.   If $y\ne v_4$, then the cycle $r_1\,r_2\,s_3\,r_8\,s_4\,r_4\,r_5\,s_1$ is clean in $D$ and separates $x\in \co{v_0,r_1,v_1}$ from $y\in \oo{r_3}$, and again $B$ is crossed in $D$, a contradiction. \end{proof}

Recall that $-1$ is equal to $9$.  The following is immediate from tidiness.

\begin{subclaim}\label{sc:r0not}  \begin{enumerate}\item\label{it:x[0,1]} If $x\in\co{a,r_9,v_0}$, then there is no attachment in $\co{v_0,r_0,v_1}$.  
\item\label{it:y[3,4]} If $y\in \oc{v_4,r_4,b}$, then there is no attachment in $\oc{v_3,r_3,v_4}$.  \hfill
\eopf \end{enumerate}\end{subclaim}

The next two subclaims are rather less trivial.

\begin{subclaim}\label{sc:r2not}\begin{enumerate}\item\label{it:x[2,3]}  If $x\in\co{a,r_9,v_0}$, then there is no attachment in $\co{v_2,r_2,v_3}$.
\item\label{it:y[1,2]}  If $y\in \oc{v_4,r_4,b}$, then there is no attachment in $\oc{v_1,r_1,v_2}$. \end{enumerate} \end{subclaim}

\begin{proof} We prove (\ref{it:x[2,3]})\wording{;}\wordingrem{(text removed)} (\ref{it:y[1,2]}) is symmetric.  For (\ref{it:x[2,3]}), suppose there is an attachment $y'$ in $\co{v_2,r_2,v_3}$.  By tidiness, there is no attachment other than $y'$  in $\oo{r_0\,r_1\,r_2\,r_3}$, and so minimality of $j-i$ implies $y'=y$.

The only other possible attachment is in $\cc{v_4,r_4,b}$.  If there is an attachment $z$  in $\cc{v_4,r_4,b}$, then either $y=v_2$ or $z=b=v_5$.  Thus, either $z$ does not exist and $B$ is the edge  $xy$, or $z$ exists, $B$ has exactly three attachments, namely $x$, $y$, and $z$,  and Lemma \ref{lm:threeAtts} shows $B$ is a $K_{1,3}$.  Let $P$ and $P'$ be the $xy$- and $yz$-paths (the latter only if $z$ exists) in $B$. 

Suppose first that $y\ne v_2$.  Then $x=v_9$, as otherwise $\cc{y,P,x,r_9,v_0}\rbsp r_0\,r_1\lbsp\cc{v_2,r_2,y}$ is an $H$-green cycle with the three $H$-nodes $v_0,v_1,v_2$ in its interior, contradicting Lemma \ref{lm:greenCycles} (\ref{it:shortJump}). 

Theorem \ref{th:BODquads} (\ref{it:aIsV0}) \dragominor{does not apply}, as $x=v_9=a$ implies $v_0\ne a$. 
 If Theorem \ref{th:BODquads} (\ref{it:bIsV5}) \dragominor{applies}, then there is a second $H$-bridge $B'$ attaching at $b=v_5$ and in $r_0\,r_1$.  But \wording{then $B$ and $B'$ must cross in $\Pi$}, a contradiction.  Therefore, Theorem \ref{th:BODquads} (\ref{it:cornerInGamma}) shows $\bQ_1$ has BOD.  

Let $e$ be an edge of $s_1$ and let $D$ be a 1-drawing of $G-e$.  Lemma \ref{lm:BODcrossed} implies $\bQ_1$ is crossed in $D$.  On the other hand, the presence of $P$ and Lemma \ref{lm:technicalV8colour} (\ref{it:2halfJump1})  and (\ref{it:V8yellow}) imply $\bQ_1$ cannot be crossed in $D$, the desired contradiction.  

Therefore, $y=v_2$.  Since $x,y\in r_9\,r_0\,r_1$, the hypothesis of the 
claim implies $z$ must exist.
 The cycles \wordingrem{(text removed)}$\cc{x,P,v_2}\rbsp r_1\,r_0\lbsp \cc{v_0,r_9,x}$ and $\cc{z,P',v_2}\rbsp r_2\,r_3\lbsp\cc{v_4,r_4,z}$ are $H$-green.   Let $e$ be an edge in $s_2$ and let $D$ be a 1-drawing of $G-e$. Theorem \ref{th:BODquads} implies $\bQ_2$ has BOD, so Lemma \ref{lm:BODcrossed} implies $\bQ_2$ is crossed in $D$.  \major{However, Lemma 7.2 (\ref{it:atMostTwo})\dragominorrem{(text removed)} shows that $r_0$ and $r_3$ are not crossed.  If $x\ne v_9$, then the same result shows $r_1$ is not crossed and likewise if $z\ne v_5$, then $r_2$ is not crossed.  If, say, $x=v_9$, then Lemma \ref{lm:technicalV8colour} (\ref{it:2halfJump3/2}) implies $r_1$ can only cross $r_8$.  However, if $z\ne v_4$, then (\ref{it:V8yellow}) shows $r_8$ cannot be crossed.}  
 
\major{ In the remaining case, $x=v_9$ and $z=v_4$.  In this case, $a=x=v_9$.  If $\bQ_1$ does not have BOD, then Theorem \ref{th:BODquads} (\ref{it:cornerInGamma}) implies $b=v_5$ and there is a $\bQ_1$-bridge $B'$ different from $M_{\bQ_1}$, having attachments at $b$ and in $r_0\,r_1$, and embedded in $\disc$.  But then $B'$ is an $H$-bridge different from $B$ that overlaps $B$ on $R$, while both are embedded in $\disc$, a contradiction.
} 
 \end{proof}

\begin{subclaim}\label{sc:r3not} \begin{enumerate}\item\label{it:x[3,4]}  If $x\in\co{a,r_9,v_0}$, then there is no attachment in $\co{v_3,r_3,v_4}$.
\item\label{it:y[0,1]}  If $y\in \oc{v_4,r_4,b}$, then there is no attachment in $\oc{v_0,r_0,v_1}$. \end{enumerate}  \end{subclaim}

\begin{proof} We prove (\ref{it:x[3,4]})\wording{;}\wordingrem{(text removed)} (\ref{it:y[0,1]}) is symmetric. For (\ref{it:x[3,4]}), suppose there is an attachment in $\co{v_3,r_3,v_4}$.   By minimality of $j-i$, Subclaim \ref{sc:r2not} and tidiness, this attachment is $y$.  Also by tidiness, there is no other attachment in $\oo{r_1\,r_2\,r_3\,r_4}$.  

Suppose there is also an attachment $z$ in $\co{v_1,r_1,v_2}$.  The preceding paragraph shows $z=v_1$.  Tidiness now implies that $x$  \dragominor{is} $v_9$ and, \dragominor{since $a\in r_9$ and $x\in \co{a,r_9,v_0}$,  $a=v_9$.}  Let $P$ and $P'$ be $H$-avoiding $xz$- and $yz$-paths in $B$, respectively.    

\dragominorrem{(Text removed.)}Theorem \ref{th:BODquads} (\ref{it:cornerInGamma}) implies $\bQ_1$ has BOD.  If $D_1$ is any 1-drawing of $G-\oo{s_1}$, then Lemma \ref{lm:BODcrossed} implies $\bQ_1$ is crossed in $D_1$.  But Lemma \ref{lm:technicalV8colour} implies (recall $z=v_1$) the two $H$-green cycles $\cc{z,P,x,r_9,v_0,r_0,z}$ and $\cc{y,P',z,r_1,v_2,r_2,v_3,r_3,y}$ are not crossed in $D_1$.  Thus, $r_9\,r_0\,r_1\,r_2$ is not crossed in $D_1$ (since $x=v_9$), so $\bQ_1$ is not crossed in $D_1$, a contradiction.

Therefore, there is no attachment in $\co{v_1,r_1,v_2}$.  Thus, we may assume that the only attachments in $\cc{a,r_9,v_0}\rbsp r_0\,r_1\,r_2\,r_3$ are $x\in \co{a,r_9,v_0}$ and $y\in \co{v_3,r_3,v_4}$.  Tidiness further shows there is no attachment in $\co{v_4,r_4,v_5}$, so the only other possible attachment of $B$ is $v_5$, in which case $y=v_3$.  

In each of the two cases $x\ne v_9$ and $x=v_9$, we show that $\bQ_4$ has NBOD by showing that $B$, $M_{\bQ_4}$, and the $\bQ_4$-bridge $B_4$ containing $s_4$ are mutually overlapping.  We remark that $B$ and $B_4$ are in different faces of $\Pi[H]$, so $B\ne B_4$.  Obviously, $B_4$ is skew to $M_{\bQ_4}$.

\medskip{\bf Case 1.}  $x\ne v_9$.

\medskip The attachments $x$ and $y$ of $B$ are skew to $v_4$ and $v_9$, so $B$ and $B_4$ overlap.  Also, $x$ and $y$ are skew to $v_8$ and $v_0$, so $B$ and $M_{\bQ_4}$ overlap, as required.    

\medskip{\bf Case 2.}  $x=v_9$.

\medskip As $x,y\in Q_3$ and $B$ is not $Q_3$-local, there is another attachment $z$ of $B$.  Our earlier remarks imply $z=v_5$ and $y=v_3$.  Now $y$ and $z$ show $B$ and $B_4$ are skew, while $x$ and $y$ show $B$ and $M_{\bQ_4}$ are skew.

\medskip We now resume our general discussion. Let $P_{xy}$ be \wording{the}  $xy$-path in $B$.  Since $x\in\co{a,r_9,v_0}$, $v_0\ne a$.  Suppose some $\bQ_1$-bridge $B'$ has an attachment at $b=v_5$ and an attachment in $r_0\,r_1$.  Since $B$ is not a $\bQ_1$-bridge and both $B$ and $B'$ are $H$-bridges, $B\ne B'$.  Then $P_{xy}$ and a $v_5\cc{r_0\,r_1}$-path in $B'$ would cross in $\Pi$, which is impossible.  Therefore, Theorem \ref{th:BODquads} shows $\bQ_1$ has BOD.  

Let $D_1$ be a 1-drawing of $G-\oo{s_1}$.  Because $\bQ_4$ has NBOD, Lemma \ref{lm:cleanBOD} implies $D_1[\bQ_4]$ is not clean in $D_1$.  Since $\bQ_1$ has BOD and $s_1$ is contained in a planar $\bQ_1$-bridge, Lemma \ref{lm:BODcrossed} implies $\bQ_1$ is crossed in $D_1$.  Therefore, $s_0$ is exposed in $D_1$.  Thus $D_1[H-\oo{s_1}]$ is one of two possible 1-drawings, depending on whether $r_9$ crosses $r_5\,r_6$ or $r_4$ crosses $r_0\,r_1$.  

If $x\ne v_9$, then $P_{xy}$ cannot be added to $D_1[H-\oo{s_1}]$ without introducing a second crossing, which is impossible.  If $x=v_9$, then the three attachments of $B$ are not all on the same face of $D_1[H-\oo{s_1}]$, so $B$ cannot be added to $D_1[H-\oo{s_1}]$ without introducing a second crossing, the final contradiction. \end{proof}

We can now complete the proof of Claim \ref{cl:no3rimPath}.  Subclaim \ref{sc:x9orY4} implies either $x\in \co{a,r_9,v_0}$ or $y\in \oc{v_4,r_4,b}$.  By symmetry, we may assume the former.  Subclaims \ref{sc:r2not} and \ref{sc:r3not} imply $y\in \cc{v_4,r_4,b}$.  If $y\ne v_4$, then Subclaims \ref{sc:r0not}, \ref{sc:r2not} and \ref{sc:r3not} (all six statements) show that there is no other attachment of $B$.  But then $B$ is $Q_4$-local, a contradiction.  Therefore, $y=v_4$, and, furthermore, there is an attachment $z$ of $B$ in $\co{v_1,r_1,v_2}$.    

If $x\ne v_9$, then both $x$ and $z$ are in $\oo{r_9\,r_0\,r_1}$, contradicting tidiness.  Thus, $x=v_9$.  

The claim will be proved once we know $z\ne v_1$.  By way of contradiction, suppose $z=v_1$.  Consider any 1-drawing $D_2$  of $G-\oo{s_2}$.  By Theorem \ref{th:BODquads},  $\bQ_2$ has BOD.  Thus, Lemma \ref{lm:BODcrossed} implies $\bQ_2$ is crossed in $D_2$.  That is, $r_0\,r_1\,r_2\,r_3$ crosses $r_5\,r_6\,r_7\,r_8$ in $D_2$.  In particular, neither $s_0$ nor $s_4$ is exposed in $D_2$.  

Since $B$ is global and has attachments at $v_4$ and $v_9$, it must be that $D_2[B]$ is in the face of $D_2[R\cup s_0\cup s_4]$ incident with $s_4$ and the crossing.   Since $v_1$ is an attachment of $B$,  $v_1$ must be in the subpath of $r_0\,r_1\,r_2\,r_3$ between the crossing and $v_4$.  But then $s_3$ is not exposed in $D_2$, implying $B$ must cross $s_3$ in $D_2$, a \dragominor{contradiction  that shows} $v_1$ is not an attachment of \dragominor{$B$,} completing the proof of the claim.
\end{proof} 

To complete the proof of the theorem,  by way of contradiction assume there is no $i$ so that $\att(B)\subseteq r_i\,r_{i+1}\,r_{i+2}$.   Claim \ref{cl:no3rimPath} shows either $\att(B)=\{v_0,v_5,z\}$, with $z\in \oo{r_2}\cup\oo{r_7}$ or $\att(B)=\{v_4,v_9,z\}$, with $z\in \oo{r_1}\cup\oo{r_6}$.  These are all the same up to the labelling of $H$, $a$, and $b$, so we may assume $\att(B)=\{v_0,v_5,z\}$, with $z\in \oo{r_2}$.  Let $H'$ be the subdivision of $V_{10}$ consisting of $H-\oo{s_0}$, together with the $v_0v_5$-path in $B$.  

In order to apply Theorem \ref{th:expSpokeNoAtt}, we show that $\Pi$ is $H'$-friendly.  If $\Pi$ is not $H'$-friendly, then Lemma \ref{lm:greenCyclesFriendly} (\ref{it:nearlyFriendly}) implies (since $H$ and $H'$ have the same nodes) $v_6v_9$ is an edge and $\Pi[v_6v_9]$ is contained in $\Mob_{H'}$, which is the same as $\Mob_{H}$.  But $v_6$ and $v_9$ are not incident with the same $H$-face in $\Mob_{H}$ and, therefore, this is impossible.  Thus, $\Pi$ is $H'$-friendly.  However, $H'$ violates Theorem \ref{th:expSpokeNoAtt}, a contradiction.  

Therefore, there is an $i$ so that $\att(B)\subseteq r_i\,r_{i+1}\,r_{i+2}$.  Claim  \ref{cl:3rimPath} implies $B$ has one of the three desired forms.\end{cproof}

We can go somewhat further in our analysis of the global $H$-bridges of a tidy $\hvfg$.  

}\begin{definition}\label{df:span}  Let $G\in \m2$ and  $\hvfg$, with $H$ tidy.  Let $B$ be a global $H$-bridge with attachments $x$ and $y$.  
\begin{enumerate}
\item The {\em span of $B$\/}\index{span} is the $xy$-subpath $R$ with the fewest $H$-nodes.
\item An edge or subpath of $R$ is {\em spanned by $B$\/}\index{spanned by} if it is in the span of $B$.  
\item \wording{\dragominor{$B$ is}: a {\em 2-jump\/}\index{2-jump}\index{jump} if, for some $i$, its attachments are $v_i$ and $v_{i+2}$; a {\em 3-jump\/}\index{3-jump} if, for some $i$, its attachments are $v_i$ and $v_{i+3}$; or else is a {\em 2.5-jump\/}\index{2.5-jump}}.
\end{enumerate}
\end{definition}\printFullDetails{

\wording{We remark that Theorem \ref{th:globalBridges} implies that, in the case of a 2.5-jump, there is an $i$ so that $v_i$ is one attachment and the other attachment is in $\oo{r_{i-3}}\cup \oo{r_{i+2}}$.}
Theorem \ref{th:globalBridges} \wording{further} implies a global $H$-bridge has precisely two attachments and its span has at most four $H$-nodes.  \minor{It follows from Definition \ref{df:green} that every global $H$-bridge combines with its span to form an $H$-green cycle.}

}\begin{lemma}\label{lm:tidyBODhyper}  Let $G\in\m2$ and $\hvfg$, with $H$ tidy.  For each $i\in\{0,1,2,3,4\}$, either $\bQ_i$ has BOD or one of $v_{i-1}v_{i-4}$ and $v_{i+1}v_{i+4}$ is a global $H$-bridge.  \end{lemma}\printFullDetails{

\begin{cproof} Let $\Pi$ be an embedding of $G$ in $\pp$ so that $H$ is $\Pi$-tidy.  Suppose neither of the edges $v_{i-1}v_{i-4}$ and $v_{i+1}v_{i+4}$ occurs in $G$.  The $\bQ_i$-bridges that are $\bQ_i$-exterior consist of $M_{\bQ_i}$, those that are contained in $\Mob$ and, therefore, attach along either $s_{i-1}$ or $s_{i+1}$, and those that are contained in $\Disc$.   Since $H$ is $\Pi$-tidy, these latter must be global.  By Theorem \ref{th:globalBridges} they are 2-, 2.5-, and 3-jumps.  

Consider any global $H$-bridge.  It is embedded in $\Disc$ so that it, together with its spanned path in $R$, bounds a face of $G$.  In particular, if we are considering a 2-jump $B$ that is a $\bQ_i$-bridge, the 2-jump is either $v_{i-1}v_{i+1}$ or $v_{i+4}v_{i+6}$.  In this case, $\bQ_i\cup B$ has no non-contractible cycle in $\pp$ and so, by Lemma \ref{lm:planeNotOverlap}, $B$ does not overlap any other $\bQ_i$-exterior $\bQ_i$-bridge.  

It is not possible for a 2.5-jump to be a $\bQ_i$-bridge.  The only 3-jumps that can be a $\bQ_i$-bridge are $v_{i+1}v_{i+4}$ and $v_{i-4}v_{i-1}$, and these are assumed not to be in $G$.  We conclude that the $\bQ_i$-exterior $\bQ_i$-bridges do not overlap and, therefore, $\bQ_i$ has BOD.
\end{cproof}

}\begin{lemma}\label{lm:globalJumps}  Let $G\in \m2$ and  $\hvfg$, with $H$ tidy.  Then:
\begin{enumerate}
\item\label{it:globalNoHnode}  no two global $H$-bridges have an $H$-node in common; 
\item\label{it:atMostOne3jump} at most one global 
$H$-bridge is a $3$-jump; 
\item\label{it:2and3jumps}  there is no $i$ so that $v_iv_{i+3}$ is a 3-jump and some 2.5-jump has an end in $\oc{v_{i-1},r_{i-1},v_i}$;
\item\label{it:spanDiffHyper} if $B_1$ and $B_2$ are global $H$-bridges, then, for every $i\in\{0,1,2,3,4\}$, there is some edge of $\bQ_i\cap R$ that is not spanned by either $B_1$ or $B_2$; and
\item\label{it:oppositeRims}  for each $i\in\{0,1,2,3,4\}$, at most one of $\oo{r_i}$ and $\oo{r_{i+5}}$ can contain an end of a 2.5-jump.  \end{enumerate} \end{lemma}\printFullDetails{

\begin{cproof}  \dragominor{We start with (\ref{it:globalNoHnode}).}

\begin{claim}\label{cl:globalNoHnode}  \dragominor{No two global $H$-bridges have an $H$-node in common.}\end{claim}

\begin{proof}\startSubclaims suppose by way of contradiction that the two global $H$-bridges $B_1$ and $B_2$ have the $H$-node $v_i$ in common.  For $j=1,2$, let $P_j$ be the subpath of $R$ spanned by $B_j$.   Then each of $B_j\cup P_j$ is a green cycle; therefore, Theorem \ref{th:twoGreenCycles} implies $P_1$ and $P_2$ are edge disjoint.   We choose the labelling so that $r_{i}\cup r_{i+1}\subseteq P_1$ and $r_{i-2}\cup r_{i-1}\subseteq P_2$. 
We treat various cases.

\begin{subclaim}\label{cl:twoThrees}  \dragominor{At least one of $B_1$ and $B_2$ is not a 3-jump.}\end{subclaim}

\begin{proof} Suppose to the contrary that $B_1$ and $B_2$ are both 3-jumps, so $B_1=v_{i}v_{i+3}$ and $B_2=v_{i-3}v_i$, respectively.  Then there is a 1-drawing $D_i$ of $(H-s_{i})\cup B_1\cup B_2$; Lemma \ref{lm:tidyBODhyper} implies $\bQ_i$ has BOD, so Lemma \ref{lm:BODcrossed} implies $\bQ_i$ is crossed in $D_i$.

Because of $B_1$, Lemma \ref{lm:technicalV8colour} (\ref{it:2halfJump1}) implies $r_{i+1}$ and $r_{i+2}$ are not crossed in $D_i$, while (\ref{it:2halfJump3/2}) of the same lemma implies that if $r_i$ were crossed, it would cross $r_{i+3}$.  However, (\ref{it:V8yellow}) shows $r_{i+3}$ is not crossed.  Therefore, no edge of $r_i\,r_{i+1}\,r_{i+2}$ is crossed in $D_i$.  Analogously, no edge of $r_{i-3}\,r_{i-2}\,r_{i-1}$ is crossed in $D_i$.  These two assertions show $\bQ_i$ cannot be crossed in $D_i$, a contradiction. \end{proof}

\begin{subclaim}\label{cl:noThree}  \dragominor{Neither $B_1$ nor $B_2$ is a 3-jump. }\end{subclaim}

\begin{proof}  By Claim \ref{cl:twoThrees}, not both $B_1$ and $B_2$ are 3-jumps.  So suppose for sake of definiteness that $B_1$ is the 3-jump $v_iv_{i+3}$ and $B_2$ is a global $H$-bridge with one end at $v_i$ and one end in $\oc{v_{i-3},r_{i-3},v_{i-2}}$.

The embedding in $\pp$ shows that $v_{i+2}v_{i+5}$ is not an edge of $G$ (it would cross $B_1$) and Claim \ref{cl:twoThrees} shows $v_{i-3}v_i$ is not an edge of $G$.  Therefore, Lemma \ref{lm:tidyBODhyper} implies $\bQ_{i+1}$ has BOD.  Thus, in any 1-drawing $D_{i+1}$ of $G-\oo{s_{i+1}}$, Lemma \ref{lm:BODcrossed} implies $\bQ_{i+1}$ is crossed in $D_{i+1}$.    

By Lemma \ref{lm:technicalV8colour} (\ref{it:2halfJump1}) (when $B_2$ is a 2.5-jump) or (\ref{it:atMostTwo}) (when $B_2$ is a 2-jump), $r_{i-1}$ is not crossed in $D_{i+1}$.  Likewise,  (\ref{it:atMostTwo}) shows that none of $r_{i}$, $r_{i+1}$, and $r_{i+2}$ is crossed in $D$.  But then $\bQ_{i+1}$ is not crossed in $D_{i+1}$, a contradiction.
\end{proof}  

By Claim \ref{cl:noThree}, we know that neither $B_1$ nor $B_2$ is a 3-jump.  By Theorem \ref{th:twoGreenCycles}, neither $v_{i-1}v_{i-4}$ nor $v_{i+1}v_{i+4}$ can occur in $G$; Lemma \ref{lm:tidyBODhyper} implies $\bQ_i$ has BOD.  Let $D_i$ be a 1-drawing of $G-\oo{s_i}$.  By Lemma \ref{lm:BODcrossed}, $\bQ_i$ is crossed in $D_i$.

Lemma \ref{lm:technicalV8colour} (\ref{it:atMostTwo}) shows that $P_1$ and $P_2$ are both not crossed in $D_i$.  This implies that $r_{i-2}\,r_{i-1}\,r_i\,r_{i+1}$ is not crossed in $D$ and, therefore, $\bQ_i$ is not crossed in $D_i$, a contradiction that completes the proof of the claim. \end{proof}

\dragominor{We move on to (\ref{it:atMostOne3jump}).}

\begin{claim}\label{cl:atMostOne3jump}  \dragominor{There is at most one global $H$-bridge that is a 3-jump.} \end{claim}

\begin{proof} Suppose  
there are distinct $3$-jumps.  Claim \ref{cl:globalNoHnode} implies that, up to relabelling, they are either $v_iv_{i+3}$ and $v_{i+4}v_{i+7}$ or $v_iv_{i+3}$ and $v_{i+5}v_{i+8}$.  Theorem \ref{th:twoGreenCycles} and  Claim \ref{cl:globalNoHnode} imply that there cannot be a third 3-jump.  Thus, Lemma \ref{lm:tidyBODhyper} implies $\bQ_{i+1}$ has BOD.    

Let $C_1$ and $C_2$ be the two $H$-green cycles containing these 3-jumps.
Lemma \ref{lm:BODcrossed} implies $\bQ_{i+1}$ is crossed in a 1-drawing $D_{i+1}$ of $G-\oo{s_{i+1}}$.  \wordingrem{(Text removed.)}
But Lemma \ref{lm:technicalV8colour} (\ref{it:atMostTwo}) implies that neither $r_i\,r_{i+1}$ nor $r_{i+5}\,r_{i+6}$ is crossed in $D_{i+1}$, a contradiction proving the claim. \end{proof}  

\dragominor{We next turn to (\ref{it:2and3jumps}).  }

\begin{claim}\label{cl:2and3jumps}  \dragominor{There is no $i$ so that $v_iv_{i+3}$ is a 3-jump and some 2.5-jump has an end in $\oc{v_{i-1},r_{i-1},v_i}$.}\end{claim}

\begin{proof} \dragominor{Suppose to the contrary that there is such an $i$.}  From Claim \ref{cl:globalNoHnode}, the 2.5-jump has an end $w\in \oo{v_{i-1},r_{i-1},v_i}$.  Its other end is $v_{i-3}$.   Lemma \ref{lm:tidyBODhyper} and  Claim \ref{cl:atMostOne3jump} imply that $\bQ_{i+2}$ has BOD. Let $D_{i+2}$ be a 1-drawing of $G-\oo{s_{i+2}}$.  
Lemma \ref{lm:BODcrossed} implies $\bQ_{i+2}$ is crossed in $D_{i+2}$.

By Lemma \ref{lm:technicalV8colour} (\ref{it:V8yellow}), $r_{i+3}$ is not crossed in $D_{i+2}$.  The same lemma (\ref{it:atMostTwo}) implies $r_i\,r_{i+1}\,r_{i+2}$ is not crossed in $D_{i+2}$.  Consequently, $\bQ_{i+2}$ is not crossed in $D_{i+2}$, contradicting the preceding paragraph and proving the claim.  \end{proof}

\dragominor{Now we prove (\ref{it:spanDiffHyper}).}

\begin{claim}\label{cl:spanDiffHyper}  \dragominor{If $B_1$ and $B_2$ are global $H$-bridges, then, for every $i\in \{0,1,2,3,4\}$, some edge of $\bQ_i\cap R$ is not spanned by either $B_1$ or $B_2$. }\end{claim}

\begin{proof} \wording{Suppose by way of contradiction that the global $H$-bridge}  $B_1$ spans the side $r_i\cup r_{i+1}$ of $\bQ_{i+1}$ and \wording{a second global $H$-bridge} $B_2$ spans $r_{i+5}\cup r_{i+6}$.   To see that $\bQ_{i+1}$ has BOD, by Lemma \ref{lm:tidyBODhyper}\wordingrem{(comma deleted)} it suffices to show that neither of the 3-jumps $v_iv_{i-3}$ and $v_{i+2}v_{i+5}$ is in $G$.   For the former, Theorem \ref{th:twoGreenCycles} implies $v_i$ is an attachment of $B_1$, contradicting  Claim \ref{cl:globalNoHnode}.  For the latter, $v_{i+2}$ is an attachment of $B_2$, with the same contradiction.  
Therefore $\bQ_{i+1}$ has BOD.

Lemma \ref{lm:BODcrossed} implies that, for any 1-drawing $D_{i+1}$ of $G-\oo{s_{i+1}}$, $\bQ_{i+1}$ is crossed in $D_{i+1}$.  However, Lemma \ref{lm:technicalV8colour} (\ref{it:atMostTwo}) implies that neither $r_i\,r_{i+1}$ nor $r_{i+5}\,r_{i+6}$ is crossed in $D_{i+1}$, showing $\bQ_{i+1}$ is not crossed in $D_{i+1}$, a contradiction proving the claim.  \end{proof}

Finally, we prove (\ref{it:oppositeRims}).  
Suppose, for $j\in \{i,i+5\}$, $\oo{r_j}$ contains an end of the 2.5-jump $B_j$.  We may use the symmetry to assume  that $B_i=wv_{i-2}$.  If $B_{i+5}$ has $v_{i+3}$ as an end, then we contradict  Claim \ref{cl:spanDiffHyper}.  Therefore, $B_{i+5}$ has $v_{i+8}=v_{i-2}$ as an end, contradicting  Claim \ref{cl:globalNoHnode}.  
\end{cproof}

We conclude this section with two observations about local bridges of a tidy subdivision of $V_{10}$.

}\begin{lemma}\label{lm:noSpokeOnlyBridge}  Let $G\in \m2$ and $\hvfg$, with $H$ tidy.  Then no $H$-bridge has all its attachments in one $H$-spoke.\end{lemma}\printFullDetails{

\begin{cproof}  By way of contradiction, suppose $B$ is an $H$-bridge and $s$ is an $H$-spoke so that $\att(B)\subseteq s$.  By Corollary \ref{co:attsMissBranch}, $B$ has precisely two attachments, so $B$ is just an edge $uw$.  Choose $B$ so that no other $H$-bridge has all its attachments in a proper subpath of $\cc{u,s,w}$.  
If $\cc{u,s,w}$ has no interior vertex, then $B$ and $\cc{u,s,w}$ are parallel edges not in the $H$-rim\wording{, contradicting Observation \ref{obs:parallel} (\ref{rim}).}\majorrem{(The preceding phrase was a principle reason for moving the parallel edges section.)}  Thus, some $H$-bridge $B'$ has an attachment $x$ in $\oo{u,s,w}$.   

Let $\Pi$ be an embedding of $G$ in $\pp$ for which $H$ is $\Pi$-tidy. Since $H\subseteq \Mob$, $B'$ is a local $H$-bridge.  Moreover, Corollary \ref{co:attsMissBranch} and the choice of $B$ show that not all  attachments of $B'$ can be in $\cc{u,s,w}$, so $B$ has an attachment $y$ not in $\cc{u,s,w}$.  But then, for at least one of the two $H$-quads $Q$ containing $s$, $B$ and $B'$ are overlapping $Q$-bridges, contradicting the definition of tidiness.
\end{cproof}

}\begin{lemma}\label{lm:noBridgeTwoInSpoke}  Let $G\in \m2$, $\hvfg$, with $H$ tidy.  For any $H$-spoke $s$, if $B$ is an $H$-bridge having an attachment in $\oo{s}$, then $B$ has no other attachment in $\cc{s}$.\end{lemma}\printFullDetails{

\begin{cproof}  Suppose $B$ is an $H$-bridge and $s$ an $H$-spoke so that $B$ has attachments $x,y$ in $s$, with $x\in\oo{s}$.    Let $\Pi$ be an embedding of $G$ in $\pp$ for which $H$ is $\Pi$-tidy.  Then $\Pi$ shows $B$ is not a global $H$-bridge.  By Lemma \ref{lm:noSpokeOnlyBridge}, $B$ has a third attachment $z$ not in $\cc{s}$.   Let $Q$ be the unique $H$-quad containing all of $x$, $y$, and $z$.  

If $y$ is not an $H$-node, then let $r$ be an $H$-rim branch of $Q$ not containing $z$.  Then $x$, $y$, and $z$ are all contained in $Q-\cc{r}$, contradicting Corollary \ref{co:attsMissBranch}.  Thus, $y$ is an $H$-node $v_i$.  We choose the labelling so that $r_i\subseteq Q$.  Corollary \ref{co:attsMissBranch} shows that $z$ is not in $Q-\cc{r_{i+5}}$ and, therefore, $z$ is in $r_{i+5}$.  Furthermore, Corollary \ref{co:attsMissBranch} now shows that $B$ can have no other attachment, so Theorem 8.2 implies $B$ is isomorphic to $K_{1,3}$.  Let $w$ be the vertex in $\Nuc(B)$.  

\begin{claim}  The cycles $\cc{y,B,w,B,x,s,y}$ and $\cc{z,B,w,B,x,Q-y,z}$ bound faces of $\Pi[G]$.\end{claim}

\begin{proof}  For the latter, $\cc{z,B,w,B,x,Q-y,z}$ is an $H$-green cycle, so the result follows from Lemma \ref{lm:greenCycles}.  The former, call it $C$, has just one vertex in $R$, so Lemma \ref{lm:CdisjointNCcycle} implies it has BOD and every one of its bridges other than the one containing $H-\oo{s}$ is planar.  If it has a second bridge $B'$, then $C$ is clean in any 1-drawing of $\comp{B'}$, contradicting Lemma \ref{lm:BODcrossed}. \end{proof}

The chosen labelling shows that $Q_{i-1}$ is the other $H$-quad containing $s$.  

\begin{claim}  There is no $Q_{i-1}$-local $H$-bridge that has an attachment in $\oo{s}$.    \end{claim}

\begin{proof} Suppose $B''$ is a $Q_{i-1}$-local $H$-bridge having an attachment $x'$ in $\oo{s}$.  Lemma \ref{lm:noSpokeOnlyBridge} implies $B''$ has an attachment $z'$ not in $\cc{s}$.  If $z'$ is in the same $H$-rim branch $r_{i-1}$ contained in $Q_{i-1}$ as $y$, then $\cc{x',B'',z',r,y,B,w,x,s,x'}$ is an $H$-green cycle $C$.  As the edge of $s$ incident with $y$ is $C$-interior, $C$ does not bound a face of $\Pi[G]$.  If $z'$ is not in $r_{i-1}$, then $\cc{z',B'',x',s,x,B,w,B,z,\bQ_i-y,z'}$ is a non-facial $H$-green cycle.  Both conclusions contradict Lemma \ref{lm:greenCycles} (\ref{it:CboundsFace}).   \end{proof}

We conclude that $s$ has length 2 and that $B$ is the only $H$-bridge attaching in $\oo{s}$.  Let $D$ be a 1-drawing of $G-wy$.  Then $D[s\cup (B-wy)]$ is clean in $D$ and we may extend $D$ to a 1-drawing of $G$ by adding in $wy$ alongside $\cc{w,B,x,s,y}$.
\end{cproof}
}

\chapter{Every rim edge has a colour}\printFullDetails{

In this section we introduce\dragominorrem{(text removed)}, for a tidy subdivision $H$ of $V_{10}$ in $G$, $H$-yellow edges.  The main result is that every $H$-rim edge has a \dragominor{colour:  $H$-green, $H$-yellow, or red}.  This is a major step on the route.  In the next section, we will analyze red edges, with the main result being that there are red edges.

}\begin{definition}\label{df:yellow}  Let $H$ be a subdivision of $V_{10}$ in a graph $G$.
\begin{enumerate}
\item\label{it:3rimPath} A {\em $3$-rim path\/}\index{3-rim path}\index{rim path} is a path contained in the union of three consecutive $H$-rim branches.
\item\label{it:closure}   The {\em closure\/}\index{closure}\index{$\cl(Q)$} $\cl(Q)$ of an $H$-quad $Q$ is the union of $Q$ and all $Q$-local $H$-bridges. 
\item  Let $H$ be\wordingrem{(text removed)} tidy in $G$.    A cycle $C$ in $G$ is {\em $H$-yellow\/}\index{yellow}\index{$H$-yellow} if $C$ may be expressed as the composition $P_1P_2P_3P_4$ of four paths so that:
\begin{enumerate}
\item $P_2$ and $P_4$ are \wording{$R$-avoiding (recall $R$ is the $H$-rim)} \dragominor{and have length at least 1};
\item\label{it:yellowP1P3} $P_1$ and $P_3$ are 3-rim paths and $P_1\cup P_3$ is not contained in a 3-rim path; and 
\item there is an $H$-green cycle $C'$ so that $P_1\subseteq \oo{C'\cap R}$.
\end{enumerate}
\item An $H$-rim edge $e$ is {\em $H$-yellow\/} if it is not $H$-green and is in an $H$-yellow cycle.
\end{enumerate}
\end{definition}\printFullDetails{ 

We remark that the $H$-rim edges that are $H$-yellow are those in $P_3$.
The next result elucidates the nature of an $H$-yellow cycle.

}\begin{lemma}\label{lm:yellowCycles}  Let $G\in \m2$, $\hvfg$, with $H$ tidy.  Let $C$ be an $H$-yellow cycle, with decomposition $P_1P_2P_3P_4$ into paths as in Definition \ref{df:yellow}, 
and let $C'$ be the witnessing $H$-green cycle.  Then:
\begin{enumerate}
\item\label{it:greenGlobal} $C'-\oo{C'\cap R}$ is a global $H$-bridge;
\item\label{it:yellowP2P4} for $i\in\{2,4\}$, $P_i$ is either $H$-avoiding or decomposes as $P_i^1P_i^2$, where $P_i^1$  is contained in some $H$-spoke, including an incident $H$-node, and $P_i^2$ is $H$-avoiding;
\item\label{it:oneBridge} there is only one $C$-bridge in $G$; and
\item\label{it:quadClosure} there is an $i\in\{0,1,2,3,4\}$ so that $C\subseteq \cl(Q_i)$.  
\end{enumerate}
\end{lemma}\printFullDetails{

\begin{cproof}   Let $\Pi$ be an embedding of $G$ in $\pp$ for which $H$ is $\Pi$-tidy; in particular, every $H$-green cycle bounds a face of $\Pi[G]$.  

 For (\ref{it:greenGlobal}), the alternative is that $C'$ is contained in $\cl(Q)$, for some $H$-quad $Q$. \dragominor{Lemma \ref{lm:greenCycles} (\ref{it:CboundsFace}) shows that} $C'$ bounds a face of $G$ in $\pp$, so $P_2$ and $P_4$ are contained in global $H$-bridges.   Each of $P_2$ and $P_4$ is in an $H$-green cycle (as is every global $H$-bridge) and, since $P_2$ has an end in $\oo{C'\cap R}$, some edge of $C'\cap R$ is in two $H$-green cycles, contradicting Theorem \ref{th:twoGreenCycles}.

For (\ref{it:yellowP2P4}), let $i\in \{2,4\}$. \dragominor{Since $P_i$ has positive length, the end $u_i$ of $P_i$ in  $P_1$ is distinct from the end $w_i$ of $P_i$ in $P_3$.}    Because $C'$ bounds a face of $G$ and is contained in $\Disc$, we see that the edges of $P_i$ incident with $u_i$ is in $\Mob$.  Since $P_i$ is $R$-avoiding, $P_i$ is contained in $\Mob$, with only its ends in $R$.

Now suppose $P_i$ has a\wording{n edge} $e$ not in $H$.    Choose $e$ to be as close to $u_i$ in $P_i$ as possible.  As $w_i$ is in $H$, there is a first vertex $y$ of $P_i$ after $e$ that is in $H$.  If $y=w_i$, then we are done, so we may assume $y\ne w_i$.  Since $P_i$ is $R$-avoiding, we see that $y$ must be in the interior of some spoke $s$.   Let $z$ be the vertex of $P_i$ incident to $e$ so that $e$ is in $\cc{z,P_i,y}$.     

As $P_i$ is contained in $\Mob$, we see that $\cc{u_i,P_i,y}$ is contained in a closed $\Pi[H]$-face bounded by some $H$-quad $Q$.  Also, $\cc{z,P_i,y}$ is $H$-avoiding and so \wording{is} contained in some $Q$-local $H$-bridge $B_i$.  By Lemma \ref{lm:noBridgeTwoInSpoke}, $y$ is the only attachment of $B_i$ in $\cc{s}$.  Since $z\ne y$ and both are attachments of $B_i$, we have that $z\notin\cc{s}$.  

The path $\cc{u_i,P_i,z}$ is $R$-avoiding and contained in $H$.  Therefore, either it is trivial or it is contained in some $H$-spoke $s'$.  In the latter case, $z\ne y$ implies $s'\ne s$.  In the former case, $u_i=z$, so $u_i\notin s$.  In both cases,  $\cc{u_i,P_i,y}\cup Q$ contains an $H$-green cycle that contains an $H$-rim edge incident with $u_i$, contradicting Theorem \ref{th:twoGreenCycles} and completing the proof of (\ref{it:yellowP2P4}).

\dragominor{For (\ref{it:oneBridge}), we start by noting that there exist $i$ and $j$ so that $P_1\subseteq r_i\,r_{i+1},\dots,r_j$ and $i-1\le j\le i+2$; we assume $P_1$ has one end in $\co{v_i,r_i,v_{i+1}}$, one end in $\oc{v_j,r_j,v_{j+1}}$, and that $j=i-1$ only if $P_1$ is just the single $H$-node $v_i$.  Item \ref{it:yellowP2P4} implies $P_2$ is contained in $\cl(Q_{i-1})\cup \cl(Q_i)$ and that $P_4$ is contained in $\cl(Q_j)\cup \cl(Q_{j+1})$.  It follows that $P_3$ has its ends in $r_{i+4}\,r_{i+5}$ and $r_{j+5}\,r_{j+6}$.  There are at most $(j+6)-(i+3)\le 5$ $H$-rim branches $r_{i+4}\,r_{i+5}\, \dots\, r_{j+6}$, so $P_3$, being a 3-rim path, must be contained in this path.  It follows that $C$ is disjoint from either $s_{i-2}$ or $s_{i+2}$.}

Let $s$ be an $H$-spoke disjoint from $C$ and let $M_C$ denote the $C$-bridge containing $s$.  

Set $R'=(R-\oo{C'\cap R})\,\cup\, (C'-\oo{C'\cap R})$.  Then $R'\cup s$ contains a non-contractible cycle $C''$ disjoint from $C$.  \dragominor{Lemma \ref{lm:CdisjointNCcycle} shows $C$ is contractible, has BOD, and every $C$-bridge other than $M_C$ is planar. } 

Suppose there is a $C$-bridge $B$ other than $M_C$; let $D$ be a 1-drawing of $\comp{B}$.
Lemma \ref{lm:BODcrossed} implies $D[C]$ is crossed.  Let $s$, $s'$, and $s''$ be the three $H$-spokes disjoint from $\oo{C'\cap R}$.   Then $R\cup s\cup s'\cup s''$ is a subdivision of $V_6$ in $\comp{B}$ that is edge-disjoint from both $P_2$ and $P_4$; this shows that some edge of $P_1\cup P_3$ is crossed in $D$. 

But now  $R'\cup s\cup s'\cup s''$ is another subdivision of $V_6$ in $\comp{B}$.  Therefore, the crossing in $D$ must involve two edges of $R'\cup s\cup s'\cup s''$.  In particular it does not involve an edge of $C'\cap R$, and, since $P_1\subseteq C'\cap R$, no edge of $P_1$ is crossed in $D$. 

Likewise, let $R''$ be obtained from $R'$ by replacing $P_3$  with $P_2P_1P_4$.  Now $R''\cup s\cup s'\cup s''$ is a third subdivision of $V_6$ in $\comp{B}$ that is disjoint from $P_3$.  Thus, the crossing in $D$ does not involve an edge of $P_3$.  Thus, none of $P_1$, $P_2$, $P_3$, and $P_4$ is crossed in $D$, contradicting the fact that $C$ is crossed in $D$.  We conclude that there is no $C$-bridge other than $M_C$, as claimed.

Finally, for (\ref{it:quadClosure}), suppose first that $P_1$ is not contained in a single $H$-rim branch.  Then there is an $H$-node $v_i$ in the interior of $P_1$.  \dragominor{However, $P_1$ is incident on one side with the face bounded by $C'$, so  the edge of $s_i$ incident with $v_i$ is on the other side of $P_1$.  Since $C$ is contractible, we conclude that there are at least two $C$-bridges, contradicting (\ref{it:oneBridge}).  Therefore, there is an $i\in \{0,1,2,3,4\}$ so that $P_1\subseteq r_i$.}

If both $P_2$ and $P_4$ are contained in $\cl(Q_i)$, then so is $P_3$, as it is a 3-rim path.   Therefore, by symmetry, we may assume that $P_2$ has some edge not in $\cl(Q_i)$.  As we traverse $P_2$ from its end in $P_1$, we come to a first edge $e$ that is not in $\cl(Q_i)$.  One end of $e$ is the vertex $u$ that is in either $s_i$ or $s_{i+1}$; for the sake of definiteness, we assume the former.  Then (\ref{it:yellowP2P4}) implies $\cc{v_i,s_i,u}\subseteq P_2$ and that the remainder of $P_2$ consists of an $H$-avoiding $uw$-path, with $w$ an end of $P_3$.
It follows that $w\in r_{i+4}$.  Let $\hat e$ be the edge of $s_i$ incident with $u$ and not in $P_2$.

\dragominor{Switching paths, we know that} $P_{4}$ has an end $x$ in $r_i$.  If $x\ne v_{i+1}$, then (\ref{it:yellowP2P4}) implies $P_{4}\subseteq \cl(Q_i)$.  In this case, $\hat e$ is in a $C$-bridge other than $M_C$, contradicting (\ref{it:oneBridge}).  Otherwise $x=v_{i+1}$, in which case $P_1P_2\lbsp\cb{w,r_{i+4},v_{i+5},r_{i+5},}$ $\bc{v_{i+6},s_{i+1},v_{i+1}}$ is an $H$-yellow cycle $\widehat C$.  There is a $\widehat C$-bridge other than $M_{\widehat C}$ containing $\hat e$, also contradicting (\ref{it:oneBridge}) for $\widehat C$.  \end{cproof}

We now turn our attention to the all-important red edges. \dragominorrem{(text omitted.)}We comment that, if $n\ge 4$ and $\hvng$, then any red edge of $G$ is in the $H$-rim.

The remainder of this section is devoted to proving the following.

}\begin{theorem}\label{th:rimColoured}  Let $G\in \m2$ and let $\hvfg$.  If $H$ is  tidy, then every $H$-rim edge is one of $H$-green, $H$-yellow, and red. \end{theorem}\printFullDetails{

We start with an easy observation.

}\begin{lemma}\label{lm:greenNotRed}  Let $G\in \m2$ and let $\hvfg$.  If $H$ is  tidy and the $H$-rim edge $e$ is either $H$-green or $H$-yellow, then $e$ is not red. \end{lemma}\printFullDetails{

\begin{cproof}  Suppose first that $e$ is $H$-green and let $C$ be the $H$-green cycle containing $e$.  There are three $H$-spokes $s$, $s'$, and $s''$ disjoint from $\oo{C\cap R}$.  Thus, $(R-\oo{C\cap R})\cup (C-\oo{C\cap R})$ together with $s$, $s'$, and $s''$ is a subdivision of $V_6$ contained in $G-e$, showing $e$ is not red.

Now suppose $e$ is $H$-yellow and let $C$ be the $H$-yellow cycle containing $e$.  Let $C'$ be the $H$-green cycle and $P_1P_2P_3P_4$ the decomposition of $C$ as in Definition \ref{df:yellow}.  Then $e$ is in $P_3$ and there are three $H$-spokes $s$, $s'$, and $s''$ disjoint from $C\cup \oo{C'\cap R}$.  In this case,  $(R-(\oo{C'\cap R}\cup \oo{P_3}))\cup (C'-\oo{C'\cap R})\cup P_2P_1P_4$, together with $s$, $s'$, and $s''$ is a subdivision of $V_6$ contained in $G-e$, showing $e$ is not red. \end{cproof}

The following concepts and lemma play a central role in the proof of Theorem \ref{th:rimColoured}.

}\begin{definition}\label{df:separated} Let $\hvfg$. 
 Let $e$ and $f$ be two edges of the $H$-rim $R$.  Then $e$ and $f$ are {\em $R$-separated in $G$\/}\index{$R$-separated}\index{separated} if $G$ has a subdivision $H'$ of $V_8$ so that the $H'$-rim is $R$ and $e$ and $f$ are in disjoint $H'$-quads.
\end{definition}\printFullDetails{

The following two observations are immediate from the definition.

}\begin{observation}\label{obs:separated}  Let $\hvfg$ and suppose $e$ and $f$ are two edges of the $H$-rim $R$ that are $R$-separated in $G$.  
\begin{enumerate}\item\label{it:separatedNoCross}  If $D$ is a 1-drawing of $G$, then $e$ and $f$ do not cross each other in $D$.
\item\label{it:twoSpokes} If $H'$ is a $V_8$ in $G$ witnessing the $R$-separation of $e$ and $f$, then there are two $H'$-spokes that have all their ends in the same component of $R-\{e,f\}$.  \hfill\eop
\end{enumerate}
\end{observation}\printFullDetails{

The following is a kind of converse of Observation \ref{obs:separated} (\ref{it:separatedNoCross}).

}\begin{lemma}\label{lm:notSeparated}  Let $G_0\in \m2$ be a graph and let $\hvfg_0$, with $H$ tidy.  Suppose $G\subseteq G_0$ with $H\subseteq G$.    Let $e\in r_i$ and $f\in r_{i+4}\cup r_{i+5}\cup r_{i+6}$ be edges that are both neither $H$-green nor $H$-yellow.  If $e$ and $f$ are not $R$-separated in $G$, then there is a 1-drawing of $G$ in which $e$ crosses $f$.\end{lemma}\printFullDetails{

\begin{cproof}  \wordingrem{(Text removed.)}
Let $\Pi$ be an embedding of $G$ in $\pp$ so that $H$ is $\Pi$-tidy.     

We may write $r_i=\cc{v_i,\dots,x_e,e,y_e,\dots,v_{i+1}}$ and\wording{, by symmetry,} we may assume $f$ is in  $$r_{i+5}\cup r_{i+6}=\cc{v_{i+5},r_{i+5},\dots,x_f,f,y_f,\dots,r_{i+6},v_{i+7}}\,.$$   If $f\in r_{i+5}$, then let $J_{e,f}=\cl(Q_i)$ and $Q=Q_i$, while if $f\in r_{i+6}$, then let $J_{e,f}=\cl(Q_i)\cup \cl(Q_{i+1})$ and $Q=\bQ_{i+1}$.  The two $H$-spokes contained in $Q$ are $s_i$ and $s_{e,f}$, which is either $s_{i+1}$ or $s_{i+2}$.

\begin{claim} There are not totally disjoint $s_i s_{e,f}$-paths in $J_{e,f}-e$.\end{claim}

\begin{proof}  Because $H$ is $\Pi$-tidy, $\Pi[J_{e,f}]$ is contained in the closed disc bounded by $\Pi[Q]$.  Therefore, one of a pair of totally disjoint $s_is_{e,f}$-paths in $J_{e,f}$ would be disjoint from $r_{i+5}\,r_{i+6}$ and it, together with a subpath of $Q-r_{i+5}\,r_{i+6}$ yields the contradiction that $e$ is $H$-green.  \end{proof}

Let $w_e$ be a cut-vertex in $J_{e,f}-e$ separating $s_i$ from \wording{$s_{e,f}$}.  Then $J_{e,f}-e$ has a separation $(H_e,K_e)$ with $s_i\subseteq H_e$, the other $H$-spoke $s_{e,f}$ contained in $Q$ is contained in $K_e$, and $H_e\cap K_e=\iso{w_e}$.   Clearly, $w_e\in r_{i+5}\,r_{i+6}$.  

There is also a separation $(H_f,K_f)$ of  $J_{e,f}-f$, so that $H_f\cap K_f$ is a single vertex $w_f$, $s_i\subseteq H_f$, and $s_{e,f}\subseteq K_f$.  For $x\in \{e,f\}$, there is a face $F_x$ of $\Pi[J_{e,f}]$ incident with both $x$ and $w_x$.  If $F_e=F_f$, then any vertex of $r_i\,r_{i+1}$ in the boundary cycle $C$ of $F_e$ may be selected as $w_f$.  Similarly, $w_e$ may be any vertex of $r_{i+5}\,r_{i+6}$ that is in $C$.  We choose $w_e$ and $w_f$ so that they are in different components of $C-\{e,f\}$.  Thus, whether $F_e=F_f$ or not, the cycle $Q$ has the form $\cc{w_e,\dots,e,\dots,w_f,\dots,f,\dots}$.  In particular, $e$ and $w_e$ are in the same component of $Q-\{w_f,f\}$, while $f$ and $w_f$ are in the same component of $Q-\{w_e,e\}$.  By interchanging the roles of $e$ and $f$ and exchanging the labels of $v_j$ and $v_{j+5}$, for $j=0,1,2,3,4$, we may assume $Q$ has the form 
$$
\cc{w_e,\dots,v_{i+5},s_i,v_i\dots,e,\dots,w_f,\dots,s_{e,f},\dots,f,\dots}\,.
$$
For technical reasons, we choose $w_e$ as close as possible to $f$ in $r_{i+5}\,r_{i+6}$ and $w_f$ as close as possible to $e$ in $r_{i+5}\,r_{i+6}$, while respecting the ordering that was just described of these four elements of $Q$.

Set $N=K_e\cap H_f$.  Then $J_{e,f}-\{e,f\}=H_e\cup N\cup K_f$, $H_e\cap N=\iso{w_e}$, and $K_f\cap N=\iso{w_f}$. See Figure \ref{fg:separating}.

\begin{figure}[!ht]
\begin{center}
\scalebox{.9}{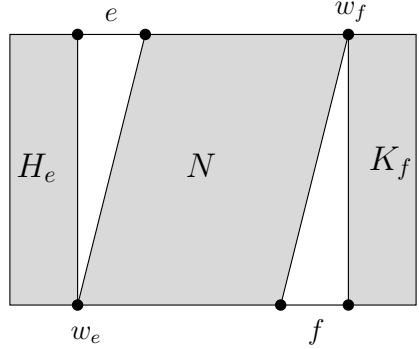}
\end{center}
\caption{The locations of $e$, $f$, $w_e$, $w_f$, $H_e$, $N$, and $K_f$.}\label{fg:separating}
\end{figure}

\begin{claim}  $N$ does not have disjoint paths both with ends in the two components of $N\cap R$.  \end{claim}

\begin{proof}  Such paths, together with the $H$-rim and the $H$-spokes $s_{i-1}$ and $s_{i+3}$, would show $e$ and $f$ are $R$-separated.  \end{proof}

Let $w$ be a cut-vertex in $N$ separating the two components of $N\cap R$, and let $(N_i,N_{i+5})$ be a separation of $N$ so that, for $j\in\{i,i+5\}$, $N_j\cap R$ is contained in $r_j\cup r_{j+1}$ and $N_i\cap N_{i+5}=\iso{w}$.   We proceed to describe a new 2-representative embedding of $G$ in $\pp$ that shows that $G$ has a 1-drawing.

Let $G'$ be the subgraph of $G$ obtained by deleting all the vertices and edges of $N$ that are not in $N\cap R$.    There is a face of $\Pi[G']$ contained in $\Mob$ and incident with both $e$ and $f$.  

\begin{claim}\label{cl:noSpan} No global $H$-bridge has a vertex in \wording{$\oo{N_i\cap R}\cup \oo{N_{i+5}\cap R}$} in its span.    \end{claim}

\begin{proof}  For sake of definiteness, suppose some vertex of \wording{$\oo{N_i\cap R}$} is in the span of the global $H$-bridge $B$.  If the $H$-node $v_{e,f}$ in $r_i\,r_{i+1}$ incident with $s_{e,f}$ is in the interior of the span of $B$, then the cycle bounding $F_f$ is $H$-yellow, contradicting the fact that $f$ is not $H$-yellow.  Letting $z$ be the vertex of $N_i$ nearest $e$ in $r_i\,r_{i+1}$, we conclude that $B$ has an attachment in $\oc{z,r_i\,r_{i+1},v_{e,f}}$, and $B$ does not span any edge of $r_{i+2}$. 

By Theorem \ref{th:globalBridges}, $B$ is either a 2-, 2.5-, or 3-jump.  It follows from the preceding paragraph that $e$ is contained in the span of $B$, yielding the contradiction that $e$ is $H$-green. \end{proof}

We can now easily complete the proof of the lemma.  By Claim \ref{cl:noSpan}, we can separately embed $N_i$ and $N_{i+5}$ in the face outside of $\Mob$.  As no global $H$-bridge can attach on both paths in $R-\{e,f\}$ without making at least one of $e$ and $f$ $H$-green, we can join the two copies of $w$ together to obtain a representativity 2 embedding $\Pi'$ of $G$ in $\pp$ having a non-contractible simple closed curve meeting $\Pi'[G]$ only in the interiors of $e$ and $f$.  This implies that $G$ has a 1-drawing, as required.  
\end{cproof}

We further investigate the detailed structure of $H$-rim edges.

}\begin{lemma}\label{lm:3jumpRed}  Let $G\in\m2$ and $\hvfg$, with $H$ tidy.  If $v_iv_{i+3}$ is a global $H$-bridge, then, for $j\in\{i-1,i+3\}$ there is an edge $e_j\in r_j$ that is neither $H$-yellow nor $H$-green.  \end{lemma}\printFullDetails{

\begin{cproof}  The two sides are symmetric, so it suffices to prove the existence of $e_{i-1}$.  Lemmas \ref{lm:tidyBODhyper} and \ref{lm:globalJumps} (\ref{it:atMostOne3jump}) imply that $\bQ_{i+1}$ has BOD.  Let $D$ be a 1-drawing of $G-\oo{s_{i+1}}$.    Lemma \ref{lm:BODcrossed} implies $\bQ_{i+1}$ is crossed in $D$.

However, the cycle $C$ consisting of $v_iv_{i+3}$ and the path it spans is \wording{$H_1$-close}, for \wording{$H_1=R\cup s_{i-1}\cup s_i\cup s_{i+3}$}.  Therefore, Lemmas \ref{lm:closeIsPrebox} and \ref{lm:preboxClean} imply that $C$ is not crossed in $D$.  We conclude from the nature of 1-drawings of $V_8$ that $r_{i-1}$ crosses $r_{i+5}\cup r_{i+6}$; let $e$ be the edge in $r_{i-1}$ that is crossed in $D$.  

Suppose, by way of contradiction, that there is a global $H$-bridge $B$ spanning $e$.  Theorem \ref{th:globalBridges} implies $B$ is either a 2-, 2.5- or 3-jump, while Theorem \ref{th:twoGreenCycles} implies $B$ does not span any edge of $r_i$ (such an edge is already spanned by $v_iv_{i+3}$).  Lemma \ref{lm:globalJumps} (\ref{it:globalNoHnode}) implies $v_i$ is not an attachment of $B$, so $B$ must be a 2.5-jump with one end in $\oo{r_{i-1}}$, contradicting Lemma \ref{lm:globalJumps} (\ref{it:2and3jumps}).  Thus, $e$ is not spanned by a global $H$-bridge.

It follows that, if $e$ is in an $H$-green cycle $C'$, then $C'\subseteq \cl(Q_{i-1})$.  But such a $C'$ is \wording{$H_2$-close}, for \wording{$H_2=R\cup s_i\cup s_{i+2}\cup s_{i+3}$}\wordingrem{(${}\topol V_6$ removed)}.  By Lemmas \ref{lm:closeIsPrebox} and \ref{lm:preboxClean}, $C'$ is not crossed in any 1-drawing of $G-\oo{s_{i+1}}$.  This contradicts the fact that $e$ is crossed in $D$.  We conclude that $e $ is not $H$-green.

So now we suppose $e$ is in the $H$-yellow cycle $C'$ and that $C''$ is a witnessing $H$-green cycle.  Then $C'\subseteq \cl(Q_{i-1})$ and $C''$ contains a global $H$-bridge $B$ that spans an edge in $r_{i+4}$.  This implies $B\ne v_iv_{i+3}$, so Lemma \ref{lm:globalJumps} (\ref{it:atMostOne3jump}) shows that $B$ is not a 3-jump.  

Moreover, (\ref{it:2and3jumps}) of the same lemma shows $B$ cannot have an attachment in $\co{v_{i+3},r_{i+3},v_{i+4}}$, while (\ref{it:spanDiffHyper}) shows $B$ cannot have $v_{i+7}$ as an attachment.  Therefore, $B$ is a 2- or 2.5-jump $v_{i+4}w$, with $w\in \co{v_{i+6},r_{i+6},v_{i+7}}$.  

The cycle $(R-\oo{C''\cap R})\cup B$, together with the $H$-spokes  $s_{i-1}$, $s_{i+2}$, and $s_{i+3}$ is a \wording{subdivision $H_3$ of} $V_6$ for which $C'$ is \wording{$H_3$-close}, showing that $e$ is not crossed in any 1-drawing of $G-\oo{s_{i+1}}$.  This contradicts the fact that $e$ is crossed in $D$ and, therefore, $e$ is not $H$-yellow.
\end{cproof}

The proof of Theorem \ref{th:rimColoured} will also depend on the following new concepts.

}\begin{definition}\label{df:scope}  Let $G$ be a graph and let $\hvfg$, with $H$ tidy.  Let $\Pi$ be an embedding of $G$ in $\pp$ so that $H$ is $\Pi$-tidy and has the standard labelling relative to $\gamma$. For $i\in \{0,1,2,\dots,9\}$:
\begin{enumerate} 
\item  $\topleft{P}_i=r_{i-2}\,r_{i-1}$\index{$\topleft{P}_i$}, $\botleft{P}{}_i=r_{i+3}\,r_{i+4}$\index{$\botleft{P}_i$}, $\topright{P}_i=r_{i+1}\,r_{i+2}$\index{$\topright{P}_i$}, and $\botright{P}{}_i=r_{i+6}\,r_{i+7}$\index{$\botright{P}_i$}.
\item the {\em spines\/}\index{spine} $\leftspine_i $\index{$\leftspine_i$} and $\rightspine i$\index{$\rightspine i$} of $Q_i$ consist of the paths $\topleft{P}_i\cup s_i\cup \botleft{P}{}_i$ and $\topright{P}_i\cup s_{i+1}\cup \botright{P}{}_i$, respectively  (see Figure \ref{fg:spines}); 
\item the {\em scope\/}\index{scope} $K_i$ of $Q_i$ consists of ${\cl({Q_i})}\cupspace\cup\cupspace {\leftspine_i }\cupspace\cup\cupspace{\rightspine i}\cupspace\cup\cupspace{\mathcal B_i}$, where $\mathcal B_i$ consists of all global $H$-bridges having both attachments either in $\topleft{P}_i \cup \topright{P}_i$ or in $\botleft{P}{} _i\cup \botright{P}{}_i$; and
\item the {\em complement\/}\index{complement} $K^\natural_i$\index{$K^\natural$} of $K_i$ is obtained from $M_{Q_i}$ by deleting the edges (but not their incident vertices) that comprise the $H$-bridges in $\mathcal B_i$.
\item  The two vertices $v_{i-2}$ and $v_{i+3}$ are the {\em  trivial\/} $\leftspine_i\,\rightspine i$-paths\index{trivial $\leftspine_i\,\rightspine i$-path} in $K_i$.  Any other $\leftspine_i\,\rightspine i$-path in $K_i$ is {\em non-trivial\/}\index{non-trivial $\leftspine_i\,\rightspine i$-path}. 
\end{enumerate}
\end{definition}\printFullDetails{

 We note that $\leftspine_i \cap \rightspine i$ is equal to $\iso{\{v_{i-2},v_{i+3}\}}$.  For our purposes, these are not ``useful"   $\leftspine_i\,\rightspine i$-paths.
\begin{figure}[!ht]
\begin{center}
\scalebox{.9}{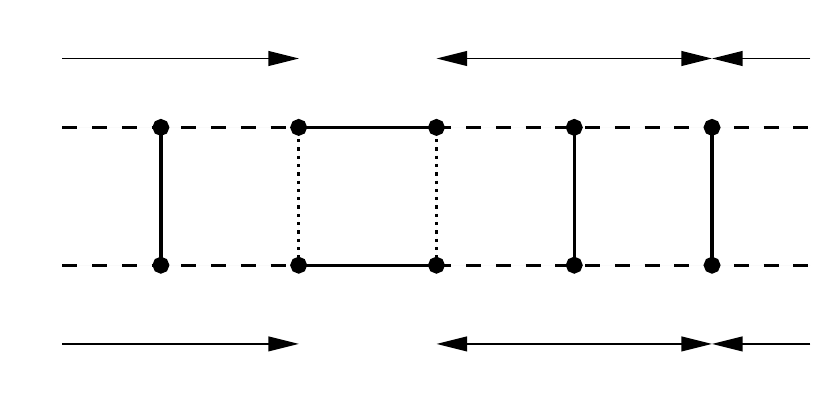}
\end{center}
\caption[The spine and its constituent paths.]{The paths with small dashes are $\topleft{P}_1$, $\botleft{P}_1$, $\topright{P}_1$, and $\botright{P}_1$.  The spine $\leftspine_1$ is the path  $r_9\,r_0\,s_1\,r_5\,r_4$, while $\rightspine 1$  is $r_3\,r_2\,s_2\,r_7\,r_8$.}\label{fg:spines}
\end{figure}

We observe that, for each $i\in\{0,1,2,3,4\}$, $G=K_i\cup K^\natural_i$.  

 The following lemma plays an important role in \wording{the rest of this section}
 .  

}\begin{lemma}\label{lm:Qspans}  Let $G\in\m2$ and $\hvfg$, with $H$ tidy.  Let $e$ be an edge of $R$ and let $i$ be such that $e\in r_i$.  Then $G-e$ has a subdivision of $V_6$ if and only if there are disjoint non-trivial $\leftspine_i\,  \rightspine i$-paths in $K_i-e$.
\end{lemma}\printFullDetails{

There is some subtlety here\dragominor{; 2-criticality is important}.  Suppose we have a subdivision $H$ of $V_{10}$ embedded in $\pp$ with representativity 2 so that all the $H$-spokes are in $\Mob$.  Give $H$ the usual labelling relative to $\gamma$.  Now delete $\oo{r_1}$ and $\oo{r_6}$, and then add the 2.5-jump $av_2$ and the 3-jump $v_6v_9$.  Then there are disjoint non-trivial $\leftspine_1\rightspine 1$-paths in the union $H'$ of $(H-\oo{r_1})-\oo{r_{6}}$ and the two jumps, but $H'$ is planar.  

We shall need the following.

}\begin{lemma}\label{lm:oneInClosure}  Let $G\in\m2$ and $\hvfg$, with $H$ tidy.  Let $e$ be an edge of $R$ and let $i$ be such that $e\in r_i$.   If there are disjoint non-trivial $\leftspine_i\, \rightspine i$-paths in $K_i-e$, then there are two such paths so that at least one of them is contained in $\cl(Q_i)$ and the other contains at most one global $H$-bridge.  \end{lemma}\printFullDetails{

\def\pathone{P_1}
\def\pathtwo{P_2}

In the proof, we consider many possibilities for the two disjoint $\leftspine_i\,\rightspine i$-paths.  For a given $i$, some possibilities might not occur because of limitations imposed by $\Pi$.  In principle, for $i=2$, all of the considered possibilities can occur, while for $i=4$, several of the considered possibilities cannot occur.

\bigskip
\begin{cproof}
Let ${\pathone}$ and ${\pathtwo}$ be the hypothesized disjoint paths.   

\begin{claim}\label{sc:disjt-ri+5}  If there is a $\leftspine_i\, \rightspine i$-path in $K_i-e$ disjoint from $r_{i+5}$, then there are disjoint $\leftspine_i\,  \rightspine i$-paths so that one of them is contained in $\cl(Q_i)$ and the other contains at most one global $H$-bridge. \end{claim}

\begin{proof}  Suppose that $P$ and $r_{i+5}$ are disjoint paths.  If $P$ contains two (or more) global $H$-bridges, then they must be 2.5-jumps having an end in $\oo{r_i}$.  By Theorem \ref{th:twoGreenCycles}, they must be of the form $v_{i-2}w_1$ and $w_2v_{i+3}$, with $w_1$ being no further from $v_i$ in $r_i$ than $w_2$ is.   By symmetry, we may assume $e$ is not in $\cc{v_i,r_i,w_1}$.  Now $\cc{v_i,r_i,w_1}\rbsp (P-v_{i-2})$ and $r_{i+5}$ are the desired disjoint $\leftspine_i\,  \rightspine i$-paths in $K_i-e$.  \end{proof}

Thus, we may assume both ${\pathone}$ and ${\pathtwo}$ intersect $r_{i+5}$.   

\begin{claim}\label{sc:notTwoGlobal}   If either of ${\pathone}$ and ${\pathtwo}$ contains two global $H$-bridges, then there are disjoint $\leftspine_i\,  \rightspine i$-paths in $K_i-e$ so that one of them is contained in $\cl(Q_i)$ and the other contains at most one global $H$-bridge.   \end{claim}

\begin{proof}  We may assume ${\pathone}$ contains two global $H$-bridges $B_1$ and $B_2$.  Both $B_1$ and $B_2$ are 2.5-jumps.  Both have ends in $\oo{r_i}\cup \oo{r_{i+5}}$.  By Lemma \ref{lm:globalJumps} (\ref{it:oppositeRims}), they both have an end in the same one of $\oo{r_i}$ and $\oo{r_{i+5}}$.     We choose the labelling so that $(B_1,B_2)$ is either $(v_{i-2}w_1,w_2v_{i+3})$ or $(v_{i+3}w_1,w_2v_{i+8})$.  We treat these cases separately.

Suppose $(B_1,B_2)=(v_{i-2}w_1,w_2v_{i+3})$.  Assume first that $e\notin \cc{w_1,r_1,w_2}$.  Then  $B_1\cup \cc{w_1,r_1,w_2}\cup B_2$ is disjoint from $r_{i+5}$, and we are done by Claim \ref{sc:disjt-ri+5}.  Therefore, we may assume $e\in \cc{w_1,r_1,w_2}$.  

In this case, ${\pathone}$ consists of $B_1$, $B_2$, and a $w_1w_2$-path $P'_1$ contained in $\cl(Q_i)$.   We know that $P'_1$ contains a vertex in $r_{i+5}$. \wording{ Lemma \ref{lm:globalJumps} (\ref{it:oppositeRims})}  implies that ${\pathtwo}$ consists of a global $H$-bridge with no vertex in $\oo{r_{i+5}}$.  Therefore, we may choose $\cc{v_i,r_i,w_1}\cup P'_1\cup{}$\wording{$ \cc{w_2,r_i,v_{i+1}}$} and ${\pathtwo}$ as the desired paths.   

We conclude the proof of this claim by considering the case $(B_1,B_2)=(v_{i+3}w_1,w_2v_{i+8})$.  \wording{First, by way of contradiction suppose $P_2$ is not contained in $\cl(Q_i)$.} Lemma \ref{lm:globalJumps} (\ref{it:oppositeRims}) implies that ${\pathtwo}$ consists of a global $H$-bridge having both ends in $\topleft{P}_i \cup \topright{P}_i$.   But then ${\pathtwo}$ is disjoint from $r_{i+5}$ and we are done by Claim \ref{sc:disjt-ri+5}.  \wording{Thus, we may assume} ${\pathtwo}\subseteq \cl(Q_i)$.  

 If ${\pathtwo}$ is disjoint from either $\co{v_{i+5},r_{i+5},w_1}$ or $\oc{w_2,r_{i+5},v_{i+6}}$, then we may replace either $B_1$ with the former or $B_2$ with the latter, and we are done again.  Otherwise, there is a $\co{v_{i+5},r_{i+5},w_1}\oc{w_2,r_{i+5},v_{i+6}}$-path $P'_2$ contained in ${\pathtwo}$ that is $r_{i+5}$-avoiding; let its ends be $w_3\in \co{v_{i+5},r_{i+5},w_1}$ and $w_4\in \oc{w_2,r_{i+5},v_{i+6}}$. 

If $P'_2$ is $r_i$-avoiding, then $P'_2\cup \cc{w_3,r_{i+5},w_4}$ is an $H$-green cycle. Since $B_1$ together with the subpath of $R$ it spans is also $H$-green, the edge of $\cc{v_{i+5},r_{i+5},w_1}$ incident with $w_1$ is in two $H$-green cycles, contradicting Theorem \ref{th:twoGreenCycles}.  

Therefore, $P'_2$ is not $r_i$-avoiding and so contains two subpaths, one being a $w_3r_i$-path $P'_2{}^1$ and the other being an $r_iw_4$-path $P'_2{}^2$.  For $k=1,2$, let $u_k$ be the vertex of $P'_2{}^k$ in $r_i$.  If $e\in \cc{v_i,r_i,u_1}$, then the paths $\cc{v_{i+5},r_{i+5},w_3}\cup P'_2{}^1\cup \cc{u_1,r_i,v_{i+1}}$ and $B_1\cup \cc{w_1,r_{i+5},v_{i+6}}$ constitute the required disjoint paths.  Otherwise, $\cc{v_i,r_i,u_1}\cup \cc{u_1,P'_2,w_4,r_{i+5},v_{i+6}}$ and $\cc{v_{i+5},r_{i+5},w_2}\cup B_2$ constitute the required disjoint paths. \end{proof}

To complete the proof of the lemma, we may now assume that, for each $j=1,2$, $P_j$ contains a unique global $H$-bridge $B_j$.

We first suppose, by way of contradiction, that both $B_1$ and $B_2$ have an end in $\oo{r_i}\cup\oo{r_{i+5}}$.  Lemma \ref{lm:globalJumps} (\ref{it:oppositeRims}) shows that such ends are in the same one of $\oo{r_i}$ and $\oo{r_{i+5}}$; let $i'\in \{i,i+5\}$ be such that, for $j=1,2$, $B_j$ has an end $w_j\in \oo{r_{i'}}$.  We may assume $B_1=v_{i'-2}w_1$ and $B_2=w_2v_{i'+3}$.  

Theorem \ref{th:twoGreenCycles} implies $w_1$ is closer to $v_{i'}$ in $r_{i'}$ than $w_2$ is.
The paths $P_1-v_{i'-2}$ and $P_2-v_{i'+3}$ are both in $\cl(Q_i)$; the former is a $w_1s_{i+1}$-path, with end $x_1\in s_{i+1}$, and the latter is a $w_2s_i$-path, with end $x_2\in s_i$.  

Recall that $\Pi[\cl(Q_i)]$ is a planar embedding of $\cl(Q_i)$ with $Q_i$ bounding a face.  The vertices $w_1,w_2,x_1,$ $x_2$ occur in this cyclic order in $Q_i$, so the disjoint paths $P_1-v_{i'-2}$ and $P_2-v_{i'+3}$ must cross in $\Pi[\cl(Q_i)]$, a contradiction.  Therefore, at most one of $B_1$ and $B_2$ has an end in $\oo{r_i}\cup \oo{r_{i+5}}$\dragominor{, while the other is equal to the path among $P_1$ and $P_2$ that contains it.}

We may choose the labelling so that $P_2$ consists only of $B_2$.  Theorem \ref{th:twoGreenCycles} implies no edge of $r_i\cup r_{i+5}$ is spanned by both $B_1$ and $B_2$; since $B_2$ spans one of $r_i$ and $r_{i+5}$ completely, one of $B_1$ and $B_2$ spans edges in $r_i$ and the other spans edges in $r_{i+5}$.  If either $B_j$ spans all of $r_i$, then, as it is disjoint from $r_{i+5}$, we are done by Claim \ref{sc:disjt-ri+5}.  In particular, $B_2$ spans $r_{i+5}$,  edges spanned by $B_1$ are in $r_i$, and $B_1$ does not span all of $r_i$.

Therefore, $B_1$ is a 2.5-jump with one end $w_1$ in $\oo{r_i}$.  We may assume the other end of $B_1$ is $v_{i+3}$.   If $e\notin \cc{v_i,r_i,w_1}$, then $\cc{v_i,r_i,w_1}\cup B_1$ is disjoint from $r_{i+5}$, and we are done by Claim \ref{sc:disjt-ri+5}.  If $e\in \cc{v_i,r_i,w_1}$, then   
 $(P_1-v_{i+3})\,\cc{w_1,r_i,v_{i+1}}$ and $P_2$ are the desired paths.
\end{cproof}

\begin{cproofof}{Lemma \ref{lm:Qspans}} The following claim settles one direction.

\begin{claim}  If there are not disjoint non-trivial $\leftspine_i\,  \rightspine i$-paths in $K_i-e$, then $G-e$ is planar.  \end{claim}

\begin{proof}  \minor{For this proof, we need to apply Menger's Theorem; in order to do so, we treat the copies of $v_{i-2}$ and $v_{i+3}$ in $\leftspine_i$ as different from their copies in $\rightspine i$.}   Let $u$ be a cut-vertex of $K_i-e$ separating $\leftspine_i $ and $\rightspine i$.    Let $\topleft{K}_i$ be the union of the $\iso{u}$-bridges in $K_i-e$ that have an edge in $\leftspine_i $ and let $\topright{K}_i$ be the union of the remaining $\iso{u}$-bridges in $K_i-e$.  Then $K_i-e=\topleft{K}_i \cup \topright{K}_ i$ and $\topleft{K}_i \cap \topright{K}_ i$ is just $\iso{u}$.  

Let $\Pi$ be an embedding of $G$ in $\pp$ so that $H$ is $\Pi$-tidy.   Since $r_{i+5}\subseteq K_i-e$, $u\in r_{i+5}$.  Because $K_i-\{e,u\}$ is not connected, there is a non-contractible, simple closed curve in $\pp$ that meets $\Pi[G-e]$ only at $u$\wording{. Thus, there is no non-contractible cycle in $G-\{e,u\}$,} showing that $G-e$ is planar.  \end{proof}

For the converse,  Lemma \ref{lm:oneInClosure} shows there are disjoint non-trivial $\leftspine_i\,  \rightspine i$-paths ${\pathone}$ and ${\pathtwo}$ in $K_i-e$ so that ${\pathone}\subseteq \cl(Q_i)$.  In particular, ${\pathone}$ is an $s_is_{i+1}$-path.   It follows from the embedding $\Pi[K_i]$ that ${\pathtwo}$ is disjoint from either $r_i$ or $r_{i+5}$.  

In every case, we find our $V_6$ by adding three spokes to the cycle contained in  {\Large (}$R-(\oo{r_i}\cup\oo{r_{i+5}})${\Large)}${}\cup {\pathone}\cup {\pathtwo}\cup s_i\cup s_{i+1}$ and containing {\Large (}$R-(\oo{r_i}\cup\oo{r_{i+5}})${\Large)}${}\cup {\pathone}\cup {\pathtwo}$.

If ${\pathtwo}$ contains no global $H$-bridges, then $s_{i+2}$, $s_{i+3}$, and $s_{i+4}$ may be chosen as the spokes.

If ${\pathtwo}$ contains precisely one global $H$-bridge $B_2$, then $B_2$ is one of: 
\begin{enumerate}
\item $v_{i-2}v_{i+1}$ (symmetrically, $v_iv_{i+3}$); \item $v_{i-1}v_{i+2}$; \item $v_{i-2}w$ and $w$ is in $\oo{r_i}$ (symmetrically, $wv_{i+3}$); \item $wv_{i+1}$ and $w$ is in $\oo{r_{i-2}}$ (symmetrically, $v_iw$, with $w\in \oo{r_{i+2}}$); \item $v_{i-1}w$ and $w$ is in $\oo{r_{i+1}}$ (symmetrically, $wv_{i+2}$, with $w\in \oo{r_{i-1}}$);  \item $v_{i-1}v_{i+1}$ (symmetrically, $v_iv_{i+2}$); \item and the comparable jumps with ends in $r_{i+3}\,r_{i+4}\,r_{i+5}\,r_{i+6}\,r_{i+7}$.  
\end{enumerate}
We choose, in all cases, $s_{i-2}$ and $s_{i+2}$ as two of the spokes, with third spoke (taking the cases in the same order):
\begin{enumerate}
\item the $P_1P_2$-subpath of $s_{i+1}$ (symmetrically, the $P_1P_2$-subpath of $s_i$);
\item  $s_{i-1}$;
\item the $P_1P_2$-subpath of $s_{i+1}$ (symmetrically, the $P_1P_2$-subpath of $s_{i}$);
\item the $P_1P_2$-subpath of $s_{i+1}$ (symmetrically, the $P_1P_2$-subpath of $s_i$);
\item  $s_{i-1}$ (symmetrically, the same);
\item  $s_{i-1}$ (symmetrically, the same); and
\item these cases are symmetric to the preceding ones.
\end{enumerate}
In every case, we have found a $V_6$ in $G-e$, as required.
\end{cproofof}

We conclude this section by proving that every rim edge is either red, $H$-green, or $H$-yellow.

\begin{cproofof}{Theorem \ref{th:rimColoured}}  Let $e$ be an edge in the $H$-rim.  There is an $i$ so that $e\in r_i$.  By Lemma \ref{lm:Qspans}, $G$ is red if and only if there are no disjoint non-trivial $\leftspine_i\,  \rightspine i$-paths in $K_i-e$.   

Now suppose there are disjoint non-trivial $\leftspine_i\,  \rightspine i$-paths $P_1$ and $P_2$ in $K_i-e$.  By Lemma \ref{lm:oneInClosure}, we may assume $P_1$ is contained in $\cl(Q_i)$, while $P_2$ contains at most one global $H$-bridge.  If $P_1$ is disjoint from $r_{i+5}$, then every maximal $r_i$-avoiding subpath of $P_1$ is contained in an $H$-green cycle.  The edge $e$ is in one of these $H$-green cycles, as required.

Thus, we may assume $P_1$ contains a vertex in $r_{i+5}$.  If $P_2\subseteq \cl(Q_i)$, then the planar embedding of $\cl(Q_i)$ shows $P_2$ is disjoint from $r_{i+5}$ and the preceding paragraph, with $P_2$ in place of $P_1$, shows $e$ is $H$-green.  Consequently, we may further assume $P_2$ contains a global $H$-bridge $B_2$.

{\bf Case 1:}  {\em $B_2$ has its ends in $\topleft{P}_i\cup r_i\cup \topright{P}_i$.}

In this case, if $e$ is spanned by $B_2$, then there is an $H$-green cycle containing $e$, namely the cycle consisting of $B_2$ and the subpath of $R$ that it spans.  The only other possibility in this case is that $B_2$ is a 2.5-jump with an end $w_2$ in $\oo{r_i}$ and that $e$ is in the one of $\cc{v_i,r_i,w_2}$ and $\cc{w_2,r_i,v_{i+1}}$ not spanned by $B_2$.  For the sake of definiteness, we suppose $B_2=v_{i-2}w_2$ and that $e$ is in $\cc{w_2,r_i,v_{i+1}}$.  

Since $P_1\subseteq \cl(Q_i)$, we see that, in this case, $P_2$ is disjoint from $r_{i+5}$ and, therefore, we may assume $P_1=r_{i+5}$.  We replace $P_2$ with $\cc{v_i,r_i,w_2}\,(P_2-v_{i-2})$ so that there are disjoint $\leftspine_i\, \rightspine i$-paths contained in $\cl(Q_i)-e$; a situation resolved in the paragraph preceding this case.

{\bf Case 2:}  {\em $B_2$ has its ends in $\botleft{P}_i\cup r_{i+5}\cup \botright{P}_i$.}

In this case, either $P_2$ is $B_2$ or, up to symmetry, $B_2$ is a 2.5-jump $wv_{i+8}$, with $w\in \oo{r_{i+5}}$, and $P_2$ is $\cc{v_{i+5},r_{i+5},w}\cup B_2$.
On the other hand, $P_1$ is an $s_is_{i+1}$-path in $\cl(Q_i)$ intersecting $r_{i+5}$.  

Let $x$ be the first vertex in $r_{i+5}$ as we traverse $P_1$ from $s_i$ and let $P'_1$ be the $s_ix$-subpath of $P_1$.  We note that $P_2$ prevents $x$ from being in $\cc{v_{i+5},r_{i+5},w}$, so $x\in \oc{w,r_{i+5},v_{i+6}}$.  Let $y$ be the end of $P'_1$ in $s_i$.  The cycle $P'_1\,\cc{x,r_{i+5},v_{i+6}}\,s_{i+1}\,r_{i}\,\cc{v_i,s_i,y}$ is $H$-yellow, as witnessed by the $H$-green cycle containing $B_2$.  Therefore, $e$ is either $H$-yellow or $H$-green (in Definition \ref{df:yellow}, an $H$-yellow edge is not $H$-green).
\end{cproofof}

}

\chapter{Existence of a red edge and its structure}\printFullDetails{

In this section, we prove that if $G$ is a 3-connected, 2-crossing-critical graph containing a tidy subdivision $H$ of $V_{10}$, then some edge of the $H$-rim is red.  Furthermore, we prove that each red edge $e$ has an associated special cycle we call $\Delta_e$. These ``deltas" will be the glue that hold successive tiles together \dragominor{and so form a vital element of the tile structure.}

The argument for proving the existence of a red edge depends on whether or not there is a global $H$-bridge that is either a 2.5- or 3-jump.  Once these cases are disposed of, matters become simple.  However, with the knowledge of the $\Delta$'s, it turns out we can show that there is no 3-jump.  This will be our first aim and so, since we need the $\Delta$'s to complete the elimination of 3-jumps, we shall begin by determining the structure of the $\Delta$ of a red edge.

}\begin{theorem}\label{th:redHasDelta}  Let $G\in\m2$, $\hvfg$, with $H$ tidy.   Let $e=uw$ be a red edge of $G$ and let $i\in \{0,1,2,\dots,9\}$ be such that $e\in r_i$.  
Then there exists a vertex $x_e\in \cc{r_{i+5}}$\index{$x_e$}\index{$\Delta_e$!$x_e$} and internally disjoint $x_eu$- and $x_ew$-paths $A_u$\index{$A_u$}\index{$\Delta_e$!$A_u$} and $A_w$\index{$A_w$}\index{$\Delta_e$!$A_w$}, respectively, in $\cl(Q_i)$ so that, letting $\Delta_e=(A_u\cup A_w)+e$\index{$\Delta_e$}:
\begin{enumerate}
\item\label{it:twoBridges} there are at most two $\Delta_e$-bridges in $G$;
\item\label{it:digon} there is a $\Delta_e$-bridge $M_{\Delta_e}$\index{$M_{\Delta_e}$}\index{$\Delta_e$!$x_e$} so that $H\subseteq M_{\Delta_e}\cup \Delta_e$, while the other $\Delta_e$-bridge, if it exists, is one of two edges in a digon incident with $x_e$;  and 
\item\label{it:deltaDetail} when there are two $\Delta_e$-bridges, let $u^e$\index{$u_e$}\index{$\Delta_e$!$u_e$} and $w^e$\index{$w_e$}\index{$\Delta_e$!$w_e$} be the attachments of the one-edge $\Delta_e$-bridge, labelled so that $u^e\in A_u$ and $w^e\in A_w$;  otherwise let $u^e=w^e=x_e$.  In both cases, $\Delta_e-e$ contains unique $uu^e$- and $ww^e$-paths $P_u$\index{$P_u$}\index{$\Delta_e$!$P_u$} and $P_w$\index{$P_w$}\index{$\Delta_e$!$P_w$}, each containing at most one $H$-rim edge, which, if it exists, is in the span of a global $H$-bridge and, therefore, is $H$-green. 
\end{enumerate}
\end{theorem}\printFullDetails{

\begin{cproof}  Let $\Pi$ be an embedding of $G$ in $\pp$ for which $H$ is $\Pi$-tidy.  We may assume $r_i=\cb{v_i,r_i,u,e,w,r_i,}$ $\bc{v_{i+1}}$.   Lemma \ref{lm:Qspans} implies $K_i-e$ has a cut vertex $x_e$ separating \wording{$\leftspine_i$ and $\rightspine i$ (again adopting the perspective that $v_{i-2}$ and $v_{i+3}$ are split into different copies in $\leftspine_i$ and $\rightspine i$)}.  As $r_{i+5}$ is a \wording{$\leftspine_i\,\rightspine i$}-path in $K_i-e$,  $x_e$ is in $r_{i+5}$.  
 
Because $\cl(Q_i)$ is 2-connected and $\Pi[\cl(Q_i)]$ has $Q_i$ bounding the exterior face, there is a face $F_e$ of $G$ in $\pp$ that is in the interior of $Q_i$ and incident with both $e$ and $x_e$.  As $G$ is 3-connected and non-planar, $F_e$ is bounded by a cycle $C_e$ and $C_e-e$ consists of a $ux_e$-path $A_u$ and a $wx_e$-path $A_w$.

For (\ref{it:twoBridges}) and (\ref{it:digon}), we begin by noticing that $C_e\subseteq \cl(Q_i)$.  Thus, there is a $C_e$-bridge $M_{C_e}$ containing the three $H$-spokes not in $Q_i$.  

\begin{claim}\label{cl:connected}  Each of  $C_e\cap s_i$, $C_e\cap s_{i+1}$, and $C_e\cap r_i$ is connected.  Either $C_e\cap r_{i+5}$ is connected or it has two components that are joined by an edge $e'$ of $r_{i+5}$ and $C_e$ has an edge parallel to $e'$.  In particular, each of  $s_i$, $s_{i+1}$, $r_i$, and $r_{i+5}-e'$ is contained in $C_e\cup M_{C_e}$. \end{claim}

\begin{proof} 
\wording{Suppose by way of contradiction that} $C_e\cap s_i$ is not connected\wording{.  As $C_e$ bounds a face of $\Pi$, it follows that} there is a $Q_i$-local $H$-bridge having all its attachments in $s_i$, contradicting Lemma \ref{lm:noSpokeOnlyBridge}.  Thus, $C_e\cap s_i$ is connected.  

It follows that any part of $s_i$ that is not in $C_e$ is in the same $C_e$-bridge as either $r_{i-1}$ or $r_{i+4}$.  That is, it is in $M_{C_e}$, and therefore, $s_i\subseteq M_{C_e}\cup C_e$.  

Symmetry shows that this also holds for $s_{i+1}$. 

Now suppose $C_e\cap r_i$ is not connected.  Then there is a $Q_i$-local $H$-bridge $B$ having all attachments in $r_i$.  Corollary \ref{co:attsMissBranch} implies $B$ has precisely two attachments $x$ and $y$, and so Lemma \ref{lm:threeAtts} implies $B$ is just the edge $xy$.  Thus, $\cc{x,r_i,y}\rbsp B$ is  an $H$-green cycle $C$. Lemma \ref{lm:greenCycles} (\ref{it:CboundsFace}) shows $C$ bounds a face of $\Pi[G]$.

By symmetry, we may assume that $x$ and $y$ are both in $\cc{v_i,r_i,u}$.  
Suppose that $z$ is any vertex in $\oc{x,r_i,u}$.   

Suppose first that $z\ne u$.   As $G$ is 3-connected, $z$ has a neighbour $z'$ not in $\cc{x,r_i,y}$.   If $zz'$ is in the interior of $Q_i$\dragominor{, it must} be parallel to an edge in $r_i$, as any other edge would go into one of the faces of $\Pi[G]$ bounded by $C_e$ and $C$.
 Therefore, $zz'$ is outside $\Mob$ and, so is an edge of another $H$-green cycle.  But then one of the edges of $\cc{x,r_i,y}$ incident with $z$ is in two $H$-green cycles, contradicting Theorem \ref{th:twoGreenCycles}.
 
This same argument, however, also applies if $z=u$, with the small variation that, by Lemma \ref{lm:greenNotRed}, $zz'$ cannot span the red edge $e$, giving the contradiction that the edge of $\cc{x,r_i,u}$ incident with $u$ is in two $H$-green cycles.
Thus, $C_e\cap r_i$ is connected.  As it did for $s_i$, this implies that $r_i\subseteq M_{C_e}\cup C_e$.  

Finally, we consider $C_e\cap r_{i+5}$.  Proceeding as we did for $r_i$, if $C_e\cap r_{i+5}$ is not connected, there is (up to symmetry) a $Q_i$-local $H$-bridge $B$ having all attachments in $\cc{v_{i+5},r_{i+5},x_e}$; $B$ is a single edge and is in an $H$-green cycle.  One end of $B$ is $x_e$, and the $H$-green cycle containing $B$ consists of two parallel edges.  

Thus, there are at most two such $H$-bridges $B$, each of which is an edge parallel to an edge in $r_{i+5}$.  If they both exist, then the 3-connection of $G$ implies $x_e$ has another neighbour, which, as above, is adjacent to $x_e$ by an edge not in $\Mob$, showing one of the edges of $r_{i+5}$ incident with $x_e$ is in two $H$-green cycles, contradicting Theorem \ref{th:twoGreenCycles}.
\end{proof}

We can now define $\Delta_e$.  If $C_e\cap r_{i+5}$ is connected, then $\Delta_e=C_e$.  Otherwise, $\Delta_e$ is obtained from $C_e$ by replacing the edge of $C_e$ incident with $x_e$ and not in $r_{i+5}$ with its parallel mate that is in $r_{i+5}$.   Notice that the $\Delta_e$- and $C_e$-bridges are the same, except for these exchanged edges incident with $x_e$.   Set $M_{\Delta_e}$ to be $M_{C_e}$.    The following is evident from what has just preceded.  

\begin{claim}  $H\subseteq M_{\Delta_e}\cup \Delta_e$ and $\Delta_e\cap r_{i+5}$ is connected. \hfill\eopf\end{claim}

Consider again $r_i\cap \Delta_e$.  It is connected, so if it is more than just $\cc{u,e,w}$, the symmetry shows we may assume it contains an edge $xu$ other than $e$.  The 3-connection of $G$ implies that $u$ is adjacent with a vertex $y$ other than $x$ and $w$.  The edge $uy$ is not interior to $Q_i$, as then it would be in the face of $G$ bounded by $C_e$.  

Thus, $uy$ is not in $\Mob$, and, as $uw$ is red, Lemma \ref{lm:greenNotRed} implies $uy$ spans $xu$.  The vertex $x$ is seen to be $H$-green by the $H$-green cycle $C_y$ containing $uy$.  Since $x$ has at least three neighbours in $G$, there is a neighbour of $x$ different from the two neighbours of $x$ in $r_{i-1}\, r_i$.  Because $C_y$ bounds a face of $G$ (Lemma \ref{lm:greenCycles} (\ref{it:CboundsFace})), \wording{every edge incident with $x$ and not in $r_{i-1}\,r_i$} is in $\Mob$. \wording{ There is a unique neighbour $z$ of $x$ so that $z$ is not in $r_{i-1}\,r_i$ and}  $xz$ is an edge of $\Delta_e$\wording{.  This shows} that $x$ is one end of $r_i\cap \Delta_e$.  These observations easily yield the following claim.

\begin{claim}\label{cl:atMostOneRiEdge} \dragominor{Each of $A_u\cap r_i$ and $A_w\cap r_i$ has at most one edge.} \hfill\eopf\end{claim}

We now turn our attention to $r_{i+5}$.

\begin{claim}\label{cl:locallyUseful}  \begin{enumerate}\item\label{it:noYellowRi+5}  No edge of $r_{i+5}\cap \Delta_e$ is $H$-yellow.
\item\label{it:xeNotSpanned} No global $H$-bridge has $x_e$ in the interior of its span.
\end{enumerate}
\end{claim}

\begin{proof} For (\ref{it:noYellowRi+5}), suppose by way of contradiction there were an $H$-yellow edge in $r_{i+5}\cap \Delta_e$.  Then Lemma \ref{lm:yellowCycles} (\ref{it:oneBridge}) shows the witnessing $H$-yellow cycle must be $\Delta_e$.  However, the witnessing $H$-green cycle must have $\Delta_e\cap r_i$ in the interior of its span, yielding the contradiction that $e$ is $H$-green.

For (\ref{it:xeNotSpanned}), suppose by way of contradiction that there is a global $H$-bridge $xy$ with $x_e$ in the interior of the span of $xy$.  Then $xy\cup (r_{i+5}-x_e)$ contains a \wording{$\leftspine_i\,\rightspine i$}-path in $K_i-\{e,x_e\}$, contradicting Lemma \ref{lm:Qspans}.  
\end{proof}

\begin{claim}\label{cl:ri+5}  \begin{enumerate}\item\label{it:threeConsec} If $\cc{v_{i+5},r_{i+5},x_e}\cap \Delta_e$ contains three vertices $x$, $y$, and $x_e$ of $r_{i+5}$, then (choosing the labelling of $x$ and $y$ appropriately) $\cc{v_{i+5},r_{i+5},x_e}\cap \Delta_e=\cc{x,xy,y,yx_e,x_e}$, $y$ and $x_e$ are joined by a digon, and $y$ is incident with a global $H$-bridge that spans $x$.  
\item\label{it:twoConsec} If $\cc{v_{i+5},r_{i+5},x_e}\cap \Delta_e$ does not contain three consecutive vertices of $r_{i+5}$, but has a vertex $x$ other than $x_e$, then either $x$ and $x_e$ are joined by a digon, or $x_e$ is incident with a global $H$-bridge that spans $x$.  
\end{enumerate}
The symmetric statements also hold for $\cc{x_e,r_{i+5},v_{i+6}}\cap \Delta_e$.
\end{claim}

\begin{proof}  For (\ref{it:threeConsec}), the fact that $\Delta_e\cap r_{i+5}$ is connected implies that there are vertices $x$ and $y$ so that $\cc{x,xy,y,yx_e,x_e}\subseteq \cc{v_{i+5},r_{i+5},x_e}$.    Because $G$ is 3-connected, $y$ is adjacent to a vertex $z$ other than $x$ and $x_e$.   The edge $yz$ cannot be in $\Mob$, as then it would be in the face of $G$ bounded by $C_e$, a contradiction.  Therefore, it is a 2.5-jump.  Claim \ref{cl:locallyUseful} (\ref{it:xeNotSpanned}) shows $yz$ does not span $x_e$.

As $G$ is 3-connected, $x$ has a neighbour $x'$ different from the two neighbours of $x$ in  $R$.  If the edge $xx'$ is in $\Disc$, then it is in the face bounded by the $H$-green cycle containing $yz$, a contradiction.  Therefore, $xx'$ is in $\Mob$ and, in particular, for that $x'$ giving the edge nearest to $xy$ in the cyclic rotation about $x$, $xx'$ is in $\Delta_e$ and, therefore, no other vertex of $\cc{v_{i+5},r_{i+5},x_e}$ is in $\Delta_e$.

Since $yx_e$ is not $R$-separated from $e$ in $G$, Lemma \ref{lm:notSeparated} implies $yx_e$ is either $H$-yellow or $H$-green.  Claim \ref{cl:locallyUseful} (\ref{it:noYellowRi+5}) implies it is not $H$-yellow; we conclude that $yx_e$ is $H$-green and let $C_{yx_e}$ be the witnessing  $H$-green cycle.

As pointed out in the first paragraph of the proof,  $C_{yx_e}$ cannot contain a global $H$-bridge that spans $x_e$.  On the other hand, $xy$ is $H$-green by the global $H$-bridge $yz$.  By Theorem \ref{th:twoGreenCycles}, this is the only $H$-green cycle containing $xy$.  Thus, the only $H$-rim edge contained in $C_{yx_e}$ is $yx_e$.  It follows that $C_{yx_e}$ is contained in $\cl(Q_i)$.  Claim \ref{cl:connected} implies $C_{yx_e}$ is a digon.


For (\ref{it:twoConsec}), the fact that $\cc{v_{i+5},r_{i+5},x_e}\cap \Delta_e$ is connected implies that $\cc{v_{i+5},r_{i+5},x_e}\cap \Delta_e=\cc{x,xx_e,x_e}$.     Lemma \ref{lm:notSeparated} implies that $xx_e$ is either $H$-yellow or $H$-green, and Claim \ref{cl:locallyUseful} (\ref{it:noYellowRi+5}) shows it is not $H$-yellow.  Therefore, it is $H$-green.  

Claim \ref{cl:locallyUseful} (\ref{it:xeNotSpanned}) shows any global $H$-bridge spanning $xx_e$ has $x_e$ as an attachment.  Otherwise, the $H$-green cycle $C_{xx_e}$ containing $xx_e$ is contained in $\cl(Q_i)$.  Again, Claim \ref{cl:connected} shows $C_{xx_e}$ is a digon.
\end{proof}

There is one more observation to make before we complete the proof of the theorem.  From Claim \ref{cl:ri+5} (\ref{it:threeConsec}), it seems possible that both $\cc{v_{i+5},r_{i+5},x_e}\cap \Delta_e$ and $\cc{v_{i+5},r_{i+5},x_e}\cap \Delta_e$ have three vertices.  However, this is not possible, as $x_e$ must have a neighbour $z$ different from its neighbours in $R$.  But now $x_ez$ cannot be in $\Mob$, as then it would be in the face bounded by $C_e$, and it cannot be in $\Disc$, as then it is a global $H$-bridge and one of the digons incident with $x_e$ is also spanned by $x_ez$, contradicting Theorem \ref{th:twoGreenCycles}.  Therefore, $r_{i+5}\cap \Delta_e$ has at most three edges, and all such edges are $H$-green.

If there are no edges, then $r_{i+5}\cap \Delta_e$ is just $x_e$.   If no edge of $r_{i+5}\cap \Delta_e$ is in a digon, then $u^e$ and $w^e$ are defined in (\ref{it:deltaDetail}) of the statement to be $x_e$.  In this case, Claim \ref{cl:ri+5} (\ref{it:threeConsec}) implies there can be at most one edge of $r_{i+5}\cap \Delta_e$ on each side of $x_e$, but any such edge is spanned by a global $H$-bridge.  If there is a digon, then it is $u^ew^e$, each of $u^e$ and $w^e$ is incident with at most one other edge in $r_{i+5}\cap \Delta_e$, and any such edge is spanned by a global $H$-bridge.

Finally, By Lemma \ref{lm:globalJumps} (\ref{it:spanDiffHyper}), not both $u$ and $u^e$, for example, can be incident with such global $H$-bridges, so $P_u$ has at most one $H$-rim edge.
\end{cproof}

}\begin{definition}  Let $G\in \m2$, $\hvfg$, with $H$ tidy, and $e$ a red edge of $G$ with ends $u$ and $w$.  With $u^e$ and $w^e$ as in the statement of Theorem \ref{th:redHasDelta}, the {\em peak\/}\index{peak}\index{$\Delta_e$!peak} of $\Delta_e$ is the subgraph of $G$ induced by $u^e$ and $w^e$.  If the peak has just one vertex, then $\Delta_e$ is {\em sharp\/}\index{sharp}\index{$\Delta_e$!sharp}.  \end{definition}\printFullDetails{

The following observations are given to summarize important points from Theorem \ref{th:redHasDelta}.

}\begin{corollary}  Let $G\in \m2$, $\hvfg$, with $H$ tidy, and $e$ a red edge of $G$.  Then the peak of $\Delta_e$ is either a single vertex or a digon and no edge of the peak is in the interior of the span of a global $H$-bridge.  \end{corollary}\printFullDetails{

\begin{cproof}  That the peak is either a single vertex or a digon is a rephrasing of Theorem \ref{th:redHasDelta} (\ref{it:digon}) and (\ref{it:deltaDetail}).  In the case the peak is a digon, neither  $u^e$ nor $w^e$ can be in the interior of the span of a global $H$-bridge, since then the $H$-rim edge in the digon is in two $H$-green cycles, contradicting Theorem \ref{th:twoGreenCycles}. 

So suppose the peak is just the vertex $u^e=w^e$, let $B$ be a global $H$-bridge with $u^e$ in the interior of its span, and let $i$ be such that $e\in r_i$.  If $\Delta_e\cap r_{i+5}$ has an edge $e'$, then $e'$ is incident with $u^e$ and, moreover, is $H$-green by a global $H$-bridge $B'$ incident with $u^e$.  But then $B$ provides a second $H$-green cycle containing $e'$, contradicting Theorem \ref{th:twoGreenCycles}.  So $\Delta_e\cap r_{i+5}$ is just $u^e$, in which case $B$ provides a witnessing $H$-green cycle that shows $\Delta_e$ is $H$-yellow.  But then $e$ is $H$-yellow, contradicting Lemma \ref{lm:greenNotRed}.  \end{cproof}

\begin{figure}[!ht]
\begin{center}
\scalebox{1.0}{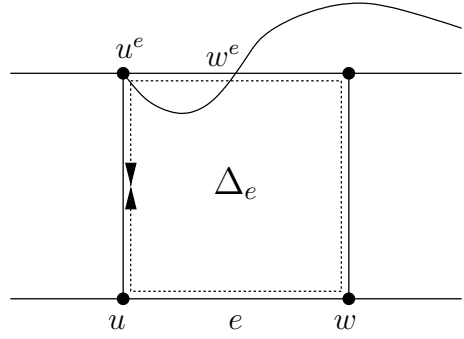}
\end{center}
\caption{One of several examples of a $\Delta$.}\label{fg:delta}
\end{figure}

Our next goal is to eliminate 3-jumps.  For this the next two lemmas are helpful.

}\begin{lemma}\label{lm:staysYellow}  Let $G\in \m2$ and $\hvfg$, with $H$ tidy.  Suppose $C$ is an $H$-yellow cycle and $C'$ is the witnessing $H$-green cycle.  Let $e$ be an edge of $G$ not in $C\cup C'\cup R$.  Suppose either $C'$ does not contain a 3-jump or $e$ is in one of the four spokes containing an $H$-node spanned by $C'$.  Then no $H$-yellow edge in $C$ is crossed in any 1-drawing of $G-e$.  \end{lemma}\printFullDetails{

\begin{proof}  There are at least four $H$-spokes contained in $G-e$.  By hypothesis, at least one of these has no end in $C'$ and, therefore, no end in $C\cup C'$.  
Therefore, Lemma \ref{lm:technicalV8colour} (\ref{it:V8yellow}) applies. \end{proof}

}\begin{lemma}\label{lm:staysGreen}  Let $G\in \m2$ and $\hvfg$, with $H$ tidy.  Suppose $C$ is an $H$-green cycle in $G$.   Suppose that $C$ does not contain a 3-jump, $e$ is an edge of $G$ not in $R\cup C$ and $D$ is a 1-drawing of $G-e$.  If an edge $e'$ of $C$ is crossed in $D$, then $C$ contains a 2.5-jump with an end in $\oo{r_i}$, for some $i$, and $e'$ is in $r_i$. \end{lemma}\printFullDetails{

\begin{cproof}  This is a straightforward consequence of Lemma \ref{lm:technicalV8colour} (\ref{it:2halfJump1} and \ref{it:2halfJump3/2}).  \end{cproof}

}\begin{theorem}\label{th:no3jump}  Let $G\in \m2$ and $\hvfg$, with $H$ tidy.  Then no global $H$-bridge is a 3-jump.\end{theorem}\printFullDetails{

\begin{cproof}  The proof begins by showing that if $v_{i-3}v_i$ is a global $H$-bridge that is a 3-jump, then there is a red edge in $r_i$.  The next step is to show that the edge of $r_i$ incident with $v_i$ is red.  The final step is to show that, if $e^*$ is the edge of $s_i$ incident with $v_i$, then $\crn(G-e^*)\ge 2$, contradicting the criticality of $G$.  Let $\Pi$ be an embedding of $G$ in $\pp$ so that $H$ is $\Pi$-tidy.

\begin{claim}\label{cl:existsRed}  There is a red edge of $G$ in $r_i$. \end{claim}

\begin{proof} Lemma \ref{lm:globalJumps} (\ref{it:atMostOne3jump}) implies neither $v_{i+5}v_{i-2}$ nor $v_iv_{i+3}$ is in $G$.  Thus, Lemma \ref{lm:tidyBODhyper} implies $\bQ_{i-1}$ has BOD.  

Let $D_{i-1}$ be a 1-drawing of $G-\oo{s_{i-1}}$.  Lemma \ref{lm:BODcrossed} implies $\bQ_{i-1}$ is crossed in $D_{i-1}$.  Let $H'$ be the subdivision of $V_6$ consisting of the $H$-rim $R$ and the three spokes $s_i$, $s_{i-3}$, and $s_{i+1}$.  Lemma \ref{lm:technicalV8colour} implies the cycle $r_{i-3}\,r_{i-2}\,r_{i-1}\,\cc{v_{i},v_{i-3}v_i,v_{i-3}}$ is clean in $D_{i-1}$.  In particular, the crossing must be of an edge in $r_{i+3}\cup r_{i+4}$ and an edge $e$ in $r_{i}$.  

We prove $e$ is red in $G$ by proving it is neither $H$-green nor $H$-yellow.  Lemma \ref{lm:globalJumps} (\ref{it:globalNoHnode}) and (\ref{it:2and3jumps}) imply that no global $H$-bridge other than $v_{i-3}v_i$ has an end in $\co{v_{i},r_{i},v_{i+1}}$.  Therefore, no $H$-green cycle containing $e$ can contain a global $H$-bridge.  Thus, any $H$-green cycle $C$ containing $e$ is contained in $\cl(Q_{i})$.  Lemma \ref{lm:staysGreen}  implies $C$ is not crossed in $D_{i-1}$, contradicting the fact that the edge $e$ is in $C$ and is crossed in $D_{i-1}$.  We conclude that $e$ is not $H$-green. 

So suppose $C$ is an $H$-yellow cycle containing $e$ and let $P_1P_2P_3P_4$ be the decomposition of $C$ into paths as in Definition \ref{df:yellow}.  By Lemma \ref{lm:yellowCycles}, there is a global $H$-bridge $B$ so that the interior of the span of $B$ contains $P_1$.  Lemma \ref{lm:globalJumps} (\ref{it:atMostOne3jump}) says there is at most one 3-jump in $G$, so $B$ is either a 2- or 2.5-jump.  

That $e$ is not $H$-yellow is an immediate consequence of Lemma \ref{lm:staysYellow}.  \end{proof}

We now aim to show that the edge of $r_i$ incident with $v_i$ is red.
By Claim \ref{cl:existsRed}, there is a red edge in $r_i$; let $e_1$ be the red edge nearest to $v_i$ in $r_i$.  Let $r'_i$ be the component of $r_i-e_1$ containing $v_i$ and let $u$ be the end of $e_1$ in $r'_i$.

\begin{claim}\label{cl:r'iNotYellow} No edge of $r'_i$ is $H$-yellow. \end{claim}

\begin{proof}  Suppose some edge $e'$ of $r'_i$ is $H$-yellow and let $C$ and $C'$ be the witnessing $H$-yellow and $H$-green cycles, respectively.  
 Lemma \ref{lm:yellowCycles} (\ref{it:greenGlobal}) implies $C'$ contains a global $H$-bridge $B$.  We note that Lemma \ref{lm:globalJumps} (\ref{it:globalNoHnode}) and  (\ref{it:2and3jumps}) imply (because $v_{i-3}v_i$ is present and $v_{i-3}=v_{i+7}$) that $B$ has no vertex in $\oc{v_{i+6},r_{i+6},v_{i+7}}$.  On the other hand, to make $C$ $H$-yellow, $B$ must have one end in $\oc{v_{i+5},r_{i+5},v_{i+6}}$.  

Due to the presence of $v_{i-3}v_i$, Lemma \ref{lm:globalJumps} (\ref{it:spanDiffHyper}) implies $v_{i+3}$ is not in $B$.  Therefore, Theorem \ref{th:globalBridges} implies $B$ has $v_{i+6}$ as one end and its other end is in $\oc{v_{i+3},r_{i+3},v_{i+4}}$.  Theorem \ref{th:redHasDelta} (\ref{it:deltaDetail}) implies the edge $e$ of $\Delta_{e_1}-e_1$ incident with $u$ is not in $H$; by Theorem \ref{th:redHasDelta}, it is in $\cl(Q_i)$.  

Let $D$ be a 1-drawing of $G-e$.  By Theorem \ref{th:BODquads}, $Q_i$ has BOD, so  Lemma \ref{lm:BODcrossed} implies $Q_i$ is crossed in $D$.  Lemma \ref{lm:technicalV8colour} implies no edge in $r_{i+4}\,r_{i+5}$ is crossed in $D$, so the crossing in $D$ is of $r_i$ with $r_{i+6}$.  

Lemmas \ref{lm:staysYellow} and \ref{lm:staysGreen} combine with Theorem \ref{th:rimColoured} to show that the edge $e''$ of $r_{i+6}$  crossed in $D$ is red in $G$.   Lemma \ref{lm:notSeparated} implies $e''$ and $e_1$ are $R$-separated in $G$ and we conclude that they are also $R$-separated in $G-e'$; in fact, $e''$ is $R$-separated from $r'_i\lbsp\cc{u,e_1,w}$.  It follows that the edge $f$ of $r_i$ crossed in $D$ is in $\cc{w,r_i,v_{i+1}}$.

Lemmas \ref{lm:staysYellow} and \ref{lm:staysGreen} combine with Theorem \ref{th:rimColoured} to show that $f$ is red in $G$; however, $e_1$ and $f$ are not $R$-separated in $G-e'$ and, therefore, not separated in $G$, contradicting Lemma \ref{lm:notSeparated}.  It follows that no edge of $r'_i$ is $H$-yellow, as required.
\end{proof}

\begin{claim}\label{cl:viIncidentRed} $u=v_i$. \end{claim}
  
\begin{proof}\startSubclaims By way of contradiction, suppose that $u\ne v_i$.  By definition of $e_1$, no edge of $r'_i$ is red, and Claim  \ref{cl:r'iNotYellow} shows no edge of $r'_i$ is $H$-yellow.  Theorem \ref{th:rimColoured} shows that every edge of $r'_i$ is  $H$-green.  Because of $v_{i-3}v_i$, Lemma \ref{lm:globalJumps} (\ref{it:globalNoHnode}) and (\ref{it:2and3jumps}) shows  no edge of $r'_i$ is $H$-green by a global $H$-bridge.   

Let $e$ be the edge of $\Delta_{e_1}-e_1$ incident with $u$; \wording{ Theorem \ref{th:redHasDelta} and the fact that $e_1$ is not incident with $v_i$ imply that} $e$ is not in $H$.  Let $D$ be a 1-drawing of $G-e$.  Note that $e$ is in a $Q_i$-local $H$-bridge.  Since $Q_i$ has BOD (Theorem \ref{th:BODquads}), it is crossed in $D$ (Lemma \ref{lm:BODcrossed}).  Every edge of $r_{i-1}$ is $H$-green in $G-e$; thus, Lemma \ref{lm:greenCycles} (\ref{it:notCrossed}) implies the following.

\begin{subclaim}\label{sc:r(i-1)notCrossed} \dragominor{No edge in $r_{i-1}$ is  crossed in $D$.}  \hfill $\Box$\end{subclaim}

\dragominor{We next rule out another possibility.}

\begin{subclaim}\label{sc:r(i+1)notCrossed} \dragominor{ No edge in $r_{i+1}$ is crossed in $D$.} \end{subclaim}

\begin{proof} Suppose some edge $e_i^D$ of $r_{i+1}$ is crossed in $D$.  Since $Q_i$ is crossed in $D$, the other crossed edge ${e'}_i^D$ is in $r_{i+5}$.  By Lemma \ref{lm:staysYellow}, no $H$-yellow edge in $r_{i+1}\cup r_{i+5}$ can be crossed in $D$.  Since $H\subseteq G-e$, Lemma \ref{lm:greenCycles} (\ref{it:notCrossed}) implies no $H$-green cycle not containing $e$ can be crossed in $D$; in particular, no $H$-green edge in $r_{i+1}\cup r_{i+5}$ can be crossed in $D$.  Now Theorem \ref{th:rimColoured} implies $e_i^D$ and ${e'}_i^D$ are both red in $G$.  

\wordingrem{(new paragraph)}Suppose first that  ${e'}_i^D$ is in $\cc{v_{i+5},r_{i+5},u^{e_1}}$.  (Recall that $u^{e_1}$ is the vertex in the peak of $\Delta_{e_1}$ nearest $u$ in $\Delta_{e_1}-e_1$.)  Lemma \ref{lm:notSeparated} implies ${e'}_i^D$ and $e_1$ are $R$-separated in $G$; this implies that $\Delta_{{e'}_i^D}$ is disjoint from $\Delta_{e_1}$.  One of the $r_ir_{i+5}$-paths in $\Delta_{{e'}_i^D}$, $s_{i+1}$, $s_{i+2}$, and $s_{i+3}$ combine with $R$ to show that ${e'}_i^D$ is $R$-separated in $G-e$ from every edge in $r_{i+1}$, a contradiction.

If, on the other hand, ${e'}_i^D$ is not in $\cc{v_{i+5},r_{i+5},u^{e_1}}$, then Lemma \ref{lm:notSeparated} shows $e_i^D$ and ${e'}_i^D$ are $R$-separated in $G$ and there is a subdivision of $V_8$ that both witnesses this separation and does not contain $e$ (the spokes are $s_{i+2}$, $s_{i+3}$, and the ``nearer" $(r_i\,r_{i+1})(r_{i+5}\,r_{i+6})$-paths in $\Delta_{e_i^D}$ and $\Delta_{{e'}_i^D}$).  This shows that $e_i^D$ and ${e'}_i^D$ are $R$-separated in $G-e$, a contradiction. \end{proof}

\dragominor{Since $Q_i$ is crossed in $D$, Subclaims \ref{sc:r(i-1)notCrossed} and \ref{sc:r(i+1)notCrossed} imply that some edge $e_i^D$ of $r_i$ is crossed in $D$.  }

\begin{subclaim}\label{sc:r'iCrossed}  \dragominor{$e_i^D\in r'_i$.}  \end{subclaim}

\begin{proof}  If $e_i^D$ is not in $r'_i$, then let ${e'}_i^D$ be the edge of $r_{i+4}\,r_{i+5}\,r_{i+6}$ that is crossed in $D$.  Then $e_i^D$ and ${e'}_i^D$ are not $R$-separated in $G-e$.  Observe that $\Delta_{e_1}$ shows no $H$-green or $H$-yellow cycle containing $e_i^D$ can also contain $e$.  Therefore, $e_i^D$ is red in $G$ and, consequently is $R$-separated from ${e'}_i^D$ in $G$.   In particular, $e$ is in every subdivision of $V_8$ that contains $R$ and witnesses the $R$-separation of $e_i^D$ and ${e'}_i^D$.   This implies that ${e'}_i^D$ is in $\cc{v_{i+5},r_{i+5},u^{e_1}}$.

As $e_1$ and ${e'}_i^D$ are both red in $G$, by Lemma \ref{lm:notSeparated} there is a subdivision $K$ of  $V_8$ containing $R$ and witnessing the $R$-separation of $e_1$ and ${e'}_i^D$.  There is an $r'_ir_{i+5}$-path $P$ in $K$ that is disjoint from $\Delta_{e_1}$.  Moreover, $P\subseteq \cl(Q_i)$.  But now, $P$ together with  the $r_ir_{i+5}$-path in $\Delta_{e_1}-u$, $s_{i+2}$, and $s_{i+3}$ make the four spokes of a subdivision of $V_8$ containing $R$ and witnessing the $R$-separation of $e_i^D$ and ${e'}_i^D$ in $G-e$, a contradiction.  \end{proof}

We now \wordingrem{(text removed)}
locate the edge ${e'}_i^D$.  To this end, let $\hat e$ be the edge of $s_{i-1}$ incident with $v_{i-1}$ and let $\widehat D$ be a 1-drawing of $G-\hat e$.  By Lemmas \ref{lm:tidyBODhyper}, \ref{lm:globalJumps} (\ref{it:atMostOne3jump}), and \ref{lm:BODcrossed}, $\bQ_{i-1}$ must be crossed in $\widehat D$.  However, Lemma \ref{lm:technicalV8colour} shows that none of $r_{i-3}\,r_{i-2}\,r_{i-1}$ can be crossed in $\widehat D$.  Since the edges in $r'_i$ are all $H$-green and none of the witnessing $H$-green cycles \wording{contains} a global $H$-bridge, Lemma \ref{lm:staysGreen} implies that no edge of $r'_i$ is crossed in $\widehat D$.  
\wordingrem{(text removed)}%
\wording{Thus,} some edge of $r_{i+3}\,r_{i+4}$ crosses an edge \minor{of $r_i-\oo{r'_i}$ in} $\widehat D$.  

\begin{subclaim}\label{sc:r(i+4)green} \dragominor{Every edge in $r_{i+4}$ is $H$-green in $G$ and no edge in $r_{i+4}$ is crossed in $\widehat D$.}  \end{subclaim}

\begin{proof}  If  $e'\in r_{i+4}$ is $H$-yellow, then $v_{i-3}v_i$ is in the witnessing $H$-green cycle and, therefore, the edge of $s_{i-1}$ incident with $v_{i+4}$ is in the interior of  an $H$-yellow cycle containing $s_{i-2}$; this contradicts Lemma \ref{lm:yellowCycles}, so $e'$ is not $H$-yellow.   

Now we eliminate the possibility that $e'$ is red.
To do this, it will be helpful to know that no $H$-green edge in $r_{i+4}$ is crossed in $\widehat D$:  fortunately, this is just Lemma \ref{lm:staysGreen}, combined with Lemma \ref{lm:globalJumps} (\ref{it:globalNoHnode}) and (\ref{it:spanDiffHyper}) to eliminate the possibility of a 2.5-jump.

Choose $e'$ to be the red (in $G$) edge in $r_{i+4}$ that is nearest in $r_{i+4}$ to $v_{i+5}$.  Lemma \ref{lm:notSeparated} implies $e'$ is $R$-separated from $e_1$ in $G$;  we may choose the witnessing subdivision $K$ of $V_8$ to contain $s_{i-2}$ and $s_{i+2}$; in particular, $K$ avoids $\hat e$.  Therefore, $e'$ is $R$-separated from $e_1$ in $G-\hat e$.  Since the edges in $r_{i+4}$ between $e'$ and $v_{i+5}$ are neither red (choice of $e'$) nor $H$-yellow (two paragraphs preceding), they are $H$-green (Theorem \ref{th:rimColoured}),  we know they are not crossed in $\widehat D$ (preceding paragraph).  The subgraph $K$ shows that none of the rest of $r_{i+3}\,r_{i+4}$ can be crossed in $\widehat D$, which is a contradiction.  Therefore, no edge of $r_{i+4}$ is red in $G$; since none is $H$-yellow by the preceding paragraph, Theorem \ref{th:rimColoured} shows they are all $H$-green. \end{proof}

It follows that an edge of $r_{i+3}$ is crossed in $\widehat D$ and it must cross some edge in $\cc{u,r_i,v_{i+1}}$. This further implies that the $uu^{e_1}$-subpath $P_u$ of $\Delta_{e_1}-e_1$ intersects $s_i$ as otherwise each edge of $\cc{u,r_i,v_{i+1}}$ is $R$-separated from $r_{i+3}$ in $G-\hat e$.  

We now return to consideration of $D$.  No edge in $r_{i+4}$ is red in $G$ and, because $P_u$ intersects $s_i$, every edge (if there are any) of $\cc{v_{i+5},r_{i+5},u^{e_1}}$ is $H$-green. This combines with Lemma \ref{lm:globalJumps} (\ref{it:globalNoHnode}) and (\ref{it:spanDiffHyper}) to show that no edge in $r_{i+4}\lbsp\cc{v_{i+5},r_{i+5},u^{e_1}}$ is in the span of a global $H$-bridge; therefore, Lemmas \ref{lm:staysYellow} and \ref{lm:staysGreen} imply that no edge of $r_{i+4}\lbsp\cc{v_{i+5},r_{i+5},u^{e_1}}$ is crossed in $D$.  Thus, \wording{the edge ${e'}_i^D$ that crosses $e_i^D$ in $D$ is} in $\cc{u^{e_1},r_{i+5},v_{i+6}}\rbsp r_{i+6}$\wordingrem{(text removed)}.  

Because of $v_{i-3}v_i$, no edge in $\cc{u^{e_1},r_{i+5},v_{i+6}}\rbsp r_{i+6}$ is in the span of a global $H$-bridge.  Therefore, Lemmas \ref{lm:staysYellow} and \ref{lm:staysGreen} imply ${e'}_i^D$ is red in $G$.  But now Lemma \ref{lm:notSeparated} implies ${e'}_i^D$ is $R$-separated in $G$ from $e_1$;  there is a witnessing subdivision $K$ of $V_8$ that contains $s_{i-1}$, $s_i$, and the nearer $(r_i\,r_{i+1})(r_{i+5}\,r_{i+6})$-paths in $\Delta_{e_1}$  and $\Delta_{{e'}_i^D}$.  \wording{Note that the path taken from $\Delta_{e_1}$ does not contain $e$.  Therefore,} $K$ is also contained in $G-e$; Observation \ref{obs:separated} (\ref{it:separatedNoCross}) shows that these edges cannot be crossed in $D$, the final contradiction that proves the claim.
\end{proof}

We now move into the \wording{final} phase of the proof that there is no 3-jump.   Let $ e^*$ be the edge of $s_i$ incident with $v_i$ and let $D^*$ be a 1-drawing of $G- e^*$.  Lemma \ref{lm:globalJumps} (\ref{it:atMostOne3jump}) implies $v_{i-3}v_i$ is the only 3-jump of $G$, so Lemma \ref{lm:tidyBODhyper} implies $\bQ_i$ has BOD.  Lemma \ref{lm:BODcrossed} implies $\bQ_i$ is crossed in $D^*$.   In particular, there is an edge $e$ in $r_{i+3}\,r_{i+4}\,r_{i+5}\,r_{i+6}$ that is crossed in $D^*$.
Lemma \ref{lm:technicalV8colour} shows that $r_{i+3}$ is not crossed in $D^*$.  

\begin{claim}  $e$ is red in $G$.  \end{claim}

\begin{proof}  If $e$ is $H$-yellow in $G$, then Lemma \ref{lm:staysYellow} shows that $e$ is not crossed in $D^*$.  Thus, $e$ is not $H$-yellow.

Suppose $e$ is $H$-green in $G$, and let $C$ be the witnessing $H$-green cycle.   Lemma \ref{lm:globalJumps} (\ref{it:atMostOne3jump}) implies $C$ does not contain a 3-jump and Lemma \ref{lm:technicalV8colour} implies both that it does not contain a 2-jump and is not contained in the union of some $Q_j$ together with a $Q_j$-local $H$-bridge.  Therefore, $C$ contains a 2.5-jump $b$ and Lemma \ref{lm:technicalV8colour} implies $e$ is in the $H$-rim branch that contains the end $x$ of $b$ that is not an $H$-node.

The edge $e$ has already been shown to be in $r_{i+4}\,r_{i+5}\,r_{i+6}$.  Suppose $e$ is in $r_{i+4}$.  If $b=v_{i+2}x$, then we contradict Lemma \ref{lm:globalJumps} (\ref{it:spanDiffHyper}) --- $v_{i-3}v_i$ and $b$ span the opposite sides of $\bQ_{i-2}$, a contradiction.  
The other alternative is that $b=xv_{i-3}$, which violates Lemma \ref{lm:globalJumps} (\ref{it:globalNoHnode}).  Thus, $e\notin r_{i+4}$.

If $e\in r_{i+5}$, then either $b=xv_{i-2}$ or $b=xv_{i+3}$.  The former does not occur, as otherwise the edges of $r_{i-3}$ are all in two $H$-green cycles, contradicting Theorem \ref{th:twoGreenCycles}.  If the latter occurs, then we contradict Lemma \ref{lm:globalJumps} (\ref{it:spanDiffHyper}) --- $v_{i-3}v_i$ and $b$ span the opposite sides of $\bQ_{i-1}$.  Thus, $e\notin r_{i+5}$.

So $e\in r_{i+6}$.  In this case $b$ is either $xv_{i-1}$ or $xv_{i+4}$. For the former, the edges of $r_{i-3}\,r_{i-2}$ are all in two $H$-green cycles, contradicting Theorem \ref{th:twoGreenCycles}.  For the latter,  the edge $e_1$ of $r_i$ incident with $v_i$ is red by Claim \ref{cl:viIncidentRed}.  The existence of $b$ \minor{shows $Q_i$ is} $H$-yellow, contradicting the fact that $e_1$ is red.  This is the final contradiction that shows $e$ is red.
 \end{proof}

\dragominor{Recall that t}he edge $e$ is in $r_{i+3}\,r_{i+4}\,r_{i+5}\,r_{i+6}$, since it is involved in a crossing with $\bQ_i$.  We have already observed that $e$ is not in $r_{i+3}$.  

Suppose first that $e\in r_{i+4}$.  Lemma \ref{lm:notSeparated} implies $e$ and $e_1$ are $R$-separated in $G$; \minor{in particular, $v_{i}$} is not in $\Delta_e$.  But then $v_{i-3}v_i$ \wording{shows $\Delta_e\subseteq \cl(Q_{i-1})-v_i$ to} be an $H$-yellow cycle, contradicting the fact that $e$ is red.

Therefore, $e\in r_{i+5}\,r_{i+6}$.   Let $\hat e$ be the edge crossed by $e$ in $D^*$.  Since Lemma \ref{lm:technicalV8colour} implies $r_{i-2}$ is not crossed in $D^*$, $\hat e\notin r_{i-2}$.  Since $e$ and $e_1$ are both red in $G$, Lemma \ref{lm:notSeparated} implies they are $R$-separated in $G$; there is a witnessing subdivision $K$ of $V_8$ that contains $s_{i-1}$ and $s_{i-2}$.  This $K$ does not contain $e^*$, and so  is contained in $G-e^*$.  Therefore, $K$ separates $e$ from $r_{i-1}$ in $G-e^*$, \wording{and so,} in $D^*$, $e$ does not cross $r_{i-1}$.  Thus, $\hat e$ is not in $r_{i-1}$. 

Therefore, $\hat e\in r_i\,r_{i+1}$.  Lemma \ref{lm:globalJumps} (\ref{it:spanDiffHyper}) implies there is no 2.5-jump $xv_{i+4}$ --- it and $v_{i-3}v_i$ would span the opposite sides of $\bQ_{i-2}$.   Also, Lemma \ref{lm:globalJumps} (\ref{it:2and3jumps}) implies there is no 2.5-jump $xv_{i+3}$ with $x\in \oo{r_i}$.  

It follows from Lemmas \ref{lm:staysYellow} and \ref{lm:staysGreen} (the preceding pararaph is used here) that the edge $\hat e$ crossed by $e$ in $D^*$ is red in $G$.  This implies that $e$ and $\hat e$ are $R$-separated in $G$ and this in turn implies that $e$ and $\hat e$ are $R$-separated in $G- e^*$, the final contradiction.
\end{cproof}

}\begin{corollary}\label{co:hyperBOD}  Let $G\in \m2$ and $\hvfg$, with $H$ tidy.  Then every $H$-hyperquad has BOD.\end{corollary}\printFullDetails{

\begin{cproof}  By Theorem \ref{th:no3jump}, no global $H$-bridge is a 3-jump.  By Lemma \ref{lm:tidyBODhyper}, every $H$-hyperquad has BOD. \end{cproof}

We are now prepared for the main result of this section.

}\begin{theorem}\label{th:redInRim}  Let $G\in \m2$ and $\hvfg$, with $H$ tidy.  Then there is a red edge in the $H$-rim. \end{theorem}\printFullDetails{

\begin{cproof}  We prove this by first considering the case there is a global $H$-bridge.  By Theorem \ref{th:no3jump}, there is no 3-jump.  By Theorem \ref{th:globalBridges}, a global $H$-bridge is either a 2.5- or a 2-jump.  

\begin{claim}  If $G$ has a 2.5-jump, then $G$ has a red edge. \end{claim}

\begin{proof}   By symmetry, we may assume $wv_{i+2}$ is a 2.5-jump with $w\in \oo{r_{i-1}}$.   By way of contradiction, we assume that $G$ has no red edge. We first treat two special cases.

\medskip
{\bf Case 1:}  {\em there is a 2.5-jump $v_{i-3}w'$, with $w'\in \oo{r_{i-1}}$.}

\medskip
In this case, let  $D$ be a 1-drawing of $G-\oo{s_{i+2}}$.  \wording{Corollary \ref{co:hyperBOD} and Lemma \ref{lm:BODcrossed} show} that $\bQ_{i+2}$ is crossed in $D$.   Lemma \ref{lm:technicalV8colour} implies \wording{each of the cycles consisting of one of these two 2.5-jumps and the subpath of $R$ it spans is} clean in $D$.  The same lemma implies that neither $r_{i+3}$ nor $r_{i+5}$ is crossed in $D$.  The combination of facts imply that some edge $e_2$ in $r_{i+2}$ crosses some edge $e_6$ in $r_{i+6}$.    

Since $G$ has no red edge,  Theorem \ref{th:rimColoured} implies each of $e_2$ and $e_6$ is either $H$-yellow or $H$-green in $G$.  There is complete symmetry between them (relative to the two 2.5-jumps), so we treat $e_6$.  If $e_6$ is $H$-yellow in $G$, then it is in some witnessing $H$-yellow cycle $C$ for which there is a witnessing $H$-green cycle $C'$.  The only possibility is that $C'$ \major{contains $wv_{i+2}$.}

\major{We have that $C\subseteq \cl(Q_{i+1})-v_{i+2}$.  Let $C=P_1P_2P_3P_4$ be the composition of paths showing $C$ is $H$-yellow, as in Definition \ref{df:yellow}.  Since $P_1\subseteq \oo{C'\cap R}$, we have $P_1\subseteq r_{i+1}-v_{i+2}$.  Choose the labelling of $P_2$ and $P_4$ so the $r_{i+1}$-end of $P_2$ is nearer $v_{i+2}$ in $r_{i+1}$ than the $r_{i+1}$-end of $P_4$ is.} 

\major{If $P_2$ is not disjoint from $\oo{s_{i+2}}$, then the edge of $r_{i+1}$ incident with $v_{i+2}$ is in two $H$-green cycles, contradicting Theorem \ref{th:twoGreenCycles}.  Therefore,} $C\cup C'$ is disjoint from $\oo{s_{i+2}}$.   But then Lemma \ref{lm:staysYellow} implies $e_6$ is not crossed in $D$ and, therefore, $e_6$ is not $H$-yellow.  Likewise, $e_2$ is not $H$-yellow.

Therefore, $e_6$ is $H$-green, so Lemma \ref{lm:staysGreen} implies $e_6$ is spanned by some 2.5-jump $J_6$ and, moreover, is not in either $H$-rim branch fully contained in the span of $J_6$.  
By Theorem \ref{th:twoGreenCycles}, no $H$-rim edge is in two $H$-green cycles.  Thus, the only possibility for the 2.5-jump $J_6$ spanning $e_6$ is $v_{i+4}w_6$, with $w_6\in \oo{r_{i+6}}$.  An analogous argument applies to $e_2$, so $e_2$ is spanned by the 2.5-jump $J_2$ $w_2v_{i+5}$, with $w_2\in \oo{r_{i+2}}$.  But now we have that every edge of $r_{i+4}$ is in the distinct $H$-green cycles containing $J_2$ and $J_6$, contradicting Theorem \ref{th:twoGreenCycles}, completing the proof in Case 1.

\medskip
{\bf Case 2:}  {\em There is a 2.5-jump $v_{i-4}w'$, with $w'\in r_{i-2}$.}

\medskip\dragominor{\major{Let $D_1$ be a 1-drawing of $G-\oo{s_{i+1}}$.  Corollary \ref{co:hyperBOD} and Lemma \ref{lm:BODcrossed} imply $\bQ_{i+1}$ is crossed in $D$.  Lemma \ref{lm:technicalV8colour} (\ref{it:atMostTwo}) shows none of $\cc{w,r_{i-1}v_i}$, $r_i$, $r_{i+1}$, and $r_{i+6}$ is crossed in $D$, while (\ref{it:V8yellow}) of the same lemma shows $r_{i+2}$ is not crossed.  It follows that some edge $e_5\in r_{i+5}$ crosses an edge $e_9\in \cc{v_{i+9},r_{i+9},w}$.  }}

\dragominor{\major{Since $e_9$ is not red, Theorem \ref{th:rimColoured} shows it is either $H$-yellow or $H$-green.  If $e_9$ is $H$-yellow as witnessed by the $H$-yellow cycle $C$ and the $H$-green cycle $C'$, then the global $H$-bridge $J$ in $C'$ is a 2- or 2.5-jump (Theorems \ref{th:globalBridges} and \ref{th:no3jump}) and $C\subseteq \cl(Q_{-1})$ (Lemma \ref{lm:yellowCycles} (\ref{it:quadClosure})).   Lemma \ref{lm:staysYellow} implies that $e_9$ is not crossed in $D$, a contradiction.
}}

\dragominor{\major{Likewise, if $e_9$ is $H$-green, the Lemma \ref{lm:staysGreen} shows it is not crossed in $D$, the final contradiction completing the proof in Case 2.}}

\medskip
{\bf Case 3:}  {\em All the remaining cases.}

\medskip Let $e_i$ be the edge of $s_i$ incident with $v_i$ and let $D_i$ be a 1-drawing of $G-e_i$.  Corollary \ref{co:hyperBOD} and Lemma \ref{lm:BODcrossed} imply $\bQ_i$ is crossed in $D_i$.

Since $G$ (in particular, $r_{i-2}$) has no red edge,  Lemma \ref{lm:staysYellow} \dragominor{shows any $H$-yellow edge in $r_{i-2}$ is not crossed in $D_i$, while Lemma \ref{lm:staysGreen}  implies that, as we are not in Case 2, no $H$-green edge of $r_{i-2}$ is crossed in $D_i$.  Lemma \ref{lm:technicalV8colour} (\ref{it:atMostTwo})} implies no edge of $\cc{w,r_{i-1},v_i}\rbsp r_i\,r_{i+1}$ is crossed in $D_i$.  Therefore, it must be that some edge $e_{i-1}$ of $\cc{v_{i-1},r_{i-1},w}$ is crossed in $D_i$.  

As $e_{i-1}$ is not red in $G$,  Theorem \ref{th:rimColoured} implies $e_{i-1}$ is either $H$-green or $H$-yellow.  If it is $H$-green, then, because we are not in Case 1,   Lemma \ref{lm:staysGreen} implies $e_{i-1}$ is in an $H$-green cycle $C$ contained in $\cl(Q_{i-1})$ and $e_i\in C$.  But then every edge in $\cc{w,r_{i-1},v_i}$ is in two $H$-green cycles, contradicting Theorem \ref{th:twoGreenCycles}.  

We conclude that $e_{i-1}$ is $H$-yellow.  Let $C$ and $C'$ be the witnessing $H$-yellow and $H$-green cycles, respectively, and let $B$ be the global $H$-bridge contained in $C'$.  Lemma \ref{lm:staysYellow} implies $e_i\in C$.  Moreover, $v_{i+5}$ is in the span of $B$, as otherwise $B$ attaches at $v_{i+2}$, contradicting Lemma \ref{lm:globalJumps} (\ref{it:globalNoHnode}).  By Lemma \ref{lm:globalJumps} (\ref{it:spanDiffHyper}), $v_{i+7}$ is not in the span of $B$.  

\dragominor{If $B$ has an end in $\oo{r_{i+2}}$, then the other end of $B$ is $v_{i+5}$.  The $R$-avoiding path (one of $P_2$ and $P_4$ in the decomposition of the $H$-yellow cycle as in Definition \ref{df:yellow}) in $C$ containing $e_i$ contains a positive-length $H$-avoiding subpath joining a vertex of $\oo{s_i}$ to a vertex of $\co{v_{i+4},r_{i+5},v_{i+5}}$.  This yields the contradiction that the edge of $r_{i+4}$ incident with $v_{i+5}$ is in two $H$-green cycles.}
Therefore, $B$ has one attachment in $\co{r_{i+5}\,r_{i+6}}$ and one attachment in $r_{i+3}$.  

 Let $D_{i+1}$ be a 1-drawing of $G-\oo{s_{i+1}}$.   Lemma \ref{lm:staysYellow} implies no $H$-yellow edge in \dragominor{either $r_{i-1}$ or $r_{i+2}$} is crossed in $D_{i+1}$.  An $H$-green edge of $r_{i+2}$ is not spanned by a global $H$-bridge (there is no room for such a jump between $B$ and $wv_{i+2}$), so Lemma \ref{lm:staysGreen} implies no $H$-green edge of $r_{i+2}$ is crossed in $D_{i+1}$.   Because we are not in Case 1 and there is no 3-jump, Lemma \ref{lm:staysGreen} implies no $H$-green edge of either $r_{i-1}$ or $r_{i+2}$ is crossed in $D_{i+1}$.  

Lemma \ref{lm:technicalV8colour} (\ref{it:atMostTwo}) implies no edge of $r_i\,r_{i+1}$ is crossed in $D_{i+1}$.  Thus, none of $r_{i-1}\,r_i\,r_{i+1}\,r_{i+2}$ is crossed in $D_{i+1}$, and therefore $\bQ_{i+1}$ cannot be crossed in $D_{i+1}$. However,  Corollary \ref{co:hyperBOD} and Lemma \ref{lm:BODcrossed} imply that $\bQ_{i+1}$ is crossed in $D_{i+1}$.  This contradiction completes the proof that $G$ has a red edge when there is a 2.5-jump. 
\end{proof}

At this point, we may assume $G$ has no 2.5-jump and no 3-jump.

\begin{claim}  If $G$ has a 2-jump $v_iv_{i+2}$, then either $r_{i-1}$ or $r_{i+2}$ has a red edge.  \end{claim}

\begin{proof}  In this case, let $D_{i+1}$ be a 1-drawing of $G-\oo{s_{i+1}}$. Corollary \ref{co:hyperBOD} and Lemma \ref{lm:BODcrossed} imply that $\bQ_{i+1}$ is crossed in $D_{i+1}$. Lemma \ref{lm:technicalV8colour} (\ref{it:atMostTwo}) shows no edge of $r_i\,r_{i+1}$ is crossed in $D_{i+1}$. Therefore, some edge of $r_{i-1}\cup r_{i+2}$ must be crossed in $D_{i+1}$.
Lemmas \ref{lm:staysYellow} and \ref{lm:staysGreen} imply that no $H$-yellow or $H$-green edge in $r_{i-1}\cup r_{i+2}$ is crossed in $D_{i+1}$.  Therefore, \wording{Theorem \ref{th:rimColoured} shows} some edge in $r_{i-1}\cup r_{i+2}$ is red. \end{proof}

In the final case, there are no global $H$-bridges.  Therefore, there are no $H$-yellow cycles and every $H$-green cycle is contained in $\cl(Q_i)$, for some $i$.  \dragominor{For $j\in\{0,1,2,3,4\}$, let} $e_j$ be the edge in $s_j$ incident with $v_j$ and let $D_j$ be a 1-drawing of $G-e_j$.  Corollary \ref{co:hyperBOD} and Lemma \ref{lm:BODcrossed} imply that $\bQ_{j}$ is crossed in $D_{j}$, so some edge in $r_{j+3}\, r_{j+4}\, r_{j+5}\, r_{j+6}$ is crossed in $D_j$.  Since $e_j$ cannot be in any $H$-green cycle containing an edge in $r_{j+3}\, r_{j+4}\, r_{j+5}\, r_{j+6}$, Lemma \ref{lm:staysGreen} implies no $H$-green edge in $r_{j+3}\, r_{j+4}\, r_{j+5}\, r_{j+6}$ can be crossed in $D_j$.  Therefore the edge in $r_{j+3}\, r_{j+4}\, r_{j+5}\, r_{j+6}$ crossed in $D_j$  is red in $G$.
\end{cproof}

\wording{We conclude this section with the technical lemma (\ref{lm:noAdjacentRed}) below that will be used in the next section.  We start with four lemmas leading to a more refined understanding of $R$-separation in cases of interest for us.  The first three are primarily used in the proof of the fourth. (Recall that an $RR$-path is an $R$-avoiding path with both ends in $R$.)}

}\major{\begin{lemma}\label{lm:RRpathWithGlobal} Let $G\in \m2$ and let $\hvfg$, with $H$ tidy, witnessed by the embedding $\Pi$.  Let $P$ be an $RR$-path in $G$.  If $B$ is a global $H$-bridge so that one end of $P$ is in the interior of the span of $B$, then there is an $H$-quad $Q$ so that $P\subseteq \cl(Q)$  and the two cycles in $R\cup P$ containing $P$ are non-contractible in $\pp$. \end{lemma}}\printFullDetails{

\begin{cproof} \major{\dragominor{As $P$ is $R$-avoiding, T}heorem \ref{th:twoGreenCycles} implies $P$ is not contained in $\disc$.    
 If $P$ is just an $H$-spoke, then both conclusions are obvious.  Otherwise, as we traverse $P$ from an end $u$ in the interior of the span of $B$, there is a first edge $e$ that is not in $H$.  Since $P\subseteq \Mob$, there is an $H$-quad $Q$ so that $e\in \cl(Q)$.   Let $P'$ be the $H$-bridge in $H\cup P$ containing $e$.  Then $P'$ is an $H$-avoiding path with both ends in $H$, so $P'\subseteq \cl(Q)$. }  

\major{\dragominor{Since $P$ is $R$-avoiding, if both ends of $P'$ are in $R$, then} $P'=P$ and $P\subseteq \cl(Q)$, as claimed.  Otherwise, one end $w$ of $P'$ is in the interior of some $H$-spoke $s_i$.  Our two claims eliminate many possibilities for the other end $x$ of $P'$.  } \dragominor{We choose the labelling so that $u\in r_{i-1}\,r_i$.}  

\major{\begin{claim}\label{cl:endNotBarQi}  $x$ is not in $\oo{s_{i-1}\,r_{i-1}\,r_i\,s_{i+1}}$. \end{claim}}

\begin{proof} \major{Suppose first that $u$ is an end of $P'$.  The choice of $e$ implies $P'=\cc{u,P,w}$ is just the edge $e$.    If $u$ is an end of $s_i$, then $e$ is an $H$-bridge having all its attachments in $s_i$, contradicting Lemma \ref{lm:noSpokeOnlyBridge}.  If $u$ is not an end of $s_i$, then there is an $H$-green cycle that contains an edge $f$ of $R$ incident with $u$.  But then $f$ is in two $H$-green cycles, contradicting Theorem \ref{th:twoGreenCycles}.  \dragominor{Thus, $u$ is not an end of $P'$.}}

\major{\dragominor{If $x\in \oo{s_{i-1}\,r_{i-1}\,r_i\,s_{i+1}}-u$, then $u=v_i$ and $P'=\cc{w,P,x}$ is} contained in either $\cl(Q_{i-1})-r_{i+4}$ or $\cl(Q_i)-r_{i+5}$.  In this case, we again have the contradiction that some edge of $R$ incident with $u$ is in two $H$-green cycles.}
\end{proof}

\major{\dragominor{Claim \ref{cl:endNotBarQi} implies $u=v_i$ and} $\cc{u,P,w}\subseteq s_i$.   Moreover, $x$ is in $Q_{i-1}\cup Q_i$ and either $P'\subseteq \cl(Q_{i-1})$ or $P'\subseteq \cl(Q_i)$.   The next claim eliminates another possibility for $x$.}

\major{\begin{claim}\label{cl:xNotInSi}  $x\notin \oo{s_i}$.  \end{claim}}

\begin{proof}  \major{Suppose by way of contradiction that $x\in \oo{s_i}$.  Let $B'$ be the $H$-bridge containing $P'$.  Observe that $B'$ is $H$-local and that $w$ and $x$ are both attachments of $B'$ in $\oo{s_i}$.  Corollary \ref{co:attsMissBranch} implies that these are the only attachments of $B'$, contradicting Lemma \ref{lm:noSpokeOnlyBridge}. }
\end{proof}

\major{We conclude from Claims \ref{cl:endNotBarQi} and \ref{cl:xNotInSi} that  $x$ is in $r_{i+4}\,r_{i+5}$.  Evidently,  $P$ is in $\cl(Q_{i-1})$ or $\cl(Q_i)$, respectively, as required for the first conclusion. \dragominor{Furthermore,  both cycles in $\Pi[R\cup P]$ that contain $P$ are non-contractible in $\pp$.}} 
\end{cproof}

}\major{\begin{lemma}\label{lm:nonCtrPaths}  Let $G\in \m2$ and let $\hvfg$, with $H$ tidy, witnessed by the embedding $\Pi$.  For $i\in \{0,1,2,3,4\}$ and  $j\in \{i+3,i+4,i+5\}$, let $e\in r_i$ and $f\in r_{j}$ be edges that are not $H$-green.  
 Suppose $P$ is an $RR$-path in $\Mob$ having both ends in the  component $R'$ of $R-\{e,f\}$ containing $r_{i+6}\,r_{i+7}\,r_{i+8}\,r_{i+9}$ and so that the  cycle in $\Pi[R'\cup P]$ is non-contractible. 
Then $$
P\subseteq \bigg(\cl(Q_j)-\co{v_{j},s_{j},v_{j-5}}\bigg)\cup \left(\bigcup_{j-5<k<i}\cl(Q_k)\right)\cup \bigg(\cl(Q_i)-\oc{v_{i+6},s_{i+1},v_{i+1}}\bigg)\,.$$
 \end{lemma}}\printFullDetails{

\begin{cproof} \major{Choose the labelling $u$ and $w$ of the ends of $P$ so that $u$ is nearer in $R'$ to the end incident with $f$ than $w$ is.}

\major{Let $\gamma$ be a non-contractible curve meeting $\Pi[G]$ in just the two points $a$ and $b$; we note that $u$ and $w$ are on different $ab$-subpaths of $R$ (allowing $a$ or $b$ to be an end of $P$).  We may choose the labelling of $a$ and $b$ so that $a\in R'$, and if both $a$ and $b$ are in $R'$, then $a$ is closer to the end of $R'$ incident with $f$ than $b$ is.
}

\major{\begin{claim}\label{cl:PmissSpoke} \begin{enumerate}
\item\label{it:vjw} If $v_{j}$ and $w$ are on the same $ab$-subpath of $R$, then $P\cap \oo{s_{j}}$ is empty.   
\item\label{it:vi+1u}  If $v_{i+1}$ and $u$ are in the same $ab$-subpath of $R$, then $P\cap \oo{s_{i+1}}$ is empty.\end{enumerate}
\end{claim}
}

\begin{proof} \major{The statements are symmetric, so it suffices to prove the first.    Suppose to the contrary that $P\cap \oo{s_{j}}$ is not empty.  As we traverse $P$ from $w$ (which is not incident with $s_j$), let $x$ be the first vertex in $\oo{s_{j}}$ and let $P'$ denote the $wx$-subpath of $P$.  Evidently, $P'$ is contained in one component $\Mob'$ of $\Mob\setminus (\gamma\cup s_{j}$).  On the other hand, $f$ is between $v_{j}$ and $w$, and so $f$ is in $\Mob'$.  If $w\in r_j$, then $P'$ and $f$ are in an $H$-green cycle, a contradiction. }

\major{Otherwise, $v_{j+1}\in \Mob'$ and $P'$ intersects $\oo{s_{j+1}}$.  In this case, some $\oo{s_{j+1}}\oo{s_j}$-subpath of $P'$ is in an $H$-green cycle with $f$, also a contradiction.}
\end{proof}

\major{If both $v_{i+1}$ and $w$ are in the same $ab$-subpath of $R$ and both $v_{i+4}$ and $u$ are in the same $ab$-subpath of $R$, then Claim \ref{cl:PmissSpoke} implies $P$ is trapped between $s_j$ and $s_{i+1}$, as required.  By symmetry,  we may assume that $v_{i+1}$ is not in the same $ab$-subpath of $R$ as $w$. \dragominor{Let $R_w$ denote the $ab$-subpath of $R$ containing $w$ and let $R_{i+1}$ denote the other $ab$-subpath of $R$, so $v_{i+1}\in R_{i+1}$.}}

\major{This implies that $v_{i+1}$, $v_{i+2}$, \dots, $v_{j}$ are all in $R_{i+1}$.  We noted above that $u\notin R_w$, so  $u$ is also in $R_{i+1}$.  From Claim \ref{cl:PmissSpoke} (\ref{it:vjw}), we conclude that $P$ is disjoint from $\oo{s_{j}}$.  Thus, $P$ is contained in the component of $\Mob-s_{j}$ disjoint from $v_{i+2}$.  }

\major{It follows from the fact that all the $H$-spokes are in $\Mob$ that $v_{j-5}$ is on the same $ab$-subpath as $w$.  This combines with the fact that $v_{i+1}$ is not in that $ab$-subpath and the fact that $P$ meets $\gamma$ at most in $a$ to tell us that $$P\subseteq \bigg(\cl(Q_{j})-\oo{s_{i-1}}\bigg)\cup\left(\bigcup_{j-5<k<i}\cl(Q_k)\right)\cup \bigg(\cl(Q_i)-\oo{s_{i+1}}\bigg)\,,$$  as required.}

\major{The additional fact that $P$ cannot include $v_j$ and $v_{i+1}$ follows from the knowledge that these vertices are not in $R'$.} \end{cproof}

\dragominor{In a similar vein, we have the following.}

}\major{\begin{lemma}\label{lm:shortSideOfSeparator}  \dragominor{Let $G\in \m2$ and let $\hvfg$, with $H$ tidy as witnessed by the embedding $\Pi$.   Suppose $e\in r_{i}$, $f\in r_{i+3}\,r_{i+4}$ and $P$ is an $RR$-path with both ends in the component of $R-\{e,f\}$ containing $r_{i+1}\,r_{i+2}$.  If $e$ is not $H$-green, then both cycles in $\Pi[R\cup P]$ containing $P$ are contractible. }\end{lemma}}\printFullDetails{

\begin{cproof}\dragominor{ Let $R'$ be the component of $R-\{e,f\}$ containing $r_{i+1}\,r_{i+2}$ and let $C$ be the cycle in $R\cup P$ that contains $P$ and is contained in $R'\cup P$.   Since $R$ is contractible, the other cycle in $R\cup P$ containing $P$ is homotopic to $C$; thus, it suffices to show $C$ is contractible.}

\dragominor{Let $u_e$ be the end of $R'$ incident with $e$.  Suppose there is
a  $(\cc{u_e,r_i,v_{i+1}}s_{i+1})s_i$-path $P'$ in $P$ contained in $\cl(Q_i)$.   Since $C$ is disjoint from $r_{i+1}$, $P'$ is contained in an $H$-green cycle containing $e$, a contradiction.}

\dragominor{Thus, there is no $(\cc{u_e,r_i,v_{i+1}}s_{i+1})s_i$-path in $P$ contained in $\cl(Q_i)$.  Since $C$ is disjoint from $r_{i+5}$, there is an arc in the disc bounded by $\Pi[Q_i]$ joining a point of $\co{v_i,r_i,u_e}$ to $r_{i+5}$ that is disjoint from $C$; this shows that $C$ is contractible, as required.}\end{cproof}

Our next lemma takes us one step closer to the useful description of $R$-separation.

}\major{\begin{lemma}\label{lm:shapeOfV8separator} Let $G\in \m2$ and let $\hvfg$, with $H$ tidy.   Suppose $e\in r_{i}$ and $f\in r_{i+3}\,r_{i+4}$ are $R$-separated as witnessed by the subdivision $H'$ of $V_8$.  If $e$ is not $H$-green, then the component of $R-\{e,f\}$ containing both ends of some $H'$-spoke is the one containing $r_{i+5}\,r_{i+6}\,r_{i+7}\,r_{i+8}\,r_{i+9}$.
\end{lemma}}\printFullDetails{

\begin{cproof} \dragominor{Let $\Pi$ be an embedding of $G$ in $\pp$ so that $H$ is $\Pi$-tidy.}  \major{\dragominor{Recall that $R$ is also the $H'$-rim.  Observation \ref{obs:separated} (\ref{it:twoSpokes}) shows that two of the four $H'$ spokes have all their ends in the same component of $R-\{e,f\}$.  Of the four $H'$-spokes, at most one can be in $\disc$.  Thus, of the two that have both ends in the same component $R'$ of $R-\{e,f\}$, there is at least one, call it $s$, that is in $\Mob$.}}

\dragominor{  In particular, the two cycles in $R\cup s$ containing $s$ are non-contractible.  Now Lemma \ref{lm:shortSideOfSeparator} shows the two ends of the $RR$-path $s$ are not in the component of $R-\{e,f\}$ containing $r_{i+1}\,r_{i+2}$ and so must be in the component containing $r_{i+5}\,r_{i+6}\,\dots\,r_{i+9}$, as claimed.
}\end{cproof}

\wording{\dragominor{Our next lemma in the series gives a quite refined description of $R$-separation.}}

}\major{\begin{lemma}\label{lm:refinedRseparation}  Let $G\in \m2$ and let ${\hvfg}$, with $H$ tidy.     Let $e\in r_i$ and $f\in r_{i+4}\, r_{i+5}$ be edges that are both not $H$-green.  If $e$ and $f$ are $R$-separated in $G$, then there is a witnessing subdivision $H'$ of $V_8$ having $s_{i+2}$ and $s_{i+3}$ as $H'$-spokes and the other two $H'$-spokes are in $\cl(Q_{i-1})\cup \cl(Q_i)$.
 \end{lemma}}\printFullDetails{

\begin{cproof} \major{ Let $\Pi$ be an embedding of $G$ in $\pp$ for which $H$ is $\Pi$-tidy.   Let $H_1$ be a subdivision of $V_8$ witnessing the $R$-separation of $e$ and $f$.  Let $s$  be an $H_1$ spoke having both ends in the same component $R'$ of $R-\{e,f\}$. }

\begin{claim} \major{The cycles in \dragominor{$\Pi[R\cup s]$} containing $s$ are non-contractible.}\end{claim}

\begin{proof} \major{Suppose first by way of contradiction that $\Pi[s]$ in not contained in $\Mob$.  Since $H$ is $\Pi$-tidy, $s$ is a global $H$-bridge.  Theorems \ref{th:globalBridges} and \ref{th:no3jump} show $s$ is either a \dragominor{2- or  a 2.5-jump}.  By hypothesis, it is not possible for both $e$ and $f$ to be in the span of $s$ and, therefore, neither is.  On the other hand, each of the other three $H_1$-spokes has precisely 1 end in the span of $s$, and is contained in $\Mob$.    Let these spokes appear in the order $t_1,t_2,t_3$ in the span of $s$.}

\major{We claim that the $t_i$ imply the existence of an $H$-yellow cycle that does not bound a face of $\Pi[G]$, contradicting Lemma \ref{lm:yellowCycles} (\ref{it:oneBridge}).    Let $P$ be the span of $s$ and, for $i=1,2,3$, let $u_i$ be the end of $t_i$ that is not in $P$.  Because $\Pi[s\cup P]$ bounds a closed disc, \dragominor{both cycles in $\Pi[R\cup t_i]$ containing $t_i$ are non-contractible.} Thus, $t_i$ has an end in each of the $ab$-subpaths of $R$.}

\major{Lemma \ref{lm:RRpathWithGlobal} implies that each $t_i$ is contained in an $H$-quad.  Thus $t_1\cup t_2\cup t_3$ is contained the the union of the closures of the $H$-quads that have an edge in $P$.  In particular, $u_1$, $u_2$, and $u_3$ occur in a 3-rim path $P_1$ having $u_1$ and $u_3$ as ends.  Letting $P_3$ be the minimal subpath of $P$ containing the ends of the $t_i$, we see that $P_1\,t_1\,P_3\,t_3$ is an $H$-yellow cycle $C$.  \dragominor{However, $\Pi[C]$ bounds a face of $\Pi[G]$; the contradiction is that $t_2$ and $s$ are on different sides of $\Pi[C]$. } }

\major{Thus, $s$ is contained in $\Mob$.   Since $s$ is one of four $H_1$-spokes, the two cycles in $\Pi[R\cup s]$ that contain $s$ are non-contractible.}\end{proof}

\major{In particular, $s$ has an end in each of the two $ab$-subpaths of $R$ determined by the standard labelling of $\Pi[G]$.  }

\major{In the case $f\in r_{i+5}$, we may,  if necessary, use the reflective symmetry $j\leftrightarrow 4-j$ (for $0\le j\le 4$), to arrange that the end $s_f$ of $s$ is, in $\Pi[R']$, between the end $u_f$ of $f$ in $R'$ and $a$, say, while the other end $s_e$ of $s$ is between $a$ and the end $u_e$ of $e$.   In particular, $v_{i+1}$, $v_{i+2}$, $v_{i+3}$, and $v_{i+4}$ are not in $R'$.  Lemma \ref{lm:shapeOfV8separator}  shows this always holds when $f\in r_{i+4}$.}

\major{\dragominor{Let $s'$ be the other $H'$-spoke having both ends in $R'$.  The arguments above for $s$ apply equally well to $s'$.  Lemma \ref{lm:nonCtrPaths} shows that $(s\cup s')\subseteq \cl(Q_{i-1})\cup \cl(Q_i)$. In particular, $s$ and $s'$ are disjoint from $s_{i+2}$ and $s_{i+3}$, so these $H$-spokes may replace the two $H_1$-spokes having ends in both components of $R-\{e,f\}$, as required.}}
\end{cproof}

\wording{The final technical lemma of this section will be used in the next.}

}\begin{lemma}\label{lm:noAdjacentRed}  Let $G\in\m2$ and $\hvfg$, with $H$ tidy.  If $e$ and $e'$ are red edges in the same $H$-rim branch, then $\Delta_e$ and $\Delta_{e'}$ are disjoint.  \end{lemma}\printFullDetails{

\begin{cproof}  \wordingrem{(Text removed.)}We may choose the labelling of $e$ and $e'$ so that $e=uw$ and $e'=xy$ are such that $r_i =\cc{v_i,r_i,u,w,r_i,x,y,r_i,v_{i+1}}$.  As we follow $\Delta_e-e$ from $w$, there is a first edge $f$ that is not in $R$.  \wording{ In fact, Theorem \ref{th:redHasDelta} (\ref{it:deltaDetail}) implies $f$ is incident with $w$, as there can be no global $H$-bridge spanning $e'$.}

\wording{  Observe that $f$ is not in $H$, so $H\subseteq G-f$.  Moreover, if $f$ is in an $H$-yellow cycle, then either $e$ or $e'$ is $H$-yellow, a contradiction.  Thus, Lemmas \ref{lm:staysYellow} and \ref{lm:staysGreen} imply the colours of an edge of $R$ are the same in $G$ and $G-f$, unless
the edge is in an $H$-green cycle in $G$ that contains $f$.   Such an edge is necessarily in $[w,r_i,x]$.}

%
%
\major{Let $D$ be a 1-drawing of $G-f$ and let $e_1$ and $e_2$ be the edges of $G-f$ crossed in $D$. Since $f$ is incident with $w\in \oo{r_i}$, Theorem \ref{th:BODquads} and Lemma \ref{lm:BODcrossed} imply that $Q_i$ is crossed in $D$, so we may assume $e_1\in r_{i-1}\,r_i\,r_{i+1}$ and $e_2\in r_{i+4}\,r_{i+5}\,r_{i+6}$.  Moreover, no $H$-green cycle containing $e_2$ contains $f$, so $e_2$ is red in $G$.   In particular, Lemma \ref{lm:notSeparated} implies $e_2$ is $R$-separated from both $e$ and $e'$.}

\major{Let $u^e$ and $w^e$ be the first vertices in $r_{i+5}$ as we traverse $\Delta_e-e$ from $u$ and $w$, respectively.  Likewise, we have $x^{e'}$ and $y^{e'}$ in  $r_{i+5}\cap (\Delta_{e'}-e')$.}

\begin{claim}\label{cl:e2location}
\major {\dragominor{ $e_2\in  \cc{u^{e},r_{i+5},y^{e'}}$.}}\end{claim}

\begin{proof}\major{Suppose by way of contradiction that $e_2\in r_{i+4}\cc{v_{i+5},r_{i+5},u^e}$; a similar argument will treat the case $e_2\in \cc{y^{e'},r_{i+5},v_{i+6}}r_{i+6}$.  }

\major{If $e_1\in r_{i-1}\cc{v_i,r_i,u}$, then $e_1$ is red in $G$, so $e_1$ and $e_2$ are $R$-separated in $G$.  Note that either $e_1\in r_i$ or $e_2\in r_{i+5}$.   Lemma \ref{lm:refinedRseparation} implies there is a witnessing subdivision $H'$ of $V_8$ that contains $s_{i+2}$ and $s_{i+3}$, while the other two spokes are in $\cl(Q_{i-1})\cup \cl(Q_i)$.  Furthermore, $\Delta_e$ shows that $f\notin H'$; therefore, $H'\subseteq G-f$ shows that $e_1$ and $e_2$ are $R$-separated in $G-f$, and therefore cannot cross in $D$, a contradiction.}  

\major{The other possibility is that $e_1\in \cc{u,r_i,v_{i+1}}r_{i+1}$.  Since $e$ and $e_2$ are both red in $G$, Lemma \ref{lm:notSeparated} implies $e_2$ is $R$-separated from $e$ in $G-f$.  As in the preceding paragraph, we may choose the witnessing subdivision $H'$ of $V_8$ to contain $s_{i+2}$ and $s_{i+3}$, while the other two spokes are in $\cl(Q_{i-1})\cup (\cl(Q_i)-f)$.  Again $H'$ witnesses the $R$-separation of $e_1$ and $e_2$ in $G-f$, a contradiction.  } \end{proof}

\major{\dragominorrem{(Text removed.)}Theorem \ref{th:redHasDelta} (\ref{it:digon}) shows that any edge in either $\Delta_e\cap r_{i+5}$ or $\Delta_{e'}\cap r_{i+5}$ is in a digon in $G$ and so is not $e_2$.  Thus, $e_2$ is further restricted to be in \dragominor{$\cc{w^e,r_{i+5},x^{e'}}$}.  Lemma \ref{lm:notSeparated} implies $\Delta_e$ and $\Delta_{e_2}$ are disjoint, as are $\Delta_{e_2}$ and $\Delta_{e'}$, which further implies that $\Delta_e$ and $\Delta_{e'}$ are disjoint, as required.}
%
%
%
%
%
%
%
%
  \end{cproof}
  }

\chapter{The next red edge and the tile structure}\printFullDetails{\label{sec:nextRed+Tiles}

We now know that there are red edges and every red edge comes equipped with a $\Delta$.  The tiles are determined by what is between ``consecutive" red edges.   In this section, we explain what ``consecutive" means, show that consecutive red edges determine one of the tiles, and complete the \wording{proof of our main result\dragominor{, Theorem \ref{th:classification},}} by demonstrating that every red edge has a consecutive red edge on each side.  

}\begin{definition}\label{df:consecutive}  Let $G\in\m2$ and $\hvfg$, with $H$ tidy.  Let $e=uw$ be a red edge in $r_i$, labelled so that $r_i=\cc{v_i,r_i,u,e,w,r_i,v_{i+1}}$.  
A red edge $e_w$  is {\em $w$-consecutive for $e$\/}\index{consecutive}\index{$w$-consecutive}\index{$u$-consecutive} if:
\begin{enumerate}\item\label{it:whereNext} $e_w\in \cc{w^e,r_{i+5},v_{i+6}}\rbsp r_{i+6}\,r_{i+7}$ (recall that $w^e$ is the vertex in the peak of $\Delta_e$ nearest $w$ in $\Delta_e-e$);
\item\label{it:betweenPeaks1} there is no red edge in  $\cc{w^e,r_{i+5},v_{i+6}}\rbsp r_{i+6}\,r_{i+7}$ between $w^e$ and $e_w$; 
\item\label{it:betweenPeaks2} there is no red edge in $\cc{w,r_i,v_{i+1}}\rbsp r_{i+1}\,r_{i+2}$ between $w$ and the peak of $\Delta_{e_w}$;
\item\label{it:consec1drawing} \wording{if $e^w$ is} the edge of $P_w$ nearest $w$ that is not in $R$, then there is a 1-drawing $D$ \wording{of $G-e^w$ in} which $e$ crosses $e_w$.
\wording{\item  There is an analogous definition for {\em $u$-consecutive}.}
\end{enumerate}
\end{definition}\printFullDetails{

Our first main goal is, therefore, the following.

}\begin{theorem}\label{th:consecRed}  Let $G\in \m2$ and $\hvfg$, with $H$ tidy.  Let $e=uw$ be red in $G$.  Then there is a $w$-consecutive red 
edge \wording{and a $u$-consecutive red edge for $e$}.
\end{theorem}\printFullDetails{

The next lemma will be helpful in the proof.

}\begin{lemma}\label{lm:separated} Let $G\in \m2$ and $\hvfg$, with $H$ tidy.   Let $e=uw$ and $\hat e$ be red edges in $G$, with $e\in r_i$ \wording{and the labelling chosen so that} $r_i=\cc{v_i,r_i,u,e,w,r_i,v_{i+1}}$\wordingrem{(comma removed)} and $\hat e\in \cc{w^e,r_{i+5},v_{i+6}}\rbsp r_{i+6}\,r_{i+7}$.  \wording{If $e^w$ is} the $w$-nearest edge of $P_w$ that is not in $R$ and $e$ and $\hat e$ are not $R$-separated \wording{in $G-e^w$, then} $e$ has a $w$-consecutive red edge. \end{lemma}\printFullDetails{

\begin{cproof}  \dragominor{Suppose there is a red \wording{edge $e'$} in $r_i\,r_{i+1}\,r_{i+2}$} between $w$ and the peak of $\Delta_{\hat e}$.  \wording{Then $e'$} is $R$-separated from $\hat e$ in both $G$ \wording{and $G-e^w$}, showing that $e$ and $\hat e$ are $R$-separated \wording{in $G-e^w$}, a contradiction.  \dragominor{Thus, no such red edge exists.} 
 
\dragominorrem{(Text removed.)}\wording{Let $\hat e'$} be the $w^e$-nearest red edge in $\cc{w^e,r_{i+5},v_{i+6}}\rbsp r_{i+6}\,r_{i+7}$. 
Lemma \ref{lm:notSeparated} \wording{implies $\hat e'$} is $R$-separated from $e$ in $G$; \wording{if $\hat e'$} were also $R$-separated from $e$ \wording{in $G-e^w$}, then so would $\hat e$, which contradicts the hypothesis.  But now Lemma \ref{lm:notSeparated} implies there is a 1-drawing \wording{of $G-e^w$} in which $e$ \wording{crosses $\hat e'$}, as required. \end{cproof}

And now the final major proof \wording{needed to prove Theorem \ref{th:classification}.}

\begin{cproofof}{Theorem \ref{th:consecRed}}  \wording{It obviously suffices to prove the existence of a $w$-consecutive red edge for $e$.}  Let $r_i$ be the $H$-rimbranch containing $e$.   Let $e^w$ be the edge of  $P_w$ nearest $w$ and not in $R$.  There are two principal cases.

\medskip\noindent{\bf Case 1:}  {\em $e^w$ is incident with $w$.}

\medskip
We note that $e^w$ is contained in a $\bQ_{i+1}$-bridge that is not $M_{\bQ_{i+1}}$.  Let $D$ be a 1-drawing of $G-e^w$.  Corollary \ref{co:hyperBOD} and Lemma  \ref{lm:BODcrossed} show that $\bQ_{i+1}$ is crossed in $D$.  

Let
\begin{itemize}\item  $f$ be the edge of $r_{i+4}\,r_{i+5}\,r_{i+6}\,r_{i+7}$ that is crossed in $D$ and \item  $f'$ be the other edge crossed in $D$; thus, $f'\in r_{i-1}\,r_i\,r_{i+1}\,r_{i+2}$. 
\end{itemize}

\begin{claim}\label{cl:fIsRed}  If $f$ is not red \wording{in $G$}, then there is a $w$-consecutive red edge for $e$.  \end{claim}

\begin{proof} \startSubclaims Because we are in Case 1, no global $H$-bridge has $w$ in the interior of its span and, therefore,  $e^w$ is not in any $H$-yellow cycle that could witness the $H$-yellowness of any edge in $r_{i+4}\,r_{i+5}\,r_{i+6}\,r_{i+7}$, (in particular, the $H$-yellowness of $f$).  Therefore, Lemma \ref{lm:staysYellow} shows $f$ is not $H$-yellow.  \dragominor{Since $f$ is not red,  Theorem \ref{th:rimColoured}} implies $f$ is $H$-green.  Lemma \ref{lm:staysGreen} implies there is a 2.5-jump $J$ that spans $f$ and so that $f$ is in the $H$-rim \dragominor{branch} whose interior contains an end of $J$. We note that if $v_{i+6}$ is in the span of $J$, then Lemma \ref{lm:technicalV8colour} \dragominor{(\ref{it:atMostTwo})} shows no edge in the span of $J$ is crossed in $D$.  Therefore, $v_{i+6}$ is not in the span of $J$.  Furthermore, if $e^w$ is not in $s_{i+1}$, then $H\subseteq G-e^w$ and, therefore Lemma \ref{lm:greenCycles} (\ref{it:notCrossed}) implies $f$ is not crossed in $D$, a contradiction.  This implies $w=v_{i+1}$.   \wording{We summarize these remarks as follows.}

\begin{subclaim}\label{cl:redOr2.5jump}  
\begin{itemize}
\item \wordingrem{(text moved)}\wording{$w=v_{i+1}$ and}
\item there is a 2.5-jump $J$ so that:
\begin{itemize} \item $f$ is spanned by $J$;
\item $f$ is in the $H$-rim branch whose interior contains an end of $J$; \wording{and}
\item $v_{i+6}$ is not in the span of $J$. \wordingrem{(semi-colon becomes period.)} \hfill$\Box$
\end{itemize}
\end{itemize}
\end{subclaim}

\begin{subclaim}\label{cl:restOFfBranchNotYellow}  Let $j\in \{i+4,i+5,i+6,i+7\}$ so that $f$ is in the $H$-rim branch $r_j$.  Then no edge of $r_j$ is $H$-yellow. \end{subclaim}

\begin{proof}  Suppose some edge $e'$ of $r_j$ is $H$-yellow.  This implies $e'$ is not $H$-green and, therefore, is not spanned by $J$. \wordingrem{(Text removed.)}%
Let $C$ and $C'$ be the witnessing $H$-yellow and $H$-green cycles, respectively.  

Suppose first that $j\in \{i+4,i+5\}$.  \minor{Then $r_j=\cc{v_j,r_j,f,r_j,e',r_j,v_{j+1}}$.}  \wording{Because $e\in r_i$} is not $H$-green, \wording{$v_{j+5}\in \{v_{i-1},v_i\}$} is in the interior of $C'\cap R$.  This implies there is an $H$-yellow cycle containing $s_j$ and the portion of $r_j$ from $v_j$ to $e'$.  By Lemma \ref{lm:yellowCycles} (\ref{it:oneBridge}), this $H$-yellow cycle must be $C$ and, therefore, $f\in C$.  Now the fact that $f$ is crossed in $D$ contradicts Lemma \ref{lm:staysYellow}.  A completely analogous argument holds for $j\in \{i+6,i+7\}$.
\end{proof}

Let $\widehat w$ be the vertex in $r_{i+5}$ that is nearest $w$ in $P_w$.  Observe that $\widehat w$ is not necessarily in the peak of $\Delta_e$.  (See Figure \ref{fg:delta}, where $\widehat w$ is the vertex of $\Delta_e$ at the top right hand corner of $\Delta_e$.)  The following claim will be helpful in completing the proof of Case 1.

\begin{subclaim}\label{cl:hatWv+6green} If $\widehat w\ne v_{i+6}$, then $\cc{\widehat w,r_{i+5},v_{i+6}}$ is in an $H$-green cycle contained in $\cl(Q_i)$. \end{subclaim}

\begin{proof}  Let $P'_w$ be the $\widehat ws_{i+1}$-subpath of $P_w$.  Since $e^w\in s_{i+1}$, $P'_w\subseteq P_w-w$.  Let $\widehat w^e$ be the end \wording{of $P'_w$ in} $s_{i+1}$.  Since $\widehat w\notin s_{i+1}$ and $\widehat w^e\in s_{i+1}$,  $\widehat w\ne \widehat w^e$.  By definition of $\widehat w$, $P'_w-\widehat w$ is disjoint from $r_{i+1}$.  Therefore,  $P'_w\lbsp\cc{\widehat w^e,s_{i+1},v_{i+6},r_{i+5},\widehat w}$ is an $H$-green cycle containing $\cc{\widehat w,r_{i+5},v_{i+6}}$, as required.  \end{proof}

The proof \wording{of Claim \ref{cl:fIsRed}} is completed now by treating separately each of the four possibilities for $f$:  $f\in r_{i+4}$, $f\in r_{i+5}$, $f\in r_{i+6}$, and $f\in r_{i+7}$.

\medskip\noindent{\bf Subcase 1:}  $f\in r_{i+4}$.

\medskip In this case, $J$ has an end $x'\in \oo{r_{i+4}}$ and the other end of $J$ is $v_{i+2}$.   Lemma \ref{lm:technicalV8colour} (\ref{it:2halfJump3/2}) implies $f'\in r_i\,r_{i+1}$.  We claim that if $f'\in r_{i+1}$, then there is another 1-drawing of $G-e^w$ in which $f$ crosses $e$.

Since $f\in r_{i+4}$ and $f'\in r_{i+1}$, we see that $s_i$ is exposed in the 1-drawing $D$ of $G-e^w$.  \dragominor{Note that $D[Q_{i-1}]$ consists} of a simple closed curve crossed by $D[f']$, with $D[r_i]$ on one side (the {\em inside\/} of $D[Q_{i-1}]$) and most of $D[H]$ on the other side (this is the {\em outside\/} of $D[Q_{i-1}]$).  

We claim that we may reroute $f$ inside $D[Q_{i-1}]$ so that it crosses $e$ instead of $f'$.  \minor{If this fails, then there is an $(H-\oo{s_{i+1}})$-avoiding path $P$ having one end in the component of $r_{i+1}-f'$ that contains $v_{i+1}$, and having its other end in $Q_{i-1}\cup \cc{v_i,r_i,u}$.}  

We note that $D[s_{i+1}-v_{i+1}]$ (which is possibly just $v_{i+6}$) is completely outside $D[Q_{i-1}]$.  Therefore, $P$ is $H$-avoiding.   In $\pp$, we conclude that $P$ cannot start inside $Q_{i+1}$.  Thus, $P$ is contained in a global $H$-bridge.  Therefore, $P$ is a global $H$-bridge; we note that $P$ has one end in the component of $r_{i+1}-f$ containing $v_{i+1}$.  No edge of $r_{i+2}$ can be spanned by $P$, as such an edge is already spanned by $J$ and therefore would contradict Theorem \ref{th:twoGreenCycles}.  In the other direction, $P$ cannot span $e$, as $e$ is red and not $H$-green.  This contradiction shows that $f$ may be redrawn as claimed.  Consequently, we may assume $f'\in r_i$.

\wording{Observe} that no global $H$-bridge can have an end $y$ in $\oo{r_i}$, since $yv_{i+3}$ shows $e$ is $H$-green, a contradiction, and $yv_{i-2}$ shows $f$ is $H$-yellow and, therefore, by Lemma \ref{lm:staysYellow} cannot be crossed in $D$.
It follows from this, using Lemmas \ref{lm:staysYellow} and \ref{lm:staysGreen} and Theorem \ref{th:rimColoured}, that $f'$ is red in $G$.

Suppose first that some edge $e'$ of $\cc{x',r_{i+4},v_{i+5}}$ is red in $G$.  Then $\Delta_e$ and $\Delta_{e'}$ are $R$-separated in $G$ \wording{as witnessed by a subdivision $H'$ of $V_8$ consisting of $R$}, $s_{i-3}$, $s_{i-2}$, and two $RR$-paths $P_1$ and $P_2$, contained in $\Delta_e$ and $\Delta_{e'}$, respectively.  
The paths $P_1$ and $P_2$ are disjoint from $s_{i+1}$ except that, possibly $P_1$ contains $v_{i+6}$.   \wording{Thus, 
$H'$ and Lemma \ref{lm:technicalV8colour} show that} $f$ cannot be crossed in $D$, a contradiction.  Therefore, there is no red edge in $\cc{x',r_{i+4},v_{i+5}}$.  

Furthermore, no global $H$-bridge other than $J$ has an end in $\co{x',r_{i+4},v_{i+5}}$, as otherwise either $e$ is $H$-yellow, or $f$ is in two $H$-green cycles, both contradictions, the latter of Theorem \ref{th:twoGreenCycles}.  We conclude that each edge of $\cc{x',r_{i+4},v_{i+5}}$ is either $H$-yellow or contained in an $H$-green cycle in $\cl(Q_{i-1})$.   
\wording{Subclaim \ref{cl:restOFfBranchNotYellow} shows the following.}

\bigskip\noindent\wording{{\bf Subcase 1 Observation:} {\em Each edge of $\cc{x',r_{i+4},v_{i+5}}$ is in an $H$-green cycle contained in $\cl(Q_{i-1})$. }}

\bigskip  

Suppose there is a red edge $e'$ in $r_{i+5}$.  By \dragominor{Lemma \ref{lm:notSeparated}, $e'$} is $R$-separated from $e$ in $G$.  Therefore, $P_u$ is disjoint from $s_i$ and now we see that $G-e^w$ contains the subdivision $H'$ of $V_{10}$  consisting of $(H-\oo{s_{i+1}})\cup P_u$.   But $J$ is in an $H'$-green cycle $C$ and so, by Lemma \ref{lm:greenCycles} (\ref{it:notCrossed}), $C$, and in particular, $f$, is not crossed in $D$, a contradiction.

Thus, no edge of $r_{i+5}$ is red in $G$.  We consider next a 1-drawing $D_{i-1}$ of $G-\oo{s_{i-1}}$.  By Corollary \ref{co:hyperBOD} and Lemma \ref{lm:BODcrossed}, $\bQ_{i-1}$ is crossed in $D_{i-1}$.  From Lemmas \ref{lm:staysYellow}, \ref{lm:staysGreen}, and \ref{lm:technicalV8colour} (\ref{it:atMostTwo}), no edge in $r_{i+2}\,r_{i+3}\,r_{i+4}$ is crossed in $D_{i-1}$.  Therefore, it is some edge $f''$ in $r_{i+5}$ that is crossed in $D_{i-1}$.   Since no edge of $r_{i+5}$ is red in $G$, Lemmas \ref{lm:staysYellow} and \ref{lm:staysGreen} imply that $f''$ is spanned by a 2.5-jump $J''=x''v_{i-2}$, with $x''\in \oo{r_{i+5}}$.  

Now consider a 1-drawing $D_{i+3}$ of $G-\oo{s_{i+3}}$.  As for $\bQ_{i-1}$ in the preceding paragraph, $\bQ_{i+3}$ is crossed in $D_{i+3}$.  In this \minor{case, $r_{i+1}$ is} contained in the $H$-yellow \minor{cycle $Q_{i+1}$}  (with witnessing $H$-green cycle containing $J''$).  \minor{Therefore, $r_{i+1}$} is not crossed in $D_{i+3}$.  Lemma \ref{lm:technicalV8colour} (\ref{it:atMostTwo}) implies no edge in the span of $J$ is crossed in $D_{i+3}$.  \wording{Subcase 1 Observation combines with Lemma \ref{lm:staysGreen} to} show that no edge in $\cc{x',r_{i+4},v_{i+5}}$ is crossed in $D_{i+3}$.  But now we see that $\bQ_{i+3}$ cannot be crossed in $D_{i+3}$, a contradiction that shows Subcase 1 cannot occur.

\medskip\noindent{\bf Subcase 2:}  $f\in r_{i+5}$.

\medskip In this case, $J$ has an end $x'\in \oo{r_{i+5}}$.  \dragominor{Subclaim \ref{cl:redOr2.5jump} implies that $v_{i+6}$ is not spanned by $J$, so the other end of $J$ is $v_{i+3}$.} Lemma \ref{lm:technicalV8colour} implies the edge $f'$ (crossed by $f$ in $D$) is in $r_{i+2}$.

We first show that there is no global $H$-bridge spanning any edge in $r_i\,r_{i+1}\,r_{i+2}$.  For if $J'$ is a global $H$-bridge that spans such an edge, then $J'$ does not span $e$, while Lemma \ref{lm:globalJumps} (\ref{it:globalNoHnode}) shows it cannot be the 2-jump $v_{i+1}v_{i+3}$.  Theorem \ref{th:twoGreenCycles} shows $J'$ cannot span any edge in $r_{i+3}$, so no edge of $r_{i+1}\,r_{i+2}$ is spanned by a global $H$-bridge.
On the other side, $J'$ would have to span $r_{i-2}\,r_{i-1}$.  In that case, $J$ and $J'$ contradict Lemma \ref{lm:globalJumps} (\ref{it:spanDiffHyper}).  

We also conclude that no edge of $r_{i+5}\,r_{i+6}\,r_{i+7}$ is $H$-yellow.  

Our next principal aim is to show that each edge of $\cc{x',r_{i+5},v_{i+6}}$ is $H$-green, witnessed by a cycle in $\cl(Q_i)$.  
We have already seen that none of the edges in $\cc{x',r_{i+5},v_{i+6}}$ is $H$-yellow; to see they are $H$-green, it suffices by Theorem \ref{th:rimColoured} to show none is red.

 If $e'$ is one of these edges that is red, then Lemma \ref{lm:notSeparated} implies it is $R$-separated from $e$.  We note that $\Delta_e$ and $\Delta_{e'}$ are disjoint, both are in $\cl(Q_i)$, and $w=v_{i+1}$.  Therefore, $e'$ is in $r_{i+5}$, between $x'$ and the peak of $\Delta_e$.   However, this shows $e'$ and $e$ are $R$-separated in $G-e^w$ and, therefore, $f$ and $r_{i+2}$ are $R$-separated in $G-e'$, showing that $f$ cannot cross anything in $D$, a contradiction.  Therefore, no edge of $\cc{x',r_{i+5},v_{i+6}}$ is red, and so they are all $H$-green.
 
\wording{We next show they are not spanned by a global $H$-bridge.  Recall that $\widehat w$ is the vertex in $r_{i+5}$ that is nearest $w$ in $P_w$.}

 If $\widehat w\ne v_{i+6}$, then \dragominor{$(P_w-e^w)\cup (s_{i+1}-e^w)\cup \cc{\widehat w,r_{i+5},v_{i+6}}$ contains an $H$-green cycle that contains $\cc{\widehat w,r_{i+5},v_{i+6}}$ and} is contained in $\cl(Q_i)$.  Theorem \ref{th:twoGreenCycles} shows no edge of $\cc{\widehat w,r_{i+5},v_{i+6}}$ is spanned by a global $H$-bridge, so no edge of $\cc{x',r_{i+5},v_{i+6}}$ is $H$-green by a global $H$-bridge.    In this case, every edge of  $\cc{x',r_{i+5},v_{i+6}}$ is $H$-green by a local cycle.  

So suppose $\widehat w=v_{i+6}$.  By way of contradiction, we suppose there is a global $H$-bridge $J''$ spanning the edge of $r_{i+5}$ incident with $v_{i+6}$. Then $J''$ must be $x''v_{i+8}$, for some $x''\in \cc{x',r_{i+5},v_{i+6}}$.  All edges in $\cc{x',r_{i+5},x''}$ are $H$-green by local cycles.  For $j\in\{i+3,i+8\}$, let $e_{j}$ be the edge of $s_{i+3}$ incident with $v_{j}$ and let $D_{j}$ be a 1-drawing of $G-e_{j}$.  Corollary \ref{co:hyperBOD} implies $\bQ_{i+3}$ has BOD and Lemma \ref{lm:BODcrossed} implies $\bQ_{i+3}$ is crossed in $D_j$.  Lemma \ref{lm:technicalV8colour} (\ref{it:2halfJump1}) implies neither $r_{i+6}\,r_{i+7}$ nor $r_{i+3}\,r_{i+4}$ is crossed in $D_j$, while (\ref{it:V8yellow}) of the same lemma implies neither $r_{i+9}$ nor $r_{i+1}$ is crossed in $D_j$.  Therefore, $r_{i+8}$ crosses $r_{i+2}$.  

If the edge $e'_{i+8}$ of $r_{i+8}$ that is crossed in $D_{i+3}$ is $H$-green because of some 2.5-jump, then  Lemma \ref{lm:technicalV8colour} implies $e'_{i+8}$ can cross only $r_{i+1}$ in $D_{i+3}$.  Therefore, Theorem \ref{th:rimColoured} and Lemmas \ref{lm:staysYellow} and (because no $H$-green cycle containing $e'_{i+8}$ can contain $e_{i+3}$)  \ref{lm:staysGreen} imply $e'_{i+8}$ is red in $G$.  Likewise the edge $e'_{i+2}$ of $r_{i+2}$ that is crossed in $D_{i+8}$ is red in $G$.

By Lemma \ref{lm:notSeparated}, $e'_{i+2}$ and $e'_{i+8}$ are $R$-separated in $G$.  Moreover,  the nearer of the  $(r_{i+7}\,r_{i+8})(r_{i+2}\,r_{i+3})$-paths $P_2$ in $\Delta_{e'_{i+2}}$ and $P_8$ in $\Delta_{e'_{i+8}}$, along with $s_{i}$ and $s_{i+1}$ witness their $R$-separation.    We now show that $P_8$ is contained in $\cl(Q_{i+8})$ and must be disjoint from $s_{i+4}$.

If $P_8$ intersects $s_{i+4}$ at a vertex other than $v_{i+4}$, then $P_8\cup s_{i+4}\cup r_{i+8}$ contains an $H$-green cycle that includes $e'_{i+8}$.  Otherwise, $P_8$ and $s_{i+4}$ intersect just at $v_{i+4}$, in which case $P_8\cup s_{i+4}\cup r_{i+8}$ contains a cycle $C$ that includes $e'_{i+8}$.  The $H$-green cycle containing $J$ shows $C$ is $H$-yellow.  Both possibilities contradict the fact that $e'_{i+8}$ is red.

Symmetrically, we use $J''$ to show that $P_2$ is disjoint from $s_{i+2}$.   Thus, $G$ contains a subdivison of $V_{12}$ consisting of $R$, $P_2$, $P_8$, $s_{i-1}$, $s_i$, $s_{i+1}$ and $s_{i+2}$.  But then $G-e^w$ contains a subdivision of $V_{10}$, yielding the contradiction that $f$  cannot be crossed in $D$.  Therefore, there is no global $H$-bridge $J''$ spanning the edge of $r_{i+5}$ incident with $v_{i+6}$.  

We conclude that every edge of $\cc{x',r_{i+5},v_{i+6}}$ is in an $H$-green cycle contained in $\cl(Q_i)$.

We are now in a position to show that $r_{i+6}$ has a red edge.   By way of contradiction, we suppose $r_{i+6}$ has no red edge.  \wording{If there were a} global $H$-bridge having an end in $\oo{r_{i+6}}$, \wording{then $r_{i+2}$ is $H$-yellow; Lemma} \ref{lm:staysYellow} shows $r_{i+2}$ is not crossed in $D$, a contradiction.  \wording{Thus, no global $H$-bridge has an end in $\oo{r_{i+6}}$.}   

\wordingrem{(New paragraph.)}Let $D_i$ be a 1-drawing of $G-\oo{s_i}$.  Then \wording{Corollary \ref{co:hyperBOD} and Lemma \ref{lm:BODcrossed} imply} $\bQ_i$ is crossed in $D_i$.  However, Lemma \ref{lm:technicalV8colour} shows none of $r_{i+3}\,r_{i+4}\,r_{i+5}\,r_{i+6}$ can be crossed in $D_i$, a contradiction.  

Thus, $r_{i+6}$ has a red edge $e'$.  Then $e$ is $R$-separated from $e'$ in $G$.  If $e$ is $R$-separated from $e'$ in $G-e^w$, then $f$ is $R$-separated from $r_{i+2}$ in $G-e^w$ and so $f$ cannot be crossed in $D$, a contradiction.  Therefore, $e$ is not $R$-separated from $e'$ in $G-e^w$, so Lemma \ref{lm:separated} implies there is $w$-consecutive red edge for $e$, completing the proof in Subcase 2.  

\medskip\noindent{\bf Subcase 3:}  $f\in r_{i+6}$.

\medskip  In this case, $J$ has an end $x'\in \oo{r_{i+6}}$ and the other end is $v_{i+9}$.   Also, Lemma \ref{lm:technicalV8colour} implies $f'$ (crossed by $f$ in $D$) is in $r_{i+9}$.

Suppose by way of contradiction that no edge of $r_{i+6}$ is red in $G$.  \major{We show that no edge of $r_{i+6}$ is $H$-yellow.  As every edge in $\cc{x',r_{i+6},v_{i+7}}$ is $H$-green (because of $J$), we assume by way of contradiction that there is an $H$-yellow edge  in $\cc{v_{i+6},r_{i+6},x'}$.  Let $C$ and $C'$ be the witnessing $H$-yellow and $H$-green cycles, respectively.  Lemma \ref{lm:yellowCycles} (\ref{it:greenGlobal}) implies there is a global $H$-bridge  $B$ contained in $C'$,  while (\ref{it:quadClosure}) shows $C\subseteq \cl(Q_{i+1})$.  
}

\major{\dragominor{The edges of the span $P_B$ of $B$ are all $H$-green, so $P_B$ does not contain the red edge $e$. One end of $B$ is in $\cc{w,r_i,v_{i+1},r_{i+1},v_{i+2}}$ and the other end is in $r_{i+3}$.  Furthermore, Lemma \ref{lm:globalJumps} (\ref{it:spanDiffHyper}) and the presence of $J$ shows $v_{i+4}$ is not the other end of $B$.}}

\major{\dragominor{
Write $C=P_1P_2P_3P_4$ as in Definition \ref{df:yellow} ($H$-yellow).  Because $C$ bounds a face $\Pi[G]$, $C\subseteq \cl(Q_{i+1})$, so that $P_1=r_{i+1}\cap C$.  In particular,  $e^w\notin C$.}
}

\major{Choose the labelling of $P_2$ and $P_4$ so that the end of $P_2$ in $r_{i+6}$ is nearer to $v_{i+6}$ than is the corresponding end of $P_4$.  Since there is an $H$-yellow cycle containing $P_2$ and $s_{i+2}$, Lemma \ref{lm:yellowCycles} (\ref{it:oneBridge}) shows this must be $C$.  It follows that $P_4=s_{i+2}$.
}

\dragominorrem{(Text removed.)}

\major{Consider the subdivision $H'$ of $V_6$ whose rim consists of $(R-\oo{P_B})-\oo{x',r_{i+6},v_{i+6}}$, $B$, $C-\oo{r_{i+6}\cap C}$, and whose spokes are $s_{i-1}$, $s_i$, and $s_{i+3},\cc{v_{i+3},r_{i+3},z}$.  Then $H'$ does not contain $e^w$ and so must contain the unique crossing in $D$.  Since $f$ is not in $H'$, this is a contradiction, showing that no edge of $\cc{v_{i+6},r_{i+6},x'}$ is $H$-yellow.}

\dragominor{ Because of $J$, a global $H$-bridge spanning an edge in $\cc{v_{i+6},r_{i+6},x'}$ would have to be a 2.5-jump having $v_{i+4}$ as an end.  But then $e$ is in an $H$-yellow cycle, which is impossible.}  \wording{Thus, for each edge $\bar e$ of  $\cc{v_{i+6},r_{i+6},x'}$, $\bar e$ is in an $H$-green cycle $C_{\bar e}$ contained in $\cl(Q_{i+1})$.   Theorem \ref{th:twoGreenCycles} implies $C_{\bar e}$ is disjoint from $s_{i+2}$.}

Let $D_{i+2}$ be a 1-drawing of $G-\oo{s_{i+2}}$.   We know that $\bQ_{i+2}$ is crossed in $D_{i+2}$ (Corollary \ref{co:hyperBOD} and Lemma \ref{lm:BODcrossed}).  \wording{Lemma \ref{lm:staysGreen} shows no edge in $\cc{v_{i+6},r_{i+6},x'}$ is crossed in $D_{i+2}$, while $J$ and Lemma \ref{lm:technicalV8colour} show no edge in $\cc{x',r_{i+6},v_{i+7}}r_{i+7}\,r_{i+8}$ is crossed in $D_{i+2}$.  Therefore,  the crossing in $D_{i+2}$ must be of an} edge $f''$ in $r_{i+5}$ crossing $r_{i+1}\,r_{i+2}$.    

If $f''$ is red in $G$, then Lemma \ref{lm:notSeparated} implies $f''$ and $e$ are $R$-separated in $G$.  Since $e^w\in s_{i+1}$,  $f''$ is 
 between (in $r_{i+5}$) $v_{i+5}$ and the peak of $\Delta_e$.  Thus, $f''$ and $e$ are $R$-separated in $G-\oo{s_{i+2}}$ (using $s_{i+3}$ and $s_{i+4}$ as two of the four spokes).  In turn, this implies $f''$ cannot cross $r_{i+1}\,r_{i+2}$ in $D_{i+2}$, a contradiction that shows $f''$ is not red.  Therefore, Lemmas \ref{lm:staysYellow} and \ref{lm:staysGreen} imply $f''$ is spanned by a 2.5-jump $v_{i+3}x''$, with $x''\in \oo{r_{i+5}}$.  

Now let $D_{i+3}$ be a 1-drawing of $G-\oo{s_{i+3}}$.  We know that $\bQ_{i+3}$ is crossed in $D_{i+3}$.  However:\begin{itemize}
\item Lemma \ref{lm:staysGreen} implies $\cc{v_{i+6},r_{i+6},x'}$ is not crossed in $D_{i+3}$;
\item Lemma \ref{lm:technicalV8colour} (\ref{it:atMostTwo}) implies $\cc{x',r_{i+6},v_{i+7}}\rbsp r_{i+7}\,r_{i+8}$ is not crossed in $D_{i+3}$; and
\item Lemma \ref{lm:staysYellow} implies $r_{i+9}$ is not crossed in $D_{i+3}$.
\end{itemize}
These three observations imply the contradiction that $\bQ_{i+3}$ cannot be crossed in $D_{i+3}$, showing that some edge $e'$ in $r_{i+6}$ is red in $G$.   

Obviously, $e'\in \cc{v_{i+6},r_{i+6},x'}$.  \major{By way of contradiction, suppose $e$ and $e'$ are $R$-separated in $G-e^w$.  Because $e\in r_i$ and $e'\in r_{i+6}$, Lemmas \ref{lm:shapeOfV8separator} and  \ref{lm:refinedRseparation} imply that there is a  a witnessing subdivision $H'$ of $V_8$ with two $H'$-spokes in $\cl(Q_i)\cup \cl(Q_{i+1})$ and the other two $H'$-spokes are $s_{i+3}$ and $s_{i+4}$. Furthermore, six of the eight ends of the $H'$-spokes are in the component $R'$ of $R-\{e,e'\}$ containing $r_{i+1}\,r_{i+2}\,r_{i+3}\,r_{i+4}$.   }

\major{Let $y$ be the end of $e'$ in $R'$.  Because $w=v_{i+1}$ and $x'\in \co{v_{i+6},r_{i+6},x}$, $R'$ is contained in 
$$
r_{i+1}\,r_{i+2}\,r_{i+3}\,r_{i+4}\,r_{i+5}\co{v_{i+6},r_{i+6},x}\,.
$$  
In particular, $J$ is not an $H'$-spoke and at most two of the $H'$-spokes have ends in the span of $J$.
Lemma \ref{lm:technicalV8colour} (\ref{it:atMostTwo}) implies the contradiction that the span of $J$, which includes $f$, cannot be crossed in $D$.  We conclude that $e$ and $e'$ are not $R$-separated in $G-e^w$.  Lemma \ref{lm:separated} implies that $e$ has a $w$-consecutive edge, as required.
}


  
\medskip\noindent{\bf Subcase 4:}  $f\in r_{i+7}$.

\medskip In this case, $J$ has an end $x'\in\oo{r_{i+7}}$.  If the other end of $J$ is $v_{i+5}$, then Lemma \ref{lm:technicalV8colour} (\ref{it:2halfJump3/2}) implies $f'$ is in $r_{i+3}$.  The contradiction is that $\bQ_{i+1}$ is not crossed in $D$.  Therefore, the other end of $J$ is $v_{i}$.  Lemma \ref{lm:technicalV8colour} (\ref{it:2halfJump3/2}) implies $f'$ is in $r_i\,r_{i+1}$.

Suppose there is no red edge in $r_{i+6}\,r_{i+7}$.  Let $e_{i+8}$ be the edge of $s_{i+3}$ incident with $v_{i+8}$ and let $D_{i+8}$ be a 1-drawing of $G-e_{i+8}$.  Corollary \ref{co:hyperBOD} and  Lemma \ref{lm:BODcrossed} imply $\bQ_{i+3}$ is crossed in $D_{i+8}$.  No edge in $r_{i+6}$ is spanned by a 2.5-jump having an end in $\oo{r_{i+6}}$, as otherwise $e$ is $H$-yellow.  Therefore, Lemmas \ref{lm:staysYellow} and \ref{lm:staysGreen} imply no edge of $r_{i+6}$ is crossed in $D_{i+8}$.  Lemma \ref{lm:technicalV8colour} (\ref{it:atMostTwo}) shows that no edge of $\cc{x',r_{i+7},v_{i+8}}\rbsp r_{i+8}\,r_{i+9}$ is crossed in $D_{i+8}$.  We conclude that some edge $\hat f$ of $\cc{v_{i+7},r_{i+7},x'}$  is crossed in $D_{i+8}$.

Lemmas \ref{lm:staysYellow} and \ref{lm:staysGreen} imply that there is a 2.5-jump $v_{i+5}x''$, with $x''\in \oc{v_{i+7},r_{i+7},x'}$, \wording{and, furthermore, that} $\hat f\in \cc{v_{i+7},r_{i+7},x''}$.  Lemma \ref{lm:technicalV8colour} (\ref{it:2halfJump3/2}) implies $\hat f$ crosses an edge $e'$ in $r_{i+4}$.  Lemmas \ref{lm:staysYellow} and \ref{lm:staysGreen} imply $e'$ is \wording{red in $G$}.

\wordingrem{(Text moved.)}Let $y$ be the end of $e'$ nearest $v_{i+5}$ in $r_{i+4}$.  The $r_ir_{i+5}$-path $P_0$ contained in the $uu^e$-subpath of $\Delta_e-e$ must have $v_{i+5}$ as an end, since otherwise $e$ is either $H$-green or $H$-yellow.  Symmetrically, \minor{the $r_{i+4}r_{i+9}$-path} $P_4$ contained in the $yy^{e'}$-subpath of $\Delta_{e'}-e'$ has $v_i$ as an end.  
   
   \wordingrem{(Slight reordering of text.)}Lemma \ref{lm:notSeparated} implies $e'$ is $R$-separated from $e$ in $G$.  \wording{Therefore, $P_0$ and $P_4$ are disjoint.  T}his implies that $R\cup P_0\cup P_4\cup s_{i+2}\cup s_{i+3}\cup s_{i+4}$ is a subdivision $V_{10}$ in $G-e^w$, showing that $f$ cannot be crossed in $D$, a contradiction \wording{that proves there} is a red edge $e''$ in $r_{i+6}\,r_{i+7}$. 

\wording{Suppose $e$ and $e''$ are $R$-separated in $G-e^w$.  }   \major{Lemma \ref{lm:shapeOfV8separator} implies that  a witnessing sub\-di\-vision $H'$ of $V_8$ is such that the component $R'$ of $R-\{e,f\}$ containing six of the eight ends of $H'$-spokes contains $r_{i+1}\,r_{i+2}\,r_{i+3}\,r_{i+4}\,r_{i+5}$.  }

\major{However, $J$ spans $\cc{x',r_{i+7},v_{i+8}}r_{i+8}\,r_{i+9}$, so at most two $H'$-spokes have ends that are in the span of $J$.  Lemma \ref{lm:technicalV8colour} (\ref{it:atMostTwo}) combines with $H'$ to yield the contradiction that the span of $J$, including $f$, cannot be crossed in $D$.
}
 It follows that $e$ and $e''$ are not $R$-separated in $G-e^w$, and now Lemma \ref{lm:separated} implies $e$ has a $w$-consecutive red edge\dragominor{, completing the proof of Claim \ref{cl:fIsRed}}.  \end{proof}

With Claim \ref{cl:fIsRed} in hand, we may assume $f$ is red.   Recall that $f$ and $f'$ are the edges crossed in $D$, with $f\in r_{i+4}\,r_{i+5}\,r_{i+6}\,r_{i+7}$ and $f'\in r_{i-1}\,r_{i}\,r_{i+1}\,r_{i+2}$. The proof in Case 1 is completed by finding a $w$-consecutive red edge for $e$.  We proceed in four cases, basically depending on which side of $\Delta_e$ each of $f$ and $f'$ is on.

\medskip\noindent{\bf Subcase 1:} {\em $f$ is in $r_{i+4}\lbsp\cc{v_{i+5},r_{i+5},u^e}$  and $f'$ is in $r_{i-1}\cc{v_i,r_i,u}$.}

\medskip  Since $f$ and $f'$ are not $R$-separated in $G-e^w$ and, therefore, not $R$-separated in $G$,  $f'$ cannot be red (Lemma \ref{lm:notSeparated}).  If $f'$ is $H$-yellow in $G$, then Lemma \ref{lm:staysYellow} shows it is not crossed in $D$.  Therefore, Theorem \ref{th:rimColoured} implies $f'$ is $H$-green in $G$.  Lemma \ref{lm:staysGreen} says there is a 2.5-jump $J$ spanning $f'$ so that $f'$ is in the partial $H$-rim branch spanned by $J$.  As $J$ cannot span $e$ ($e$ is not $H$-green), Lemma \ref{lm:technicalV8colour} (\ref{it:2halfJump3/2}) and our current context ($f$ in $r_{i+4}\,r_{i+5}$ and $f'$ in $r_{i-1}\,r_i$) implies \wording{this is possible only if} $f'\in r_{i-1}$ and $f\in r_{i+5}$.  However, the red edges $f$ and $e$ are $R$-separated in $G$, implying that $G-e^w$ still has five spokes (we may replace $s_{i+1}$ with the $r_ir_{i+5}$ subpath of $P_u$).  Thus, $f'$ is $H'$-green in $G-e^w$, for some $H'\topol V_{10}$.  This is impossible, as $f'$ is crossed in $D$ (Lemma \ref{lm:greenCycles} (\ref{it:notCrossed})).

\medskip\noindent{\bf Subcase 2:} {\em $f\in r_{i+4}\lbsp\cc{v_{i+5},r_{i+5},u^e}$ and $f'\in \cc{u,r_i,v_{i+1}}\rbsp r_{i+1}\,r_{i+2}$.}

\medskip   In this subcase, $f$ is $R$-separated in $G$ from $e$. The witnessing subdivision $H'$ of $V_8$ can be chosen to contain the ``nearer" $(r_{i-1}\,r_i)(r_{i+4}\,r_{i+5})$-paths, one from each of $\Delta_f$ and $\Delta_e$, \wording{\dragominor{together} with the $H$-spokes $s_{i+2}$ and $s_{i+3}$ to construct $H'$.}   

We claim that this $H'$ also shows that $f$ is $R$-separated from $f'$ in $G-e^w$.  If $f\in r_{i+4}$, then, since $\bQ_{i+1}$ is crossed in $D$, $f'\in r_i\, r_{i+1}$.  In this case, $H'$ contains the spokes $s_{i+2}$ and $s_{i+3}$, so indeed $f$ and $f'$ are in disjoint $H'$-quads, as required.  If $f\in r_{i+5}$, then $f'\in r_{i+2}$ by Lemma \ref{lm:technicalV8colour} (\ref{it:2halfJump3/2}), and again $f$ and $f'$ are in disjoint $H'$-quads\dragominor{, showing $f$ and $f'$ are $R$-separated in $G-e^w$.  Observation \ref{obs:separated} (\ref{it:separatedNoCross}) yields the contradiction that $f$ and $f'$ do not cross each other in $D$.}

\medskip\noindent{\bf Subcase 3:} {\em  $f\in \cc{w^e,r_{i+5},v_{i+6}}\rbsp r_{i+6}\,r_{i+7}$ and $f'\in r_{i-1}\cc{v_i,r_i,u}$. } 

\medskip If $f$ is $R$-separated from $e$ in $G-e^w$, then it cannot cross $f'$ in $D$, a contradiction.  Otherwise, Lemma \ref{lm:separated} implies there is a $w$-consecutive red edge for $e$.

\medskip\noindent{\bf Subcase 4:} {\em $f\in \cc{w^e,r_{i+5},v_{i+6}}\rbsp r_{i+6}\,r_{i+7}$ and $f'\in \cc{u,r_i,v_{i+1}}\rbsp r_{i+1}\,r_{i+2}$}.

\medskip If $f'=e$, then we are done:  Lemma \ref{lm:separated} implies $e$ has a $w$-consecutive edge.   

So we assume $f'\ne e$. If $f'$ is red in $G$, then Lemma \ref{lm:notSeparated} implies it is $R$-separated from $f$ in $G$.  Therefore, $f'$ is $R$-separated from $f$ in $G-e^w$, a contradiction; so $f'$ is not red in $G$. 

\major{Suppose by way of contradiction that $f'$ is $H$-yellow, with witnessing $H$-yellow and $H$-green cycles $C$ and $C'$, respectively.  If $e^w$ is not in $C$, then Lemma \ref{lm:staysYellow} yields the contradiction that $f'$ is not crossed in $D$.}  

\major{If $e^w$ is in $C$, then let $P_2$ be the $RR$-subpath of $C$ containing $e^w$, let $P'$ be the $RR$-subpath of $\Delta_e-e$ that contains $e^w$, and let $J$ be the global $H$-bridge contained in $C'$.  The end of $P'$ in $r_{i+5}$ cannot be in the interior of the span of $J$, as then either the peak of $\Delta_e$ is a vertex, in which case we have that  $\Delta_e$ is $H$-yellow, yielding the contradiction that $e$ is $H$-yellow, or the peak of $\Delta_e$ consists of parallel edges, both in the span of $J$, contradicting Theorem \ref{th:twoGreenCycles}.}

\major{It follows that $P'$ has its end in $r_{i+5}$, but not in the interior of the span of $J$.  On the other hand, $P_2$ has, by Definition \ref{df:yellow}, one end in the interior of the span of $J$.  But now $(P_2\cup P')-e^w$ contains an $R$-avoiding subpath that intersects at most the one spoke $s_{i+1}$.  Therefore, this subpath is in an $H$-green cycle and contains an edge spanned by $J$, contradicting Theorem \ref{th:twoGreenCycles}.  It follows that $f'$ is $H$-green.
}  

\wording{Theorem \ref{th:no3jump} implies that $H$ has no 3-jumps.  If $f'$ is} $H$-green by a 2.5-jump $J$, then, because $J$ cannot span $e$, Lemma \ref{lm:technicalV8colour}  (\ref{it:2halfJump3/2}) implies $f\in \cc{w^e,r_{i+5},v_{i+6}}\rbsp r_{i+6}$ and $f'\in r_{i+2}$.  Let $x$ be the end of $f$ closest to $w^e$ in $r_{i+5}\,r_{i+6}$.  Let $H'$ be the subdivision of $V_8$ obtained from $H-\oo{s_{i+1}}$ by replacing $s_{i+2}$ with $P_x$ \dragominor{(recall this is defined in Theorem \ref{th:redHasDelta} (\ref{it:deltaDetail}))}.   Now $f$ and $f'$ violate  Lemma \ref{lm:technicalV8colour} (\ref{it:2halfJump3/2}) relative to $H'$.    Therefore, $f'$ is not $H$-green by a 2.5-jump.

Lemma \ref{lm:technicalV8colour} implies $f'$ is not $H$-green by a 2-jump, as then it is not crossed in $D$.  Thus, $f'$ is $H$-green by a local $H$-green cycle $C$.  Lemma \ref{lm:staysGreen} implies $e^w$ is in $C$.  Since $f$ cannot be $R$-separated from $f'$ in $G-e^w$, we see that $f$ is not $R$-separated from $e$ in $G-e^w$.  Now Lemma \ref{lm:separated} implies there is a $w$-consecutive red edge for $e$, \dragominor{concluding the proof for Case 1.}


\bigskip
\noindent{\bf Case 2:}  {\em $e^w$ not incident with $w$.}

\medskip  By Theorem \ref{th:redHasDelta} (\ref{it:deltaDetail}), $w$ is incident with a global $H$-bridge $J_w$.  Since $w$ is not incident with $e^w$, $w\ne v_{i+1}$, and therefore $J_w$ is the 2.5-jump $wv_{i+3}$.

We observe that, since $e_w$ is not incident with $w$, its incident vertex in $r_i$ is in the interior of the span of $J_w$.  Moreover, $e_w$ is the first edge of an $R$-avoiding $r_ir_{i+5}$-path $P$ in $\Delta_e-e$, which, together with a subpath of $r_i\,r_{i+1}$, $s_{i+2}$, and a subpath of $r_{i+5}\,r_{i+6}$ makes an $H$-yellow cycle $C$.  By Lemma \ref{lm:yellowCycles} (\ref{it:oneBridge}), there is only one $C$-bridge in $G$ and, therefore, $P=s_{i+1}$.  In particular, $e_w\in s_{i+1}$.

\begin{claim}\label{cl:noYellow} No edge in  $r_{i+7}\,r_{i+8}$ is $H$-yellow.\end{claim}

\begin{proof}\startSubclaims  Suppose some edge $e'$ in $r_{i+7}\,r_{i+8}$ is $H$-yellow.  Let $C$ and $C'$ be the witnessing $H$-yellow and $H$-green cycles, respectively.  By Lemma \ref{lm:yellowCycles} (\ref{it:greenGlobal}), $C'$ contains a global $H$-bridge $J'$.

In the case $e'$ is in $r_{i+7}$, the span of $J'$ contains a vertex of $r_{i+2}$ in its interior.  Theorem \ref{th:twoGreenCycles} implies  $J'=J_w$.  But now $C\cup Q_{i+1}$ contains an $H$-yellow cycle $C''$ for which there is a $C''$-interior $C''$-bridge containing an edge of $s_{i+2}$, contradicting Lemma \ref{lm:yellowCycles} (\ref{it:oneBridge}).  Therefore, no edge in $r_{i+7}$ is $H$-yellow.

Now we suppose $e'$ is in $r_{i+8}$.  Lemma \ref{lm:globalJumps} (\ref{it:globalNoHnode}) shows $J'$ does not have $v_{i+3}$ as an end, so $J'$ has one end $x\in \oo{r_{i+3}}$ and its other end is $v_{i+6}$.  But now $C\cup Q_{i+4}$ contains an $H$-yellow cycle $C''$ having a $C''$-interior $C''$-bridge containing an edge of $s_{i+4}$, contradicting Lemma \ref{lm:yellowCycles} (\ref{it:oneBridge}).  
\end{proof}

\begin{claim}\label{cl:ri+7HasRed} Some edge of $r_{i+7}$ is red.\end{claim}

\begin{proof}\startSubclaims  Suppose no edge of $r_{i+7}$ is red.  By Theorem \ref{th:rimColoured} and Claim \ref{cl:noYellow}, every edge in $r_{i+7}$ is $H$-green.  

\begin{subclaim}\label{sc:niceRed}  If there is a red edge in either $r_{i+3}\,r_{i+4}$ or $r_{i+8}\,r_{i+9}$, \dragominor{then there is a red edge  in $r_{i+8}\,r_{i+9}$.  Furthermore, among all such red edges, the one $e''$ with an end $x''$  nearest $v_{i+8}$ in $r_{i+8}\,r_{i+9}$ is such that $(e'')^{x''}$ is not incident with $x''$ (that is, Case 1 does not apply to $e''$ and $x''$).} \end{subclaim}

\begin{proof} We first suppose no edge of $r_{i+3}\,r_{i+4}$ is red.  Then there is a red edge  in $r_{i+8}\,r_{i+9}$.  \dragominor{For any such red edge $e''$, if  the end $x''$ of $e''$ nearest to $v_{i+8}$ is incident with $(e'')^{x''}$\!, then Case 1 shows there is an $x''$-consecutive red edge $\hat e$ for $e''$.}  By Definition \ref{df:consecutive} (\ref{it:whereNext}),  $\hat e\in r_{i+1}\,r_{i+2}\,r_{i+3}\,r_{i+4}$. \dragominor{ Since the edges in $r_{i+1}\,r_{i+2}$ are $H$-green,  $\hat e\notin r_{i+1}\,r_{i+2}$.}  But then $\hat e$ is a red edge in $r_{i+3}\,r_{i+4}$, a contradiction.   Therefore, $x''$ is not incident with $(e'')^{x''}$, as required.

The alternative is that there is a red edge in $r_{i+3}\,r_{i+4}$.  Among all such edges, let $e'$ be the one having an incident vertex $x'$ nearest $v_{i+3}$ in $r_{i+3}\,r_{i+4}$.  Because of Theorem \ref{th:twoGreenCycles} and $J_w$, $x'$ is not incident with a 2.5-jump \dragominor{$x'v_{i+1}$ or $x'v_{i+2}$}.  Therefore, \dragominor{$x'$ is incident with $(e')^{x'}$}, and we conclude \dragominor{from Case 1} that there is an $x'$-consecutive red edge $e''$ for $e'$.  Because of $J_w$, every edge in $r_{i+6}$ is either $H$-yellow or $H$-green and so, in particular, is not red.  By assumption, no edge of $r_{i+7}$ is red.  \wording{By Definition \ref{df:consecutive} (\ref{it:whereNext}),}  $e''\in r_{i+8}\,r_{i+9}$.  Also, $\Delta_{e''}$ separates  $s_{i+3}$ from $\Delta_{e'}$ in $\cl(Q_{i+3})\cup \cl(Q_{i+4})$.

Let $x''$ be the end of $e''$ nearest $v_{i+8}$ in $r_{i+8}\,r_{i+9}$.  By way of contradiction, suppose $x''$ \dragominor{is incident with $(e'')^{x''}$\!.  Then Case 1 shows} there is an $x''$-consecutive red edge $\hat e$ for $e''$.  But $\hat e$ is not in $r_{i+1}\,r_{i+2}$ because $J_w$ makes every one of those edges $H$-green.  Therefore, $\hat e$ is in $r_{i+3}\,r_{i+4}$.  Since $\Delta_{\hat e}$ separates $s_{i+3}$ from $\Delta_{e''}$ in $\cl(Q_{i+3})\cup \cl(Q_{i+4})$, we see that $\hat e$ is  nearer to $v_{i+3}$ than $e'$ is, contradicting the choice of $e'$.  Therefore \dragominor{$x''$ is not incident with $(e'')^{x''}$}\!, as required.     \end{proof}

\begin{subclaim}\label{sc:noRed} No edge in either $r_{i+3}\,r_{i+4}$ or $r_{i+8}\,r_{i+9}$ is red. \end{subclaim}

\begin{proof}  Suppose by way of contradiction that there is a red edge in either $r_{i+3}\,r_{i+4}$ or $r_{i+8}\,r_{i+9}$.  By Subclaim \ref{sc:niceRed}, there is a red edge $e''$ in $r_{i+8}\,r_{i+9}$ so that the end $x''$ of $e''$ nearest $v_{i+8}$ in $r_{i+8}\,r_{i+9}$ \dragominor{is not incident with $(e'')^{x''}$}\!.  Therefore, Theorem \ref{th:redHasDelta} (\ref{it:deltaDetail}) implies $x''$ is incident with a 2.5-jump that is either $x''v_{i+6}$ or $x''v_{i+7}$.  It cannot be the former, as the 2.5-jumps $x''v_{i+6}$ and $J_w$ contradict Lemma \ref{lm:globalJumps} (\ref{it:spanDiffHyper}).  Therefore, $x''$ is in the interior of $r_{i+9}$ and the 2.5-jump is $x''v_{i+7}$.
The contradiction is obtained by showing that $\crn(G)\le 1$.  

Let $D$ be a 1-drawing of $G-\oo{r_{i+7}}$.  There is still a subdivision $H'$ of $V_8$ in $G-\oo{r_{i+7}}$ consisting of the rim $(R-\oo{r_{i+7}})\,\cup\, x''v_{i+7}$ and the four spokes $s_i$, $s_{i+1}$, $s_{i+2}$ and 
$s_{i+3}\,r_{i+8}\lbsp\cc{v_{i+9},r_{i+9},x''}$
.  We note that $x''v_{i+7}$ is an $H'$-rim branch, contained in an $H'$-quad $Q$ consisting of $s_{i+2}$, $r_{i+2}$, $s_{i+3}\,r_{i+8}\lbsp\cc{v_{i+9},r_{i+9},x''}$, and $x''v_{i+7}$.

We aim to show $D[Q]$ is clean, so by way of contradiction, we assume  $D[Q]$ is not clean.  The $H'$-rim branches of $Q$ are $r_{i+2}$ and $x''v_{i+7}$.  
Since $r_{i+1}\,r_{i+2}$ is not crossed in $D$ (Lemma \ref{lm:technicalV8colour} (\ref{it:2halfJump1})), we deduce that $x''v_{i+7}$ is crossed in $D$.  Furthermore, the cycle $r_{i+3}\,s_{i+4}\,r_{i+8}\,s_{i+3}$ (which is $Q_{i+3}$ in $G$) is $H'$-close and, therefore Lemmas \ref{lm:closeIsPrebox} and \ref{lm:preboxClean} imply $Q_{i+3}$ is not crossed in $D$.  It follows that $x''v_{i+7}$ crosses $r_{i+4}$ in $D$, so $s_{i+3}\,r_{i+8}\lbsp\cc{v_{i+9},r_{i+9},x''}$ is exposed in $D$, from which $D[H']$ is completely determined.  (See Figure \ref{fg:D[H']determined}\wordingrem{(This figure has been modified.)}.)

\begin{figure}[!ht]
\begin{center}
\scalebox{1.2}{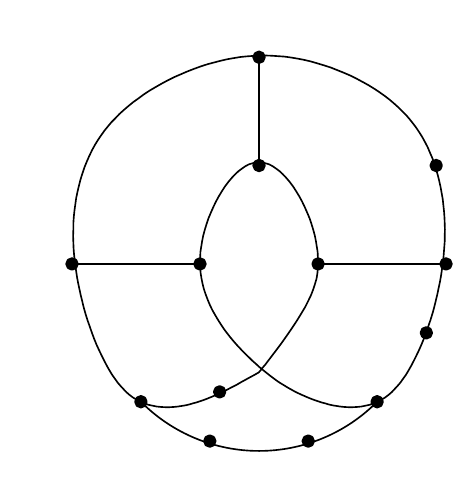}
\end{center}
\caption{\wordingrem{(Figure modified)}$D[H']$}\label{fg:D[H']determined}
\end{figure}

Our contradiction is obtained from a detailed consideration of $\Delta_{e''}$.  We first show that $v_{i+4}$ is in the peak of $\Delta_{e''}$.  To see this, we note that the $r_{i+9}r_{i+4}$-subpath of $\Delta_{e''}-e''$ that starts nearest $x''$ is simply $s_{i+4}$, as otherwise there is an $H$-yellow cycle $C$ with more than one $C$-bridge.  Theorem \ref{th:redHasDelta} (\ref{it:deltaDetail}) implies the subpath of $\Delta_{e''}-e''$ from $x''$ to the peak of $\Delta_{e''}$ has at most one edge in $R$; therefore, there is no edge of $r_{i+4}$ between $v_{i+4}$ and the peak of $\Delta_{e''}$.  That is, $v_{i+4}$ is in the peak of $\Delta_{e''}$.

Let $y''$ be the end of $e''$ different from $x''$.  Because $y''$ is too close to $J_w$, it is not incident with a global $H$-bridge.  Thus, the edge of $\Delta_{e''}-e''$ incident with $y''$ is not in $R$ and, therefore, is the first edge of an $r_{i+9}r_{i+4}$-subpath $P$ of $\Delta_{e''}-e''$.  Let $z''$ be the other end of $P$.

We note that $z''\ne v_{i+4}$, as $D[P]$ cannot cross $D[H']$.  Therefore, $z''\in \ob{v_{i+4},r_{i+4},}$ $\bc{v_{i+5}}$.  If $z''$ is in the peak of $\Delta_{e''}$, then $z''$ and $v_{i+4}$ are joined by parallel edges, one of which is not in $H'$.  That one must cross $D[H']$, which is a contradiction.  Therefore, $z''$ is not in the peak of $\Delta_{e''}$.  But now Theorem \ref{th:redHasDelta} (\ref{it:deltaDetail}) implies $z''$ is in the interior of the span of a global $H$-bridge $J''$ that has an end in the peak of $\Delta_{e''}$; therefore, this end of $J''$ is in $r_{i+4}$.  

The end of $J''$ in $r_{i+4}$ must be $v_{i+4}$, as otherwise $J''$ is a 2.5-jump with one end being $v_{i+7}$, which, together with $x''v_{i+7}$, contradicts Lemma \ref{lm:globalJumps} (\ref{it:globalNoHnode}).  Therefore, $J''$ is either $v_{i+4}v_{i+6}$ or $v_{i+4}u''$, with $u''\in \oo{r_{i+6}}$.  However, Lemma \ref{lm:technicalV8colour} (\ref{it:atMostTwo}) or  (\ref{it:2halfJump1}) and $J''$ show that $r_{i+4}$ cannot be crossed in $D$, a contradiction that finally shows $D[Q]$ is clean.

We can now obtain the claimed 1-drawing of $G$.  Observe that 
$x''v_{i+7}$ is in an $H$-green cycle that, by Lemma \ref{lm:greenCycles} (\ref{it:CboundsFace}), has only one bridge.  Also, if there is a $Q_{i+2}$-bridge other than $M_{Q_{i+2}}$, then $\cl(Q_{i+2})$  has an edge $f$ not in $Q_{i+2}$.  But Theorem  \ref{th:BODquads} and Lemma \ref{lm:BODcrossed} imply $Q_{i+2}$ would be crossed in any 1-drawing of $G-f$; however, both $r_{i+2}$ and $r_{i+7}$ are $H$-green courtesy of $J_w$ and $x''v_{i+7}$.  Therefore, $Q_{i+2}$ has only one bridge.  It follows that there are only two $Q$-bridges in $G$, one of which is $r_{i+7}$.   Since $D[Q]$ is clean,  it bounds a face of $D[G-\oo{r_{i+7}}]$ and it is easy to put $r_{i+7}$ into this face so as to obtain a 1-drawing of $G$.  That is, $\crn(G)\le 1$, a contradiction completing the proof of the subclaim. \end{proof}

We are now in a position to finish the proof of Claim \ref{cl:ri+7HasRed}.  Let $e_3$ be the edge of $s_{i+3}$ incident with $v_{i+3}$ and let $D$ be a 1-drawing of $G-e_3$.  Corollary \ref{co:hyperBOD} and Lemma \ref{lm:BODcrossed} imply $\bQ_3$ is crossed in $D$.  It follows that there is an edge $\hat e$  in $r_{i+6}\,r_{i+7}\,r_{i+8}\,r_{i+9}$ that is crossed in $D$.  

The $H$-yellow cycle $Q_{i+1}$ contains $r_{i+6}$, so Lemma \ref{lm:staysYellow} implies $r_{i+6}$ is not crossed in $D$.  By assumption for $r_{i+7}$ and by Subclaim \ref{sc:noRed} for $r_{i+8}\,r_{i+9}$, no edge of $r_{i+7}\,r_{i+8}\,r_{i+9}$ is red.  Lemmas \ref{lm:staysYellow} and \ref{lm:staysGreen} imply that $\hat e$ is spanned by some 2.5-jump $J'$, and, moreover, $\hat e$ is in the $H$-rim branch whose interior contains the end $x'$ of $J'$.  

If $\hat e\in r_{i+7}$, then $J'$ is either $x'v_{i+5}$ or $x'v_{i}$.  Suppose first that $J'=x'v_{i+5}$.  Lemma \ref{lm:technicalV8colour} (\ref{it:2halfJump3/2}) implies $\hat e$ crosses an edge in $r_{i+4}$.  But Theorem \ref{th:twoGreenCycles} shows $r_{i+4}$ cannot be in the span of a 2.5-jump, so Lemmas \ref{lm:staysYellow} and \ref{lm:staysGreen} imply no edge of $r_{i+4}$ is crossed in $D$.  Thus, $J'\ne x'v_{i+5}$.

Now we suppose $J'=x'v_{i}$.  In this case, Lemma \ref{lm:technicalV8colour} (\ref{it:2halfJump3/2}) implies $\hat e$ crosses an edge in $r_{i+1}$, while (\ref{it:atMostTwo}) of the same lemma implies no edge in the span of $J$, which includes $r_{i+1}$, is crossed in $D$.  We conclude that $\hat e\notin r_{i+7}$.

If $\hat e\in r_{i+8}$, then $J'$ is either $x'v_{i+6}$ or $x'v_{i+1}$.  Theorem \ref{th:twoGreenCycles} shows the latter does not happen.  Lemma \ref{lm:globalJumps} (\ref{it:spanDiffHyper}) shows the former does not happen.  Therefore, $\hat e\notin r_{i+8}$.

The last possibility is that $\hat e\in r_{i+9}$.  In this instance, $J'$ is either $x'v_{i+7}$ or $x'v_{i+2}$.  Theorem \ref{th:twoGreenCycles} precludes the latter possibility, so we assume $J'=x'v_{i+7}$.  However, in this case, Lemma \ref{lm:technicalV8colour} (\ref{it:2halfJump3/2}) implies $\hat e$ crosses an edge $\tilde e$ in $r_{i+5}$, in which case neither $\hat e$ nor $\tilde e$ is in $\bQ_{i+3}$, contradicting the fact that $\bQ_{i+3}$ is crossed in $D$.
\end{proof}

We now finish the proof of Case 2 and, therefore, Theorem \ref{th:consecRed}.  By Claim \ref{cl:ri+7HasRed}, we may let 
 $e'=xy$ be the red edge in $r_{i+7}$ that is nearest $v_{i+7}$ in $r_{i+7}$, labelled so that $x$ is nearer $v_{i+7}$ in $r_{i+7}$ than $y$ is.  We look for the $x$-consecutive red edge for $e'$.  As the edges spanned by $J_w$ are $H$-green, $e$ is the only possibility for the $x$-consecutive red edge for $e'$.

Suppose first that $e'$ and $x$ satisfy the condition for Case 1.  We have proved there is an $x$-consecutive red edge for $e'$ and, as just mentioned, this can only be $e$.  This implies that $x=v_{i+7}$.  To see that $e'$ is the $w$-consecutive red edge for $e$, it remains to show that $e$ and $e'$ can be crossed in $G-e^w$.  (This is the only asymmetric condition in the definition of consecutive.)

The $H$-quad $Q_{i+1}$ is also an $H$-yellow cycle and so (Lemma \ref{lm:yellowCycles} (\ref{it:oneBridge})) bounds a face of $G$.  It follows that $e$ and $e'$ are not $R$-separated in $G-e^w$ and, therefore Lemma \ref{lm:notSeparated} implies there is a 1-drawing of $G-e^w$ in which $e$ and $e'$ are crossed, as required.

The alternative is that $e'$ and $x$ do not satisfy the condition for Case 1.  Then, just as for $w$ above, there is a 2.5-jump $J_x=xv_{i+5}$ incident with $x$.  Also, the edge $e^x$ of $\Delta_{e'}-e'$ that is nearest $x$ and not in $R$ is in $s_{i+7}$.  Since $Q_{i+1}$ bounds a face of $G$, $e$ and $e'$ are not $R$-separated in $G-e^w$ and, therefore, Lemma \ref{lm:notSeparated} implies there is a 1-drawing of $G-e^w$ in which they are crossed.
\end{cproofof}

The following is a consequence of Definition \ref{df:consecutive} and Theorem \ref{th:consecRed}.   

}\begin{lemma}\label{lm:consecSymm}  Let $G\in \m2$ and $\hvfg$, with $H$ tidy.  With the labelling of $e=uw$ and $e_w$ as in Definition \ref{df:consecutive}, if $x$ is the end of $e_w$ nearest $w^e$ in $\cc{w^e,r_{i+5},v_{i+6}}\rbsp r_{i+6}\,r_{i+7}$, then $e$ is $x$-consecutive for $e_w$. \end{lemma}\printFullDetails{

\begin{cproof}  By Theorem \ref{th:consecRed}, there is an $x$-consecutive red edge $e''$ for $e_w$.  Conditions (\ref{it:betweenPeaks1}) and (\ref{it:betweenPeaks2}) of Definition \ref{df:consecutive} applied to $e_w$ being $w$-consecutive for $e$ and the same conditions applied to $e''$ being $x$-consecutive for $e_w$ imply that $e=e''$. \end{cproof}

The main goal of this work is to prove Theorem \ref{th:classification}.  The following lemma will be very helpful.

}\begin{lemma}\label{lm:findGreen}  Let $G\in \m2$, $\hvfg$, and let $\Pi$ be an embedding of $G$ in $\pp$ so that $H$ is $\Pi$-tidy.   Let $C$ be a contractible cycle contained in $\Mob$ so that $C$ is the union of  a 3-rim path $C\cap R$ (recall Definition \ref{df:yellow} (\ref{it:3rimPath})) and an \wording{$R$-avoiding path $P$}.  Then, for every edge $e$ of $C\cap R$, there is an $H$-green cycle containing $e$ and \wording{contained in $H\cup P$}. \end{lemma}\printFullDetails{

\begin{cproof}   The \wording{graph $H\cup P$ is} 2-connected and not planar, so every \wording{face of $\Pi[H\cup P]$ is} bounded by a cycle. There is a face \wording{$F$ of $H\cup P$ contained} in $\Mob$ and incident with $e$; by the preceding \dragominor{remark, $F$} is bounded by a cycle $C'$.

Let $j$ be the index so that $e\in r_j$; thus, $F$ is $Q_j$-interior.   Since $F$ is also $C$-interior, $C'\cap H\subseteq \oo{s_j\,r_j\,s_{j+1}}$.  In particular, there is at least one edge of $C'$ \wording{that is in $P$ but} not in $H$.

Observe that $\oo{s_j\,r_j\,s_{j+1}}-e$ has two components $K_1$ and $K_2$.  Since $C'$ contains a vertex in each of $K_1$ and $K_2$ (namely the ends of $e$), $C'$ contains \wording{an $\oo{s_j\,r_j\,s_{j+1}}$-avoiding $K_1K_2$-path $P'$.  Thus, $P'\subseteq P$}.  

Let $C''$ be the \wording{cycle in $\oo{s_j\,r_j\,s_{j+1}}\cup P'$}.  Then $C''$ is evidently an $H$-green cycle containing $e$, as required.  
\end{cproof}

Now for the main result.

}\bigskip\noindent{\bf Theorem \ref{th:classification}}  {\em If $G$ is a 3-connected\wording{,}\wordingrem{(comma added)} 2-crossing-critical graph containing a subdivision of $V_{10}$, then $G\in \mathcal T(S)$.}\printFullDetails{

\bigskip
\begin{cproof}  By Theorem \ref{th:existTidy}, $G$ contains a tidy subdivision $H$ of $V_{10}$;  let $\Pi$ be an embedding of $G$ in $\pp$ so that $H$ is $\Pi$-tidy.  The strategy is to show that, between every red edge $e=uw$ and its $w$-consecutive red edge $e_w$, there is one of the thirteen pictures (as defined just before Lemma \ref{lm:2degenerate}).  This is accomplished by showing that $e$ produces ``one side" of the picture and $e_w$ produces the other.   Let $i\in \{0,1,2,\dots,9\}$ be such that $e\in r_i$; we choose the labelling so that $r_i=\cc{v_i,r_i,u,e,w,r_i,v_{i+1}}$.  Thus, $e_w\in r_{i+5}\,r_{i+6}\,r_{i+7}$.

Let $x$ be the end of $e_w$ so that $e$ is the $x$-consecutive red edge for $e_w$.  Let $P_1$ be the $w^ex$-subpath of $R$ that is a 3-rim path \wording{ (Definition \ref{df:yellow} (\ref{it:3rimPath}))}; likewise $P_2$ is the $x^{e_w}w$-subpath of $R$ that is a 3-rim path.  

\begin{claim}\label{cl:P1jump}  Let $B$ be a global $H$-bridge spanning an edge of $P_1$.  Then:\begin{enumerate}[label=(\alph*)]
\item  $B$ has ends $w^e$ and $x$;
\item  $w^e=v_{i+5}$; 
and 
\item $e^w\in s_{i+1}$ and $(e_w)^x\in s_{i+2}$.
\end{enumerate}
The analogous claims holds for $P_2$.\end{claim}

\begin{proof}
We remark that the span of $B$ does not include in its interior a peak vertex of $\Delta_e$, and does not include $e_w$.   Therefore, $B$ has both its attachments in $P_1$.  

\wording{Consequently,} the attachments of $B$ are contained in $ r_{i+5}\,r_{i+6}\lbsp\co{v_{i+7},r_{i+7},v_{i+8}}$.  Theorem \ref{th:globalBridges} implies one end of $B$ is $v_{i+5}$ and the other end is in $\co{v_{i+7},r_{i+7},v_{i+8}}$.    

It follows that $w^e=v_{i+5}$.  At the other end, we claim $x$ is in $B$.  We note that $e_w$ is in $r_{i+7}$, so that $H-\oo{s_{i+4}}$ shows that $e$ and $e_w$ are $R$-separated.   Let $x'$ be the end of $B$ in $r_{i+7}$.

\dragominor{If $(e_w)^x$ is not in $s_{i+2}$, then let $e_{i+7}$ be the edge of $s_{i+2}$ incident with $v_{i+7}$ and let $D$ be a 1-drawing of $G-e_{i+7}$.   Corollary \ref{co:hyperBOD} and Lemma \ref{lm:BODcrossed} imply $\bQ_{i+2}$ is crossed in $D$. The presence of $J$ and $B$ combine with Lemma \ref{lm:technicalV8colour} (\ref{it:atMostTwo}) to show that neither $\cc{w,r_i,v_{i+1}}r_{i+1}\,r_{i+2}$ nor $r_{i+6}\cc{v_{i+7},r_{i+7},x'}$, respectively, is crossed in $D$.}

\dragominor{It follows that some edge $e'$ in $\cc{x',r_{i+7},v_{i+8}}$ is crossed in $D$.   Let $P_w$ be the path in $\Delta_e$ described in Theorem \ref{th:redHasDelta} (\ref{it:deltaDetail}).  Since $P_w$ does not have $v_{i+1}$ as one end, and its other end is $v_{i+5}$, its only intersection with $s_{i+1}$ can be in $\oo{s_{i+1}}$.  Such an intersection produces an $H$-green cycle that shows the edge of $r_i$ incident with $v_{i+1}$ is in two $H$-green cycles, contradicting Theorem \ref{th:twoGreenCycles}.  Therefore, $P_w$ is disjoint from $s_{i+1}$.}

\dragominor{Using $P_w$, $s_{i+1}$, $s_{i+3}$ and $s_{i+4}$ as spokes and $R$ as the rim yields a $V_8$ that shows $r_{i+7}$ is $R$-separated in $G-e_{i+7}$ from $\cc{v_i,r_i,w}$; thus, Observation \ref{obs:separated} (\ref{it:separatedNoCross}) shows $e'$ crosses an edge $e''$ of $r_{i+3}$ in $D$.  Lemmas \ref{lm:staysYellow} and \ref{lm:staysGreen} shows $e'$ and $e''$ are red in $G$.  Lemma \ref{lm:notSeparated} shows that $e'$ and $e''$ are $R$-separated in $G$.  Lemma \ref{lm:refinedRseparation} shows that a witnessing $V_8$ can be chosen to avoid $e_{i+7}$.  But now $D$ contradicts Observation \ref{obs:separated} (\ref{it:separatedNoCross}).
Therefore,  $(e_w)^x$ \wording{is in $s_{i+2}$} and incident with $v_{i+7}$.   }

If $B$ has an end in $\oo{r_{i+7}}$, then Theorem \ref{th:redHasDelta} (\ref{it:deltaDetail}) implies $P_x\cap r_{i+7}$ has just one edge, namely $xv_{i+7}$ and, consequently, $x$ is in $B$.

If, on the other hand, $v_{i+7}$ is an end of $B$, then Theorem \ref{th:redHasDelta} (\ref{it:deltaDetail}) implies $x$ must be incident with $e^x$ and, \wording{therefore $x=v_{i+7}$}.  Again, we see that $x$ is in $B$. 

\dragominor{Observe that $J_w$ and $B$ are now seen to be completely symmetric with respect to $(e,w)$ and $(x,e_w)$; in particular, we conclude that $e^w\in s_{i+1}$.  }\end{proof}

  If there is a global $H$-bridge $B$ spanning an edge of $P_1$, then we let $P'_1=B$.  Otherwise,  we let $P'_1=P_1$.
A completely analogous discussion holds for $P_2$ to yield the $wx^{w_e}$-path $P'_2$.  

Our next claim identifies the cycle that is the boundary of our picture.

\begin{claim}\label{cl:CeCycle} The closed walk $P_wP'_1P_xP'_2$ is a cycle. \end{claim}

\begin{proof}  If the edge of $P_w$ incident with $w$ is in $R$, then Theorem \ref{th:redHasDelta} (\ref{it:deltaDetail}) shows that $w$ is incident with a global $H$-bridge $B$.  Claim \ref{cl:P1jump} implies $B=P'_2$.  Thus, $w$ has degree 2 in the closed walk $P_wP'_1P_xP'_2$.  Otherwise, $P'_2=P_2$, $w$ is incident with $e^w$, and again $w$ has degree 2 in $P_wP'_1P_xP'_2$.

The other ``corners" $w^e$, $x$, and $x^{e_w}$ \dragominor{are treated similarly.}  \end{proof}

}\begin{definition}\label{df:Ce}\minorrem{(Reworking of definition of $C_e$ so that it becomes referenceable.)}\minor{Let $e$ and $e'$ be red edges and let $w$ and $x$ be the ends of $e$ and $e'$, respectively, so that $e'$ is the $w$-consecutive red edge for $e$ and $e$ is the $x$-consecutive red edge for $e'$.  Let $P_1$ be the $x\Delta_{e}$-path in $R$ that is a 3-rim path and let $P_2$ be the $w\Delta_{e'}$-path in $R$ that is a 3-rim path.  Let $P_w$ be the $ww^e$-path in $\Delta_e-e$ and let $P_x$ be the $xx^{e'}$-path in $\Delta_{e'}-e'$.  For $i=1,2$, let $P'_i$\index{$P'_i$} be $P_i$ unless there is a global $H$-bridge $B_i$ spanning an edge of $P_i$, in which case $P'_i=B_i$.}

\minor{The cycle $C_e$\index{$C_e$} is the composition $P_wP'_1P_xP'_2$.} \end{definition}\printFullDetails{

We will see that $C_e$ is the outer boundary of the one of the thirteen pictures that occurs.   We observe that $C_e$ is in the boundary of the closed disc in $\pp$ consisting of the union of the closed discs bounded by $r_i\,r_{i+1}\,r_{i+2}\,s_{i+3}\,r_{i+7}\,r_{i+6}\,r_{i+5}\,s_i$, $P'_1P_1$ (if $P'_1\ne P_1$), and $P'_2P_2$ (if $P'_2\ne P_2$).  Therefore, $C_e$ is the boundary of a closed \wording{disc $\Disc_e$ in $\pp$}.  

We now prove three claims that will be useful for finding the various parts of the picture.

\begin{claim}\label{cl:P1disjointBoundsFace}  Let $C$ be a cycle contained in $\Disc_e$.  If either $C\cap P'_1$ or $C\cap P'_2$ is empty, then $C$ bounds a face of $\Pi[G]$.\end{claim}

\begin{proof}   By symmetry, we may suppose $C\cap P'_1$ is empty. Let $M$ be the $C$-bridge containing $s_{i+4}$. 

\startSubclaims
\begin{subclaim}\label{sc:contractible} If $B$ is a $C$-bridge different from $M$, then $\Pi[C\cup B]$ is contractible in $\pp$.\end{subclaim}

\begin{proof}  We start by noting that $\Pi[B]\subseteq \Mob$, since $P'_2$ is either just an edge that is a global $H$-bridge \dragominor{(and so in $\Disc$ and forcing $B$ to be in $\Mob$)} or $P'_2=P_2$ and there is no global $H$-bridge having an attachment in $\oo{P_2}$.   In the latter case, any global $H$-bridge having an attachment at an end of $P_2$ (say $w$), has its other attachment in the $H$-rim $R-\oo{P_2}$.  Such an attachment is in $\Nuc(M)$, contradicting the assumption  that $B\ne M$.

It follows that $\Pi[C\cup B]$ is contained in $\Mob$ and totally disjoint from $s_{i+4}$.  Therefore, $\Pi[C\cup B]$ is contractible, as claimed.\end{proof}

Let $H'$ be the subgraph of $H\cup P'_1\cup P'_2$ consisting of $(R-(\oo{P_1}\cup \oo{P_2}))\cup (P'_1\cup P'_2)$ and the three $H$-spokes $s_{i+3}$, $s_{i+4}$, and $s_{i}$.  The following claim shows that $H'$ is a subdivision of $V_6$.  (The notation $\iso{y}$ is in Definition \ref{df:bridge} (\ref{it:isolated}).)

\begin{subclaim}\label{sc:CeDelta}  $C_e\cap s_i\subseteq \iso{v_{i+5}}$ and $C_e\cap s_{i+3}\subseteq \iso{v_{i+3}}$.  \end{subclaim}

\begin{proof}  Recall that $P_w$ is contained in $\Delta_e$.  Theorem \ref{th:redHasDelta}  (the existence of $A_u$ and $A_w$\wording{, together with}  (\ref{it:deltaDetail})) implies $P_w$ is internally disjoint from $P_u$ and, therefore, cannot intersect  $s_i$, except possibly at their common end point $v_{i+5}$.   The analogous argument using $\Delta_{e_w}$ applies for $s_{i+3}$.  \end{proof}

If $C$ does not bound a face of $\Pi[G]$, then let $e'$ be any edge of any $C$-interior $C$-bridge and let $D$ be a 1-drawing of $G-e$.
Subclaim \ref{sc:CeDelta} implies that $C$ is $H'$-close (Definition \ref{df:Hclose}).  Lemmas \ref{lm:closeIsPrebox} and \ref{lm:preboxClean} imply $C$ is clean in $D$.  Therefore, $D$ contains a 1-drawing of $C\cup M$ in which $C$ is clean and Lemma \ref{lm:cleanBOD} implies $C$ has BOD.
  It now follows from Corollary \ref{co:TutteTwo} that $\crn(G)\le 1$, the final contradiction.
\end{proof}

We find structures in the $C_e$-interior that lead to the pictures.   Our discussion will be $w$-centric; there is a completely analogous discussion for $x$.

A useful observation is the following.  Recall that $P_w$ is the $ww^e$-path in $\Delta_e-e$ (Theorem \ref{th:redHasDelta} (\ref{it:deltaDetail})) and $P_x$ is the analogous $xx^{e_w}$-path in $\Delta_{e_w}$.

\begin{claim}\label{cl:noTileCrossing}   \begin{enumerate}\item\label{it:wSide} No $C_e$-interior $C_e$-bridge has an attachment in each of the components of $(C_e-P_x)-e^w$. 
\item\label{it:xSide} No $C_e$-interior $C_e$-bridge has an attachment in each of the components of $(C_e-P_w)-e^x$.
\end{enumerate}
 \end{claim}

\begin{proof}   Let $H'$ be a subdivision of $V_8$ witnessing the $R$-separation of $e$ and $e_w$.  As $e$ and $e_w$ are $R$-separated in neither $G-e^w$ nor $G-e^x$, $e^w$ and $e^x$ are both in $H'$.  Since $e$ and $e_w$ are in disjoint $H'$-quads, $e^w$ and $e^x$ are in disjoint $H'$-spokes, \dragominor{which we denote as $P^w$ and $P^x$\!, respectively; $P^w$ and $P^x$ are contained in the closed disc bounded by $\Pi[C_e]$. }   

\startSubclaims

\begin{subclaim}\label{sc:niceSeparation}  There is such an $H'$ so that $P^x=P_x$.  \end{subclaim}

\begin{proof}  As a first case, suppose $C_e\cap s_i=\varnothing$.  Then we may choose $H'$ to be $R$, $s_i$, $s_{i+4}$, $P_w$, and $P_x$, and we are done.  In the second case, $C_e\cap s_{i+3}=\varnothing$; replace $s_i$ with $s_{i+3}$. 

In the final case, $C_e\cap s_i$ and $C_e\cap s_{i+3}$ are not empty.  In this instance, $e_w\in r_{i+7}$.  We may choose $H'$ to consist of $R$, $s_{i+4}$, $s_i$, $s_{i+1}$, and $P_x$, the latter being contained in  $\cl(Q_{i+2})$.
\end{proof}

By symmetry, it suffices to prove (\ref{it:wSide}).  Suppose by way of contradiction that there is a $C_e$-interior $C_e$-bridge $B$ having an attachment in each component of $(C_e-P_x)-e^w$.  Subclaim \ref{sc:niceSeparation} implies there is a subdivision $H'$ witnessing the $R$-separation of $e$ and $e_w$ so that $P_x\subseteq H'$.    Let $P^w$ be the other $H'$-spoke contained in the interior of $C_e$.

Let $C'$ be the cycle bounding the $C_e$-interior face of $C_e\cup P^w$ that is incident with $e^w$.  The $C_e$-bridge $B$ contains a subpath $P'$ joining the two components of $(C'-P_x)-e^w$.  Now $((C'-P_x)-e^w)\cup P'$ contains an $R$-avoiding path $P''$ that can replace $P^w$ in $H'$ to get another subdivision of $V_8$ that witnesses the $R$-separation of $e$ and $e_w$ in $G-e^w$.  However, this contradicts the fact that $e$ and $e_w$ are not $R$-separated in $G-e^w$.
\end{proof}

Here is our \wordingrem{(text deleted)}final preliminary claim.

\begin{claim}\label{cl:twoAtts}  Let $B$ be a $C_e$-interior $C_e$-bridge.  Then $B$ \wording{is just an edge and its ends}.  \end{claim}

\begin{proof}  Suppose to the contrary that $B$ is a $C_e$-interior $C_e$-bridge with at least three attachments.

\startSubclaims
\begin{subclaim}\label{sc:P1andP2}  $B$ has at most two attachments in each of $C_e-P'_1$ and $C_e-P'_2$. \end{subclaim}

\begin{proof} By symmetry, it suffices to prove the first of these. Suppose $B$ has at least two attachments in $C_e-P'_1$.  Let $y$ and $z$ be the ones nearest the two ends of $C_e-P'_1$.  There is a cycle in $B\cup C_e$ consisting of a $C_e$-avoiding $yz$-path in $B$ and the $yz$-subpath of $C_e-P'_1$.  Claim \ref{cl:P1disjointBoundsFace} implies this cycle bounds a face of $\Pi[G]$ and, therefore, $B$ can have no other attachment in $C_e-P'_1$.
\end{proof}

\begin{subclaim}\label{sc:corners}  $\att(B)\cap P'_1\subseteq \{x,w^e\}$ and $\att(B)\cap P'_2\subseteq \{w,x^{e_w}\}$. \end{subclaim}

\begin{proof}  By symmetry, it suffices to prove the first of these.  
By way of contradiction, suppose $B$ has an attachment $y$ in $\oo{P'_1}$.  Because $B$ has at least three attachments, Subclaim \ref{sc:P1andP2} implies $B$ has an attachment $z$ in $P'_2$.  Any $C_e$-avoiding $yz$-path in $B$ contradicts Claim \ref{cl:noTileCrossing}.
\end{proof}

From these two subclaims, we easily deduce that:
\begin{itemize}\item $B$ has at most four attachments; \item one of $w$ and $x^{e_w}$ is an attachment of $B$; and \item one of $x$ and $w^e$ is an attachment of $B$. 
\end{itemize}

Observe that Claim \ref{cl:noTileCrossing} (\ref{it:wSide}) implies that not both $w$ and $w^e$ are attachments of $B$, while (\ref{it:xSide}) implies that not both $x$ and $x^{e_w}$ are attachments of $B$.  Therefore, $\att(B)\cap (P'_1\cup P'_2)$ is either $\{w,x\}$ or $\{w^e,x^{e_w}\}$.

\begin{subclaim}\label{sc:wxNot}   $\att(B)\cap (P'_1\cup P'_2)=\{w^e,x^{e_w}\}$.\end{subclaim}

\begin{proof} Suppose by way of contradiction that  $\att(B)\cap (P'_1\cup P'_2)=\{w,x\}$.  As $B$ has at least three attachments, there is an attachment $y$ in $\oo{P_w}\cup \oo{P_x}$.   By symmetry, we may assume $y\in \oo{P_w}$.   Let $P^{yw}$ be a $C_e$-avoiding $yw$-path in $B$.  Then the union of $P^{yw}$ and the $yw$-subpath of $P_w$ is a cycle $C^{yw}$ in $\Disc_e$.  

Since $y$ and $w$ are in $P_w-w^e$, $C^{yw}$ is disjoint from $P'_1$.   Claim \ref{cl:P1disjointBoundsFace} implies $C^{yw}$ bounds a face of $\Pi[G]$.  On the other hand, $P_w$ is contained in the boundary of the face bounded by $\Delta_e$ and, therefore, $C^{yw}\cap P_w$ is in the boundary of two faces of $\Pi[G]$. We deduce that $C^{yw}\cap P_w$ is just the edge $wy$.  

Furthermore, Claim \ref{cl:noTileCrossing} implies $w$ and $y$ are in the same component of $P_w-e^w$.  Therefore, the definition of $e^w$ implies $wy$ is in $R$, and consequently $P'_2$ is a global $H$-bridge spanning $wy$.   However, any edge of $B$ incident with $w$ --- and there is at least one such --- must be \wording{in the interior} of the face of $\Pi[G]$ bounded by the $H$-green cycle containing $P'_2$ (Lemma \ref{lm:greenCycles} (\ref{it:CboundsFace})).  This contradiction proves the subclaim.\end{proof}

We are now ready to complete the proof of the claim.  Any vertex in $\att(B)\setminus \{w^e,x^{e_w}\}$ is in $\oo{P_w}\cup \oo{P_x}$.  Subclaim \ref{sc:P1andP2} implies there is at most one of these.  Since $B$ has at least three attachments, there is at least one of these.  We conclude there is exactly one such attachment $y$.  We may choose the labelling so that $y\in \oo{P_w}$.  
\minorrem{(Two paragraphs replaced with the following sentence.)}\minor{Lemma \ref{lm:threeAtts} implies $B$ is isomorphic to $K_{1,3}$.}



The vertex $y$ is in the interior of $P_w$.  Thus, both edges of $P_w$ incident with $y$ are in the boundary of the face bounded by $\Pi[\Delta_e]$. Consequently, any edge of $G$ incident with $y$ is in $\Disc_e$.  

Let $c$ be the vertex of degree 3 in $B$.  Claim \ref{cl:P1disjointBoundsFace} implies that the cycles $\cc{y,c,w^e,y}$ and  $\cc{y,c,x^{e_w},P'_2,w,P_w,y}$ both bound faces in $\Disc_e$.  Therefore, $y$ has degree 3 in $G$.

Let $e'$ be the edge $cw^e$ \wording{of $B$ and let} $D'$ be a 1-drawing of $G-e'$.   Consider the subdivision $H'$ of $V_6$ consisting of $(R-(\oo{P_1}\cup \oo{P_2})\cup (P'_1\cup P'_2)$, $P_x$, $s_i$, and $s_{i+4}$.  Then $H'$ shows that $P_w\cup (B-e')$ is not crossed in $D'$.  

The path $P'=\cc{c,cy,y,P_w,w^e}$ is not crossed in $D'$.  Since $y$ has degree 3 in $D'$, we may add the edge $w^ec$ to $D'$ alongside $P'$ without crossing to obtain a 1-drawing of $G$.  This is the final contradiction that \wording{shows $B$ has only two attachments.  Lemma \ref{lm:threeAtts} shows $B$ is just an edge and its ends.}
\end{proof}

We now have our preliminary lemmas in hand and proceed to complete the proof of Theorem \ref{th:classification}.

}\begin{definition}\label{df:halfTiles}  Let $C_e$ be decomposed as $P_wP'_1P_xP'_2$ \wording{as in Definition \ref{df:Ce}}.
\begin{enumerate}\item\label{it:wDigon} \wording{If $f$ is an edge not in $C_e$ with ends $w$ \dragominor{and $x_{e_w}$} and $P'_2$ has length 1, then $f$ is a {\em $w$-chord\/}\index{$w$-chord}\index{chord}.}

\item\label{it:wSlope} \wording{If $f$ is an edge not in $C_e$ joining $w$ to a vertex $y\in \oo{P_x}$ and the $yx^{e_w}$-subpath of $P_x$ has length 1, then $f$ is a {\em $w$-slope\/}\index{$w$-slope}\index{slope}.} 
\item\label{it:wDigonSlope} \wording{If $f$ and $f'$ are edges not in $C_e$, with $f$ joining $w$ with $z\in \oo{P'_2}$ and $f'$ joining $z$ to $z'\in \oo{P_x}$, and if $P'_2$ has length 2, while the $z'x^{e_w}$-subpath of $P_x$ has length 1, then $\{f,f'\}$ is a {\em $w$-chord+$w$-slope\/}\index{$w$-chord+$w$-slope}\index{chord+slope}.}
\item\label{it:wBackSlope} \wording{If $f$ is an edge not in $C_e$ joining $x^{e_w}$ to a vertex $y$ in $\oo{P_w}$, and both $P'_2$ and the $yw$-subpath of $P_w$ have length 1, then $f$ is a {\em $w$-backslope\/}\index{$w$-backslope}\index{backslope}.}
\item\label{it:wHgraph} \wording{If $f$ is an edge not in $C_e$ joining $y\in \oo{P_w}$ and $z\in \oo{P_x}$, and the paths $P_w$ and $P_x$ have length 2, while $P'_1$ and $P'_2$ have length 1, then $f$ is a {\em crossbar\/}\index{crossbar}.}
\end{enumerate}
\end{definition}\printFullDetails{

\dragominor{The five situations in Definition \ref{df:halfTiles} are illustrated in Figure \ref{fg:halfTiles}.}

\begin{figure}[!ht]
\begin{center}
\scalebox{1.0}{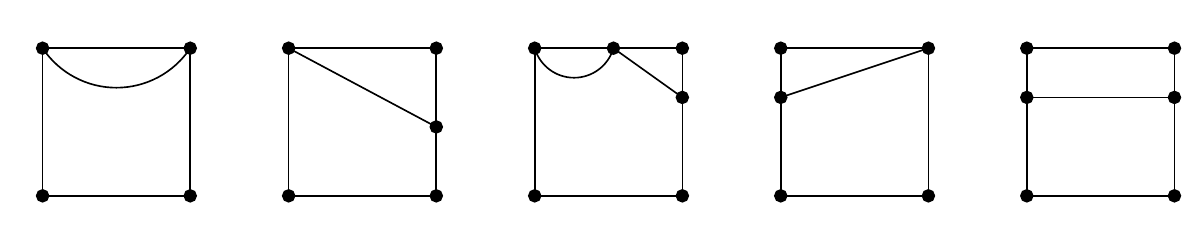}
\end{center}
\caption{Definition \ref{df:halfTiles}.}             
\label{fg:halfTiles}
\end{figure}

\wordingrem{(Text removed.)}%

\begin{claim}\label{cl:ewNotGreenYellow}  If $e^w$ is in neither an $H$-yellow nor an $H$-green cycle, then every edge of $P_2$ is $H$-green.  If $\mathcal C$ is the \wording{set of  $H$-green cycles containing edges} of $P_2$, then $C_e\cup (\bigcup_{C\in\mathcal C})C$ contains either: 
\begin{enumerate}[label=(\alph*)]
\item\label{it:digonA} $C_e$ plus a $w$-chord; \item\label{it:slopeA} $C_e$ plus a $w$-slope; or \item\label{it:digonSlopeA} $C_e$ plus a $w$-chord+$w$-slope.  
\end{enumerate}\end{claim}

\begin{proof}   Because $e^w$ is not in an $H$-yellow cycle, Theorem \ref{th:redHasDelta} (\ref{it:deltaDetail}) implies $w$ is incident with $e^w$. 

\medskip{\bf Case 1:}  {\em some edge of $P_2$ is spanned by a global $H$-bridge.}

\medskip  Let $B$ be a global $H$-bridge spanning an edge of $P_2$.  Claim \ref{cl:P1jump} implies $B$ has ends $x^{e_w}$ and $w$,  $x^{e_w}=v_{i+3}$, $e^w\in s_{i+1}$,  and $e^x\in s_{i+2}$.  Since $w$ is incident with $e^w$, we have $w=v_{i+1}$.  

We show \ref{it:slopeA} occurs by proving that $r_{i+1}$ is a $w$-slope.  We show that $r_{i+1}$ and $r_{i+2}$ are both paths of length 1, starting with the latter.

We note that $P_x$ is equal to $s_{i+2}\,r_{i+2}$.  Moreover, $r_{i+2}$ has the face of $\Pi[G]$ bounded by the $H$-green cycle $C_g$ containing $B$ on one side and the face bounded by $\Delta_{e_w}$ on the other.  Thus, $r_{i+2}$ is just a single edge.  

\wording{Claim \ref{cl:twoAtts} shows that the $C_e$-bridge $B'$ containing $r_{i+1}$ is just an edge and its ends. Thus, $r_{i+1}$ is $B'$ and so has length 1, as required, completing the proof in Case 1.}

\medskip{\bf Case 2:}  {\em no edge of $P_2$ is spanned by a global $H$-bridge.}

\medskip In this case, $P'_2=P_2$.  We start by showing that every \wording{edge of $P_2$ is} $H$-green.

Because $e_w$ is $w$-consecutive for $e$,  Definition \ref{df:consecutive} implies no edge of $P_2$ is red.  By Theorem \ref{th:rimColoured}, we need only show that none is $H$-yellow.  Suppose to the contrary that there is an $H$-yellow edge $f$ in $P_2$, as witnessed by the $H$-yellow cycle $C_y$ and the $H$-green cycle $C_g$.  Lemma \ref{lm:yellowCycles} implies there is a global $H$-bridge $B$ contained in $C_g$.  

The face of $\Pi[G]$ bounded by $C_y$ (Lemma \ref{lm:yellowCycles} (\ref{it:oneBridge})) is in $\Mob$.  Now the faces of $\Pi[G]$ bounded by $\Delta_e$ and $\Delta_{e_w}$ separate $\Mob$ into two parts, one of which contains $f$, and therefore $P_2$.   It follows that $P_1$ is also in this part and $C_y$ has at least a vertex in $P_1$. We conclude that $B$ spans an edge of $P_1$.  Claim \ref{cl:P1jump} implies $B=P'_1$, $w^e=v_{i+5}$, $e^w\in s_{i+1}$, and $e^x\in s_{i+2}$.  Because $P'_2=P_2$, and $e^w\in s_{i+1}$, we deduce that $w=v_{i+1}$.  

Since $e^w$ is not in an $H$-yellow cycle, we  conclude that $Q_{i+1}$ is not an $H$-yellow cycle.  The other attachment of $B$, namely $x$, which is in $\co{v_{i+7},r_{i+7},v_{i+8}}$, must therefore be $v_{i+7}$. 

If the $H$-yellow edge $f$ is in $r_{i+1}$, then $C_y\cap P_1$ is contained in the interior of the span of $B$.    This implies that $s_{i+1}$ is in an $H$-yellow cycle and, therefore, $e^w$ is in an $H$-yellow cycle, contrary to the hypothesis.

We have noted that $x=v_{i+7}$ is an end of $B$.  Consequently, no edge of $\cc{v_{i+2},r_{i+2},x^{e_w}}$ can be $H$-yellow.  That is, every edge of $P_2$ is $H$-green.

We now complete the proof in Case 2.   Let $C$ be the $H$-green cycle containing the edge of $P_2$ that is incident with $w$.    Because $e^w$ is not in any $H$-green cycle,  $w$ is incident with an edge $e'$ in $C$ that is not in $C_e$.  Let $B$ be the $C_e$-bridge containing $e'$.  

Claim \ref{cl:twoAtts} implies $B$ is just an edge with the two ends $w$ and a second vertex $z$.
 The path $C\cap P_2$ is in the boundary of the face of $\Pi[G]$ bounded by $C$ (Lemma \ref{lm:greenCycles} (\ref{it:CboundsFace})).  Also, there is no global $H$-bridge spanning an edge of $P_2$ (\wording{we are in Case 2}).  These two facts imply $C\cap P_2$ is just an edge.

Suppose first that $z\in P_x-x^{e_w}$.   Because $C$ is $H$-green, it is \wording{disjoint from $P_1$}.  Thus, Claim \ref{cl:P1disjointBoundsFace} implies that the cycle $C'$ that is the union of the $wz$-subpath of $P_2P_x$ and $B$ bounds a face of $\Pi[G]$.  This face is contained in $\Mob$, as is the face bounded by $C$.  Both are incident with the edge of $P_2$ incident with $w$ and so they are the same face.  We conclude that $C=C'$.

Now $C\cap P_x$ is in the boundary \wording{of a face inside the disc bounded by $\Delta_{e'}$ on} one side and the face bounded by $C$ on the other.  \dragominor{Because $G$ is 3-connected}, this subpath has length 1.  In this case, we have \ref{it:slopeA}.

The other possibility is that $z$ is in $P_2$.  We have already shown that $w$ and $z$ are the ends of a digon.  If $z=x^{e_w}$, then we have \ref{it:digonA}.  Therefore, we may suppose $z\ne x^{e_w}$.

Since $G$ is 3-connected, $z$ has a neighbour $y$ distinct from its neighbours in $P_2$.  Let $B'$ be the $C_e$-bridge containing $zy$. Claim \ref{cl:twoAtts} implies $B'$ is just an edge joining $z$ and $y$.  

The choice of $y$ shows $y\ne w$.  Claim \ref{cl:noTileCrossing} (\ref{it:wSide}) and (\ref{it:xSide})  imply, respectively, that  $y\notin \oo{P_wP'_1}$  and that $y\ne x$.   If $y\in P_2$, then (just as for $w$ and $z$) $z$ and $y$ are the ends of a digon, so $y$ is a neighbour of $z$ in $P_2$, contradicting the choice of $y$.  Therefore, $y\in \oo{P_x}$.  

Let $C'$ be the cycle consisting of $zy$ and the $zy$-subpath of $P_2P_x$.  Claim \ref{cl:P1disjointBoundsFace} implies $C'$ bounds a face of $\Pi[G]$.

To see that \ref{it:digonSlopeA} holds, notice that  $C'\cap P_x$ is in the boundary of the faces bounded by $C'$ and $\Delta_{e_w}$. \dragominor{ Again, the 3-connection of $G$ shows} $C'\cap P_x$ is a path of length 1.   Likewise, $C'\cap P_2$ is in the boundary of the face bounded by $C'$.  There is no global $H$-bridge spanning any edge of $P_2$, so $C'\cap P_2$ is also a path of length 1, completing the proof that \ref{it:digonSlopeA} occurs and the proof of Claim \ref{cl:ewNotGreenYellow}.
\end{proof}  

It remains to consider the possibilities that $e^w$ is in either an $H$-yellow or an $H$-green cycle.  We do the latter first.

\begin{claim}  If $e^w$ is in an $H$-green cycle $C$, then either  \begin{enumerate}[label=(\alph*), start=4]\item\label{it:backSlope} $C_e\cup C$ contains $C_e$ plus a back-slope or \item\label{it:hgraph} $C_e\cup C$ is $C_e$ plus a crossbar. \end{enumerate}\end{claim}

\begin{proof}  Let $F$ be the face bounded by $C$ (Lemma \ref{lm:greenCycles} (\ref{it:CboundsFace})). Obviously $F$ is \wording{not inside the} face bounded by $\Delta_e$, and, since $F$ is contained in $\Mob$, $F$ is $C_e$-interior.  Let $y$ be the end of $e^w$ nearer $w$ in $P_2$; then $y\in r_i$.  From the definition of $H$-green cycle (Definition \ref{df:green}), the edge of the $yx^{e_w}$-subpath of $P_2$ incident with $y$ is in $C$.

If $w$ is an attachment of a global $H$-bridge, then every edge of $C\cap R$ is in two $H$-green cycles, which is impossible by Theorem \ref{th:twoGreenCycles}.  Therefore, $P_2=P'_2$,  $y=w$, and $C$ is the union of the $wz$-path $C\cap P_2$ (this defines $z$) and an $R$-avoiding $wz$-path $P$.  

The path $P$ contains a subpath $P'$ joining a vertex of the $zx$-subpath of $P_2P_x$ to a vertex of the component of $P_w-e^w$ containing $w^e$; we may assume $P'$ is $C_e$-avoiding.  Claim \ref{cl:P1disjointBoundsFace} implies that the cycle contained in $P'\cup P_wP_2P_x$ bounds a face of $\Pi[G]$.  As $z$ is in this cycle, it must be that $z$ is an end of $P'$ and, moreover, this cycle is $C$.  In particular, $P$ is just $P'$ plus a subpath of $P_w$.
We know that $C\cap P_2$ is just an edge.  Since the path $C\cap P_w$ is in the boundary of the faces bounded by both $C$ and $\Delta_e$, it is \wording{also just the edge $e_w$}.

If $z\ne x^{e_w}$, then $P'=P$ and the $zw^e$-path contained in $P\cup P_w$ contradicts Claim \ref{cl:noTileCrossing} (\ref{it:wSide}).  Therefore, $z=x^{e_w}$.

Let $B$ be the $C_e$-bridge containing $P'$.   Claim \ref{cl:twoAtts} implies $B$ has precisely two attachments $w'\in P_w$ and $x'\in P_x$:  therefore, $B$ is just the edge $w'x'$ (this is also $P'$).  If $x'$ is $x^{e_w}$, then $B$ is a $w$-backslope.  

Finally, suppose $x'$ is in $P_x-x^{e_w}$.  Then $C$ bounds a face incident with $C\cap P_x$.  Since $C\cap P_x$ is also in the boundary of the face bounded by $\Delta_{e_w}$,  it has length 1.   

On the $P'_1$ side, $B$ together with the $w'x'$-subpath of $\oo{P_wP'_1P_x}$ is a cycle $C'$ disjoint from $P'_2$.  By Claim \ref{cl:P1disjointBoundsFace},  $C'$ bounds a face of $\Pi[G]$.  As above, each of $C'\cap P_w$, $C'\cap P_x$, and $P'_1$ all have length 1.  Therefore, $B$ is a crossbar.\end{proof}
 
Our final case is that $e^w$ is in an $H$-yellow cycle.

\begin{claim}  If $e^w$ is in an $H$-yellow cycle $C$, then  either \begin{enumerate}[label=(\alph*), start=4]\item\label{it:backSlopeA} $C_e\cup C$ contains $C_e$ plus a back-slope or \item\label{it:hgraphA} $C_e\cup C$ is $C_e$ plus a crossbar. \end{enumerate}\end{claim}

\begin{proof}  Let $C'$ be the $H$-green cycle witnessing that the cycle $C$ containing $e^w$ is $H$-yellow.  Then $C'$ contains an $H$-jump $J$ and either both ends of $J$ are in $P_1$ or both ends of $J$ are in $P_2$.    In either case, Claim \ref{cl:P1jump} implies the span of $J$ is all of $P_1$ or $P_2$.  We treat these two possibilities separately.

\startSubclaims
\begin{subclaim}\label{sc:bothP2}  If both ends of $J$ are in $P_2$, then \ref{it:hgraphA} occurs.
\end{subclaim}

\begin{proof} In this case, Claim \ref{cl:P1jump} implies $J$ has ends $w$ and $x^{e_w}$, $x^{e_w}= v_{i+3}$, $e^w\in s_{i+1}$, and $e^x\in s_{i+2}$.  

Because $e_w$ is both incident with $v_{i+1}$ and in an $H$-yellow cycle as witnessed by the $H$-green cycle $C'$ containing $J$, $v_{i+1}$ is in the interior of the span of $J$; consequently, $w\in \oo{r_i}$.  Therefore, the edge of $r_i$ incident with $v_{i+1}$ is $H$-green.  

We observe that $J$ witnesses that $Q_{i+1}$ is an $H$-yellow cycle.  It follows from Lemma \ref{lm:yellowCycles} (\ref{it:oneBridge}) that $C=Q_{i+1}$.  The same part of the same lemma combines with the fact that $e$ is not $H$-green to show that $P_w$ consists of $\cc{w,r_i,v_{i+1},s_{i+1},v_{i+6}}$ and that $P_w$ has length precisely two.  Symmetrically, $P_x$ consists of $s_{i+2}\,r_{i+2}$ and has length 2.   Therefore, we have \ref{it:hgraphA}, as required. \end{proof}

It remains to consider the possibility that both ends of $J$ are in $P_1$.
 Claim \ref{cl:P1jump} implies $J=w^ex$,  $w^e=v_{i+5}$, $e^w\in s_{i+1}$, and $e^x\in s_{i+2}$.  Also, $r_{i+5}$ is in the $H$-green cycle $C'$ containing $J$, and so $P_w$ contains $r_{i+5}\,s_{i+1}$.  Since Theorem \ref{th:redHasDelta} (\ref{it:deltaDetail}) implies $P_w$ has at \wording{most one $H$-rim edge}, we conclude that $w=v_{i+1}$.  Recall that $P_w$ is in the boundary of the face of $\Pi[G]$ bounded by $\Delta_e$.  The path $r_{i+5}$ is also in the boundary of the face bounded by $C'$ and so is just an edge.  The path $s_{i+1}$ is also in the boundary of the face bounded by $C$, so it too is just an edge.

If $J$ is not incident with $v_{i+7}$, then the situation is precisely that Subclaim \ref{sc:bothP2} with the roles of $(e,w)$ and $(e_w,x)$ interchanged.  Therefore, $C_e\cup C$ is \ref{it:hgraphA}, as required.

Therefore, we may assume $J$ is incident with $v_{i+7}$.   At this point, we know that $s_{i+1}\,r_{i+5},J$ and at least the edge $e^x$ of $s_{i+2}$ are contained in $C_e$.  There is a $C_e$-bridge containing $r_{i+6}$; Claim \ref{cl:twoAtts} implies this $C_e$-bridge is precisely $r_{i+6}$ and this is just an edge.

The cycle $C$ has a second edge $e'$ incident with $v_{i+6}$.  There is a $C_e$-bridge $B$ containing $e'$.  Claim \ref{cl:twoAtts} implies $B$ has precisely two attachments, namely $v_{i+6}$ and some other vertex $y$.  

If $y\in P_x-x^{e_w}$, then $B$ together with the $yv_{i+6}$-subpath of $C_e-P'_2$ contains a cycle disjoint from $P'_2$ and yet does not bound a face (it contains $r_{i+6}$).  We know that $r_{i+5}\,r_{i+6}\,J$ bounds a face of $\Pi[G]$, so $y$ is not in $J\,r_{i+5}$.   Claim \ref{cl:noTileCrossing} implies $y\notin P'_2-x^{e_w}$.  Thus, $y=x^{e_w}$.  

To finish the proof that \ref{it:backSlopeA} occurs, note first that $s_{i+1}$ and $B$ are both edges; thus, it suffices to prove that $P'_2$ is just an edge.  In fact, Claim \ref{cl:P1disjointBoundsFace} implies $P'_2\,B\,s_{i+1}$ bounds a face of $\Pi[G]$.  In particular, $P_2$ is not inside this face; therefore,  $P'_2=P_2$.  Consequently, $P'_2=P_2$ is just an edge. \end{proof}

In order to determine the 13 pictures, we remark that, from the perspective of both $e$ and $e_w$, any of (\ref{it:wDigon})-(\ref{it:wHgraph}) in Definition \ref{df:halfTiles} can occur.  However, if (\ref{it:wHgraph}) occurs for either, then Claim \ref{cl:P1disjointBoundsFace} implies $C_e$ and this crossbar is all that is in $\Disc_e$.  In the cases (\ref{it:wSlope})(\ref{it:wBackSlope}) and (\ref{it:wDigonSlope})(\ref{it:wBackSlope}), there are two possibilities, as the slope and the back-slope can have either distinct or common ends in the spoke; the latter is denoted by a ${}^+$ in the listing below.  There is no third possibility, since the slope and back-slope do not cross in $\Disc_e$.    Thus, there are the 13 pictures (\ref{it:wDigon})(\ref{it:wDigon}), (\ref{it:wDigon})(\ref{it:wSlope}), (\ref{it:wDigon})(\ref{it:wDigonSlope}), (\ref{it:wDigon})(\ref{it:wBackSlope}), (\ref{it:wSlope})(\ref{it:wSlope}), (\ref{it:wSlope})(\ref{it:wDigonSlope}), (\ref{it:wSlope})(\ref{it:wBackSlope}), (\ref{it:wSlope})(\ref{it:wBackSlope})${}^+$, (\ref{it:wDigonSlope})(\ref{it:wDigonSlope}), (\ref{it:wDigonSlope})(\ref{it:wBackSlope}), (\ref{it:wDigonSlope})(\ref{it:wBackSlope})${}^+$, (\ref{it:wBackSlope})(\ref{it:wBackSlope}), and (\ref{it:wHgraph})(\ref{it:wHgraph}).

\dragominor{Label the red edges in $G$ as $e_0,e_1,\dots,e_{k-1}$ so that, for $i=0,1,\dots,k-1$, $e_i$ has ends $u_i$ and $v_i$ and so that, reading indices modulo $k$, $e_{i+1}$ is the $v_i$-consecutive red edge for $e_i$.  This implies that $e_i$ is the $u_{i+1}$-consecutive red edge for $e_{i+1}$.  }

\dragominor{Since there are no red edges between $e_{i-1}$ and $e_{i+1}$ on the ``peak of $\Delta_{e_i}$" portion of $R$, defining adjacency to mean ``consecutive" shows the set of red edges make a cycle.  Furthermore, $v_i$ and $u_{i+1}$ are both in the cycle $C_{e_i}$ that determines the picture $\mathfrak P_i$ between $e_i$ and $e_{i+1}$.  Taking any $v_iu_{i+1}$-path $P_i$ in $\mathfrak P_i$, we see that $P_i$ together with either of the $v_iu_{i+1}$-subpaths of $R$ makes a non-contractible cycle in $\pp$.  In this sense, $e_i$ and $e_{i+1}$ are on opposite sides of $R$. }

\dragominor{If we think of $e_0$ as being on ``top" and $e_1$ on the ``bottom", then $e_2$, $e_4$, \dots are all on top and $e_3$, $e_5$ \dots, are on the bottom.    When we get back to $e_0$ from $e_{k-1}$, we have gone once around the M\"obius strip, so $e_0$ is now on the bottom.  It follows that $e_{k-1}$ is on top and, therefore, $k-1$ is even, showing $k$ is odd.}

It follows that $G$ contains a subgraph $H$ that is in $\mathcal T(S)$.  (There may be edges in the interior of $C_e$ ``between" the structures we identified ``near" $P'_1$ and $P'_2$.)  However,   Theorem \ref{th:tiledAre2cc} implies $H\in \m2$, so we conclude $G=H$.  That is, $G\in \mathcal T(S)$.
\end{cproof}
}

\chapter{Graphs that are not 3-connected}\label{sec:not2conn}\printFullDetails{

The rest of this work is devoted to: describing all the \2cc\ graphs that are not 3-connected, discussed in this section; finding all 3-connected \2cc\ graphs that do not contain a subdivision of $V_8$, treated in Section \ref{sec:3conNotI4c}; and showing that the number of 3-connected \2cc\ graphs that do not contain a subdivision of $V_{2n}$ is finite, which is Section \ref{sc:noV2n}.  These last two combine with the preceding work to show that there are only finitely many 3-connected \2cc\ graphs to be determined, namely those that have a subdivision of $V_8$ but no subdivision of $V_{10}$.

In this section we show that every 2-crossing-critical graph that is not 3-connected is either one of a few known examples or is obtained from a graph in $\m2$ by replacing 2-parallel edges with a ``digonal" path (that is, a path in which every edge is duplicated).   We remark that we continue assuming that the minimum degree is at least 3, as subdividing edges does not affect crossing number.  We first determine all the 2-crossing-critical graphs that are not 2-connected.

}\section{2-critical graphs that are not 2-connected}\printFullDetails{

Since the crossing number is additive over components,
any $2$-crossing-critical graph can have 
at most two components, each of them equal to 
either $K_{3,3}$ or $K_5$.  Thus, there are only three different such graphs:  two disjoint copies of $K_5$, two disjoint copies of $K_{3,3}$, and disjoint copies of each.

Similarly, the crossing number is easily seen to be additive over blocks.  Thus, the blocks of a connected, 
but not $2$-connected, $2$-crossing-critical graph must be $1$-critical 
graphs, and therefore all such graphs can be obtained from the aforementioned 
disconnected $2$-crossing-critical graphs by identifying two vertices 
from distinct components.    The identified vertex may be a new vertex that subdivides some edge.   For example, there are three possibilities in which both blocks are $K_5$:  the identified vertex is a node in both, or only in one, or in neither.  Likewise for $K_{3,3}$.  There are four 2-crossing-critical graphs in which one block is a subdivision of $K_5$ and the other is a subdivision of $K_{3,3}$.

}\begin{proposition}
The thirteen graphs in  Figure \ref{fg:1connected} are precisely those 2-crossing-critical graphs that are not 2-connected.
\end{proposition}\printFullDetails{

\begin{figure}[!ht]
\begin{center}
\includegraphics[scale=.275]{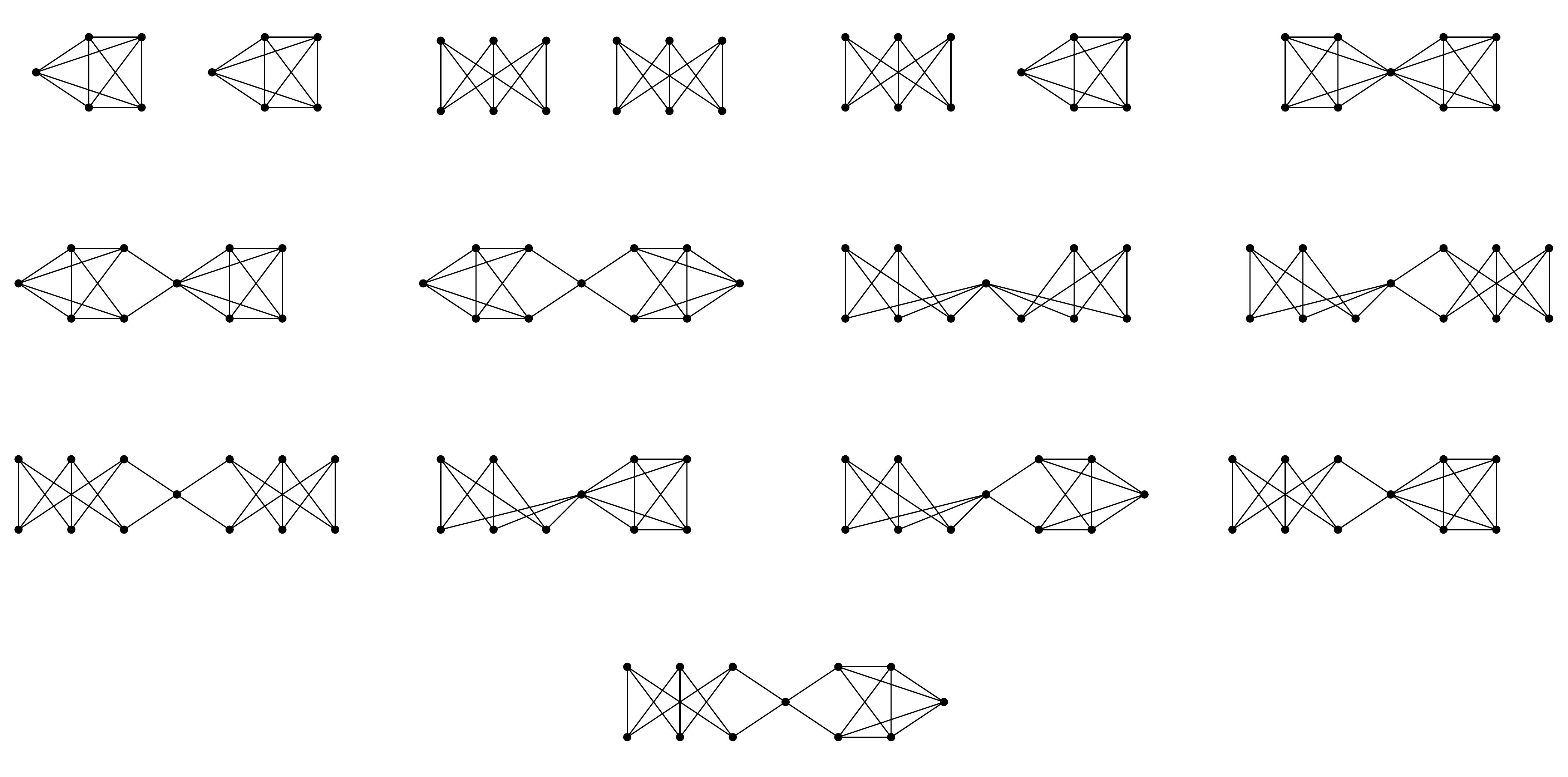}
\caption{The $2$-crossing-critical graphs that are not $2$-connected.}             
\label{fg:1connected}
\end{center}
\end{figure}

}
\section{2-connected 2-critical graphs that are not 3-connected}{\label{sec:2conNot3conn}\printFullDetails{

\def\iso#1{\|#1\|}

\def\isouv{\iso{\{u,v\}}}

In this subsection, we treat
$2$-crossing-critical graphs that are $2$-connected, but not 
$3$-connected.  With 36 exceptions,  these all arise from 3-connected 2-crossing-critical graphs that have {\em digons\/}\index{digon} (i.e., two edges with the same two ends).  The digons may be replaced with arbitrarily long  ``digonal paths"\index{digonal path} --- these are simply paths in which every edge is converted into a digon.

Tutte  \cite{TutCinG, TutEncy} developed a decomposition theory of a 2-connected graph into its {\em cleavage units\/}\index{cleavage unit}, which are either 3-connected graphs, cycles of length at least 4, or for $k\ge 4$, $k$-bonds (a {\em $k$-bond\/}\index{bond}\index{$k$-bond} is a graph with $k$ edges, all having the same two ends).  We provide here a brief review of this theory.  A {\em $2$-separation\/}\index{2-separation}\index{separation}  of a 2-connected graph $G$ is a pair $(H,K)$ of edge-disjoint subgraphs of $G$, each having at least two edges, so that $H\cup K=G$ and $H\cap K=\isouv$ \wording{(recall $\isouv$ is the graph with just the vertices $u$ and $v$ and no edges.)}.   Notice that a 3-cycle and a 3-bond have no 2-separations and, therefore, are to be understood in this context to be 3-connected graphs.

\wordingrem{(Redundant text removed.)}%
The 2-separation $(H,K)$ \wording{with $H\cap K=\isouv$ is} a {\em hinge-separation\/}\index{hinge}\index{hinge-separation} if at least one of $H$ and $K$ is a $\isouv$-bridge and at least one of them is 2-connected.  Another way to say the same thing, but in terms of $H\cap K$:  $\isouv$ is a {\em hinge\/} if either there are at least three $\isouv$-bridges, not all just edges, or there are exactly two $\isouv$-bridges, at least one of which is 2-connected.

The theory of cleavage units develops as follows.  Let $G$ be a 2-connected graph.   

\begin{enumerate} \item  If $\isouv$ is a hinge and $(H,K)$ is a hinge-separation (possibly of another hinge), then there is some 
$\isouv$-bridge containing either $H$ or $K$.  

\item $G$ has no hinge if and only if $G$ is 3-connected, a cycle of length at least 4, or a $k$-bond, for some $k\ge 4$.  (Recall that a 3-cycle and a 3-bond are 3-connected.)  In each of these cases, $G$ is its own cleavage unit.

\item If  $(H,K)$ is a hinge-separation and $H\cap K=\isouv$, then the cleavage units of $G$ are the cleavage units of the two graphs $H+uv$ and $K+uv$ obtained from $H$ and $K$ by adding a {\em virtual edge\/}\index{virtual edge} between $u$ and $v$, respectively.  This inductively determines the cleavage units.

\item There is a decomposition tree $T$ whose vertices are the cleavage units of $G$ and whose edges are the virtual edges.  A virtual edge joins in $T$ the two cleavage units of $G$ containing it.

\item $G$ contains a subdivision of each of its cleavage units.

\item\label{subdiv} If $G$ contains a subdivision of some 3-connected graph $H$, then some cleavage unit of $G$ contains a subdivision of $H$.

\end{enumerate}

In attempting to reconstruct $G$ from its decomposition tree and its cleavage units, each time we combine two graphs along a virtual edge, there are two possibilities for how to identify the vertices of the corresponding hinge.  This ambiguity will play a small role in constructing the 2-crossing-critical graphs that are 2- but not 3-connected.

It is easy to see that $G$ is planar if and only if every cleavage unit is planar. (We could apply Kuratowski's Theorem and Item \ref{subdiv} or prove it more directly.)  Since we are interested in non-planar graphs, there are two relevant possibilities:  one or more than one of the cleavage units of $G$ is not planar.  We start by treating the latter case.  We remark that the following discussion makes clear that the crossing number is not additive over cleavage units.  Related discussions can be found in \v Sir\'a\v n \cite{siran},  Chimani, Gutwenger, and Mutzel \cite{chimani} (but see \cite{bb} for significant comments about the latter), Beaudou and Bokal \cite{bb}, and Lea\~nos and Salazar \cite{ls}.

}\begin{lemma}\label{twounits}  Let $G$ be a 2-connected graph.  If two cleavage units of $G$ are not planar,  then $\crn (G)\ge 2$. \end{lemma}\printFullDetails{

It is an important consequence that, if $G$ is 2-crossing-critical, 2-connected, and has 2 non-planar cleavage units, then $G$ is simple, i.e., has no digons.

\begin{cproofof}{Lemma \ref{twounits}}
Among all 2-separations $(H,K)$ of $G$, we choose the one that has $K$ minimal so that both $H+uv$ and $K+uv$ are not planar, where $H\cap K=\isouv$.  If the crossing number of $G$ is not at least 2, then $\crn(G)\le1$\wording{, so, by way of contradiction, suppose $D$ is a} 1-drawing of $G$.

Let $P_K$ and $P_H$ be $uv$-paths in $K$ and $H$ respectively.  Since $G$ contains the subdivision $H\cup P_K$ of $H+uv$, $G$ is not planar.  Therefore, $D$ has a crossing.
Evidently, $D(H\cup P_K)$ and $D(K\cup P_H)$ both contain the crossing.  We conclude that the crossing in $D$ is of an edge of $P_H$ with an edge of $P_K$.  It follows that there are not edge-disjoint $uv$-paths in either $H$ or $K$ and that the crossed edges are cut-edges in their respective subgraphs.

Let $w$ and $x$ be the ends of the edge in $K$ that is crossed, labelled so that $w$ is nearer to $u$ in $P_K$ than $x$ is.  Let $K_u$ and $K_v$ be the two components of $K-wx$, with the former containing $u$.  Since $K+uv$ is not planar,  either $K_u+uw$ or $K_v+vx$ is not planar.  We may assume it is the former. Notice that $(H\cup K_v)+xu$ contains a subdivision of $H+uv$ and, therefore, is not planar.   But then $((H\cup K_v),K_u)$ is a 2-separation contradicting the minimality of $K$. 
\end{cproofof}

We are now in a position to determine the 2-connected, 2-crossing-critical graphs having two non-planar cleavage units.

}\begin{theorem}\label{2conn}  Let $G$ be a 2-connected, 2-crossing-critical graph having two non-planar cleavage units.  Then $G$ is one of the 36 graphs in Figures \ref{2units} and \ref{3units}.\end{theorem}\printFullDetails{

\begin{figure}[!ht]
\begin{center}
\includegraphics[scale=.275]{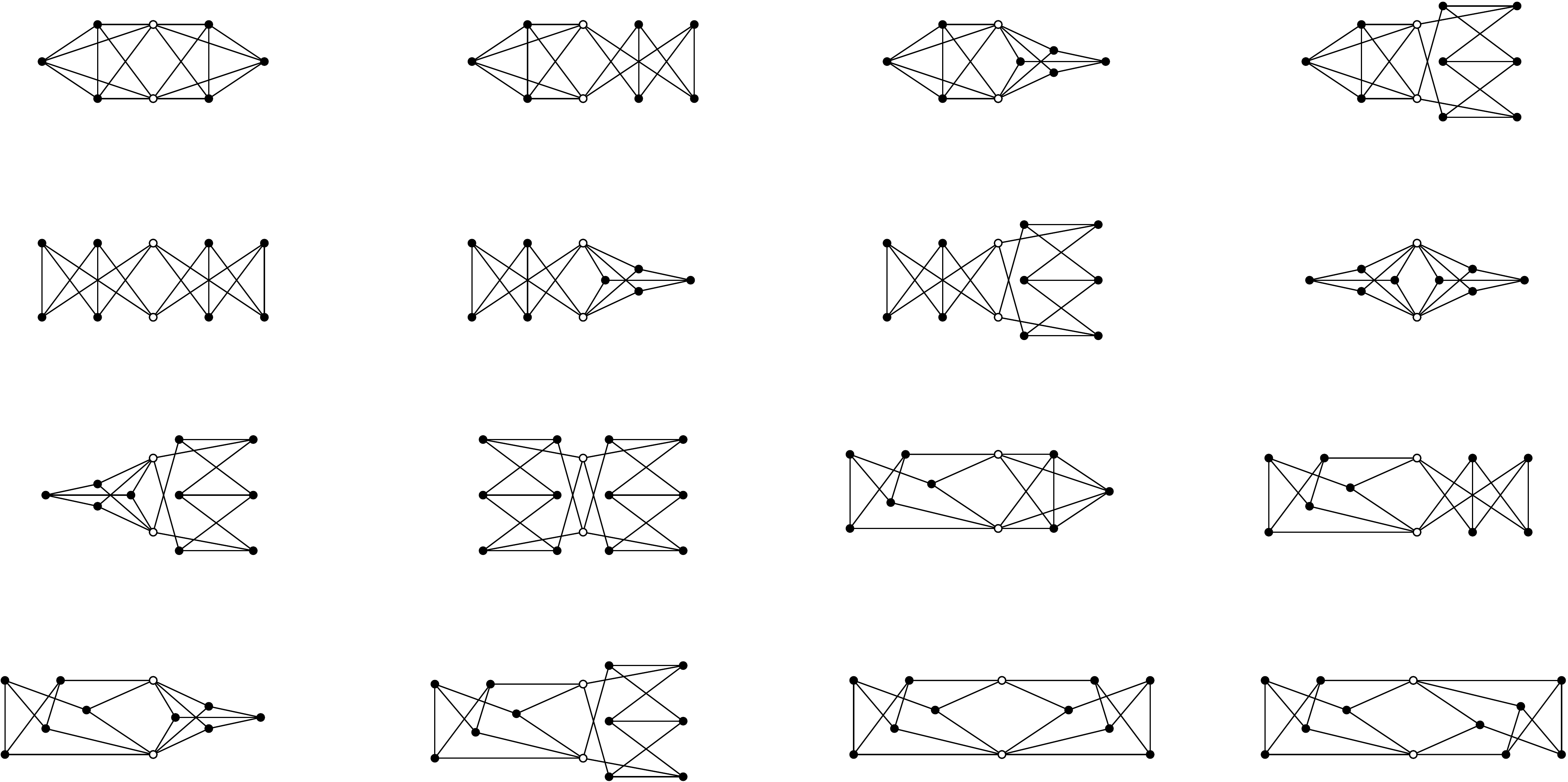}
\caption{\wording{2-connected, not 3-connected, $2$-crossing-critical graphs, 2 non-planar cleavage units}}          
\label{2units}
\end{center}
\end{figure}

\begin{figure}[!ht]
\begin{center}
\includegraphics[scale=.75]{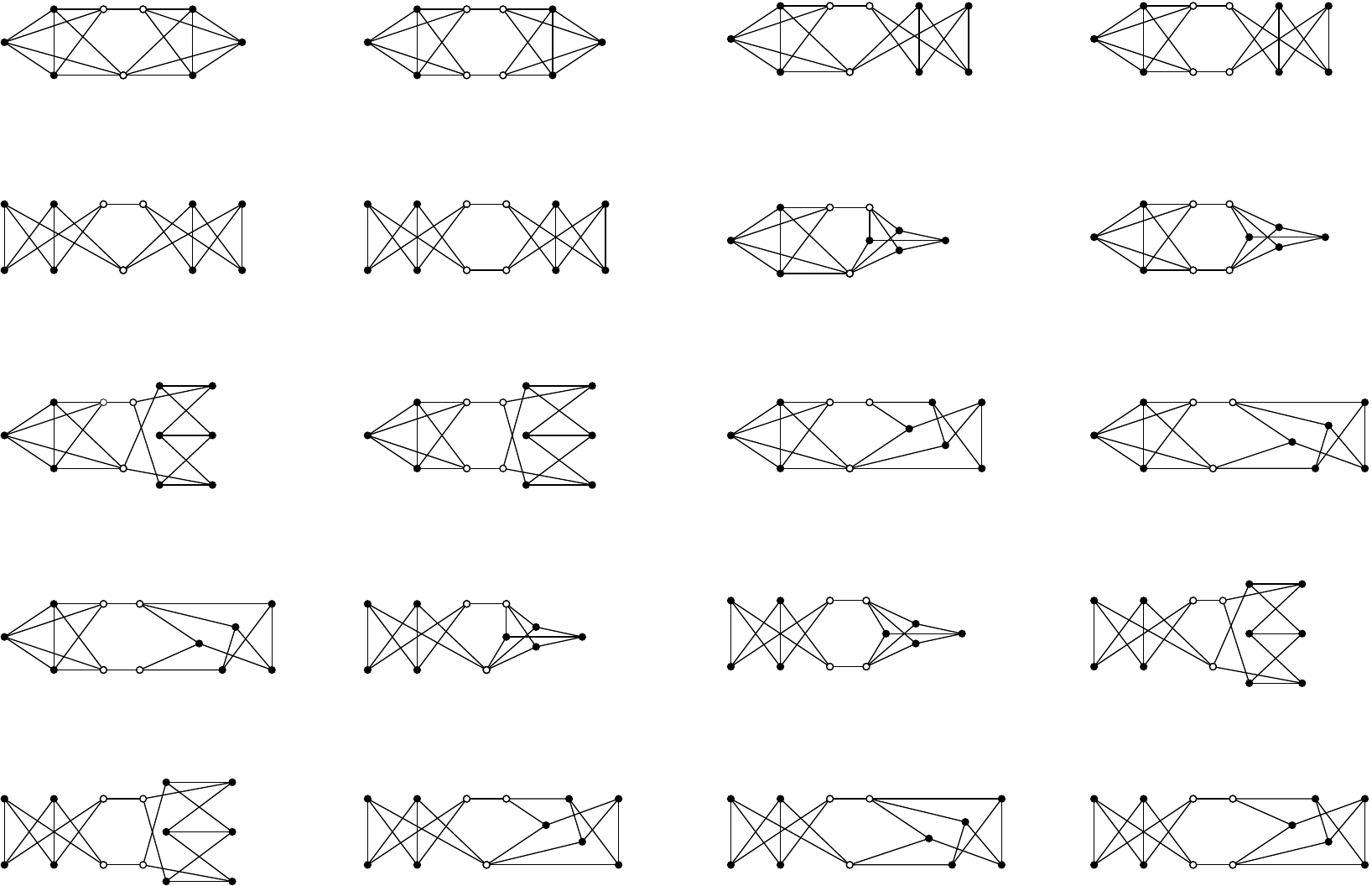}
\caption{\wording{2-connected, not 3-connected, $2$-crossing-critical graphs, 3 cleavage units, 2 of which are non-planar.}}             
\label{3units}
\end{center}
\end{figure}

\begin{cproof} Let $C_1$ and $C_2$ be non-planar cleavage units of $G$.

\begin{claim}  $G$ has at most three cleavage units:  $C_1$, $C_2$ and possibly a 3- or 4-cycle; if there are three, then the 3- or 4-cycle is the internal vertex in the decomposition tree. \end{claim}

\begin{proof} For $i=1,2$, let $\{u_i,v_i\}$ be the hinge of $G$ contained in $C_i$ such that $C_1$ and $C_2$ are contained in different $\iso{\{u_i,v_i\}}$-bridges.  For any other virtual edge $xy$ in $C_i$, there is a path $P_{xy}$ in $G$ that is $C_1\cup C_2$-avoiding.  Let $\widetilde C_i$ be $C_i\cap G$ (i.e., $C_i$ with none of its virtual edges) together with all these $P_{xy}$.  Let $H$ be the subgraph of $G$ consisting of $\widetilde C_1\cup \widetilde C_2\cup Q$, where $Q$ consists of two  disjoint $\{u_1,v_1\}\{u_2,v_2\}$-paths in $G$.  Evidently, $H$ is 2-connected and $C_1$ and $C_2$ are cleavage units of $H$.

Lemma \ref{twounits} implies $\crn(H)\ge 2$.  Since $H\subseteq G$ and $G$ is 2-crossing-critical, $H=G$.  Since $G$ has no vertices of degree 2, $G$ consists of either two or three cleavage units, namely $C_1$, $C_2$, and possibly a 3- or 4-cycle between them. \end{proof}

We next determine the possibilities for $C_1$ and $C_2$.

\begin{claim}  For each $i=1,2$, one of the following occurs:
\begin{enumerate} \item $C_i$ is $K_5$;
\item  $C_i$ is $K_{3,3}$;
\item  $C_i-u_iv_i$ is a subdivision of $K_{3,3}$.
\end{enumerate}\end{claim}

\begin{proof} Hall proved that every 3-connected non-planar graph is either $K_5$ or contains a subdivision of $K_{3,3}$ \cite{dwh}.  Since $G$ is simple and $C_i$ is 3-connected, we deduce that $C_i$ is either $K_5$ or contains a subdivision of $K_{3,3}$.   So suppose $C_i$ contains a subdivision $K$ of $K_{3,3}$.

Suppose $C_i-u_iv_i$ has an edge $e$ for which $C_i-e$ is not planar.  Since $C_i-e$ is \wording{2-connected, $G-e$ is} 2-connected and has at least two non-planar cleavage units ($C_{3-i}$ and another contained in $C_i-e$).  By Lemma \ref{2units}, $\crn(G-e)\ge 2$, contradicting 2-criticality of $G$.  So $C_i-u_iv_i\subseteq K$.    Thus, either $C_i=K$ or $C_i-u_iv_i=K$, as claimed. \end{proof}

\begin{claim}\label{unittypes}  There are five possibilities for $C_i$,  namely:
\begin{enumerate}
\item\label{k5} $C_i$ is $K_5$;
\item\label{k33} $C_i$ is $K_{3,3}$; 
\item\label{k33a}   $C_i-u_iv_i$ is $K_{3,3}$ and $u_iv_i$ joins two non-adjacent nodes of $K_{3,3}$;
\item\label{k33b}  $C_i-u_iv_i$ is $K_{3,3}$ with one edge subdivided once and $u_iv_i$ joins the degree 2 vertex to a  node of $K_{3,3}$ that is not incident with the subdivided edge; and
\item\label{k33c}  $C_i-u_iv_i$ is $K_{3,3}$ with two non-adjacent edges both subdivided once  and $u_iv_i$ joins the two degree 2 vertices.
\end{enumerate}\end{claim}

\begin{proof}  If $C_i$ is neither $K_5$ nor $K_{3,3}$, then it must be a subdivision $K$ of  $K_{3,3}$ with the additional edge $u_iv_i$.  Clearly $K$ has at most two vertices of degree 2.  If $K$ has no vertices of degree 2, then, since $C_i$ is simple, we have (\ref{k33a}).  Likewise, if $K$ has only one vertex of degree 2, that vertex (one of $u_i$ and $v_i$)  cannot be in  a branch incident with the other one of $u_i$ and $v_i$, which is (\ref{k33b}).  Finally, suppose $u_i$ and $v_i$ are both of degree 2 in $K$.  Then their containing branches cannot be incident with a common vertex $w$, as otherwise, we could delete the edge $u_iw$ and still have two non-planar cleavage units, contradicting 2-criticality.  This proves (\ref{k33c}).  \end{proof}

Note that in all five cases of Claim \ref{unittypes}, there is only  one possibility for $C_i$, up to isomorphism.  Only (\ref{k33b}) has non-isomorphic labellings of $u_i$ and $v_i$.

\begin{claim}  If $G$ has just two cleavage units, then $G$ is one of the 16 graphs in Figure  \ref{2units}.\end{claim}

\begin{proof}  If neither  $C_1$ nor  $C_2$ is (\ref{k33b}) from Claim \ref{unittypes}, then, with repetition allowed, there are 10 possible unordered pairs for $C_1$ and $C_2$.  Each of the pairs uniquely produces the graph $G$.  There are four graphs having $C_1$ but not $C_2$ satisfying Claim \ref{unittypes} (\ref{k33b}), and there are two graphs having both $C_1$ and $C_2$  satisfying Claim \ref{unittypes} (\ref{k33b}).\end{proof}

\begin{claim}\label{threeunits} If $G$ has three cleavage units, then at least one of $C_1$ and $C_2$ is either $K_{5}$ or $K_{3,3}$.\end{claim}

\begin{proof}  Let $e$ be an edge of $G$ in the third cleavage unit of $G$; recall that this cleavage unit is either a 3- or a 4-cycle.  The blocks of $G-e$ include $C_1-u_1v_1$ and $C_2-u_2v_2$; if these were both non-planar, then $\crn(G-e)\ge2$, contradicting 2-criticality of $G$.  Hence, at least one of $C_1-u_1v_1$ and $C_2-u_2v_2$ is  planar.  By Claim \ref{unittypes}, such a one must be either $K_5$ or $K_{3,3}$. \end{proof}

\begin{claim} If $G$ has three cleavage units, then $G$ is one of the 20 graphs in Figure \ref{3units}.\end{claim}

\begin{proof}  There are three pairs in which both $C_1$ and $C_2$ are one of $K_5$ and $K_{3,3}$ and two possibilities for the third cleavage unit, yielding six graphs.  Now suppose $C_1$ is one of $K_5$ and $K_{3,3}$ and $C_2$ is not.  There are three possibilities for $C_2$ and two possibilities for the third bridge.  However, when the third bridge is a 3-cycle, there are two ways to attach $C_2$ when it is of Type (\ref{k33b}) from Claim \ref{unittypes}.  Thus, there are 6 graphs with the third cleavage unit a 4-cycle and 8 when it is a 3-cycle.  \end{proof}

From the claims, we see that the 36 graphs shown in Figures  \ref{2units}  and \ref{3units} are all the cases in which $G$ is 2-connected, but not 3-connected, and has two non-planar cleavage units.
\end{cproof}

In the remaining cases of 2-connected, but not 3-connected, 2-crossing-critical graphs, there is only one non-planar cleavage unit $C$.   The graph $C$ is simple.  The following result shows how to obtain $G$ from a 3-connected 2-crossing-critical graph.  \wording{It requires the following definition.}

}\begin{definition}\label{df:digonalPath}   \wording{A {\em digonal path\/}\index{digonal path} is a graph obtained from a path $P$ by adding, for every edge $e$ of $P$, an edge parallel to $e$.} \end{definition}\printFullDetails{

}\begin{theorem}  Let $G$ be a 2-crossing-critical graph with minimum degree at least 3.  Suppose that $G$ is 2-connected but not 3-connected and has only exactly one non-planar cleavage unit, $C$.  The graph $\widetilde C$ obtained from $C$ by replacing each of its virtual edges with a digon is 2-crossing-critical and 3-connected.  The graph $G$ is recovered from $\widetilde C$ by replacing these virtual edge pairs by digonal paths. \end{theorem}\printFullDetails{

\begin{cproof}  That $\widetilde C$ is 3-connected is a trivial consequence of the fact that $C$ is 3-connected.   

As for the rest, let $uv$ be a virtual edge in $C$.  Then $\isouv$ is a hinge of $G$.  We consider the $\isouv$-bridges in $G$; let $B_{uv}$ be the one that contains $C\cap G$, and \wording{let $\comp B_{uv}$ be} the union of the remaining $\isouv$-bridges.   We have two objectives:  to show that $\widetilde C$ is 2-crossing-critical and that\wording{, for each $uv$, $\comp B_{uv}$ is a digonal $uv$-path}.

For the former, we first show $\crn(\widetilde C)\ge 2$.  Otherwise $\widetilde C$ has a 1-drawing $D$.  Obviously no edge in a digon of $\widetilde C$  is crossed in $D$.  For each virtual edge $uv$ \wording{of $C$, $\comp B_{uv}+uv$ is} planar, so it may be inserted into $D$ in place of the $uv$-digon  in $D$ to obtain a 1-drawing of $G$, which is a contradiction.  \wording{Therefore, $\crn(\widetilde C)\ge 2$.}

We next claim that \wording{each $\comp B_{uv}$ consists} of  digonal  $uv$-paths.  Assume first \wording{that $\comp B_{uv}$ has} a cut-edge $e$ separating $u$ and $v$.  Since $G$ has no vertices of degree 2 \wording{and $\comp B_{uv}$ is not} just a single \wording{edge, $\comp B_{uv}$ contains} some edge $e'$ so \wording{that $\comp B_{uv}-e'$ still} contains a $uv$-path.  

If no \wording{edge of $\comp B_{uv}$ is} crossed \wording{in a 1-drawing $D_{e'}$ of $G-e'$, then, since $\comp B_{uv}-e'$ contains a $uv$-path, $\comp B_{uv}$} may be substituted \wording{for $\comp B_{uv}-e'$ in $D_{e'}$ to} obtain a 1-drawing of $G$, which is impossible.  So some \wording{edge of $\comp B_{uv}$ is} crossed in $D_{e'}$.  Deleting edges from $\comp{B}_{uv}-e'$ to leave only a $uv$ path shows that $D_{e'}$ restricts to a 1-drawing of $B_{uv}+uv$ in which there is at most one crossing; if there is a crossing, then $uv$ is crossed.  Since every planar \wording{embedding of $\comp B_{uv}+uv$ has} $uv$ and $e$ on the same face, the 1-drawing of $B_{uv}+uv$ and a planar embedding \wording{of $\comp B_{uv}+uv$ may be} merged to produce a 1-drawing of $G$ \wording{in which $e$ is crossed.  This} contradiction that \wording{shows $\comp B_{uv}$ contains} edge-disjoint $uv$-paths.

Let $e$ be an \wording{edge of $\comp B_{uv}$}.  Then \wording{a 1-drawing $D_e$ of $G-e$ must} have a crossing of some edge $e'$ \wording{of $\comp B_{uv}$.  If $\comp B_{uv}-\{e,e'\}$ has} a $uv$-path $P$, then $D_e$ restricts to a planar embedding of $C$ by using $P$ to represent $uv$.  But $C$ is non-planar, so every \wording{edge of $\comp B-{uv}$ is} in an edge-cut of size at most 2 separating $u$ and $v$.  \wording{Combining this with the preceding paragraph shows that every edge of $\comp B_{uv}$ is in an edge-cut of size exactly 2.  It is an easy exercise to see that this implies $\comp B_{uv}$ is} a pair of digonal paths.

We conclude by showing that, for every edge $e$ of $\widetilde C$, $\crn(\widetilde C-e)\le 1$.   Suppose first that $e$ is not in a digon.  \wording{Each $\comp B_{uv}$ has} a  $uv$-path $P_{uv}$ that is clean in $D_e$.  Thus, $D_e[G-e]$ contains a subdivision of $C-e$ in which no virtual edge (represented in the subdivision by $P_{uv}$) is crossed.  Therefore, the virtual edges may be replaced with digons to give a 1-drawing of $\widetilde C-e$, as claimed.

Now suppose $e$ is in the $uv$-digon.  Let $e'$ be any \wording{edge of $\comp B_{uv}$}.  Then $D_{e'}$ contains a 1-drawing of $C$, in which every other virtual edge $wx$ is represented by a $wx$-path $P_{wx}$ \wording{in $\comp B_{wx}$ that} is clean in $D_{e'}$.  All these other virtual edges may be replaced with digons to give a 1-drawing of $\widetilde C-e$, as required.  \end{cproof}

}

\chapter{On 3-connected graphs that are not \wording{peripherally}-4-connected}\label{sec:3conNotI4c}\printFullDetails{

In this chapter, we reduce the problem of finding all 3-connected \2cc\ graphs to the consideration of non-planar, \wording{peripherally}-4-connected\index{\p4c} graphs.  \wording{Our motivation for doing this is to use a known characterization of internally-4-connected graphs (a concept intimately related to peripherally-4-connected graphs) with no subdivision of $V_8$ to find all the 3-connected, 2-crossing-critical graphs with no subdivision of $V_8$.}

}\begin{definition}\label{df:i4c}   A graph $G$ is {\em \wording{peripherally}-4-connected\/} if $G$ is 3-connected and, for any 3-cut $S$ of $G$ and any partition of the components of $G-S$ into two non-null subgraphs $H$ and $K$, at least one of $H$ and $K$ has just one vertex. 
\end{definition}\printFullDetails{

We begin this section by finding the four 3-connected, not \wording{peripherally}-4-connected, 2-crossing-critical graphs that are not obtained from planar substitutions into a \wording{peripherally}-4-connected graph.  The bulk of the section is devoted to explaining in detail how to obtain the remaining 3-connected 2-crossing-critical graphs from \wording{peripherally}-4-connected graphs.  \wording{Finally, this theory} is used to explain how to find all the 3-connected \2cc\ graphs that do not contain a subdivision of $V_8$.

}\section{A 3-cut with two non-planar sides}\label{ssc:twoNonPlanar}\printFullDetails{

In this section we find the four 3-connected, not \wording{peripherally}-4-connected, \2cc\ graphs that are not obtained by substituting planar pieces into degree-3 vertices in a \wording{peripherally}-4-connected graph (this substitution process being the remainder of the section).   We start by describing the four graphs and showing that they are \2cc.

}\begin{definition}\label{df:k34*} The graph $K_{3,4}^*$\index{$K_{3,4}^*$} is obtained from disjoint copies of $K_{2,3}$ by joining the parts of the bipartition having three vertices in each of the copies by a perfect matching $M$.  \end{definition}\printFullDetails{

\wording{Observe that $K_{3,4}$ is obtained from $K_{3,4}^*$ by contracting all the edges of the matching $M$.}  The following generalizes the well-known fact that $K_{3,4}$ is 2-crossing-critical.

}\begin{lemma}\label{lm:k34*} If $H$ is obtained from $K_{3,4}^*$ by contracting some subset of $M$, then $H$ is 2-crossing-critical.\end{lemma}\printFullDetails{

\begin{cproof} Suppose $e$ is an edge of $K_{3,4}^*$ not in $M$.  Then there is a 1-drawing of $K_{3,4}^*-e$ in which no edge of $M$ is crossed.  Thus, $\crn(H-e)\le 1$.  If $e\in M$, then $H-e$ is planar. It remains to show $\crn(H)\ge 2$.

Suppose to the contrary that $H$ has a 1-drawing $D$.  Let $H_1$ and $H_2$ be the $K_{2,3}$ subgraphs of $H$ contained in $K_{3,4}^*-M$.  For each vertex $v$ of degree 3 in $H_2$, there are three disjoint $vH_1$-paths in $H$; adding $v$ and these paths to $H_1$ yields a subdivision $H_v$ of $K_{3,3}$ in $H$.  Thus, $D[H_v]$ has a crossing, and, since there are two choices for $v$, this crossing involves only edges of $H_1$ and $M$.  

Interchanging the roles of $H_1$ and $H_2$ shows the crossing in $D$ involves only edges of $M$.  But then $D[H_v]$ has its only crossing on branches incident with $v$, which is impossible.  \end{cproof}

We remark that there are splits of $K_{3,4}$ \wording{that have} crossing number 1 --- split two of the degree 4 vertices so that the two partitions of the four neighbors are different.  Fortunately, they do not occur in our context.  

In order to show that these are the only four graphs with ``non-planar 3-cuts", we need to understand just what ``non-planar 3-cuts" are. 

}\begin{definition}\label{df:3cut+}  Let $S$ be a 3-cut in a 2-connected graph, so there are subgraphs $H$ and $K$ of $G$ such that $G=H\cup K$ and $H\cap K=\iso S$.  For $L\in \{H,K\}$,  $L^+$\index{$L^+$} denotes the graph obtained from $L$ by the addition of a new vertex adjacent to precisely the vertices in $S$.  \end{definition}\printFullDetails{

 We will see that, in the case $G$ is 2-crossing-critical,  with the exception of $K_{3,4}$, there are at most three non-trivial $S$-bridges, and so at least one of $H$ and $K$ is an $S$-bridge.  Our next goal is to show that the four graphs in Lemma \ref{lm:k34*} are the only four that have both $H^+$ and $K^+$ non-planar.  We start with the following\wording{, which is likely well-known; however, we could not find a reference.  It extends Hall's Theorem \cite{dwh} that there is a subdivision of $K_{3,3}$}.

}\begin{lemma}\label{lm:k33node} Let $G$ be a 3-connected non-planar graph different from $K_5$ and let $v$ be a vertex of $G$.  Then $G$ has a subdivision $H$ of $K_{3,3}$ in \wording{which $v$ is an} $H$-node.\end{lemma}\printFullDetails{

\begin{cproof} \wording{Here is an outline of the easy, but tedious, proof.}
As a first step, we show that there is a subdivision of $K_{3,3}$ containing $v$.  By Hall's Theorem \cite{dwh}, $G$ contains a subdivision $L$ of $K_{3,3}$. \wordingrem{(Text removed.)}If $v\notin L$, then there are three disjoint $vL$-paths.  \minor{There are three possibilities for the ends of these paths in $L$:  two are in the same closed $L$-branch; two are in $L$-branches incident with a common $L$-node; and the $L$-branches containing the ends of the paths are pairwise disjoint.  In the first case, $v$ is incorporated into the interior of a branch of a new subdivision of $K_{3,3}$, while in the other cases, $v$ is incorporated as a node of the new subidivision of $K_{3,3}$.}

\ignore{If two of these land on the same closed $L$-branch, then obviously we can replace part of that $L$-branch with a path through $v$.

If there is a closed $L$-branch $\cc{u,w}$ and a half-open $L$-branch $\co{u,x}$ so that the ends of two of the paths are in these, then it is easy to see that $v$ is a node of some subdivision of $K_{3,3}$.  

Finally, if all three land in disjoint $L$-branches, then $v$ can be one node, the three ends of the $vL$-paths are nodes and two other nodes of $L$ are nodes of a subdivision of $K_{3,3}$ in $G$.}

So now assume that $v$ is in $L$, but not as a node.  Then $v$ is interior to some $L$-branch $b$ with ends $u$ and $w$.  Let $L'=L-\oo{b}$ --- this is a subdivision of $K_{3,3}$ less an edge.  \wording{Because there are, in $G$, disjoint} $L'$-avoiding $v\{u,w\}$-paths, standard proofs of Menger's Theorem imply that \wordingrem{(text removed)}there are three disjoint $L'$-avoiding \wording{$vL'$-paths}, having $u$ and $w$ among their \wording{three $L'$-ends}.  Therefore, we may assume not only is $v$ in the interior of $b$, but there is an $L$-avoiding $vx$-path from $v$ to some other vertex $x$ \wording{of $L'$}.

Up to symmetry, there are three possibilities for $x$:  it is a node of $L$ other than $u$ and $w$; it is interior to an $L$-branch incident with $u$ but not with $w$; and it is interior to an $L$-branch not incident with either $u$ or $w$.  Let $y$ and $z$ be nodes of $L'$ (note that $u$ and $w$ are not actually nodes of $L'$).  We can assume $x$ is either $y$, or in the $L$-branch $\cc{w,y}$, or in the $L$-branch $\cc{y,z}$.  Let $Y$ be a $\{u,w,x\}$-claw with centre $v$, so that $Y\cap L'$ is just $u$, $w$, and $x$.   

If $x$ is either $z$ or in $\oo{y,z}$, then $(L'\cup Y)-\oo{w,x}$ is a subdivision of $K_{3,3}$ with $v$ as a node.  If $x$ is in $\oo{y,z}$, then $(L'\cup Y)-\oo{w,y}$ is a subdivision of $K_{3,3}$ having $v$ as a node.  \end{cproof}

We are now ready for the classification of the 3-connected \2cc\ graphs with two non-planar sides to a 3-cut.

}\begin{theorem}\label{th:3CutBothNonPlanar} Let $G\in \m2$ have subgraphs $H$ and $K$ of $G$ and a set $S$ of three vertices of $G$ such that:
\begin{enumerate}
\item
 $G=H\cup K$;
 \item 
$H\cap K=\iso S$;
\item
\label{it:Sbridge}
$H$ and $K$ both \wording{have an} $\iso S$-bridge having all of $S$ as attachments; and
the two graphs $H^+$ and $K^+$  are both non-planar.
\end{enumerate}
Then $G$ is one of the four graphs obtained from $K_{3,4}^*$ by contracting some subset of $M$.\end{theorem}\printFullDetails{

\begin{cproof} \wording{Let $u$, $v$, and $w$ be the vertices in $S$.}  For $L\in\{H,K\}$, let $v^+_L$ denote the vertex in $L^+$, but not in $L$.  The graph $L^+$ is a subdivision of a 3-connected graph (the only possible vertices of degree 2 are $u$, $v$, and $w$).  Since $L^+$ is not planar and has a vertex of degree 3, it is not a subdivision of $K_5$ and, therefore, by Lemma \ref{lm:k33node} contains a subdivision $L'$ of $K_{3,3}$ in which $v^+_L$ is a node.  Now $G'=(H'-v^+_H)\cup (K'-v^+_K)$ is a subdivision of $K_{3,4}^*$, with some subset of $M$ contracted.  By Lemma \ref{lm:k34*}, $\crn(G')=2$, so $G'=G$, as required.  \end{cproof}

}\section{3-reducing to \wording{peripherally}-4-connected graphs}\printFullDetails{

In this section, we discuss the general details of reducing a 3-connected graph to a \wording{peripherally}-4-connected graph.  These results apply in some generality and not just in the context of \2cc graphs.  These are the first of several steps toward finding all the 3-connected \2cc\ graphs that do not contain a subdivision of $V_8$.

These results are fairly technical but essential to this part of the theory.

}\begin{definition}\label{df:non-peripheral}\begin{enumerate}\item A 3-cut $S$ in a 3-connected graph is {\em reducible\/}\index{reducible (3-cut)} if $G-S$ has at most 3 components and they partition into two subgraphs each having at least two vertices.
\item\label{it:noK34} The set $\mfK$\index{$\mfK$} consists of those 3-connected graphs that do not contain a subdivision of $K_{3,4}$.
\end{enumerate}
\end{definition}
\printFullDetails{

The following result is obvious from the definitions and begins to explain the appearance of $K_{3,4}$ in Definition \ref{df:non-peripheral} (\ref{it:noK34}).  

}\begin{lemma}  Let $G$ be a 3-connected graph that is not \wording{peripherally}-4-connected.  Then either $G$ has a reducible 3-cut or $G$ has $K_{3,4}$ as a subgraph.  \end{lemma}\printFullDetails{

The next result sets up the basic scenario that we will use throughout our reduction to \wording{peripherally}-4-connected graphs.
 
}\begin{lemma}\label{lm:stay3conn}  Let $G\in \mfK$.  Then there is a sequence $G_0,G_1,\dots,G_k$ of 3-connected graphs in $\mfK$ so that:  $G_0=G$; $G_k$ is \wording{peripherally}-4-connected; and, for each $i=1,2,\dots,k$, there is a 3-cut $S_i$ in $G_{i-1}$ and a non-trivial, planar $S_i$-bridge $B_i$ so that $\Nuc(B_i)$ has at least two vertices and $G_ i$ is obtained from $G_{i-1}$ by contracting the nucleus of $B_i$.  \end{lemma}\printFullDetails{

\begin{cproof}  Suppose $G_{i-1}$ is 3-connected.  Among all the choices of $S_i$ and  $S_i$-bridges $B_i$ so that $\Nuc(B_i)$ has at least two vertices, choose $B_i$ to be inclusion-wise maximal.    We claim that the graph $G_i$ obtained from $G_{i-1}$ by contracting $\Nuc(B_i)$ to a vertex is 3-connected.  

Otherwise, there is some pair $\{u,v\}$ of vertices so that $G_{i}-\{u,v\}$ is not connected.  If the vertex of contraction of $\Nuc(B_i)$ is neither $u$ nor $v$, then $\{u,v\}$ is a 2-cut in $G_{i-1}$, a contradiction.  Therefore, we can assume $u$ is the contraction of $\Nuc(B_i)$.   

Let $H$ and $K$ be components of $G_i-\{u,v\}$, with the labelling chosen so that $|S_i\cap V(H)|\ge |S_i\cap V(K)|$; in particular, $|S_i\cap V(K)|\le 1$.   Let $h\in V(H)$; if there is a vertex $k\in V(K)\setminus S_{i}$, then $\{v\}\cup (S_i\cap V(K))$ separates $k$ from $h$ in $G_{i-1}$, which contradicts the assumption that $G_{i-1}$ is 3-connected.  

Therefore $V(K)\subseteq S_i$, so there is a single vertex $s$ in $K$, and $s\in S_i$.  It follows that $s$ is adjacent to only vertices in $\Nuc(B_i)$ and possibly to $v$.  But this contradicts the maximality of $B_i$:  let $S'=(S\setminus\{s\})\cup \{v\}$. Observe that $B_i+s$ is a planar $S'$-bridge, contradicting maximality of $B_i$.

\wording{Lastly, we show} that if $G_{i-1}$ does not have a subdivision of $K_{3,4}$, then \wording{neither does} $G_i$.  Any subdivision of $K_{3,4}$ in $G_{i}$ must contain the vertex $v_i$ of contraction.  Since $v_i$ has degree 3 \wording{in $G_i$} and $B_{i-1}$ is an $S$-bridge, we can reroute the subdivision of $K_{3,4}$ in $G_i$ into $B_{i-1}$ to obtain a subdivision of $K_{3,4}$ in $G_{i-1}$. \end{cproof}

}\begin{definition}\label{df:planar3reduction}  Let $G\in \mfK$.    
\begin{enumerate}\item Then $G$ {\em reduces to $G'$ by 3-reductions\/}\index{reduces (by 3-reductions)}\index{3-reductions} if there is a sequence $G_0,G_1,\dots,G_k$ of 3-connected graphs so that $G_0=G$; $G_k=G'$; and, for each $i=1,2,\dots,k$, there is a 3-cut $S_i$ in $G_{i-1}$ and an $S_i$-bridge $B_i$, whose nucleus at least two vertices, so that $G_ i$ is obtained from $G_{i-1}$ by contracting the nucleus of $B_i$. 
\item For each vertex $v$ of $G'$ and each $i=0,1,2\dots,k$,  $K_v^i$\index{$K_v^i$} denotes the connected subgraph of $G_i$ that contracts to $v$.  We also set $K_v=K_v^0$\index{$K_v$}.
\item If $v$ has just three neighbours $x$, $y$, and $z$ in $G'$, then $G_v$\index{$G_v$} is the graph obtained from $K_v$ by adding $x$, $y$, and $z$, and, for each $t\in \{x,y,z\}$ and each edge $v't'$ of $G$ with $v'\in K_v$ and $t'\in K_t$, adding the edge $v't$.
\end{enumerate}
\end{definition}\printFullDetails{

We now commence a lengthy series of technical lemmas that all play vital roles in usefully reducing the 3-connected graph \2cc\ graph $G$ to a smaller 3-connected \2cc\ graph $G_{\rep(v)}$.   The culmination of this part of the work is Theorem \ref{th:TUreplace} in the next section, showing that $G_{\rep(v)}$ is \2cc.  This will lead to a program for determining all the 3-connected \2cc\ graphs that reduce to a particular \wording{peripherally}-4-connected graph.

}\begin{lemma}\label{lm:no2matching}  Let $G\in \mfK$ and suppose $G$ reduces by 3-reductions to the \wording{peri\-phe\-rally}-4-connected graph $G^{\ei4c}$.  For any two vertices $u,v$ of $G^{\ei4c}$, there is a single vertex in $G$ incident with all edges having one end in $K_u$ and one end in $K_v$.
\end{lemma}\printFullDetails{ 

\begin{cproof}  Let $G=G_0,G_1,\dots,G_k=G^{\ei4c}$ be a sequence of 3-reductions.  Choose $i$ to be largest so that there are disjoint $K_u^{i-1}K_v^{i-1}$-edges
 $ab$ and $cd$  with $a,c\in K_u^{i-1}$ and $b,d\in K_v^{i-1}$.  In $G_{i}$, either $a$ and $c$ have been identified or $b$ and $d$ have;  by symmetry, we may assume the former.  

The vertices $b$ and $d$ are obviously attachments of $B_i$ and so these are in $S_i$.  Let $z_i$ be the third vertex in $S_i$.  \wording{Since $K_u^{i-1}$ is} connected and since, by Definition \ref{df:planar3reduction}, $u$ has three neighbours in $G^{\ei4c}$, $z_i\in K_u$.  
Continue using the label $a$ for the vertex obtained by contracting $\Nuc(B_i)$.

At some point in the later 3-reductions\wording{,}\wordingrem{(comma added)} $a$ and $z_i$ are identified and at another point $b$ and $d$ are identified.  We show that neither can be done before the other, which is impossible.

Suppose $z_i$ and $a$ are identified first.  When this identification occurs, a 3-cut $S_j$ and an $S_j$-bridge $B_j$ so that $z_i$ and $a$ are in $\Nuc(B_j)$.  The vertices $b$ and $d$ are again attachments of $B_j$ and so are in $S_j$; let $z_j$ be the third vertex in $S_j$.

Because \wording{$i$ is largest so there are disjoint} $K_u^{i-1}K_v^{i-1}$-edges, all edges between $K_u^j$ and $K_v^j$ at this moment are incident with $a$.  It follows that $\{a,z_j\}$ is a 2-cut in the current graph, separating $z_i$ from $b$.  But this contradicts the fact that $G_{j-1}$ is 3-connected.  Therefore, $z_i$ and $a$ are not identified before $b$ and $d$.

On the other hand, suppose $b$ and $d$ are identified first, by the contraction of $\Nuc(B_j)$.  When $b$ and $d$ are identified, the only neighbours of $a$ are $b$, $d$, and $z_i$.  Following the identification of $b$ and $d$, the only neighbours of $a$ are $z_i$ and the vertex of identification, again contradicting 3-connection of $G_j$.
\end{cproof}

We need a slight variation on a standard definition.

}\begin{definition}\label{df:isthmus}  Let $G$ be a connected graph.  
\begin{enumerate}  \item An {\em isthmus\/}\index{isthmus} is a set $I$ of parallel edges so that $G-I$ is not connected. 
\item A {\em cut-edge\/}\index{cut-edge} is an edge $e$ so that $G-e$ is not connected. 
\end{enumerate}
\end{definition}\printFullDetails{

Obviously,  $e$ is a cut-edge of $G$ if and only if $\{e\}$ is an isthmus, but an isthmus may have more than one edge.  The distinction comes into play because at various points we will consider edge-disjoint paths in certain subgraphs of our \2cc\ graph; if there are not two edge-disjoint $uv$-paths, then there is a cut-edge separating $u$ and $v$.  On the other hand, the 3-connection of $G$ does not preclude the possibility of parallel edges; at several points we will be able to identify that two vertices $u$ and $v$ have the property that they must be adjacent, but \wording{be unable to distinguish} whether they are joined by 1 or 2 edges.  A common scenario will have the set of edges between them making an isthmus in some subgraph.

In particular, the case that $K_v$ has an isthmus is a central one in reducing \2cc\ graphs.

}\begin{lemma}\label{lm:isthmus3reduction}  Let $G\in\mfK$ reduce to the \wording{peripherally}-4-connected graph $G^{\ei4c}$ by a sequence of 3-reductions.  Suppose there is a vertex $v$ of $G^{\ei4c}$ so that the graph $K_v$ has an isthmus $I$.  Then, for each component $K$ of $K_v-I$, there are at least two neighbours $x$ and $y$ of $v$ in $G^{\ei4c}$ so that there are $KK_x$- and $KK_y$-edges in $G$.  
\end{lemma}\printFullDetails{

\begin{cproof}    At some moment in the reduction of $G$, $G_{i-1}$ has a 3-cut $S_i$ and $B_i$ is the planar $S_i$-bridge in $G_{i-1}$ that contains $I$.  Then $B_i-I$ is not connected; the ends $u$ and $w$ of the edge or edges in $I$ are in different components $K$ and $L$, respectively, of $B_i-I$.    

Let $x$, $y$, and $z$ be the neighbours of $v$ in $G^{\ei4c}$ and let $t$ be any vertex of $G_{i-1}$ not in $K_v^i\cup K_x^i\cup K_y^i\cup K_z^i$.  (Since $G^{\ei4c}$ is not planar, it has at least five vertices.)   In $G_{i-1}$ there are three pairwise internally-disjoint $ut$-paths.  These three paths leave $B_i$ through distinct attachments of $B_i$; these are the vertices in $S_i$.  The same argument applies for $wt$-paths.  

In particular, two of the $ut$-paths leave $K$ on edges incident with vertices in  $S_i$.  Likewise for $L$.  Therefore, $K$ and $L$ are both joined by edges to the same attachment $s\in S_i$.  It follows that $s$ is not in $K_v^i$, so $s$ is in  $K_x^i$, say.  Moreover, since the $K_v^i$-ends of these two edges are not the same, Lemma \ref{lm:no2matching} implies all the edges between $K_v^i$ and $K_x^i$ are incident with $s$.  

Since $G_{i-1}$ is 3-connected, $G_{i-1}-(\{s\}\cup I)$ is connected.  Therefore, there are edges of $G_{i-1}$ leaving each of $K$ and $L$; each of these edges is also leaving $K_v^i$ and, therefore, has its other end in one of $K_x^i$, $K_y^i$, and $K_z^i$.  However, this other end cannot be $s$ and, consequently, cannot be in $K_x^i$, as required.
\end{cproof} 

\minorrem{(Useless -- and incomprehensible -- lemma deleted.)}%
The connectivity of $G$ has further implications about the structure of the $K_v$.

}\begin{lemma}\label{lm:attsDistinct}  Let $G\in\mfK$ reduce by 3-reductions to a \wording{peripherally}-4-connected graph $G^{\ei4c}$.   Let $v$ be a vertex of $G^{\ei4c}$ with just the three neighbours $x$, $y$, and $z$ and suppose $K_v$ has at least two vertices.  For each $t\in \{x,y,z\}$, let $t'$ be any vertex incident with all the $K_vK_t$-edges.  Then $x'$, $y'$, and $z'$ are all distinct. \end{lemma}\printFullDetails{

\begin{cproof}  Suppose $x'=y'$.  Then $x'$ is in $K_v$.  Observe that  no vertex of $K_v-\{x',z'\}$ is adjacent to any vertex of of $G-\{x',z'\}$ not in $K_v$.  Since $G$ is 3-connected, it follows that $K_v$ consists of just $x'$ and $z'$.  In particular, $z'\ne x'$.  Also, recall that $K_v$ contracts to a single vertex in the sequence of planar 3-reductions.

At the moment of contraction of $K_v$, $G_{i-1}$ is 3-connected and $x'z'$ is an isthmus.   Therefore, Lemma \ref{lm:isthmus3reduction} implies that  $z'$ is joined to at least one of $K_x^i$ and $K_y^i$; this contradicts the fact that all edges from $K_v^i$ to $K_x^i\cup K_y^i$ are incident with $x'$.
\end{cproof}

The vertices $x'$, $y'$, and $z'$ are not uniquely determined. It is possible that there is only one vertex in each of $K_v$ and $K_x$ incident with all $K_vK_x$-edges; one obvious instance is if there is only one $K_vK_x$-edge.   We will follow up on this a little later.

Here is a very simple and very useful observation.

}\begin{lemma}\label{lm:triangle}  Let $H$ be a simple, non-planar, \wording{peripherally}-4-connected graph.   There is no 3-cycle of $H$ having two  \wording{vertices with just 3 neighbours}.  \end{lemma}\printFullDetails{

\begin{cproof}  Suppose to the contrary there are three vertices $x,y,z$ making a 3-cycle, with $x$ and $y$ having \wording{only three neighbours each}.  Let $v$ and $w$ be the other neighbours of $x$ and $y$.  Then \wordingrem{(word deleted)}$x$ and $y$ are the vertices of one component of $H-\{v,w,z\}$. 

Observe that $H$ is non-planar, 3-connected, and has a vertex of degree 3.  Therefore $H$ is not $K_5$ and so contains a subdivision of $K_{3,3}$.  It follows that $H$ has at least six vertices.  Thus, there is another component of $H-\{v,w,z\}$.

Since $H$ is \wording{peripherally}-4-connected, the only possibility is that there is exactly one other component and it consists of a single vertex $u$, adjacent to all of $v$, $w$, and $z$.  The only other possible edges in $H$ are between $v$, $w$, and $z$.  However, the resulting graph is planar, a contradiction.  \end{cproof}

The following result assures us that useful (and expected) paths exist in each $K_v$.

}\begin{lemma}\label{lm:x'y'z'paths1}  Let: \begin{enumerate}
\item  $G\in \mfK$ reduce by 3-reductions to the \wording{peripherally}-4-connected graph $G^{\ei4c}$;
\item $G^{\ei4c}$ have at least five vertices;
\item  $v$ be a vertex of $G^{\ei4c}$ so that $K_v$ has at least two vertices; and
\item   $x$, $y$, and $z$ be the neighbours of $v$ in $G^{\ei4c}$, with corresponding vertices $x'$, $y'$, and $z'$ in $G$ as in Lemma \ref{lm:attsDistinct}.  
\end{enumerate}
Then:\begin{enumerate}[label=\alph*)]\item\label{it:3disjt} for any vertex $w$ in $K_{v}-\{x',y',z'\}$, there are three $w\{x',y',z'\}$-paths in $G_v$ that are pairwise disjoint except for $w$; and
\item\label{it:primeDisjt} if $x'\in K_v$, then there are $x'y'$- and $x'z'$-paths in $G_v-x$ that are disjoint except for $x'$.
\end{enumerate}
  \end{lemma}\printFullDetails{

\begin{cproof} For \ref{it:3disjt}, let $u$ be any vertex of $G$ not in $K_{v}\cup K_{x}\cup K_{y}\cup K_{t_3}$.  Since $G$ is 3-connected, there are three pairwise internally-disjoint $wu$-paths in $G$.  The result follows from the observation that $w$ and $u$ are in different components of $G-\{x',y',z'\}$.

If \ref{it:primeDisjt} fails, then there is a vertex $w$ of $G_{v}-x$ that separates $x'$ from $\{y',z'\}$.  Since $K_v$ is an $\iso{\{x,y,z\}}$-bridge in $G_v$, $w$ is in $K_v$ (possibly $w=y'$ or $w=z'$).   Since $\{x',w\}$ is not a 2-cut in $G$, $x'$ and $w$ are adjacent in $K_v$.  But now they are joined by an isthmus $I$ in $K_v$.  Since $x'$ is a component of $K_v-I$ joined only to $K_x$, we have a contradiction of Lemma \ref{lm:isthmus3reduction}.
\end{cproof}

}\section{Planar 3-reductions}\printFullDetails{

In this section we now turn our attention to the particular case $G\in \m2$.  We want to show that the 3-reductions can be taken to be contractions of planar bridges.  So suppose $S$ is a non-peripheral 3-cut in $G$.  
 
If there are four or more non-trivial $S$-bridges (that is, having a nucleus), then $G$ has a subdivision of $K_{3,4}$ and so is $K_{3,4}$.  In the remaining cases, there are at most three non-trivial $S$-bridges.  If there are three and $B$ is one of them so that $B^+$  is not planar (as in Subsection \ref{ssc:twoNonPlanar}), then the union $K$ of the remaining $S$-bridges has $K^+$ not planar.  Theorem \ref{th:3CutBothNonPlanar} implies that $G$ is one of four \2cc\ graphs.  Thus, if there are three non-trivial $S$-bridges, we may assume that, for each one $B$, $B^+$ is planar.  Finally, consider the case that  there are  precisely two non-trivial $S$-bridges $B_1$ and $B_2$.   Since $S$ is not \wording{peripheral}, both $B_i$ have at least two vertices.  If both $B_i^+$ are non-planar, then we are in the case dealt with in Theorem \ref{th:3CutBothNonPlanar}, so we may assume that one of them is planar.  In summary, in every case, we may assume that $G^{\ei4c}$ is obtained from 3-reductions in $G$ in which the contracting $S_{i}$-bridge $B_{i}$ is always planar.

}\begin{definition}  Let $G$ be a 3-connected graph and let $G^{\ei4c}$ be a \wording{peripherally}-4-connected graph. Then {\em $G$ reduces to $G^{\ei4c}$ by planar 3-reductions\/}\index{planar 3-reductions}\index{3-reductions!planar} if there is a sequence $G=G_0,G_1,G_2,\dots,G_k=G^{\ei4c}$ of 3-reductions so that, for each $i=1,2,\dots,k$, $G_i$ is obtained from $G_{i-1}$ by contracting $\Nuc(B_{i-1})$ and $B_{i-1}^+$ is planar. \end{definition}\printFullDetails{

We need two results about $K_v$ in the context of planar 3-reductions.  This requires further definitions.

}\begin{definition}\label{df:KmaxKmin} Let $G$ be a 3-connected graph that reduces by 3-reductions to the \wording{peri\-pherally}-4-con\-nec\-ted graph $G^{\ei4c}$.  Suppose $v$ is a vertex of $G^{\ei4c}$ having only the neighbours $x$, $y$, and $z$.   For each $t\in \{x,y,z\}$, let $m_t$ denote the number of vertices in $K_v$ adjacent to vertices in $K_t$ and let $n_t$ denote the number of vertices in $K_t$ adjacent to vertices in $K_v$.  (Lemma \ref{lm:no2matching} implies that at least one of $m_t$ and $n_t$ is 1.)
\begin{enumerate}\item The subgraph $K^{\max}_v$\index{$K^{\max}_v$} induced by $K_v$ together with, for each $t\in\{x,y,z\}$ with $n_t=1$,  the vertex of $K_t$ adjacent to vertices in $K_v$.
\item The subgraph $K^{\min}_v$\index{$K^{\min}_v$} induced by $K_v$ together with, for each $t\in \{x,y,z\}$ with $m_t>1$, the vertex of $K_t$ adjacent to vertices in $K_v$.
\end{enumerate}
\end{definition}\printFullDetails{

We remark that $K_v\subseteq K^{\min}_v\subseteq K^{\max}_v$, and, for $t\in\{x,y,z\}$, $K^{\max}_v$ has a vertex $t'\in K_t$ that is not in $K^{\min}_v$ precisely when $n_t=m_t=1$.

}\begin{lemma}\label{lm:x'y'z'Cycle} \wording{Let $G\in\mfK$ reduce by 3-reductions to a \wording{peripherally}-4-connected graph $G^{\ei4c}$.   Let $v$ be a vertex of $G^{\ei4c}$ with just the three neighbours $x$, $y$, and $z$ and suppose $K_v$ has at least two vertices.  Then} there is a cycle $C$ in $K^{\min}_v$ containing all of $x'$, $y'$ and $z'$.  \end{lemma}\printFullDetails{

\begin{cproof}  Suppose $w$ is a cut-vertex of $K^{\min}_v$, so there are subgraphs $X$ and $Y$ of $K^{\min}_v$ with  $X\cup Y=K^{\min}_v$, $X\cap Y=\iso w$, and both $X-w$ and $Y-w$ are not empty.  We may choose the labelling so that $X$ has at least the two vertices $x'$ and $z'$ from $\{x',y',z'\}$, while $Y-w$ has at most one; we may further assume $x'\ne w$.  If $y'\notin Y-w$, then $w$ is a cut-vertex of $G$, contradicting the fact that $G$ is 3-connected.  Therefore, $y'\in Y-w$.

However, if $y'\in K_v$, then we have a contradiction to Lemma \ref{lm:x'y'z'paths1} (\ref{it:primeDisjt}.  Therefore, $y'\notin K_v$.  If there is a vertex in $Y$ other than $w$ and $y'$, then we contradict 3-connection of $G$, so $y'$ is adjacent only to $w$ in $G_v$.  But then $y'\notin K^{\min}_v$.    

It follows that there is no cut-vertex in $K^{\min}_v$.  Thus, there is a cycle $C$ in $K^{\min}_v$ containing $x'$ and $y'$.  Obviously, we are done if $z'\in C$, so we assume $z'\notin C$.  

Since there is no cut-vertex in $K^{\min}_v$, there are two $z'C$-paths $P_1$ and $P_2$ that are disjoint except for $z'$.  If the $C$-ends of $P_1$ and $P_2$ are not both on the same $x'y'$-subpath of $C$, then $G^+_v$ contains a subdivision of $K_{3,3}$.  This contradicts the fact that we are doing planar 3-reductions.  Therefore, the $C$-ends of $P_1$ and $P_2$ are on the same $x'y'$-subpath of $C$ and it is easy to find the desired cycle through all of $x'$, $y'$, and $z'$.
\end{cproof}

The following is the last lemma we need to get the main result of this section.

}\begin{lemma}\label{lm:twoNeighbours}   Let $G\in \m2$ and suppose $G$ reduces by planar 3-reductions to the \wording{peri\-pher\-ally}-4-connected graph $G^{\ei4c}$.  Let $v$ and $x$ be adjacent vertices in $G^{\ei4c}$.  Then there are at most two vertices in $K_v$ adjacent to vertices in $K_x$.  \end{lemma}\printFullDetails{

\begin{cproof}  This is obvious if 
$
K_v
$ 
has at most one vertex.  In the remaining case, 
$
v
$ 
has degree 3 in 
$
G^{
\ei4c
}
$; let 
$y$ and 
$z$ be its other neighbours.

 Suppose by way of contradiction that 
$s$, 
$t$, and 
$u$ are distinct vertices in 
$K_v$ all adjacent to vertices in 
$K_x$.  By Lemma 
\ref{lm:no2matching}, there is a vertex 
$x'$ incident will all the 
$K_vK_x$-edges and, evidently, 
$x'\in K_x$. 

In the planar embedding 
$D^+_v$ of 
$G_v^+$, letting 
$w$ denote the new vertex adjacent to each of 
$x$, 
$y$, and 
$z$, we may choose the labelling so that the edges 
$xw,xs,xt,xu$ occur in this cyclic order around 
$x$.\wordingrem{(Text moved and adjusted.)} 

\begin{claim}
\label{cl:stuPath}  
There is an 
$su$-path  in 
$K_v$ containing 
$t$.  
\end{claim}

\begin{proof}   
As 
$K_v$ is connected, there is an 
$su$-path 
$P$ in 
$K_v$. We are obviously done if 
$t\in P$, so we assume 
$t\notin P$.  Let 
$C$ be the cycle obtained by adding 
$x'$ to 
$P$ and joining it to 
$s$ and 
$u$.

\wording{The rotation at $x$ implies} that 
$t$ is on one side of 
$D^+_v[C]$, \wording{while 
$w$, $y$, and, consequently,
$z$, are} on the other.  Therefore, every 
$t\{
y,z
\}$-path in 
$G^+_v$ goes through either 
$x$ or 
$P$.  

If there is a \wording{cut-vertex 
$r$ in} 
$K_v$ separating 
$t$ from 
$P$, \wording{then 
$\{
r,x'
\}$ 
is a} 2-cut in the 3-connected graph 
$G$, which is impossible.  Therefore, there are 
$tP$-paths 
$Q$ and 
$R$ in 
$K_v$ that are disjoint except for 
$t$.  We can now reroute 
$P$ through 
$t$ to obtain the desired path.
\end{proof}

\wording{Since 
$G$ is 2-crossing-critical, there is a 1-drawing 
$D$ of 
$G-x't$.   From Claim \ref{cl:stuPath}, there is an $su$-path $P$ in $K_v$ containing $t$.  Let $C$ be the cycle obtained from $P$ by adding $x'$, $x's$ and $x'u$.}

\medskip
\begin{claim}\label{cl:D[C]faces}  All the vertices of $G-(K_v\cup K_x)$ are in the same face of $D[C]$.  \end{claim}

\begin{proof}
Suppose by way of contradiction that there are vertices in 
$G-(K_v\cup K_x)$ that are in different faces of 
$D[C]$.

\medskip
{\bf Case 1:}  {\em there is a vertex 
$p$ in 
$G^{\ei4c}$ so that 
$K_p$ contains vertices that are in different faces of 
$D[C]$. 
}

\medskip In this case there is an edge 
$f$ of 
$K_p$ that crosses 
$D[C]$.  As 
$D$ has at most one crossing, 
$f$ is a cut-edge of 
$K_p$.  Lemma \ref
{lm:isthmus3reduction} 
implies each component of 
$K_p-f$ is adjacent to at least two different 
$K_n$'s.  If one of them is adjacent to both 
$K_x$ and 
$K_v$, then we have a 3-cycle 
$pxv$ in 
$G^{\ei4c}$ in which both 
$p$ and 
$v$ have degree 3, contradicting Lemma \ref
{lm:triangle}.

Therefore, we may assume each is adjacent to one, say 
$K_q$ and 
$K_r$, that is neither 
$K_x$ nor $K_v$.  However, now 
$\{v,x,p\}$ is a 3-cut in 
$G^{\ei4c}$ separating $q$ and $r$ in 
$G^{\ei4c}$.  Therefore one of \wording{them --- say $q$ --- is} adjacent to precisely these three \wording{vertices in $G^{\ei4c}$,} producing \wording{the  3-cycle $\{q,v,x\}$ in} 
$G^{\ei4c}$ that contradicts Lemma 
\ref{lm:triangle}.

\medskip 
{\bf Case 2:}  
{\em any two vertices of $G-(K_v\cup K_x)$ in different faces of $D[C]$ are in different $K_p$'s.}

\medskip    Since $G-(K_v\cup K_x)$ is connected, there is a path in $G-(K_v\cup K_x)$ joining vertices in different faces of $D[C]$.  Therefore, there is, for some vertices $q$ and $r$ of 
$G^{\ei4c}$, a $K_qK_r$-edge $f$ that crosses $D[C]$.   It follows \wording{that $D[C]$ has no self-crossings, so $D[C]$} has only two faces.

Clearly 
$G^{\ei4c}-
\{x,v,f\}$ has $K_q$ and $K_r$ in different components.  
Since $G^{\ei4c}$ has at least six vertices, it has a vertex $m$ different from all of $v$, $x$, $q$ and $r$.  We may choose the labelling so that $D[K_q]$ is in one face of $D[C]$, while $D[K_r\cup K_m]$ is contained in the other.  It follows that $\{v,x,r\}$ is a 3-cut in $G^{\ei4c}$ separating $q$ from $m$.  

Since $G^{\ei4c}$ is \wording{peripherally}-4-connected, \wording{one of $q$ and $m$ --- say $q$ --- is} adjacent precisely to $v$, $x$, and $r$, yielding \wording{the 3-cycle $\{v,x,q\}$ in $G^{\ei4c}$ that} has two vertices \wording{with only three neighbours}, contradicting Lemma \ref{lm:triangle}. \end{proof}

We note that the crossing in $D$ cannot involve two edges, each incident with a vertex in $K_v$, as otherwise $G^{\ei4c}$ is planar.  In particular, $D[C]$ is not self-crossing.

\begin{claim}\label{cl:Cbridges}  $OD_{G^+_v}( C)$ is isomorphic to $OD_G( C)$.  In particular, $OD_G( C)$ is bipartite. \end{claim} 

\begin{proof}  The main point is that there is a single $C$-bridge in $G$ containing $G-(K_v+x')$.  To prove this, we show that any two vertices in $G-(K_v+x')$ are connected by a $C$-avoiding path.  For vertices not in $K_v\cup K_x$, this is easy:  for any two vertices $p$ and $q$ in $G^{\ei4c}-\{v,x\}$, there is a $pq$-path in $G^{\ei4c}-\{v,x\}$, showing that any two vertices in $K_p\cup K_q$ are joined by a path in $G-(K_v\cup K_x)$.  

If $p\in K_x-x'$, then Lemma \ref{lm:attsDistinct} implies that the three vertices separating $K_x$ from its neighbours are distinct.   \wording{For one of these vertices $w'$ that is not $x'$,} Lemma \ref{lm:x'y'z'paths1} implies there is a $pw'$-path in $K_x-x'$, completing the proof that there is a single $C$-bridge $B$ in $G$ containing $G-(K_v+x')$.  

Every other $C$-bridge in $G$ is contained in $K_v+x'$.  These are all $C$-bridges in $G_v^+$; the only other $C$-bridge in $G_v^+$ is the one containing the vertex joined to $x$, $y$, and $z$.  This $C$-bridge has precisely the same attachments as $B$.  This shows that $OD_G( C)$ and $OD_{G^+_v}( C)$ are isomorphic.

Since $G^+_v( C)$ is planar, $OD_{G^+_v}( C)$ is bipartite, yielding the fact that $OD_G(C )$ is bipartite.
\end{proof}

Suppose first that $C$ is clean in $D$.  \wording{Since $B$ is the} unique non-planar $C$-bridge in $G$, $D$ yields a 1-drawing of $C\cup B$ with $C$ clean.  Therefore, Corollary \ref{co:TutteTwo} implies $\crn(G)\le 1$, a contradiction.  

If, on the other hand, $C$ is not clean in $D$, then $C$ is crossed by an edge $f$\major{.  By Claim \ref{cl:D[C]faces}, $f$ is incident with a vertex in $K_v\cup K_x$.  If $f$ is incident with a vertex in $K_v$, then contract $K_v$ (with a vertex inserted at the crossing point, if necessary, to get a 1-drawing of $G^{\ei4c}$ so that both edges incident with the crossing are incident with $v$.  This implies the contradiction that $G^{\ei4c}$ is planar.}

\major{If $f$ is not incident with $x'$, then $K_x-x'$ has vertices on both sides of $D[C]$.  One of these is in a component $K^1_x$ of $K_x-f$ that is on the side of $D[C]$ that does not contain any vertex of $G-(K_v\cup K_x)$.  Lemma \ref{lm:isthmus3reduction} implies $K^1_x-x$ is joined to a vertex in some other $K_w$, $w\ne v$, which cannot happen without crossing $D[C]$ a second time, a contradiction.   It follows that $f$ is incident with $x'$.  Furthermore, Lemmas \ref{lm:x'y'z'Cycle} and \ref{lm:x'y'z'paths1} (\ref{it:3disjt} imply that $f$ is in a cycle $C_f$ in $G-K_v$.} 
The ends of the edge $e_v$ of $K_v$ crossed in $D$ are separated by $D[C_f]$, so $e_v$ is a cut-edge of $K_v$.  Moreover, $e_v$ is in $C$.

We now see that the $C$-bridges are $B$, those contained in one component of $K_v-e_v$, and those contained in the other component of $K_v-e_v$.  Notice that $B$ is a cut-vertex of $OD_G(C )$, and so it overlaps $C$-bridges of both the other types.  

Since $OD_G(C )$ is connected and bipartite, it follows that the $C$-bridges in either of the components of $K_v-e_v$ occur on the same side of $D[C]$ that they do in $D^+_v$.  In particular, $x't$ may be reintroduced to $D$ to obtain a 1-drawing of $G$, which is impossible.\end{cproof}

\noindent{\bf Strategy.}   {\em The strategy now is to show that if we replace any $K_v$ with a smallest possible representative subject to the preceding observations, then we produce a 2-crossing-critical graph.   This is the last part of this section.  This implies that $G^{\ei4c}$ turns into a 2-crossing-critical graph by choosing these smallest possible representatives.  From this, it is then  possible to determine (although not in a theoretical sense, but rather in a definite, finite --- really manageable --- way that we shall describe) all the 3-connected 2-crossing-critical graphs that have these configurations and reduce to $G^{\ei4c}$ by planar 3-reductions.  

There will remain the issue of determining all the possible $G^{\ei4c}$.  Of course, one can list them all, but it is not clear at what point to stop.   Fortunately, Theorem 2.14 shows that we do not need to do this when $G$ contains a subdivision of $V_{10}$, as we already know what $G$ looks like.  When $G$ does not contain a subdivision of $V_8$, a theorem of Robertson plus some analysis implies that $G^{\ei4c}$ has at most 9 vertices.  We are left with the open question of finding the graphs in $\m2$ that contain a subdivision of $V_8$ but do not contain a subdivision of $V_{10}$.  In Section \ref{sc:noV2n}, we show that any such graph has at most about 4 million vertices.}

\bigskip

We next characterize certain properties of the graphs $G_v$; our goal is to show that these (more or less) determine the crossing number of $G$.

}\begin{definition}\label{df:TUconfig}   Let $x$, $y$, and $z$ be vertices in a graph $H$ so that $H$ is an $\iso{\{x,y,z\}}$-bridge. Then:
\begin{itemize}\item  $T$ is the set of vertices $w\in \{x,y,z\}$ so that there are edge-disjoint $w(\{x,y,z\}\setminus\{w\})$-paths in $H$; and
\item $U$ is the set of vertices $w\in \{x,y,z\}$ for which there are edge-disjoint paths in $H-w$ joining the two vertices in $\{x,y,z\}\setminus\{w\}$.
\item\label{it:TUconfigPlanar} $(H,\{x,y,z\})$ is a {\em $(T,U)$-configuration\/}\index{$(T,U)$-configuration}\index{configuration} if the graph $H^+$ obtained from $H$ by adding a new vertex adjacent just to $x$, $y$, and $z$ is planar.
\end{itemize}
\end{definition}\printFullDetails{

Our entire argument depends on the fact, to be proved in the next section, that the pairs $(T,U)$ effectively characterize 2-criticality.  Theorem \ref{th:Regrow}, the main point of this section, shows that substituting one $(T,U)$-configuration for another retains the fact that the crossing number is at least 2.

For a $(T,U)$-configuration, obviously there are only four possibilities for $|T|$.  It is a routine analysis of cut-edges to see that, if $|T|\le 1$, then $U$ is  empty, while if, for example, $T=\{x,y\}$, then $U=\{z\}$.  Thus, for $|T|\le 2$, $U$ is determined by $T$.  This is not the case for $|T|=3$.  In this instance,  if $z\notin U$, then there is a cut-edge in $G_v-z$ separating $x$ and $y$.  From here and the fact that $T=\{x,y,z\}$,  one easily sees that $x,y\in U$.  Thus, if $T=\{x,y,z\}$, then $|U|$ can be either 2 or 3.  Therefore, there are in total five possibilities for the pair $(|T|,|U|)$.

We first show that replacing a $(T,U)$-configuration with another $(T,U)$-con\-fi\-gu\-ra\-tion does not lower the crossing number below 2.  First the definition of substitution.

}\begin{definition}\label{df:substitution}  Let $G$ reduce by planar 3-reductions to the \wording{peripherally}-4-connected graph $G^{\ei4c}$\index{$G^{\ei4c}$}.  Suppose $v$ is a vertex of $G^{\ei4c}$ with neighbours $x$, $y$, and $z$ so that $(G_v,\{x,y,z\})$ is, for some subsets $T$ and $U$ of $\{x,y,z\}$, a $(T,U)$-configuration.  Let $N$ be the set of vertices $t$ in $\{x,y,z\}$ for which $K^{\max}_v\cap K_t$ is null.  (See Definition \ref{df:KmaxKmin} for $K^{\max}_v$.)  Let $\bar N_v$ denote the attachments of $K^{\max}_v$:  these are the vertices that are of the form $t'$, $t\in \{x,y,z\}$, chosen to be in $K_t$ whenever possible.  
\begin{enumerate}
\item A $(T,U)$-configuration $(H,\{x,y,z\})$ is {\em $(G,K_v)$-compatible\/}\index{compatible}\index{$K_v$!compatible} if: \begin{enumerate}\item  for each $t\in N$, then there is only one neighbour of $t$ in $H$;
\item\label{it:equalDegrees} the degrees of each $t\in \{x,y,z\}$ are the same in both $G_v$ and $H$; and 
\item \wording{setting $\bar N_H$\index{$\bar N_H$} to consist of the union of  the} set of vertices of $H$ in $\{x,y,z\}\setminus N$ together with the neighbours in $H$ of the vertices in $N$, $H-N$ either has a single vertex or contains a cycle through all the vertices in $\bar N_H$. 
\end{enumerate} 
\item  The {\em substitution of the $K_v$-compatible $(T,U)$-configuration $(H,\{x,y,z\})$ for $K_v$ in $G$\/}\index{substitution}\index{compatible!substitution}\index{$K_v$!compatible substitution} is the graph $G^H_v$\index{$G^H_v$} obtained from $G$ by adding  $H-N$ by identifying the vertices in $\bar N_v$ with those in $\bar N_H$ in the natural way, and then deleting all vertices in $K^{\max}_v-\bar N_v$.
\end{enumerate}
\end{definition}\printFullDetails{

We are almost ready for a major plank in the theory.  

Our plan is to show that we can replace a ``large" $(T,U)$-configuration by a ``small" $(T,U)$-configuration and still be 3-connected and 2-crossing-critical.    There is one special case that requires particular attention.

}\begin{definition}\label{df:doglike}  A $(T,U)$-configuration $(H,\{x,y,z\})$ is {\em doglike \wording{with nose $n$}\/}\index{doglike}\index{doglike!nose}\index{nose} if $|T|=3$ and $|U|=2$ \wording{and $n$ is the vertex in $T\setminus U$}.  \end{definition}\printFullDetails{

}\begin{theorem}\label{th:Regrow}
Let $G$ reduce by planar 3-reductions to the \wording{peripherally}-4-connected graph $G^{\ei4c}$.  Suppose $v$ is a vertex of $G^{\ei4c}$ with precisely the neighbours $x$, $y$, and $z$ so that $K_v$ has at least two vertices so that $(G_v,\{x,y,z\})$ is, for some subsets $T$ and $U$ of $\{x,y,z\}$, a $(T,U)$-configuration.   Let $(H,\{x,y,z\})$ be a $(G,K_v)$-compatible  $(T,U)$-configuration.  If $\crn(G)\ge 2$, then $\crn(G^H_v)\ge 2$.
\end{theorem}\printFullDetails{   

\begin{cproof} We remark that the non-planarity of $G$ and the fact that we are doing {\em planar\/} 3-reductions implies $G^{\ei4c}$ is not planar.  This fact will be used throughout the proof.

Let $H'=H-\{x,y,z\}$ and let $N$ be the set of vertices $t$ in $\{x,y,z\}$ so that $K^{\max}_v\cap K_t$ is null.  By way of contradiction, we suppose $G^H_v$ has a 1-drawing $D$.  

We start with two simple observations.

\begin{claim}\label{cl:noH'crosses}  Some edge of $H'$ is crossed in $D$. \end{claim}

\begin{proof}
  If no edge of $H'$ is crossed in $D$, then Definition \ref{df:substitution} (\ref{it:equalDegrees}) implies we may resubstitute $K_v$ for $H'$  to obtain a 1-drawing of $G$, a contradiction.  \end{proof}    

\begin{claim}\label{cl:allH'Crosses}  There is no drawing $D'$ of $G^H_v$  in which each crossed edge is incident with a vertex in $H'$.  \end{claim}

\begin{proof}  Otherwise, insert a vertex at each crossing point, and add this vertex to $H'$.  Then contract every edge in the new graph that has both ends in $H'$, and also contract all the $K_u$ to single vertices.  The result is a planar embedding of $G^{\ei4c}$, a contradiction.  \end{proof}

Therefore, we may assume the crossing edges are $e_v\in H'$ with some other edge $f$ not incident with any vertex in $H'$. Observe that $H'$ cannot be a single vertex.

\begin{claim}\label{cl:fNotIsthmus} $f$ is not a cut-edge of $G^H_v-H'$.\end{claim}

\begin{proof} Suppose $f$ is a cut-edge of $G^H_v-H'$.  Since $D[G^H_v-H']$ has no crossing, it is planar.  Therefore, the faces on each side of $f$ in $D[G^H_v-H']$ are the same.  Thus, the ends of $e_v$ are in the same face of $D[G^H_v-H']$.

Consider now the planar embedding $D[G^H_v-e_v]$.  The two ends of $e_v$ are in the same face of the subembedding $D[G^H_v-H']$ and so may be joined by an arc that is disjoint from $D[G^H_v-H']$.  This produces a drawing of $G^H_v$ in which all the crossings involve $e_v$ and edges incident with at least one vertex in $H'$.  This contradicts  Claim \ref{cl:allH'Crosses}. \end{proof}

Since $f$ is not a cut-edge of $G^H_v-H'$, there is a 
cycle $C_f$ of $G^H_v-H'$ containing $f$.  Moreover, $D[C_f]$ separates the two ends of $e_v$, so $e_v$ is an cut-edge of $H'$.  Let $H^1$ and $H^2$ be the two components of $H'-e_v$.

The next claim is central to the remainder of the argument.

\begin{claim}\label{cl:ft'}  Let $t\in \{x,y,z\}$ be a common neighbour of $H^1$ and $H^2$.  Then $f$ is incident with $t'\in K_t$ and one of the faces of $H'+t'$ incident with both $t'$ and $e_v$ is empty except for the segment of $f$ from $t'$ to the crossing with $e_v$. \end{claim}

\begin{proof}  Let $C$ be any cycle in $H'+t'$ containing $e_v$.  Since $e_v$ is a cut-edge of $H'$, $t'\in C$.  Since $G^{\ei4c}-\{v,t\}$ is connected, $G-(K_v\cup K_t)$ is connected.  

Suppose by way of contradiction that there are vertices $u$ and $w$ of $G-(K_v\cup K_t)$ on both sides of $D[C]$.  By the preceding paragraph, there is a $uw$-path $P$ in $G-(K_v\cup K_t)$.  Since $P$ is graph-theoretically disjoint from $C$, but $D[u]$ and $D[w]$ are on different sides of $D[C]$, $D[P]$ crosses $D[C]$; this must be at the unique crossing of $D$, so $f\in P$ and the crossing of $D[P]$ with $D[C]$ is the crossing of $f$ with $e_v$.  

Moreover, $D[C_f]$ crosses $D[C]$ at the crossing of $D$ and so they must cross somewhere else.  As $C_f$ and $H'$ are disjoint, the second crossing is at the vertex $t'$.  Since this is true of any cycle $C_f$ in $G-K_v$, $f$ is a cut-edge of $(G-K_v)-t'$.  

We now consider two cases.

\medskip{\bf Case 1:}  {\em there are distinct vertices $t_1$ and $t_2$ of $G^{\ei4c}-\{t,v\}$ so that $D[K_{t_1}]$ and $D[K_{t_2}]$ are on different sides of $D[C]$.}

\medskip In this case, either (i) for some vertex $s$ of $G^{\ei4c}$,  $f\in K_s$, in which case $t_1$ and $t_2$ are in different components of $G^{\ei4c}-\{t,v,s\}$, or (ii) since $G^{\ei4c}$ is non-planar and so has at least five vertices, for some vertex $s$ of $G^{\ei4c}$ that is an end of $f$, we may choose $t_1$ and $t_2$ to again be in different components of $G^{\ei4c}-\{t,v,s\}$.

In either case, the internal 4-connection of $G^{\ei4c}$ implies that there is an $i\in\{1,2\}$ so that $t_i$ is the only vertex in its component of $G^{\ei4c}-\{t,v,s\}$.  But then $tvt_i$ is a 3-cycle in $G^{\ei4c}$ having $v$ and $t_i$ as degree 3 vertices, contradicting Lemma \ref{lm:triangle}.

\medskip{\bf Case 2:} {\em there are not distinct vertices $t_1$ and $t_2$ of $G^{\ei4c}-\{t,v\}$ so that $D[K_{t_1}]$ and $D[K_{t_2}]$ are on different sides of $D[C]$.}

\medskip In this case, there is a vertex $s$ of $G^{\ei4c}-\{t,v\}$ so that $f\in K_s$ and all the vertices of $G-(K_v\cup K_x)$ on one side of $D[C]$ are in one component $K_s^1$ of $K_s-f$, while all the other vertices of $G-(K_v\cup K_x)$, including the other component $K_s^2$ of $K_s-f$, are on the other side of $D[C]$.

Lemma \ref{lm:isthmus3reduction} implies that $K_s^1$ has neighbours in two $K_r$'s.  According to $D$, these can only be $K_v$ and $K_t$.  But now the 3-cycle $tvs$ has the two degree 3 vertices $v$ and $s$, contradicting Lemma \ref{lm:triangle}.  

Since $f$ is on both sides of $D[C]$, but one side has no vertex, it must be that the end of $f$ on that side is in $C$.  But $f$ is disjoint from $H'$, and so this end can only be $t'$.\end{proof}

Our proof proceeds by considering how many common neighbours among $K_x$, $K_y$, and $K_z$ there are for $H^1$ and $H^2$.  We start by noting that there cannot be three, since then the graph $H^+$ is not planar, contradicting Definition \ref{df:TUconfig}.

\begin{claim}\label{cl:H1H2oneNbr}  $H^1$ and $H^2$ have exactly one common neighbour.  \end{claim}

\begin{proof}  We have already ruled out the possibility that $H^1$ and $H^2$ have three common neighbours.

To rule out two common neighbours, suppose by way of contradiction that $H^1$ and $H^2$ have the two common neighbours $K_x$ and $K_y$.  By the preceding remark, at least one of $H^1$ and $H^2$ does not have a neighbour in $K_z$.  Since $H'$ does have a neighbour in $K_z$, we may choose the labelling so that $H^1$ has a neighbour in $K_z$ and $H^2$ does not. 

Claim \ref{cl:ft'} implies $f$ is incident with both $x'$ and $y'$.  But now $D[f]$ can be rerouted along the other side of the $x'H^2$-edges, around $H^2$, and on to $y'$ so that $G^H_v$ has no crossings.  This implies the contradiction that $G^{\ei4c}$ is planar.  We conclude that $H^1$ and $H^2$ have at most one common neighbour.

If they have no common neighbours, then $H^1$ has neighbours just in $K_x$, while $H^2$ has neighbours in $K_y$ and $K_z$, but not in $K_x$.  In this case, $e_v$ is a cut-edge in $H$ separating $x$ from $\{y,z\}$.  It follows that $x\notin T$.  Since $G_v$ is also a $(T,U)$-configuration, there is an cut-edge $e'_v$ of $G_v$ separating $x$ from $\{y,z\}$.  Now we can replace $H'$ in $T$ with $K_v$ in such a way that $e'_v$ (in fact the only edge of $G_v$ incident with $x$) is crossed by $f$ to yield a 1-drawing of $G$.  This contradiction completes the proof of the claim.
\end{proof}

We conclude from Claim \ref{cl:H1H2oneNbr} that $H^1$ and $H^2$ have precisely one common neighbour $x'$.  Claim \ref{cl:ft'} implies that $f$ is incident with $x'$.  

If, for some $i\in \{1,2\}$, $H^i$ has no other neighbour, then we may reroute $f$ to go around $D[H^i]$, yielding a planar embedding of $G^H_v$ and, therefore, of the non-planar graph $G^{\ei4c}$, a contradiction.

Thus, we may choose the labelling so that $H^1$ has at least one neighbour in $K_y$, while $H^2$ has at least one neighbour in $K_z$.  If, say, $H^1$ is joined to $K_y$ by only one edge, then $y\notin T$; therefore, $y$ is incident with a unique edge in $G_v$ and we can replace $D[H]$ with the planar embedding of $K_v$ so that it is the $yK_v$-edge that is crossed by $f$.  This yields that contradiction that $G$ has a 1-drawing.

Thus, we may assume that $T=\{x,y,z\}$.  However,  there are not edge-disjoint $yz$-paths in $H-x$ ($e_v$ is a cut-edge separating $y$ and $z$).  Therefore, $U=\{y,z\}$, showing $G_v$ is doglike.  It follows that $G_v-x$ has a cut-edge $e'_v$ separating $y$ and $z$.  We may substitute the planar embedding of $K_v$ for $D[H]$ so that $e'_v$ crosses $f$, yielding the final contradiction that $G$ has a 1-drawing.
\end{cproof}

}\section{Reducing to a basic 2-crossing-critical example}\printFullDetails{ 

In this section, we show that if $G$ is a 3-connected \2cc\ graph that reduces by planar 3-reductions to a \wording{peripherally}-4-connected graph, then there is a ``basic" 3-connected \2cc\ graph from which $G$ is obtained by the regrowth mechanism of the preceding section.

}\begin{theorem}\label{th:TUreplace} Let $G\in\m2$ reduce by planar 3-reductions to a \wording{peripherally}-4-connected graph $G^{\ei4c}$.   Let $v$ be a vertex of $G^{\ei4c}$ with just the three neighbours $x$, $y$, and $z$, so that $(G_v,\{x,y,z\})$ is a $(T,U)$-configuration and $K_v$ has at least two vertices.  Let $G_{\rep(v)}$ be the graph obtained from $G$ by contracting as indicated in the following cases. 
\begin{enumerate}
\item If $(G_v,\{x,y,z\})$ is doglike, then let $e$ be the cut-edge of $K_v$ and contract each component of $K_v-e$ to a vertex.
\item If $(G_v,\{x,y,z\})$ is not doglike, then we have the following subcases.
\begin{enumerate}
\item If none of $G_x$, $G_y$, and $G_z$ is doglike, then contract $K_v$ to a vertex.
\item If $(|T|,|U|)=(3,3)$, then contract $K_v$ to a vertex.
\item\label{it:dogNbr} If $G_x$ is doglike and $y\notin T$, then let $C$ be a cycle in $G_v^+$ containing $x'$, $y'$, and $z'$,  delete everything in $K_v-E( C)$ and contract the edges of  $C$ to the 3-cycle $x'y'z'$.
\end{enumerate}
\end{enumerate}
Then $G_{\rep(v)}\in\m2$.
 \end{theorem}\printFullDetails{

There is one clarification that is required to understand one fine detail of $G_{\rep(v)}$.  If, for example, the vertex $x'$ is in $K_v$, then we proceed precisely as described in the statement.  If, however, $x'$ is in $K_x$ and $x\in T$, then in $G_{\rep(v)}$ we retain only two edges between $x'$ and the contracted vertex in $K_{\rep(v)}$ to which it is joined.  This especially applies in the case \ref{it:dogNbr}:  if $z'\in K_z$, then we keep only the two edges of $C$ incident with $z'$, while if $z'\in K_v$, then we keep all the $z'K_z$-edges. 

There is also an important remark to be made.  We had long thought that it was possible to reduce each $K_v$ to a single vertex and retain 2-criticality.  This might be true in the particular cases of 3-connected \2cc\ graphs with no subdivision of $V_8$, but it is certainly not true of all 3-connected \2cc\ graphs.  

In Definition \ref{df:theTiles} we described the set $\mathcal S$ of all graphs that can be obtained from the 13 tiles and the two frames.  These graphs are all 3-connected and \2cc.  Consider any one of these that uses the right-hand frame in Figure \ref{fg:frames} and uses the second picture in the third row of Figure \ref{fg:pictures}. With appropriate choices of the neighbouring pictures, the 3-cycle in the upper half of the picture is \wording{part of a doglike} $G_v$ that contains the parallel edges in the picture and the parallel edges in the frame:  the horizontal edge in the 3-cycle is $K_v$.  The vertical edge in the other 3-cycle in the picture is a $K_x$.    When we do the planar 3-reductions in this case, the contractions of $K_x$ and $K_v$ produce a pair of parallel edges not in the rim.  The conclusion is that the resulting \p4c\ graph plus parallel edges is not \2cc.  Thus, the technicalities we must endure in the statement of Theorem \ref{th:TUreplace} seem to be unavoidable.
\minorrem{(This paragraph refers to Figure \ref{fg:pictures}.  When the pictures are done properly, we will likely have to rewrite the reference to that picture.)}

\bigskip

\begin{cproof}   We use the notation $K_{\rep(v)}$\index{$K_{\rep(v)}$} for the contraction of $K_v$ in $G_{\rep(v)}$. 

\medskip{\bf Phase 1:}  {\em showing $G_{\rep(v)}$ is 3-connected.}

\medskip Let $t$ and $u$ be vertices of $G_{\rep(v)}$.  We show $G_{\rep(v)}-\{t,u\}$ is connected. 

Let $w_t$ and $w_u$ be the vertices of $G^{\ei4c}$ so that $t\in K_{w_t}$ and $u\in K_{w_u}$ (taking, for example, $K_{w_t}$ to be $K_{\rep(v)}$ if $t\in K_{\rep(v)}$).
It follows from Lemma \ref{lm:x'y'z'paths1} that every vertex of every $K_s$ has a path in $G-\{t,u\}$ to at least \wording{one neighbour of $K_s$} that is not one of $K_{w_t}$ or $K_{w_u}$.   This is also true of $K_{\rep(v)}$, as may be seen by checking the analogues for $K_{\rep(v)}$ of Lemma \ref{lm:x'y'z'paths1} in the three cases for which $K_{\rep(v)}$ has at least two vertices.   (Note there are two possible outcomes for $K_{\rep(v)}$ in Case \ref{it:dogNbr}, depending on whether $z'\in K_v$, in which case $K_{\rep(v)}$ is a 3-cycle, or $z'\in K_z$, in which case $K_v$ is an edge.)

Since each $K_s$ is connected, $G_{\rep(v)}-\{t,u\}$ is connected.

\medskip {\bf Phase 2:}  {\em showing $\crn(G_{\rep(v)})\ge 2$.}

\medskip  The graph $\bar K_{\rep(v)}$ obtained from $K_{\rep(v)}$ by adding $x$, $y$, and $z$ is a $(G,K_v)$-compatible $(T,U)$-configuration.  Therefore, Phase 2 follows immediately from Theorem \ref{th:Regrow}.

\medskip{\bf Phase 3:}  {\em showing that $G_{\rep(v)}$ is 2-crossing-critical.}

\medskip Let $e$ be any edge of $G_{\rep(v)}$.  Then there is \wording{an edge $e_G$ in $G$ naturally} corresponding to $e$ (in the sense that precisely the same contractions and deletions of $G$ and $G-e_G$ can be used to obtain both $G_{\rep(v)}$ \wording{and $G_{\rep(v)}-e$)}.   

\medskip {\bf Special situation.}  {\em There is one case where the choice of $e_G$ must be made with special care.  Suppose $K_v$ contracts down to the single vertex $v$ and $e$ is one of two parallel edges $vx$.  In the case $K_v$ has a cut-edge $e'$, Lemma \ref{lm:isthmus3reduction} implies each component of $K_v-e'$ is joined to two of the neighbours of $v$.  Suppose that $K_x$ is the only common neighbour of these two components.  Since $G_v$ is not  doglike, some component $L$ of $K_v-e'$ is joined by exactly one edge to its other neighbour; choose $e_G$ to be an $xL$-edge.}

\medskip
}\begin{definition}  For each vertex $w$ of $K_{\rep(v)}$, $L_w$\index{$L_w$} denotes the subgraph of $K_v$ that contracts to $w$.  \end{definition}\printFullDetails{

Since $G$ is 2-crossing-critical, there is a 1-drawing $D$ of $G-e_G$.  If no edge of any $L_w\subseteq K_v$ is crossed in $D$, then these may each be contracted to obtain a 1-drawing of $G_{\rep(v)}-e$, and we are done.    

\begin{claim}\label{cl:allL_wCrosses}  If there is a drawing of $G-e_G$ in which all the crossings are between edges incident with vertices in $L_w$, then $G_{\rep(v)}-e$ is planar.  \end{claim}

\begin{proof}  Insert vertices at each crossing point and contract every edge in the new graph that has both ends in some $L_u$.  The result is a planar embedding of $G_{\rep(v)}-e$.  \end{proof}

Therefore, we may assume the crossing edges are $e_v\in L_w\subseteq K_v$ with some other edge $f$ not incident with any vertex in $L_w$.

\medskip{\bf Case 1:} {\em $f$ is a cut-edge of $(G-e_G)-L_w$.}

\medskip In this case, $D[(G-e_G)-L_w]$ has no crossing, so it is planar.  Therefore, the faces on each side of $f$ in $D[(G-e_G)-L_w]$ are the same.  Thus, the ends of $e_v$ are in the same face of $D[(G-e_G)-L_w]$.

Consider now the planar embedding $D[(G-e_G)-e_v]$.  The two ends of $e_v$ are in the same face of the subembedding $D[(G-e_G)-L_w]$ and so may be joined by an arc that is disjoint from $D[(G-e_G)-L_w]$.  This produces a drawing of $G-e_G$ in which all the crossings involve $e_v$ and edges incident with at least one vertex in $L_w$.  Claim \ref{cl:allL_wCrosses} implies $G_{\rep(v)}-e$ is planar, as required.

\medskip{\bf Case 2:}  {\em $f$ is not a cut-edge of $(G-e_G)-L_w$.}

\medskip In this case, $f$ is in a cycle $C_f$ of $(G-e_G)-L_w$.  Moreover, $D[C_f]$ separates the two ends of $e_v$, so $e_v$ is a cut-edge of $L_w$.  Let $L_w^1$ and $L_w^2$ be the components of $L_w-e_v$.

We \wording{consider separately} two cases for $G_v$.

\medskip{\bf Subcase 2.1:}  {\em $G_v$ is doglike.}

\medskip In this subcase, $K_{\rep(v)}$ is two vertices $w$ and $\bar w$ joined by a cut-edge $e'$ of $G_v-x$, each joined by an edge to $x'$, $w$ is joined by at least two edges to $K_y$ and $\bar w$ is joined by at least two edges to $K_z$.   Lemma \ref{lm:twoNeighbours} implies that $K_x$ has at most two neighbours in $K_v$.  We already know there is one in each of $L_w$ and $L_{\bar w}$.  Lemma \ref{lm:no2matching} now implies there is a vertex $x'\in K_x$ incident with all the $K_vK_x$-edges in $G$. Thus, we may choose the labelling of $L_w^1$ and $L_w^2$ so that the neighbour of $x'$ in $L_w$ is in $L_w^1$.  

We see that $x'$ and the end of $e_v$ in $L_w^2$ are neighbours of vertices in $L_w^1$, and neither of these vertices is in $L_w^1$.  The only other possibilities for neighbours of $L_w^1$ outside of $L_w^1$ are in $K_y$ and $L_{\bar w}$, the latter being the end of $e'$.  
A similar remark holds for $L_w^2$:  it has the neighbour (via $e_v$) in $L_w^1$, and can have at most neighbours in $K_y$ and $L_{\bar w}$ (via $e'$).

Since $G$ is 3-connected, for each $i=1,2$,  $L_w^i$ has at least two neighbours outside of $L_w^i$ other than $x'$.  From the neighbour analysis of the preceding paragraph, there are at most three in total:  two to $K_y$ and one to $L_{\bar w}$.    There are two ways this can happen.

In the first way, both edges from $L_w$ to $K_y$ have their ends in $L_w^2$, while $e'$ has an end in $L_w^1$.  But then $e_v$ is a cut-edge of $K_v$ that violates Lemma \ref{lm:isthmus3reduction}:  the edge $e_G$ cannot connect $L_w^2$ to either $x'$ (Lemma \ref{lm:twoNeighbours} or $K_z$ (because $e'$ is a cut-edge of $G_v-x$), so the component $L_w^2$ of $K_v-e_v$ is joined only to $K_y$. 

Therefore, $e'$ has one end in $L_w^2$ and the two $K_vK_y$ edges have ends in different ones of $L_w^1$ and $L_w^2$.  It follows that $y'$ is incident with these \wording{edges, so Lemma \ref{lm:twoNeighbours}} implies $y'$ has precisely these neighbours in $K_v$.

\wording{Contract $D[e_v]$ so that} $L_w^1$ is pulled across $f$ and, if necessary, shrink $D[L_w^1]$ so that we obtain a new drawing $D^1$ of $G-e_G$ in which $f$ crosses the edges from $x'$ and $y'$ to $L_w^1$.  

\begin{claim}\label{cl:fNotInLbarW}  $f\notin L_{\bar w}$.\end{claim}

\begin{proof}  If $f\in L_{\bar w}$, then exactly the same analysis as for $L_w$ implies that $L_{\bar w}-f$ has two components $L_{\bar w}^1$, from which there is an edge to $x'$ and an edge to $z'$, and $L_{\bar w}^2$, from which there is an edge to $z'$ and $L_w^2$.  But now the graph-theoretically disjoint cycles in $L_w+y'$ containing $e_v$ and $L_{\bar w}+z'$ containing $f$ cross exactly once in $D$, which is impossible.  \end{proof}

It follows from Claim \ref{cl:fNotInLbarW} that $f\notin L_{\bar w}$.  We contract the uncrossed $D^1[L_w]$ and $D^1[L_{\bar w}]$ to obtain a drawing $D^2$ of $G_{\rep(v)}-e$, in which the only crossings are of $f$ with the edges from $x'$ and $y'$ to $L_w^1$.  In $D^2$, there are parallel edges $y'w$; the one from $y'$ to $L_w^2$ is not crossed in $D^2$, so we may make all the others go alongside the uncrossed one.  This yields a drawing $D^3$ of $G_{\rep(v)}-e$ in which the only crossing is $x'w$ with $f$, so $D^3$ is a 1-drawing of $G_{\rep(v)}-e$, as required.

\medskip{\bf Subcase 2:}  {\em $G_v$ is not doglike.}

\medskip{\bf Subcubcase 2.1:} {\em there is a neighbour $x$ of $v$ in $G^{\ei4c}$ so that $G_x$ is doglike and $x'\in K_v$ is the nose of $G_x$.}

\medskip  Let $C$ be the cycle in $G_v$ that we contracted to the 3-cycle $x'y'z'$.  We let $G^C$ be the subgraph of $G$ obtained by deleting all edges between the various $L_u$ except the one or three edges in $C$.  Choose the labelling so that $y$ is a neighbour of $v$ in $G^{\ei4c}$ so that there is exactly one $K_vK_y$-edge in $G$; thus $y'\in K_v$.

Let $r$ be that element of $\{x,y,z\}$ so that $r'\in L_w$.  There are precisely two edges $e_1$ and $e_2$ in $G^C$  coming out of $L_w$ in $G_v-r$.     

Let $L_w^1$ be the component of $L_w-e_v$ containing $r'$ and let $L_w^2$ be the other.  Since  $C$ goes through $r'$, at least one of $e_1$ and $e_2$ is incident with a vertex in $L_w^1$.  Therefore, at most one of $e_1$ and $e_2$ has an end in $L_w^2$.  

We claim that $L_w^2$ is not joined to any other vertex in $G^C$.  The only possibility is that there is an edge from $L_w^2$ to $K_x\cup K_y\cup K_z$.  Since all the $K_vK_x$- and $K_vK_y$-edges in $G$ are incident with $x'$ and $y'$, respectively and $x'$ and $y'$ are not in $L_w^2$,  there are no edges in $G$ from $L_w^2$ to $K_x\cup K_y$.  

As for the possibility of an $L_w^2K_z$-edge, this can only exist if $z'\in K_z$.  But $z'$ already has two known neighbours in $K_v$, namely the $K_v$-ends of the edges of $C$ incident with $z'$.  Lemma \ref{lm:twoNeighbours} implies these are the only vertices of $K_v$ adjacent to vertices in $K_z$.  Therefore these known $z'$-neighbours are the only ones; in particular, $z'$ has no neighbour in $L_w^2$, as claimed.

We obtain a 1-drawing of $G_{\rep(v)}-e$ by partially contracting $D[e_v]$ and, if necessary, scaling $D[L_w^2]$ down so that $L_w^1$ and $L_w^2$ are now drawn on the same side of $f$.  The only crossing in this new drawing is of the edge of $D[G^C]$, if it exists, that is not $e_v$ and joins $L_w^2$ to the rest of $G^C$.  Now we may contract all the $L_u$ to single vertices to obtain the required 1-drawing of $G_{\rep(v)}-e$.

\medskip{\bf Subsubcase 2:}  {\em there is no neighbour $x$ of $v$ in $G^{\ei4c}$ so that $G_x$ is doglike and $x'\in K_v$ is the nose of $G_x$..}

\medskip At this stage, $K_v$ contracts to a single vertex of $G_{\rep(v)}$.  In this case, $K_v-e_v$ has two components $K_v^1$ and $K_v^2$.   Lemma \ref{lm:isthmus3reduction} implies each of $K_v^1$ and $K_v^2$ are connected in $G$ to at least two of $K_x$, $K_y$ and $K_z$.  Because $G^+_v$ is planar, at most two of $K_x$, $K_y$, and $K_z$ can be adjacent to both $K_v^1$ and $K_v^2$.

If both $K_x$ and $K_y$ have neighbours in both $K_v^1$ and $K_v^2$, then there is an $i\in\{1,2\}$ so that $K_v^i$  has adjacencies only in those two.    Now pull $D[K_v^i]$ across $f$ and, scaling $D[K_v^i]$ if necessary, to obtain a planar embedding of $G-e_G$.  This contracts to a planar embedding of $G_{\rep(v)}-e$, as required.

Thus, we may assume $K_v^1$ and $K_v^2$ have precisely one common neighbour in $G$.  Each has its own neighbour.  Since $G_v$ is not doglike, one of these, say $K_v^1$, is joined by a single edge to that unique neighbour and now we can drag $K_v^1$ across $f$.  This works unless $e$  goes to $K_v^2$ and $K_v^2$ is joined to its unique neighbour by two edges.  But this is the special situation, and $e$ is joined to $K_v^1$, not $K_v^2$.
\end{cproof}

}\section{Growing back from a given peripherally-4-connected graph}\printFullDetails{

The important corollary of Theorem \ref{th:TUreplace} is that, if we replace each $K_v$ with its $K_{\rep(v)}$, then we get a \2cc\ model of $G^{\ei4c}$ with very simple replacements for the vertices of $G^{\ei4c}$.     In this section, we explain how to obtain all the 3-connected \2cc\ graphs that reduce by planar 3-reductions to a particular \p4c\ graph.

Let  $L$ be a non-planar \p4c\ graph.  For each vertex $v$ of $L$ having only three neighbours $x$, $y$, and $z$, we decide on the type of $v$; that is, we choose $T_v\subseteq \{x,y,z\}$ and, in the case $|T_v|=3$, we decide on $U_v$: either $U_v=\{x,y,z\}$, or $U_v$ consists of two of $\{x,y,z\}$.  For each edge of $L$ joining two vertices of degree at least 4, we decide whether the edge will be a single edge or a parallel pair.

The choices must be made so that $x\in T_v$ if and only if $v\in T_x$.  If, for some $v$, $(|T_v|,|U_v|)=(3,2)$ ($v$ is chosen to be doglike), then some other implications (as in Theorem \ref{th:TUreplace}) must be maintained.  Choose the labelling so that $x\notin U_v$.  Then $x$ is the nose of the dog, $v$ is replaced with $K_v$, so that $K_v$ is an edge $y'z'$, so that $y'$ incident with two edges going to $K_y$, and likewise for $z'$ to $K_z$.  Each of $y'$ and $z'$ is also incident with an edge to $x'\in K_x$.  Furthermore, $K_x$ can be either a vertex, or, if $|T_x|\ne 3$, an edge, or a 3-cycle.

Once all these choices have been made, the resulting graph is tested for 2-criticality.  Thus, for a given \wording{\wording{peripherally}-4-connected graph} $L$, there will be many graphs that require testing.  If one of the resulting graphs $L'$ is found to be \2cc, then there may be many other 3-connected \2cc\ graphs that arise from $L'$.  Recall that, for each vertex of $L$ that has only three neighbours, we have made a choice as to what type that vertex has.   The following lemma explains what may replace the vertex of each type.

\begin{lemma}\label{lm:buildingUp}  Suppose the \p4c\ graph $L$ has choices as explained in the preceding paragraphs to produce a 3-connected \2cc\ graph $L'$.  Suppose $G$ is a 3-connected \2cc\ graph that reduces by planar 3-reductions to $L$ so that $L'$ is the graph obtained from $G$ by the replacements described in Theorem \ref{th:TUreplace}.  Then, for each $K_v$ in $L'$, $K_v$ is replaced by one of the possibilities shown in Figures \ref{fg:TUreplace}, depending on $(T_v,U_v)$.  \end{lemma}

\begin{figure}
\begin{center}
\includegraphics[scale=.4]{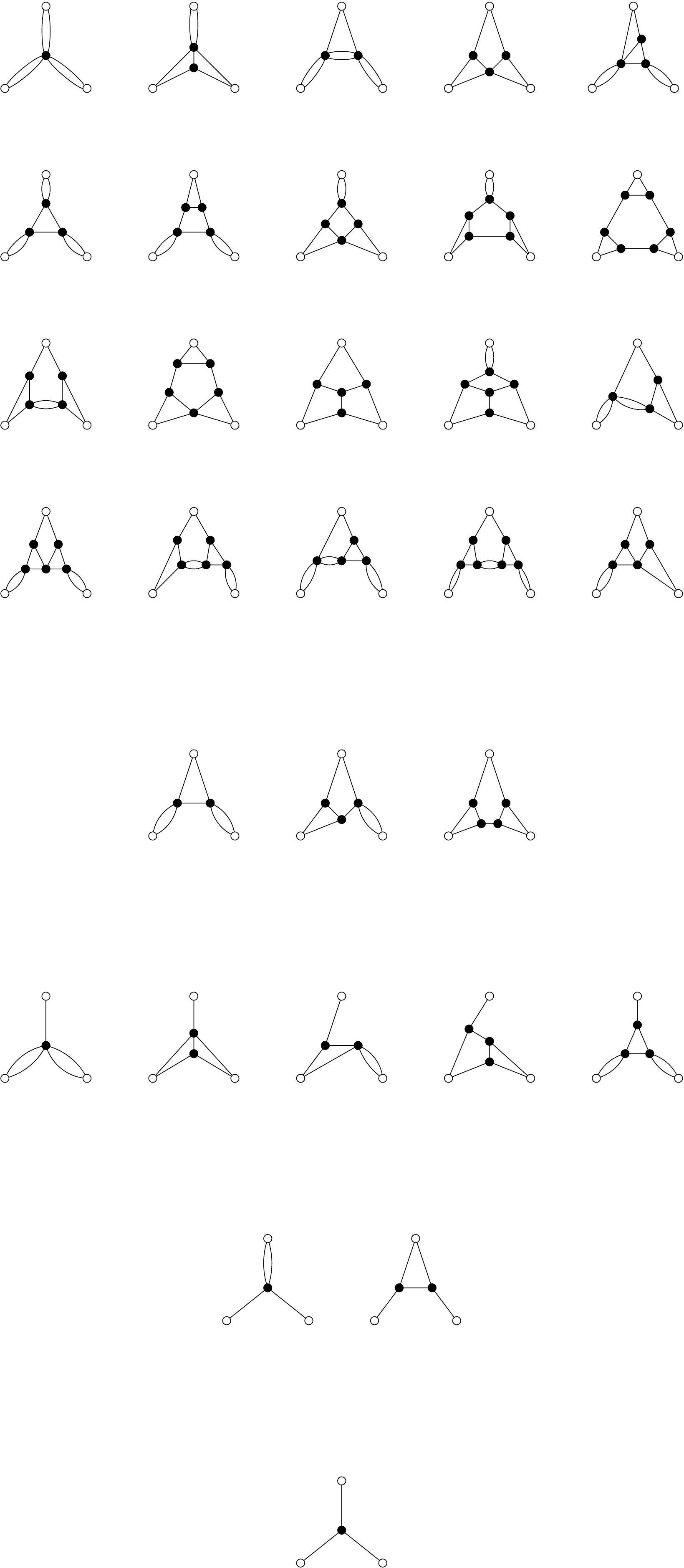}
\caption{The possible $(T,U)$-configurations.}
\label{fg:TUreplace}
\end{center}
\end{figure}

\begin{cproof}  We only illustrate the tedious proof in a couple of cases.

\bigskip {\bf Case 1:}  $(T_v,U_v)=(\{x,y,z\},\{y,z\})$.

\medskip Let $e$ be a cut-edge in $G_v-x$ separating $y$ and $z$.  Let $K_v-e$ have the two components $K_v^y$, containing the neighbour(s) of $y$, and $K_v^z$, containing the neighbour(s) of $z$.  If $K_v^y$, for examples, is not just either a single vertex or an edge joining the two neighbours of $y$, then it contains a subdivision of one of these (either pick a path in $K_v^y$ joining the neighbour of $y$ to the $K_v^y$-end of $e$ or pick a path joining the two neighbours of $y$).  It is easy to see that the subdivision (making a similar choice on the $z$-side) is also a $(T_v,U_v)$-configuration.  By Theorem \ref{th:Regrow}, the subgraph has crossing number 2, and so is all of $G$.  Thus, $K_v$ can be at most one of the three figures in Figure \ref{fg:TUreplace} corresponding to $(|T|,|U|)=(3,2)$.

\bigskip {\bf Case 2:}  $T_v=\{x,y,z\}=U_v$.

\medskip  In this case, $G_v-x$ contains edge-disjoint $yz$-paths.  Therefore, it contains two such paths $P$ and $Q$ that make a digonal pair.  
  If $P$ and $Q$ are internally disjoint, then there is a $(P-\{y,z\})(Q-\{y,z\})$-path $R$.  If $P$ and $Q$ are not internally disjoint, then set $R=\varnothing$.  In either case, set $M=P\cup Q\cup R$.   There are two $x(M-\{y,z\})$-paths  $R_1$ and $R_2$ in $G_v$.

If the ends of $P$ and $Q$ are in the same digon of $P\cup Q$, then planarity of $G_v^+$ implies $R_1$ and $R_2$ have their ends in the same one of $P$ and $Q$.   It follows that $M\cup R_1\cup R_2$ is a $(T_v,U_v)$-configuration, and so is $G_v$ by 2-criticality and Theorem \ref{th:Regrow}.

The fact that $G$ is 3-connected implies that there cannot be more than four common internal vertices to $P$ and $Q$, as if there were six digons, then some two consecutive ones would not contain an end of either $R_1$ or $R_2$.  This would readily yield a 2-cut in $G$, which is impossible.  This is why the number of possibilities for $G_v$ in this case is finite. \end{cproof}

\minor{In some of the larger $(T,U)$-configurations, there are edges that are not required to produce the relevant paths between $s$, $t$, and $u$, but, rather, are there to maintain the connectedness of the configuration.  These edges might be deletable without reducing the crossing number below 2.  Thus, each candidate 3-connected graph produced by the method described needs to have its criticality checked.  
}

}\section{\major{Further reducing to internally-4-connected graphs}}\printFullDetails{

\major{In order to find the \2cc\ graphs that do not contain $V_8$, we wish to use the characterization by Robertson of $V_8$-free graphs.  This characterization, described in the next section, is in terms of {\em \i4c\ graphs\/}.  These graphs are very closely related to \p4c\ graphs and it is the purpose of this section to describe the reduction of a \p4c\ graph to an \i4c\ graph, and back again.}

\begin{definition}  \major{A {\em \hug\/}\index{hug} in a graph $G$ is an edge $e$ in a triangle $T$ whose vertex $v$ not incident with $e$ has degree 3.  The triangle $T$ is the {\em \triang{$e$}\/}\index{triangle ($e$-)}\index{$e$-triangle}, $v$ is the {\em \head\/}\index{head (of a hug)}\index{hug!head} of the \hug\ and the two edges of $T$ other than $e$ are the {\em \arm s\/}\index{arm (of a hug)}\index{hug!arm} of the \hug. } \end{definition}

\begin{definition}  \major{A $G$ is \i4c\ \index{\i4c} if it is \p4c\ and has no \hug s.} \end{definition}

\major{It is not correct that simply deleting (successively) the hugs from a \p4c\ graph produces an \i4c\ graph.  There is a particular situation that arises that needs special care.}

\begin{definition} \major{\begin{enumerate}\item A \hug\ $e$ with \head\ $v$ is a {\em \bearhug\/}\index{bearhug}\index{hug!bearhug} if there is an end $u$ of $e$, incident with a second hug $uy$ whose head $t$ is different from $v$, and so that, with $w$ the other end of $e$, the neighbours of $u$ are contained in the union of $\{t,v,w\}$ and the set of neighbours of $t$.  (See Figure \ref{fg:bearHug}.)
\item A \hug\ is {\em deletable\/}\index{deletable (hug)}\index{hug!deletable} if it is not a \bearhug.
\item A pair of \bearhug s having a common end is {\em simultaneously deletable\/}\index{simultaneously deletable}\index{deletable (hug)!simultaneously deletable}\index{hug!simultaneously deletable}.
\end{enumerate}}\end{definition}

\begin{figure}
\begin{center}
\scalebox{1.0}{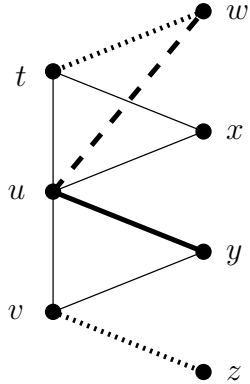}
\end{center}
\caption{The thick edge is a \bearhug.  The dotted edges $tw$ and $vz$ might be subdivided, and the dashed edge $uw$ need not be present. If $uw$ is not present, then $\{ux,uy\}$ is a simultaneously deletable pair of \bearhug s.}\label{fg:bearHug}
\end{figure}

\major{We are now in a position to reduce a \p4c\ graph to an \i4c\ graph. }

\begin{theorem}\label{th:hugElimination}  \major{Let $G$ be a non-planar \p4c\ graph and let $G=G_0,G_1,$ $\dots,G_k$ be a sequence of graphs so that, for each $i=1,2,\dots,k$, there is either a \hug\ $h_i$ or a simultaneously deletable pair $h_i$ of \bearhug s  in $G_{i-1}$ so that $G_i=G_{i-1}-h_i$.  Then, for $i=0,1,2,\dots,k$: 
\begin{enumerate} \item $G_i$ is a subdivision of a non-planar \p4c\ graph;
\item if $v$ has degree 2 in $G_i$ but not in $G_{i-1}$, then $h_i$ is a simultaneously deletable pair of \bearhug s in $G_{i-1}$, both incident with $v$; and 
\item every degree 2 vertex in $G_i$ has two degree 3 neighbours in $G_i$.
\end{enumerate}}

\major{Furthermore, if the sequence $G_0,G_1,\dots,G_k$ is maximal, then $G_k$ is a subdivision of an \i4c\ graph.} \end{theorem}

We emphasize that, in the reduction process described in the statement, $G_i$ is obtained from $G_{i-1}$ by the deletion of either one  or two edges.

\bigskip
\begin{cproof}   \major{Suppose by way of contradiction that $i$ is least so that $G_i$ is planar.  Since $G_0$ is not planar, $i>0$, so $G_i=G_{i-1}-h_i$.  Each edge in $h_i$ joins two neighbours of a degree 3 vertex in $G_i$ and so may be added to the planar embedding of $G_i$ to produce a planar embedding of $G_{i}$ together with that edge of $h_i$. {\bogdan{In the case $|h_i|=2$, the heads of the hugs are not adjacent.  Thus, both hugs may be added simultaneously, while preserving planarity.  Thus, $G_{i-1}$ is planar, contradicting the choice of $i$.} }}

\major{By way of contradiction, we may let $i$ be least so that $G_i$ is not a subdivision of a \p4c\ graph.  Thus, $i\ge 1$.
  Throughout the proof, when we refer to the vertices $t,u,v,w,x,y,z$, we are always referring to the labelling in Figure \ref{fg:bearHug}.  In each of the three cases, there are two possibilities for $h_i$ to be considered.
}
  
\major{It will be helpful to notice that, in the case $h_i$ consists of a simultaneously deletable pair of bear hugs, the vertex $u$ is not a node of $G_i$ and is incident with both deleted edges.  
}

\begin{claim} \major{$G_i$ is a subdivision of a 3-connected graph.} \end{claim}

\begin{proof}  \major{Let $a$ and $b$ be distinct nodes of $G_i$.  Then $a$ and $b$ are distinct nodes of $G_{i-1}$, so there are three internally disjoint $ab$-paths $P_1,P_2,P_3$ in $G_{i-1}$.  }

\major{If $e\in h_i$, then the head $c$ of the $e$-triangle has degree 3.  {\bogdan{If $e$ is in some $P_i$ and $T$ is the triangle containing $e$ and its head, then we may replace $P_i\cap T$ with the path in $T$ complementary to $P_i\cap T$.  The at most two modifications result in three internally disjoint paths that}} are also paths in $G_i$. } \end{proof}

\medskip

\begin{claim} \major{If $a$ has degree at least three in $G_{i-1}$ and degree 2 in $G_i$, then:\begin{enumerate}
\item $|h_i|=2$;
\item  $a$ is incident with both edges in $h_i$; and
\item  both neighbours of $a$ have degree 3 in $G_i$.  
\end{enumerate}}\end{claim}

\begin{proof} \major{ Let $e\in h_i$.    The head $b$ of the $e$-triangle has degree 3 in $G_{i-1}$ and, since $G_{i-1}$ is a subdivision of a \p4c\ graph, no other vertex of the $e$-triangle has degree 3, so Lemma \ref{lm:triangle} shows they both have degree at least 4.  
It follows that if $e$ is the only edge in $h_i$, then the ends of $e$ have degree at least 3 in $G_i$ and no new vertex of degree 2 is introduced in $G_i$.}

\major{
Therefore $h_i$ is a deletable pair.  The only new vertex of degree 2 in $G_i$ is $u$, so $a=u$. Also, the only neighbours of $u$  in $G_i$ have degree 3 in $G_i$. 
}\end{proof}

\major{
The remaining possibility is that there is a set $\{a,b,c\}$ of nodes of $G_i$ and a 3-separation  $(H,J)$  of $G_i$ so that $H\cap J=\iso{a,b,c}$ and both $H-\{a,b,c\}$ and $J-\{a,b,c\}$ have at least two nodes of $G_i$.  
}
  
\major{
Because $G_{i-1}$ is a subdivision of a \p4c\ graph, there is an edge $e\in h_i$ having one end $r_H$ in $H-\{a,b,c\}$ and one end $r_J$ in $J-\{a,b,c\}$.  
}
  
\major{
Suppose for the moment that $h_i$ has a second edge.  Since $G_{i-1}$ is a subdivision of a \p4c\ graph, not all the neighbours of $u$ in $G_{i-1}$  can be in the same one of $H$ {\bogdan{and $J$}}.  We may choose the labelling so that $x=r_J$.  As $t$ is a common neighbour of $u=r_H$ and $x=r_J$, we conclude that $t\in \{a,b,c\}$, say $t=a$.  
}
  
\major{
It follows that at least one of $v$ and $y$ (the other two neighbours of $u$) is in $H-\{t,b,c\}$.  Since $v$ and $y$ are adjacent, it follows that both are in $H$ and, furthermore,  $uy$ is also in $H$.  In particular, there is a unique edge in $h_i$ that has one end in $H-\{a,b,c\}$ and one end in $J-\{a,b,c\}$.
}
  
\major{
Now the two possibilities for $h_i$ are merged:  $e$ is the unique edge in $h_i$ having one end $r_H$ in $H-\{a,b,c\}$ and one end $r_J$ in $J-\{a,b,c\}$.  The head $q$ of the $e$-triangle must be in $\{a,b,c\}$, say $q=a$.
}
  
\major{
  Since $q$ has degree 3, we may choose the labelling so that $r_H$ is the only neighbour of $q$ in $H-\{q,b,c\}$.  The neighbour $r_J$ of $q$ is in $J-\{q,b,c\}$.   Note that $r_H$ and $r_J$ are both nodes of $G_{i-1}$.
}
  
\major{  
 The third neighbour $s$ of $q$ is in $J$, so $\{r_H,b,c\}$ is a 3-cut in $G_{i-1}$.  Since $G_{i-1}$ is \p4c, there is a unique node $p$ in $H-r_H$, which is {\bogdan{joined by branches in $G_{i-1}$}} to all of $r_H$, $b$, and $c$.
}
  
\major{  
  If $s\in \{b,c\}$, then the discussion in the preceding paragraph applies with $r_J$ and $J$ in place of $r_H$ and $H$, respectively.  The nodes of $G_{i-1}$ are now all known (there are only 7), and the edges are almost completely determined.  In particular, $G_{i-1}$ is a subgraph of the planar graph shown in Figure \ref{fg:guoli}, contradicting the fact that $G_{i-1}$ is non-planar.   Therefore,  $s$ is in $J-\{q,b,c\}$.
    }
    
    \begin{figure}
\begin{center}
\scalebox{.9}{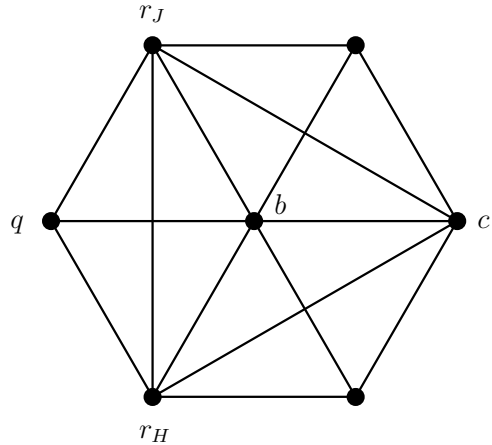}
\end{center}
\caption{\major{When $s=b$, $G_{i-1}$ is a subgraph of the illustrated planar graph.}}\label{fg:guoli}
\end{figure}

\major{
The vertex $r_H$ is the only candidate for the second branch vertex {\bogdan{(after $p$)}} of $G_i$ in $H-\{q,b,c\}$, so it must be joined by a $G_i$-branch to at least one of $b$ and $c$; choose the labelling so that $b$ is an end of such a $G_i$-branch.
}
  
\major{
If $b$ has only one neighbour  in $J-\{q,b,c\}$, then $p$ and $b$ are both degree 3 vertices in a triangle in $G_{i-1}$; since $G_{i-1}$ is a subdivision of a \p4c\ graph, this contradicts Lemma \ref{lm:triangle}.  The same reasoning implies that both $r_H$ and $b$ have degree at least 4 in $G_{i-1}$.  These imply that $r_Hp$, $r_Hb$, and $pb$ are all edges of $G_{i-1}$.
}
  
\major{
Because $r_Hr_J$ is in $h_i$ and $q$ is the head of the $r_Hr_J$-triangle, we know that $r_Hr_J$, $qr_H$, and $qr_J$ are all edges of $G_{i-1}$.   Furthermore, $r_Hs$ is not a $G_{i-1}$-branch (it would yield a second edge with one end in each of $H-\{a,b,c\}$ and $J-\{a,b,c\}$). 
}
  
\major{
The triangles $pr_Hb$ and $qr_Hr_J$ show that $r_Hr_J$ is a \bearhug.  Since it was deleted, it must be in a simultaneously deletable pair of \bearhug s.  This implies that $r_Hb$ is the other edge in that pair.  Thus, $H-\{a,b,c\}$ has only one node in $G_i$, a contradiction that completes the proof that each $G_i$ is a subdivision of a \p4c\ graph.}  

\major{We move on to showing that a maximal sequence ends in a subdivision of an \i4c\ graph.  So suppose $G_i$ is not a subdivision of an \i4c graph.  Since it is a subdivision of a \p4c\ graph $H$, there is a 3-cut $\{a,b,c\}$ in $H$ so that $ab$ is an edge of $G_i$.   Since $H$ is \p4c, there is a vertex $p$ adjacent in $H$ to all of $a,b,c$ and with no other neighbours in $H$.  Lemma \ref{lm:triangle} shows that the triangle $p,a,b$ has at most one vertex of degree 3; since $p$ is such a vertex, $a$ and $b$ have degree at least 4 in $H$.  It follows that $pa$ and $pb$ are edges of $G_i$ and, therefore, $ab$ is a \hug\ in $G_i$.   }

\major{
It is evident from the definitions that, as soon as $G_i$ has a \hug, then either $G_i$ has a \hug\ that is not a \bearhug\ or $G_i$ has a pair of simultaneously deletable \bearhug s.  In either case, $G_i$ is not the last in a maximal sequence.}\end{cproof}

\major{We conclude this section with a brief discussion of the reverse process:  how to generate all the \p4c\ graphs that reduce to a given non-planar \i4c\ graph $G$.  Every graph created through iterating the following procedure is \p4c\ and  non-planar.  We choose either two non-adjacent neighbours of a degree 3 vertex and add the edge between them, or we choose an edge $e$ joining degree 3 vertices and a neighbour of each vertex incident with $e$, subdivide $e$ once, and join both the chosen neighbours to the vertex of subdivision.}

\major{Every \i4c\ graph produces only finitely many \p4c\ graphs through this process, as the number of possible additions is initially finite and strictly decreasing.}

}\section{The case of $V_8$-free 2-crossing-critical graphs}\label{V8free}\printFullDetails{

In this section, \wording{we complete our analysis} of \p4c\ \2cc\ graphs by considering the case of 3-connected \2cc graphs that do not contain a subdivision of $V_8$.  This is the whole reason for studying \p4c\ graphs, since there is a characterization of \wording{the closely related} \i4c\ graphs that do not contain a subdivision of $V_8$.

Two important classes of graphs in this context are the following.

}\begin{definition}\label{df:bicycle4covered}  \begin{enumerate}\item A {\em bicycle wheel\/}\index{bicycle wheel} is a graph consisting of a {\em rim}\index{rim (of a bicycle wheel)}\index{bicycle wheel!rim}, which is a cycle $C$, and an {\em axle\/}\index{axle}\index{bicycle wheel!axle}, which is consists of two adjacent vertices $x$ and $y$ not in the rim, together with {\em spokes\/}\index{bicycle wheel!spokes}\index{spokes (of a bicycle wheel)}, which are edges from $\{x,y\}$ to $C$.
\item A {\em 4-covered graph\/}\index{4-covered graph} is a graph $G$ containing a set $W$ of four vertices so that $G-W$ has no edges.
\end{enumerate}
\end{definition}\printFullDetails{

\wording{Maharry and Robertson \cite{mr} prove Robertson's Theorem that an internally-4-connected graph with no subdivision of $V_8$ is one of the following:}
\begin{enumerate}\item \wording{a planar graph;}
\item a non-planar graph with at most seven vertices;
\item $C_3\,\Box\,C_3$;
\item \wording{a bicycle wheel; and}
\item a 4-covered graph.
\end{enumerate}

Suppose $G$ is a 3-connected graph that does not contain a subdivision of $V_8$ and $G$ reduces by planar 3-reductions to the \p4c\ graph $G^{\ei4c}$.    It follows that $G^{\ei4c}$ has no $V_8$.  \major{Eliminating hugs as described in Theorem \ref{th:hugElimination} produces an internally-4-connected graph $G^{\eifc}$.   Deleting hugs does not affect the planarity of the graph; since $G^{\ei4c}$ is not planar, so is $G^{\eifc}$. By Robertson's Theorem, one of the following happens:
\begin{enumerate}\item $G^{\eifc}$ is not planar and has at most seven vertices;
\item $G^{\eifc}$ is $C_3\,\Box\,C_3$;
\item $G^{\eifc}$ is a bicycle wheel; and
\item $G^{\eifc}$ is a 4-covered graph.
\end{enumerate}}

\major{Our ambition in the remainder of this section is to explain how to determine all the peripherally-4-connected graphs $G^{\ei4c}$ that can be the outcome of a sequence of planar 3-reductions starting from a 3-connected, 2-crossing-critical graph $G$ that has no subdivision of $V_8$.  Any peripherally-4-connected graph with no subdivision of $V_8$ that either has crossing number exactly 1 or is itself 2-crossing-critical needs to be tested.  Those with crossing number 1 might extend to a 2-crossing-critical example by duplication of edges and/or replacing vertices of degree 3 by one of the basic $(T,U)$-configurations, as explained in the preceding section.}

\minor{The first two items arising from Robertson's Theorem are easily dealt with.  A computer program can easily find all internally-4-connected graphs with at most 7 vertices and determine which ones either have crossing number 1 or are 2-crossing-critical. 
The graph  $C_3\,\Box\,C_3$ is itself \2cc, so this is one of the 3-connected, \2cc\ graphs that do not contain a subdivision of $V_8$. }

}\begin{definition} \major{Let $G^{\ei4c}$ be a peripherally-4-connected graph and let $G^{\eifc}$\index{$G^{\eifc}$} be the internally 4-connected graph obtained from $G^{\ei4c}$ by simplifying (that is, leaving only one edge in each parallel class) and eliminating hugs.  Then $G^{\ei4c}$ is a {\em peripherally-4-connected extension of $G^{\eifc}$\/}\index{\p4c!extension}.}\end{definition}\printFullDetails{

\major{We conclude this section by showing how to which bicycle wheels and 4-covered graphs $G^{\eifc}$ can have such a \2cc\ $G^{\ei4c}$ as an extension.  In particular, $G^{\eifc}$ must either have crossing number 1 or itself be 2-crossing-critical; in the latter case $G^{\ei4c}=G^{\eifc}$.}

}\bigskip\noindent {\bf CASE 1:} \wording{{\em the bicycle wheels.}}\printFullDetails{

\bigskip\major{Let $x$ and $y$ be the adjacent vertices making the axle of the bicycle wheel $G^{\eifc}$, and let $C$ be the cycle that is the rim.   Our goal is to provide sufficient limitations on $C$ to show that the computation is feasible.
Here is our first limitation, which can very likely be improved.
}

}\begin{lemma}\label{lm:no6consecX} \major {Suppose $G\in\m2$ reduces by planar 3-reductions to the graph $G^{\ei4c}$ that is a peripherally-4-connected extension of $G^{\eifc}$.  If $G^{\eifc}$ is a bicycle wheel with axle $xy$ and rim $C$, then $x$ is not adjacent in $G^{\eifc}$ to six consecutive vertices on $C$, none of which is adjacent to $y$.}\end{lemma}\printFullDetails{

\begin{cproof}\major{Suppose by way of contradiction that $x_1,x_2,x_3,x_4,x_5,x_6$ are six consecutive (in this order) vertices of $C$ adjacent to $x$ but not $y$.  Lemma \ref{lm:triangle} implies no two consecutive ones of these vertices have only three neighbours in $G^{\ei4c}$.   By symmetry, we may assume $x_3$ has a neighbour $u$ that is not adjacent to $x_3$ in $G^{\eifc}$.   
}

\major{Because $G^{\ei4c}$ is a peripherally-4-connected extension of $G^{\eifc}$, there are vertices $w$ and $z$ so that $x_3$, $u$, and $w$ are the neighbours (in both graphs) of $z$ and no other vertex has just these three neighbours.  Since $y$ is not adjacent to $x_3$ and $x$ has more than 3 neighbours, $z\in C$.  If follows that $x_3$ and $u$ are the $C$-neighbours of $z$ and $w$ is the neighbour of $z$ that is in $\{x,y\}$.  In particular, $z$, being a neighbour of $x_3$ is either $x_2$ or $x_4$, so $w=x$.  In either case, three consecutive vertices from $x_1,x_2,\dots,x_6$ are such that the outer two are adjacent by a chord in $G^{\ei4c}$; if necessary, we relabel so these are $x_1,x_2,x_3$.  In particular, $x_2$ has just three neighbours in $G^{\ei4c}$.}

\major{Let $D$ be a 1-drawing of $G^{\ei4c}-xx_2$ and let $K$ be the subgraph of $G^{\ei4c}-xx_2$ induced by $x$, $x_1$, $x_2$, and $x_3$.   }

\begin{claim} \major{$K$ is clean in $D$.}\end{claim}

\begin{proof} \major{In $G^{\ei4c}-xx_2$, $x_2$ has only two neighbours, so the edge $x_1x_3$ and the path $(x_1,x_2,x_3)$ make a pair of parallel edges.  Therefore, we may assume neither of these is crossed in $D$.}

\major{It suffices to prove that $xx_1$ is not crossed in $D$, as the proof for $xx_3$ is symmetric.  Suppose by way of contradiction that $xx_1$ is crossed in $D$ and consider the planar embedding of $G^{\ei4c}-\{xx_1,xx_2\}$ induced by $D$.   Since $G^{\eifc}-\{xx_1,xx_2\}$ is a subgraph, it is also planar, embedded in the plane by $D$.}

\major{Since $x_3$ has only three neighbours in $G^{\eifc}-\{xx_1,xx_2\}$, we can add the edge $xx_2$ alongside the path $(x,x_3,x_2)$ to obtain a planar embedding of $G^{\eifc}-xx_1$.   Then we may add the edge $xx_1$ alongside the path $(x,x_2,x_1)$ to get a planar embedding of $G^{\eifc}$. However, this contradicts the fact that $G^{\eifc}$ is not planar.}\end{proof}

\major{Now let $K$ be the subgraph of $G^{\ei4c}-xx_2$ induced by $x$, $x_1$, $x_2$, and $x_3$.  Because $x_1$, $x_2$, and $x_3$ are consecutive along $C$, there is a unique $K$-bridge $B$ in $G^{\ei4c}-xx_2$. The claim shows $K$ is clean in $D$, so $D[B]$ is contained in one face $F$ of $D[K]$.}

\major{Adjusting which of $D[x_1x_3]$ and $D[(x_1,x_2,x_3)]$ is which, if necessary, we may arrange $D$ so that both $x$ and $x_2$ are incident with a face of $D[K]$ that is not $F$.  This permits us to add $xx_2$ to $D$ without additional crossings, to obtain a 1-drawing of $G$.  This final contradiction yields the result.
}
\end{cproof}

\major{Along the same lines, we have the following limitation.}

}\begin{lemma}\label{lm:noFourXy} \major {Suppose $G\in\m2$ reduces by planar 3-reductions to the graph $G^{\ei4c}$ that is a peripherally-4-connected extension of $G^{\eifc}$.  If $G^{\eifc}$ is a bicycle wheel with axle $xy$ and rim $C$, and there are four distinct vertices of $C$ adjacent to both $x$ and $y$, then these are the only six vertices of $G^{\eifc}$.}
\end{lemma}\printFullDetails{

\begin{cproof} \major{Suppose to the contrary that $u_1$, $u_2$, $u_3$,  and $u_4$ are distinct vertices of $C$ adjacent to both $x$ and $y$ in $G^{\ei4c}$ and there is another vertex $u_5$.  We may choose the labelling of $x$ and $y$ so that $xu_5\in G^{\eifc}$.  Let $D$ be a 1-drawing of $G^{\ei4c}-xu_5$.    
}

\major{Let $K$ be the subgraph of $G^{\ei4c}-xu_5$ consisting of $C$ and all edges between $x$ and vertices of $C$.  (We do not include any chords of $C$ that might exist in $G^{\ei4c}$.)  If $x$ and $y$ are both in the same face of $D[C]$, then $y$ is in some face $F$ of $D[C]$ and at least two of $u_1$, $u_2$, $u_3$, and $u_4$ are not incident with $F$.  This implies the contradiction that $D$ has at least two crossings.}

\major{We conclude that $y$ is not in the same face of $D[C]$ with $x$.  It follows that $xy$ crosses $C$ in $D$ and this is the only crossing.   We claim we can add the edge $xu_5$ to $D$ to obtain a 1-drawing of $G^{\ei4c}$.}

\major{Let $F$ be the unique face of $D[K]$ incident with both $u_5$ and $u$ and let $C'$ be the cycle bounding $F$.  If we cannot add $xu_5$ in $F$, then there is an edge $e$ of $G^{\ei4c}$ that has an end in each of the two components of $C'-\{x,u_5\}$.  Since $C'-x\subseteq C$, it follows that both ends $w_1$ and $w_2$ of $e$ are in $C$.}

\major{Since $e$ is not an edge of $G^{\eifc}$, there are vertices $w_3$ and $z$ of $G^{\ei4c}$ so that $z$ has just the neighbours $w_1$, $w_2$, and $w_3$.    Since both $x$ and $y$ have at least four neighbours, $z\notin\{x,y\}$.  Since one of $x$ and $y$ is a neighbour of $z$, $w_3\in\{x,y\}$.  Finally, $z$ has at least two neighbours in $C$, so these are $w_1$ and $w_2$.  We conclude that $z=u_5$.}

\major{We note that $xy$ cannot cross the 3-cycle $u_5w_1w_2$ in $D$.  Therefore, we can move $w_1w_2$ to the face of $D[C]$ that contains $y$; in this new 1-drawing of $G^{\ei4c}-xu_5$, $x$ and $u_5$ are incident with the same face, giving the contradiction that $G^{\ei4c}$ has a 1-drawing.   
}
\end{cproof}

\major{The final limitation is the following.}

}\begin{lemma}\label{lm:noFourAlternations}   \major {Suppose $G\in\m2$ reduces by planar 3-reductions to the graph $G^{\ei4c}$ that is a peripherally-4-connected extension of $G^{\eifc}$.  Suppose $G^{\eifc}$ is a bicycle wheel with axle $xy$ and rim $C$, and there are six distinct vertices $x_1$, $y_1$, $x_2$, $y_2$, $x_3$, $y_3$ in this cyclic order on $C$, so that, for $i=1,2,3,$, $x_i$ is adjacent to $x$ and $y_i$ is adjacent to $y$.   Then these are the only six vertices of $C$.} \end{lemma}\printFullDetails{

\major{We remark that we allow for the possibility that some (or all) of the $x_i$ are also adjacent to $y$ and, likewise, some of the $y_i$ can be adjacent to $x$.}

\bigskip

\begin{cproof} \major{ Suppose to the contrary that there is another vertex $u$ in $C$.  If possible, choose the $x_i$, $y_i$ and $u$ so that $u$ is adjacent to only one of $x$ and $y$.  We may assume that $u$ occurs between $x_1$ and $y_1$ in the cyclic order on $C$.  By the choice of the $x_i$, $y_i$, and $u$, if $u$ is adjacent to both $x$ and $y$, then so are $x_1$ and $y_1$ and all vertices between them on $C$.}

\major{Let $D$ be a 1-drawing of $G^{\ei4c}-xu$.  Let $K$ be the subgraph of $G^{\ei4c}-xu$ consisting of $C$ and all edges between $x$ and vertices of $C$.  (We do not include any chords of $C$ that might exist in $G^{\ei4c}$.)  If $x$ and $y$ are on the same side of $D[C]$, then at most one of the $y_i$ is incident with the face of $D[K]$ containing $y$, showing $D$ has at least two crossings, a contradiction.  Therefore, the crossing of $D$ is of $xy$ with an edge of $C$.
}

\major{There is a face of $D[K]$ incident with both $x$ and $u$; let $C'$ be its bounding cycle.  If we cannot add $xu$ to $D$, it is because there is an edge $e$ of $G^{\ei4c}-xu$ with an end in each of the components of $C'-\{x,u\}$.  Since $C'-x\subseteq C$, it follows that the ends $w_1$ and $w_2$ of $e$ are both in $C$.  Because $G^{\ei4c}$ is a peripherally-4-connected extension of a bicycle wheel, there are vertices $z$ and $w_3$ so that $z$ has only the neighbours $w_1$, $w_2$, and $w_3$.
}

\major{Both $x$ and $y$ have at least four neighbours in $G^{\eifc}$, so $z\notin \{x,y\}$; thus, $z\in C$.  Since $z$ has two neighbours in $C$ and at least one in $\{x,y\}$, it follows that $w_3\in \{x,y\}$, while $w_1$ and $w_2$ are the two $C$-neighbours of $z$.  Therefore, $z=u$.  As $u$ is adjacent to $x$, we conclude that $u$ is not also adjacent to $y$.  But now we can move the edge $w_1w_3$ to the other side of $C$ so that the resulting 1-drawing of $G^{\ei4c}-xu$ extends to a 1-drawing of $G$, a contradiction.
}
\end{cproof}

\major{Lemmas \ref{lm:no6consecX}, \ref{lm:noFourXy}, and \ref{lm:noFourAlternations} effectively limit the possibilities for $G^{\eifc}$.  Each of these must be checked for either having crossing number 1 or being 2-crossing-critical.  Those with crossing number 1 must have their peripherally-4-connected extensions tested for 2-criticality.  No matter what improvement is made to Lemma \ref{lm:no6consecX}, this will require computer work to complete.}  

}\bigskip\noindent {\bf CASE 2:} \wording{{\em the 4-covered graphs.}}\printFullDetails{

\bigskip

We begin our analysis by describing three particular \i4c\ \2cc\ graphs that are 4-covered.

\begin{definition}  \begin{enumerate}\item The 3-cube $Q_3$\index{$Q_3$} is the 3-regular, 3-connected, planar, bipartite graph with 8 vertices.
\item The graph $Q_3^v$\index{$Q_3^v$} is the bipartite graph obtained from $Q_3$ by adding one new vertex joined to all four vertices on one side of the bipartition of $Q_3$.
\item The graph $Q_3^{2e}$\index{$Q_3^{2e}$} is the bipartite graph obtained from $Q_3$ by adding two of the four missing (bipartite-preserving) edges.
\item The graph $Q_3^t$\index{$Q_3^t$} is the graph obtained from $Q_3$ by adding a 3-cycle $abc$ on one side of the bipartition of $Q_3$ together with one edge joining the fourth vertex $d$ of the same part to the non-adjacent vertex in the other part of the bipartition.
\end{enumerate}
\end{definition}

}\begin{lemma}\label{lm:3cube}  The graphs $Q_3^v$, $Q_3^{2e}$, and $Q_3^t$ are all \2cc. \end{lemma}\printFullDetails{

\begin{cproof}  We start with the following observation.

\begin{claim}\label{cl:1drawingQ3} If $D$ is a 1-drawing of $Q_3$, then $D$ is the unique planar embedding of $Q_3$. \end{claim}

\begin{proof}  If $e$ and $f$ are two non-adjacent edges of $Q_3$, then it is easy to see that they are in disjoint cycles.  Therefore, no two edges of $Q_3$ cross in $D$. \end{proof}

We use Claim \ref{cl:1drawingQ3} to show that $\crn(Q_3^v)\ge 2$, $\crn(Q_e^{2e})\ge 2$, and $\crn(Q_3^t)\ge 2$.  

Adding the one vertex to the planar embedding of the 3-cube yields 2 crossings, since each face of the 3-cube is incident with only 2 of the four vertices joined to the new vertex.  This shows $\crn(Q_3^{v})\ge 2$. 

For $Q_3^{2e}$, each of the two new edges joins vertices not on the same face of $Q_3$ and so each has a crossing with $Q_3$.  Thus, $\crn(Q_3^{2e})\ge 2$.

For $Q_3^t$, the new edge $e$ incident with $d$ must cross $Q_3$ in any drawing $D$ of $Q_3^t$ for which $D[Q_3^t]$ has no crossings.  If the 3-cycle $D[abc]$ also has a crossing with $Q_3$, then $D$ has two crossings.  Otherwise, $D[abc]$ separates the two ends of $D[e]$, so $D[e]$ crosses $D[abc]$.  Thus, $\crn(Q_3^t)\ge 2$.

We now consider 2-criticality in each case.

For $Q_3^v$, deleting any edge of the 3-cube makes a face incident with 3 of the four vertices and so yields a 1-drawing.  Likewise deleting one of the edges incident with the new vertex yields a 1-drawing.

For $Q_3^{2e}$, obviously deleting either of the edges not in $Q_3$ yields a 1-drawing.  On the other hand, if $e$ is an edge of $Q_3$ incident with at most one of the vertices of $Q_3^{2e}$ of degree 4, then deleting $e$ makes one of the newly adjacent pairs now lie on the same face, yielding the required 1-drawing.  If $e$ is one the remaining two edges of $Q_3$, there is a 1-drawing of $Q_3-e$ with one crossing that extends to a 1-drawing of $Q_3^{2e}-e$.

For $Q_3^t$, criticality of all the edges not incident with $d$ is obvious, as it is the new edge  $e$ incident with $d$.  The remaining three edges are symmetric.  Deleting any one of these results in a subgraph that has crossing number 1 (we may move the other end of $e$ to the other side of $abc$ to get a 1-drawing). \end{cproof}

}\begin{lemma}\label{lm:4covered}  Suppose $G\in \m2$ reduces by planar 3-reductions to \wording{a peripherally}-4-connected $G^{\ei4c}$ with at least 8 vertices \wording{that is an extension of the internally-4-connected 4-covered graph $G^{\eifc}$.}  Then either $G$ is one of the graphs $Q_3^{v}$, $Q_3^{2e}$, \minor{or $G^{\ei4c}$ has exactly 8 vertices.}\end{lemma}\printFullDetails{

\ignore{\begin{figure}
\begin{center}
\includegraphics[scale=.25]{4covered}\qquad\includegraphics[scale=.25]{4covered1a}\qquad\includegraphics[scale=.25]{4covered1b}\qquad\includegraphics[scale=.25]{4covered2a}\qquad \includegraphics[scale=.25]{4covered2b}
\end{center}
\end{figure}

\begin{figure}
\begin{center}\includegraphics[scale=.25]{4covered2c}\qquad\includegraphics[scale=.25]{4covered2d}\qquad\includegraphics[scale=.25]{4covered3b}\qquad\includegraphics[scale=.25]{4covered3b}\qquad\includegraphics[scale=.25]{4covered3c}
\end{center}
\end{figure}

\begin{figure}
\begin{center}
\includegraphics[scale=.25]{4covered3d}\qquad\includegraphics[scale=.25]{4covered4b}\qquad\includegraphics[scale=.25]{4covered4c}\qquad\includegraphics[scale=.25]{4covered4d}\qquad\includegraphics[scale=.25]{4covered5b}
\caption{The 4-covered graphs}\label{fg:rest4covered}
\end{center}
\end{figure}
}

\begin{cproof}  Let $a,b,c,d$ be the four vertices so \wording{that $G^{\eifc}-\{a,b,c,d\}$ is} an independent set $I$.  For each $x\in \{a,b,c,d\}$, let $X$ be the set of vertices in $I$ adjacent to everything in $\{a,b,c,d\}\setminus \{x\}$, and let $R$ be the remaining vertices in $I$; a vertex in $R$ is joined to all of $\{a,b,c,d\}$.  

Note that a vertex in $R$ has degree 4 \wording{in $G^{\eifc}$, so} it is also a vertex of $G$; it cannot be the outcome of any 3-reductions.   If $|R|\ge 3$, then $G$ contains \wordingrem{(text removed)}$K_{3,4}$ and so $G=K_{3,4}$, a contradiction.  Thus, $|R|\le 2$.

If, for some $x\in \{a,b,c,d\}$, $|X|\ge 2$, then $\{a,b,c,d\}\setminus\{x\}$ is a 3-cut \wording{in $G^{\eifc}$ that} separates any two vertices $v,w$ in $X$ from all the other vertices in $I\setminus \{v,w\}$, of which there are at least two.  This contradicts the fact \wording{that $G^{\eifc}$ is} internally 4-connected.  Thus, $|X|\le 1$.

This implies that $G^{\ei4c}$ has at most 10 vertices, but we can proceed a little further. 

If $R=\varnothing$, \wording{then $G^{\eifc}$ is} planar (adding the $K_4$ on $\{a,b,c,d\}$ does not affect planarity), which is a contradiction.  Thus, $|R|>0$.

If, for each $x\in\{a,b,c,d\}$, $|X|=1$, then the bipartite subgraph \wording{of $G^{\eifc}$ consisting} of $\{a,b,c,d\}$ and the four vertices in $A\cup B\cup C\cup D$ is the 3-dimensional cube $Q_3$. 
Adding one of the vertices in $R$ to $Q_3$ produces $Q_3^v$.  That is, if all of $A$, $B$, $C$, and $D$ are not empty, $|R|= 1$ and $G=Q_3^v$. 

Thus, we may assume $R\ne \varnothing$ and $D=\varnothing$.

If $|R|=2$, then \wording{for $G^{\eifc}$ to} have at least 8 vertices, at least two of  $A$, $B$,  and $C$ are not empty.  Thus, \wording{$Q_3^{2e}\subseteq G^{\ei4c}$, so $G^{\ei4c}=Q_3^{2e}$}.

\minor{In the final situation, we have $|R|=1$ and, because $G^{\ei4c}$ has at least 8 vertices, all of $A$, $B$, and $C$ are not empty.  In particular, $G^{\ei4c}$ has exactly 8 vertices, as required.}  
\end{cproof}

\wording{A computer search can find all the peripherally-4-connected graphs having 8 vertices.  These will include all the examples that are peripherally-4-connected extensions of internally-4-connected, 4-covered graphs having 8 vertices.  This completes our analysis of 3-connected, 2-crossing-critical graphs with no subdivision of $V_8$.}
}

\chapter{Finiteness of 3-connected 2-crossing-critical graphs with no $V_{2n}$}\label{sc:noV2n}\printFullDetails{

This section is devoted to showing that, for each $n\ge 3$, there are only finitely many 3-connected \2cc\ graphs that do not contain a subdivision of $V_{2n}$.   In particular, Theorem \ref{th:allbounded} asserts that if $G$ has a subdivision of $V_{2n}$ but no subdivision of $V_{2n+2}$, then $|V(G)|= O(n^3)$.

   The finiteness has been proved previously by completely different methods in \cite{DOTV}.   In our particular context, this shows that there are only finitely many 3-connected \2cc\ graphs that have a subdivision of $V_8$ but do not have a subdivision of $V_{10}$; these are the only ones missing from a complete determination of the \2cc\ graphs.

The first subsection shows that, if $G$ is a 3-connected \2cc\ graph that does not contain a subdivision of $V_{2n+2}$, then, for any $\hvng$, each $H$-bridge in $G$ has at most 88 vertices.  The second subsection shows that, for a particular subdivision $H$ of $V_{2n}$, there are only $O(n^3)$ $H$-bridges having a vertex that is not an $H$-node.  These easily combine to give the $O(n^3)$ bound of Theorem \ref{th:allbounded}.

}

\section{$V_{2n}$-bridges  are small}\label{sec:bridgesSmall}\printFullDetails{

The main result of this subsection is to show that if \wording{$G\in\m2$ and $\hvng$}, then any $H$-bridge $B$ is a tree with a bounded number of leaves, so that $|V(B)|\le 88$.  In the next subsection, we show that there are only $O(n^3)$ non-trivial $H$-bridges.

The next lemma will have as a corollary the first main result of this subsection.

}\begin{lemma}
\label{lm:bridgesAreSmall}
Let $G\in\m2$, $\hvng$, $n\ge 3$, and $B$ an $H$-bridge. 
Then $|\att(B)|\le 11n+12$.
\end{lemma}\printFullDetails{
\begin{cproof}
Let $e$ be an edge of $B$ incident with $x\in \att(B)$ and $y\in \Nuc(B)$. Then $D_e[B-e]$ is contained 
in a face $F$ of $D_e[H]$.  Because we know the 1-drawings of $V_{2n}$, we know that each face of $D_e[H]$ is incident with at most $n+1$ $H$-branches.  Moreover, $B-e$ is an $H$-bridge in $G-e$ and $\att_{G-e}(B-e)$ is either $\att_G(B)$ or $\att_G(B)\setminus\{x\}$.

 If $B$ has at least $11(n+1)+2$ attachments, then 
some $H$-branch $b$ contains at least 12 attachments of $B-e$. Let $a_1\ldots a_{12}$ be any 12 distinct attachments of $B-e$ occurring in this order in $b$. 
Let $T\subseteq B$ be a minimal tree that meets 
$\att(B)$ at $a_1$, $a_3$, $a_4$, $a_6$, $a_7$, $a_9$, $a_{10}$, 
and $a_{12}$, so that these $a_i$ are the leaves of $T$, and let $Q
=\cc{a_1,b,a_{12}}$. Set $Y=T\cup Q$. 

For $i=1,4,7,10$, 
there is a unique cycle $C_i\subseteq Y$ that meets $b$ precisely 
in  $a_iQa_{i+2}$. 
Let $I\subseteq\{1,4,7,10\}$ be the subset such that, for $i\in I$, 
$x\notin C_i$; clearly $|I|\ge 3$. 

For each $i\in I$, let $M_i$ be the $C_i$-bridge in $G-e$ with 
$H\subseteq M_i\cup C_i$. As $x\notin C_i$, $x\in\Nuc(M_i)$. 
Let $B_i$ be the $C_i$-bridge in $G-e$ containing $y$ or $B_i=y$ 
if $y\in C_i$. Let $P_i$ be a minimal subpath of $C_i$ containing 
$B_i\cap C_i$, so that $a_iQa_{i+2}\not\subseteq P_i$.

\begin{claim}
\label{cl:BiBj}
Let $i,j,k\in I$ be distinct. If $y\notin M_i\cup M_j$, then:\begin{itemize}\item $B_i=B_j$;\item $P_i=P_j\subseteq C_i\cap C_j$; and \item  $y\in M_k$.\end{itemize}
\end{claim}

\begin{proof}  If $u$ and $v$ are vertices in $C_i\cap C_j$, then \wording{$u$ and $v$ are not in $b$ and} there is a unique $uv$-path $P$ in $T$.  \wording{We note that} $P\subseteq C_i\cap C_j$.  Thus, $C_i\cap C_j$ is a path.

If there were a $yC_i$-path disjoint from $C_j$, then $y\in M_i$, a contradiction.  Therefore, every $yC_i$-path meets $C_j$ and, symmetrically, every $yC_j$-path meets $C_i$.  Thus, every $y(C_i\cup C_j)$-path has one end in $C_i\cap C_j$.   It follows that if $y\in C_i\cup C_j$, then $y\in C_i\cap C_j$, so in this case $B_i=B_j=\iso y$.

In the case $y\notin C_i\cup C_j$, let $B$ be the $(C_i\cup C_j)$-bridge containing $y$.  The preceding paragraphs show that $\att(B)\subseteq C_i\cap C_j$, so that in fact $B$ is also both a $C_i$- and a $C_j$-bridge.  In particular, $B_i=B_j=B$.

For the last part, we assume $y\notin M_k$ and note that  $B=B_i=B_j=B_k$ and $C_i\cap C_j\cap C_k$ is a non-null path $P'$.  If $P'$ has length at least one, then $P'\cup C_i\cup C_j\cup C_k$ contains a subdivision of $K_{2,3}$ and yet has all three of the vertices on one side incident with a common face, which is impossible.  Therefore, $P'$ consists of a single vertex $z$.  

If $z$ is not $y$, $B$ has only $z$ as an attachment in $G-e$.  It follows that either $z$ or $\{z,x\}$ is a cut-set of $G$, contradicting the fact that $G$ is 3-connected.   Thus, $z=y$, and so, for some $t\in \{i,j,k\}$, $y$ is an attachment of $M_t$; in particular, $y\in M_t$, a contradiction.
\end{proof}

By Claim \ref{cl:BiBj}, there is an $i\in I$ such that $y\in M_i$. For such an $i$, set $C=C_i$ and note that $x\in M_i-\att(M_i)$, so that $M=M_i+e$ is a $C$-bridge in $G$.  Furthermore, $\att_G(M)=\att_{G-e}(M-e)$.

Notice that $D_e[C]$ is clean, since the crossing of $D_e$ is between disjoint $H$-branches.  Thus, $C$ has BOD in $G-e$. Also, any $C$-bridge $B'\neq M$ 
has $C\cup B'$ planar. As $\att_G(M)=\att_{G-e}(M-e)$, $C$ has BOD in $G$.

Recall that the $H$-bridge $B$ has $a_i$, $a_{i+1}$, and $a_{i+2}$ as attachments.  For any vertex $u$ of $B$ not in $b$, there is an $H$-avoiding $ua_{i+2}$-path, whose edge $e'$ incident with $u$ is in some $C$-bridge $B'$.  Since $x$ and $y$ are on the same side of $D_e[C]$, $M$ is contained on that side of $D_e[C]$ and $e'$ is on the other side.  Therefore, $B'\ne M$.  

In $D_{e'}$, the crossing is in $H$ and $D_{e'}[C]$ is clean.  That is, $D_{e'}[C\cup M]$ is a 1-drawing with $C$ clean.  Corollary \ref{co:TutteTwo} shows $\crn(G)\le 1$, the final contradiction.  \end{cproof}

The following corollary is the first main result of this section.  

}\begin{corollary}
\label{cr:bridgesAreSmall}
Let $G\in\mc{3}$, $\hvng$, $n\ge 3$, $B$ an $H$-bridge. Then $|\att(B)|\le 45$.
\end{corollary}\printFullDetails{

\begin{cproof}  If $n=3$, then the result is an immediate consequence of Lemma \ref{lm:bridgesAreSmall}.  Thus, we may assume $n\ge 4$.   If $B$ has attachments in the interiors of non-consecutive spokes, then $G$ is the Petersen graph and the result clearly holds.  

Otherwise, $B$ has attachments in at most two consecutive spokes.  Thus, there is a subdivision $H'$ of $V_6$ contained in $H$ that contains all the attachments of $B$.  Applying Lemma \ref{lm:bridgesAreSmall} to $H'$, we again see that $|\att(B)|\le 45$. \end{cproof}

We now turn to the other half of the argument that bounds the number of vertices in an $H$-bridge, namely, that the bridge is a tree.  We need a new notion.

}\begin{definition} Let $T^*$ be a graph consisting of subdivision of a $K_{2,3}$ together with three pendant edges, one incident with each of the three degree 2 vertices in the $K_{2,3}$.   A {\em tripod\/}\index{tripod} is any graph $T$ obtained from $T^*$ by contracting any subset of the pendant edges; if all three pendant edges are contracted, then an edge is added between the two copies of $K_{1,3}$, but not having a vertex of contraction as an end --- this may be done in any of three essentially different ways.  The {\em attachments\/}\index{attachments!of a tripod}\index{tripod!attachments} of the tripod are the degree 1 and 2 vertices in $T$.\end{definition}\printFullDetails{

We are now ready for the second half of the main result of this section.

}\begin{lemma}  
\label{lm:bridgeIsTree}
Suppose $G\in \mc{3}$, $\hvng$, $n\ge 3$, $G$ has no \wording{subdivision of} $V_{2(n+1)}$, and  
$B$ is an $H$-bridge.  
Then either $B$ is a tree or $B$ has a tripod, $n = 3$ and $|V(G)|\le 10$.
\end{lemma}\printFullDetails{

\begin{cproof}  By way of contradiction, suppose $B$ has a cycle $C$.
 If $ \att(B)\cap C\ne\emptyset$, let $e$ 
be an edge of $C$ incident with $u\in \att(B)$.  If $C\cap \att(B)=\emptyset$, 
then let $e$ be any edge of $C$.  The choice of $e$ shows that $B-e$ is an $H$-bridge in $G-e$ and that $\att_{G-e}(B-e)= \att_G(B)$.  
Since $D_e[H]$ contains the crossing in $D_e[G-e]$,  $D_e[B-e]$ 
is \wording{contained in a face $F$} of $D_e[H]$.  

Let $C'= \partial F^\times$, so $C'$  is a cycle in $G'=(G-e)^\times$.
Since $G'$ is planar, $C'$ has BOD in $G'$ and $C'\cup B'$ is planar for each $C'$-bridge
$B'$ in $G'$. If $C'\cup B$ were planar, then $G'+e$ would be planar, in which case $\crn(G)\le1$, a contradiction.  Therefore, $C'\cup B$ is not planar.

We now introduce a convenient notion.

}\begin{definition} Let $G$ be a graph.  The graph $G^t$\index{$G^t$} is the graph whose vertices are the $G$-nodes and whose edges are the $G$-branches. 
\end{definition}\printFullDetails{

\begin{claim}\label{cl:3conn}
 $(C'\cup B)^t$ is $3$-connected.
\end{claim}

\begin{proof}
Let $L=(C'\cup B)^t$.
If $|V(\Nuc(B))|=1$, then $L$ is a wheel and the claim follows.
So assume $|V(\Nuc(B))|\ge 2$. We show that any two vertices of $L$ are 
joined by three internally disjoint paths. For $u,w\in \Nuc(B)$, this is 
true in $G$, so let $P_1$, $P_2$, $P_3$ be such paths in $G$. If at least
one $P_i$ is contained in $B-C'$, then we can easily modify the others
to use $C'$ rather than $G-B$ to get three paths in $L$. If all 
three intersect $C_e$, then $B\cap (P_1\cup P_2\cup P_3)$ is two claws
$Y_u$ and $Y_w$. There is a $Y_uY_w$-path in $\Nuc(B)$, which returns us to
the previous case.

\wording{If $u\in \Nuc(B)$ and} $w\in C'$, then $w$ is an attachment of $B$.  Let $Y$ be a claw in $B$ with centre $u$ and talons on $C'$.  Using a $C'$-avoiding $wY$-path in $B$, if necessary, we can assume $w$ is a talon of $Y$.  It is then easy to use $C'$ to extend the other two paths in $Y$ to $w$.

Finally, if $u,w\in C'$, then both $u$ and $w$ are attachments of $B$, so there is a $C'$-avoiding path joining them.  This path and the two $uw$-paths in $C'$ yield the required three paths.
\end{proof}

}\begin{definition}  Let $C$ be a cycle in a graph $G$ and let $P_1$ and $P_2$ be disjoint $C$-avoiding paths in $G$.  Then {\em $P_1$ and $P_2$ are $C$-skew paths\/}\index{$C$-skew paths}\index{skew paths}\index{bridge!skew paths} if the two $C$-bridges in $C\cup P_1\cup P_2$ overlap.\end{definition}\printFullDetails{

As $C'\cup B$ has no planar embedding, \cite{mohar} implies $B$ has either  a tripod whose attachments are in $C'$ or two $C'$-skew paths.

\begin{claim}
\label{cl:tripod}
If $B$ has a tripod $T$, then $n=3$, $G=H\cup T$  and $|V(G)|\le 14$.
\end{claim}

\begin{proof}  Let $S$ be the attachments of $T$.  As $H\cup T$ is 2-connected and, relative to the cut $S$, both $H'{}^+$ (taking $H'$ to be any $V_6$ containing $S$)  and $T^+$ are non-planar.  By Theorem \ref{th:3CutBothNonPlanar}, $\crn(H'\cup T)\ge 2$.  Thus, $G=H'\cup T$, so $n=3$ and, again by Theorem \ref{th:3CutBothNonPlanar}, $|V(G)|\le 10$.  \end{proof}




Thus, we can assume $B$ has no tripod.
Then $B$ has $C'$-skew paths, say $P_1$ and $P_2$. Since these do not
exist in $B-e$, $e$ is in one of them. If $C\cap \att(B)=
\emptyset$, choose $e'$ any edge of $C$ not in $P_1\cup P_2$. If
$C\cap \att(B)\neq\emptyset$, choose $e'$ to be the other edge of $C$
incident with the same attachment as $e$. 


 Repeat with $G-e'$.
This yields $C''$ so that $B$ has 
$C''$-skew paths $u'_1 u_2'$ and $w_1' w_2'$ ($e'$ incident with
$u'_1$).  Since $u_1 u_2 \cup w_1 w_2 \subseteq B-e'$, they are
not $C''$-skew. In $C'$, we have the cyclic 
order $u_1, w_1, u_2, w_2$, say. 
In $C''$ we have $u_1 u_2 w_1 w_2$. 
Likewise in $C'$ we have $u'_1 u_2' w_1' w_2'$, while in $C''$ we have
$u'_1 w_1' u_2' w_2'$. 

Let $D$ and $D'$ be
$1$-drawings of $H$ having all attachments of $B$ on faces $F, F'$,
respectively, so that the cyclic orders of $\att(B)$ are different in
$\partial{F}$ and $ \partial{F'}$.

\begin{claim}\label{cl:nisnot3}
$n\ge 4$.
\end{claim}

\begin{proof}
Let $H$ be a subdivision of $V_6$ in $G$. We remark that if $f$ and $f'$ are any disjoint $H$-branches having internal vertices that are ends of an $H$-avoiding path $P$ in $G$, then $H\cup P$ is a subdivision of $V_8$ in $G$.  

We consider first the case that  $\att(B)$ is not contained in any $4$-cycle of
$H$. Because we know the $1$-drawings of $H$ and $\att(B)$ is contained in the boundary $\partial F$ of a face $F$ of such a 1-drawing,  $\partial{F}$ is
$\times v_1v_2v_3\times$. If $B$ has attachments in both $\oo{\times v_1}$ and
$\oo{v_3\times}$, then $G$ has a subdivision of $V_8$, as required.  Thus, we may assume that $\att(B)$ is contained \wording{in a 4-cycle $Q$} of $H$, which we may take to be
$\cc{v_1v_2v_3v_4v_1}$.

In at least one of
$D$ and $D'$, $Q$ is self-crossed (otherwise the cyclic orders of $\att(B)$ are the same) and $B$ is 
drawn in the face $\times v_1v_6v_3\times$. However, in this case $\att(B) \subseteq \oc{\times,v_1} \cup \co{v_3,\times}$
and at least two attachments of $B$ are in each. In this case, we again have a subdivision of $V_8$ in $G$, as required.
\end{proof}

\begin{claim}\label{cl:nospokeatt}
$B$ has no (interior) spoke attachment.
\end{claim}


\begin{proof}  From Claim \ref{cl:nisnot3}, we know that $n\ge 4$. By way of contradiction, we assume $B$ has an attachment in $\oo{s_0}$.  From the listing of the faces of 1-drawings of $V_{2n}$, the only possibilities for each of $\partial{F}$ and $ \partial{F'}$ are:
\begin{description}
\item{(1)} $\cc{v_0,r_0,v_1,s_1,v_{n+1},r_n,v_n,s_0,v_0}$;
\item{(1')}
$\cc{v_0,r_{-1},v_{-1},s_{-1},v_{n-1},r_{n-1},v_{n},s_0,v_0}$;
\item{(2)} 
$\oo{v_1, r_0, v_0, s_0, v_n, r_n, v_{n+1}, r_{n+1}, v_{n+2}}$;
\item{(2')}
$\oo{v_{-1}, r_{-1}, v_0, s_0, v_n, r_{n-1}, v_{n-1}, 
r_{n-2}, v_{n-2}}$;
\item{(3)} $\oo{v_{n-1}, r_{n-1}, v_n, s_0, v_0, r_{-1}, v_{-1}, 
r_{-2}, v_{-2}}$;
\item{(3')}
$\oc{v_{n+1},r_n,v_n,s_0,v_0,r_0,v_1,r_1,v_2}$;
\item{(4)} $\oo{v_{-1}, r_{-1}, v_0, s_0, v_n, r_n, v_{n+1}}$;
\item{(4')}
$\oo{v_{n-1}, r_{n-1}, v_n, s_0, v_0, r_0, v_{1}}$;
\item{(5)} $\cc{v_0, v_1, v_2, \ldots, v_n, s_0, v_0}$;
\item{(5')}
$\cc{v_0,s_0,v_n,v_{n+1},v_{n+2},\ldots,v_{-1},v_0}$.
\end{description}

We now consider these possibilities in pairs.  In every case, the ends of the skew paths will occur in the same cyclic order on the boundaries of the two faces, which is impossible.

\begin{description}
\item{(1,1')} $\att(B) \subseteq s_0$; same cyclic order, a contradiction.
\item{(2,2')} $\att(B) \subseteq s_0$; same cyclic order, a contradiction.
\item{(3,3')} $\att(B) \subseteq s_0$; same cyclic order, a contradiction.
\item{(4,4')} $\att(B) \subseteq s_0$; same cyclic order, a contradiction.
\item{(5,5')} $\att(B) \subseteq s_0$; same cyclic order, a contradiction.
\item{(1,2)} $\att(B) \subseteq \oc{v_1,r_0,v_0,s_0,v_n,r_n,v_n+1}$;
  same cyclic order, a contradiction.
\item{(1,2')} $\att(B) \subseteq \cc{v_0, s_0, v_n}$;
  same cyclic order, a contradiction.
\item{(1,3)} $\att(B) \subseteq s_0$;
  same cyclic order, a contradiction.
\item{(1,3')} $\att(B) \subseteq \oc{v_{n+1},r_n,v_n,s_0,v_0,r_0,v_1}$;
  same cyclic order, a contradiction.
\item{(1,4)} $\att(B) \subseteq \oc{v_1,r_0,v_0,s_0,v_n}$;
  same cyclic order, a contradiction.
\item{(1,4')} $\att(B) \subseteq \cc{v_1,r_0,v_0,s_0,v_n}$;
  same cyclic order, a contradiction.
\item{(1,5)} $\att(B) \subseteq \cc{v_1,r_0,v_0,s_0,v_n}$;
  same cyclic order, a contradiction.
\item{(1,5')} $\att(B) \subseteq \cc{v_{n+1},r_n,v_n,s_0,v_0}$;
  same cyclic order, a contradiction.

\item{(2,3)} $\att(B) \subseteq s_0$;
  same cyclic order, a contradiction.
\item{(2,3')} $\att(B) \subseteq \oc{v_{n+1},r_n,v_n,s_0,v_0,r_0,v_1}$;
  same cyclic order, a contradiction.
\item{(2,4)} $\att(B) \subseteq \oc{v_{n+1},r_n,v_n,s_0,v_0}$;
  same cyclic order, a contradiction.
\item{(2,4')} $\att(B) \subseteq \oc{v_1,r_0,v_0,s_0,r_n}$;
  same cyclic order, a contradiction.
\item{(2,5)} $\att(B) \subseteq \oc{v_1,r_0,v_0,s_0,v_n}$;
  same cyclic order, a contradiction.
\item{(2,5')} $\att(B) \subseteq \co{v_0,s_0,v_n,r_n,v_{n+1},r_{n+1},v_{n+2}}$;
  same cyclic order, a contradiction.

\item{(3,4)} $\att(B) \subseteq \oc{v_{-1},r_{-1},v_0,s_0,v_n}$;
  same cyclic order, a contradiction.
\item{(3,4')} $\att(B) \subseteq \oc{v_{n-1},r_{n-1},v_n,s_0,v_0}$;
  same cyclic order, a contradiction.
\item{(3,5)} $\att(B) \subseteq \oc{v_{n-1},r_{n-1},v_n,s_0,v_0}$;
  same cyclic order, a contradiction.
\item{(3,5')} $\att(B) \subseteq \oc{v_{-2},r_{-2},v_{-1},r_{-1},v_0,
  s_0, v_n}$;
  same cyclic order, a contradiction.

\item{(4,5)} $\att(B) \subseteq \cc{v_0,s_0,v_n}$;
  same cyclic order, a contradiction.
\item{(4,5')} $\att(B) \subseteq \oo{v_{-1},r_{-1},v_0,s_0,v_n,r_n,v_{n+1}}$;
  same cyclic order, a contradiction.
\end{description}

As any pair \wording{gives the same cyclic order}, we always get a contradiction.
\end{proof}

\begin{claim}\label{cl:bnotlocal}
$B$ is not a local $H$-bridge.
\end{claim}

\begin{proof} Suppose $B$ is local, with $\att(B)\subseteq Q_0$. From Claims \ref{cl:nisnot3} and \ref{cl:nospokeatt}, we may assume $n\ge4$ and  $B$ has no spoke attachment.   Thus, $\att(B)\subseteq r_0\cup r_n$.  
 Moreover, $B$ cannot have attachments in both $\oo{r_0}$ and $\oo{r_n}$
because $G$ has no subdivision of $V_{2(n+1)}$. On the other hand, $B$ has at least two attachments
in both $r_0$ and $r_n$ or else the cyclic order of the ends of the skew paths is always the
same. So we may assume $\att(B) \cap r_0 = \{v_0,v_1\}$. We need two
attachments in $r_n$. From the listing of faces in 1-drawings of $V_{2n}$, the only possibilities for $\partial F$ and $\partial F'$ occur when $Q_0$ is not self-crossed and so the cyclic orders of the attachments of $B$ are the same in both cases, a contradiction.\end{proof}

\begin{claim}\label{cl:nearlyLocal}
For some $i$, $\att(B) \subseteq r_i\cup r_{i+n+1}$.
\end{claim}

\begin{proof} 
By Claims \ref{cl:nisnot3}, \ref{cl:nospokeatt}, and \ref{cl:bnotlocal}, $n\ge 4$, $B$ has no spoke attachments, and $B$ is not local.

We consider in turn the possibilities for the face of $D
_e[H]$ that contains $B-e$.  We know $B$ is not local, so it can only be contained in a face whose boundary has one of the following forms:
\begin{enumerate}
\item\label{it:smallSpoke} $\cc{\times,r_i, v_i, s_i,v_{i+n},r_{i+n-1},\times}$;
\item\label{it:smallInternalNotLocal} $\cc{\times,r_i,v_i,r_{i-1},v_{i-1},s_{i-1},v_{n+i-1},r_{n+i-1},\times}$;
\item\label{it:largeExposedNoSpoke} $\cc{\times,r_i,v_{i+1},r_{i+2},\dots,v_{i+n-1},r_{i+n-1},\times}$;
\item\label{it:largeExposedSpoke} $\cc{v_i,s_i,v_{n+i},r_{n+i},v_{n+i+1},\dots,r_{i-1},v_i}$; or 
\item\label{it:largeNoExposedSpoke} $\cc{\times,r_i,v_{i+1},r_{i+1},\dots,r_{n+i-1},v_{n+i},r_{n+i},\times}$.
\end{enumerate}

As in the proof of Claim \ref{cl:nospokeatt}, the faces of $D_e[H]$ and $D_{e'}[H]$ containing $B-e$ and $B-e'$, respectively, cannot both be of one of the types (\ref{it:largeExposedNoSpoke}, \ref{it:largeExposedSpoke}, \ref{it:largeNoExposedSpoke}): the vertices of $\att(B)$ will occur in the same order in both cases.

If one of the drawings has $B-e$ or $B-e'$ in a face of type (\ref{it:smallSpoke}), then we are done:  $\att(B)\subseteq r_i\cup r_{i+n-1}$.  The remaining case is that one of the drawings has $B-e$ or $B-e'$ drawn in a face of type (\ref{it:smallInternalNotLocal}).

All other possibilities having been eliminated, we may assume (taking $i=n+1$) $$\att(B) \subseteq
\cc{\times,r_1,v_1,r_0,v_0,s_0,v_{n},r_n,\times}\,.$$ 
Because $B$ is not local, $\att(B)\cap \oo{r_{1}} \neq
\emptyset$. Because $\att(B)$ occurs in different orders in $\partial{F}$
and
$\partial{F'}$, $\att(B) \cap r_n \neq \emptyset$. By way of
contradiction, we suppose $B$ also has an attachment in 
$\co{v_0,r_0,v_1}$. The only other face which could allow these three
attachments is  $\cc{\times,r_0,v_1,r_1,\dots,v_{i-1},r_{i-1},v_n,r_n,\times}$. Notice $v_0$ is not in this second boundary, so one attachment
is in $\oo{r_0}$.  Because $V_{2(n+1)} \not\subseteq G$, no attachment
is in $\oo{r_n}$. Thus $\att(B) \cap r_n = \{ v_{n}\}$. But then, once again, the attachments of $B$ occur in the same cyclic orders in $\partial F$ and $\partial F'$, a contradiction.\end{proof}

  As we have seen above, the alternative to ``$B$ is neither a tree nor contains a tripod" is that $B$ has the $C'$-skew paths $P_1$ and $P_2$, as well as the $C''$-skew paths $P'_1$ and $P'_2$.  Claim \ref{cl:nearlyLocal} shows the four ends of $P_1$ and $P_2$ are in $r_0 \cup
r_{n+1}$. If three of them are in $r_0$, say, then they occur
in the same cyclic order in $\partial F$ and $\partial F'$, a contradiction. So two are in $r_0$ and two
in $r_{n+1}$. If $P_1$ has both ends in $r_0$, say, then the ends of
$P_1$ and $P_2$ can never interlace, a contradiction as they interlace
in $\partial{F}$. So each has one end in each of $r_0$ and
$r_{n+1}$. Likewise for $P_1', P_2'$. 

Adding at most $3$ paths in $B - \att(B)$ to $P_1\cup P_2 \cup P_1'
\cup P_2'$, we obtain $B'\subseteq B$ containing $P_1\cup P_2\cup P_1'
\cup P_2'$ so that $B'$ is an $H$-bridge in $H\cup B'$. 

Recall that $n\ge 4$ by Claim \ref{cl:nisnot3}.  All the attachments of $B'$ are in $H-\oo{s_3}$.  
Suppose $D''$ is a $1$-drawing of $(H\cup B')-\oo{s_3}$. Then $D''[B']$ is in a
face $F''$ of $D''[H-\oo{s_3}]$. Since $r_0$ and $r_{n+1}$ both have at least two
attachments of $B'$, they are both incident with $F''$. Thus one of
the pairs $P_1, P_2$ and $P_1', P_2'$ is a $\partial F''$-skew pair\wording{.  Therefore, $\crn((H\cup B')-\oo{s_3}) \ge 2$, contradicting the fact that $G$ is 2-crossing-critical.} 
\end{cproof} 

Combining Corollary \ref{cr:bridgesAreSmall} and Lemma \ref{lm:bridgeIsTree}, we immediately have the main result of this section.

}\begin{theorem}\label{th:b88} Let $G\in \m2$, $\hvng$, $n\ge 3$, and suppose $G$ has no \wording{subdivision of} $V_{2(n+1)}$.  If $B$ is an $H$-bridge, then $|V(B)|\le 88$.\end{theorem}\printFullDetails{

This completes the first main step of our effort to show that 3-connected, 2-crossing-critical graphs \wording{with no subdivision of  $V_{2n}$} have bounded size.  

}

\section{The number of bridges is bounded}\printFullDetails{\label{sec:fewBridges}

This subsection, the final leg of this work, is devoted to showing that there is a particular subdivision $H$ of $V_{2n}$ in $G$ so that there are at most $O(n^3)$ $H$-bridges in $G$ that have a vertex that is not an $H$-node.  \wording{Theorem \ref{th:b88} shows that}, for any $\hvng$, all $H$-bridges have  at most 88 vertices (when there is no subdivision of $V_{2(n+1)}$).  The combination easily implies $G$ has at most $O(n^3)$ vertices.

}\begin{definition}  Let $G$ be a graph and let $n$ be an integer, $n\ge 3$.  A subdivision $H$ of $V_{2n}$ in $G$ is {\em smooth\/}\index{smooth} if, whenever $B$ is an $H$-bridge with all its attachments in the same $H$-branch, $B$ is just an edge that is in a digon with an edge of $H$.
\end{definition}\printFullDetails{
 
We begin by showing that every $G\in\m2$ with a subdivision $V_{2n}$ has a smooth subdivision $H$ of $V_{2n}$.  For such an $H$, every vertex of $G$ either is an $H$-node or is in an $H$-bridge that does not have all its attachments in the same $H$-branch.  So it will be enough to show that the number of these $H$-bridges is $O(n^3)$.  

This analysis is completed in three parts.  We start with the result that there are not many $H$-bridges having an attachment in a particular vertex of $H$ and an attachment in the interior of some $H$-branch.  This is useful for $H$-bridges having both node and branch attachments, but is also used in the second part, which is to bound the number of $H$-bridges having attachments in the interiors of the same two $H$-branches.  The final part puts these together with those $H$-bridges having attachments in three or more $H$-nodes.

We start by showing that every $G\in \m2$ with a  subdivision of $V_{2n}$ has a smooth subdivision of $V_{2n}$.

}\begin{lemma}\label{lm:smooth}   Let $G\in \m2$ and suppose $G$ contains a subdivision of $V_{2n}$, with $n\ge 3$.  Then $G$ has a smooth subdivision of $V_{2n}$.\end{lemma}\printFullDetails{

\begin{cproof}  Choose $H$ to be a subdivision of $V_{2n}$ in $G$ that minimizes the number of edges \wording{of $G$ that are in $H$}.  We claim $H$ is smooth.  

To this end, let $B$ be an $H$-bridge with all attachments in the same $H$-branch $b$ and let $P$ be a minimal subpath of $b$ containing $\att(B)$.  Set $K=B\cup P$ and notice that $K$ is both $H$-close and 2-connected.  By Lemma \ref{lm:closeIsCycle}, $K$ is a cycle, so $B$ is just a path and, since $G$ is 3-connected, just an edge.  It remains to prove that $P$ is just an edge as well.

Let $H'=(H\cup B)-\oo P$.  Evidently $H'$ is a subdivision of $V_{2n}$ in $G$ and $|E(H')|=|E(H)|-|E(P)|+1$.  Since $|E(H)|\le |E(H')|$ by the choice of $H$, we see that $|E(P)|\le 1$, and, therefore, $P$ is just an edge, as required.  \end{cproof}

We now turn our attention to the $H$-bridges of a smooth subdivision $H$ of $V_{2n}$.   There are three main steps.

\bigskip {\bf Step 1:}  {\em Bridges attaching to a particular vertex and branch.}

\bigskip

The first step in bounding the number of $H$-bridges is to bound the number of them that can have an attachment at a particular vertex of $H$ and in the interior of a particular $H$-branch. This is the content of this step.  

}\begin{lemma}\label{lm:fromvertextobranch}
Let $G\in \mc{3}$, $\hvng$, $n\ge 3$ and suppose $H$ is smooth. For a vertex (not necesssarily a
node) $u$ of $H$ and an $H$-branch $b$, there are at
most $41$ $H$-bridges with an
attachment at $u$ \wording{and an attachment} in $\oo{b}-u$.
\end{lemma}\printFullDetails{

\def\bof{{(\partial{F})^\times}}

\begin{cproof}
Suppose there are $42$ such $H$-bridges. Let $B_0$ be one
of them, let $e\in E(B_0)$ and \wording{let $D$ be a 1-drawing of $G-e$}. If
$u\notin\oo{b}$, then at most $4$
faces of $D[H]$ are incident with $\oo{b}$, and therefore at least $11$ of these
$H$-bridges (other than $B_0$) are in the same face $F$ of $D[H]$. If $u\in\oo{b}$, then precisely
two faces of $D[H]$ are incident with $u$, so at least $21$ of these
bridges are in the same face $F$ of $D[H]$ and of these at least $11$ have an
attachment in the same component of $D[b-u] \cap (\partial{F})^{\times}$. In both
cases, let $\bbb$ be the set of $11$ bridges, contained in $F$, having
$u$ as an attachment and an attachment in the same component $b'$ of $D[b-u]
\cap (\partial{F})^\times$. As $D[\bof \cup (\cup_{B\in\bbb} \bbb)]$ is planar
with $\bof$ bounding a face, no two $\bof$-bridges in $\bbb$ overlap.

Let $P=b'$ and $Q=\bof - \oo P$. Lemma \ref{lm:orderingLemma} applies to $\bof$, $P$, $Q$,
$\bbb$. As there are no digons disjoint from $H$, there is a unique (up
to inversion) ordering $B_1,\ldots, B_{11}$ of $\bbb$ so that
$P=P_{B_1}\ldots P_{B_{11}}$ and $Q=Q_{B_1}\ldots Q_{B_{11}}$. 

Because
$u\in Q_{B_1} \cap Q_{B_2} \cap \cdots \cap Q_{B_{11}}$ and the
$Q_{B_i}$ are internally disjoint subpaths of $Q$, all of
$Q_{B_2}, \ldots, Q_{B_{10}}$ are just $u$.
For $i=1,\ldots,11$, let  $a_i$ and $a_i'$ be the ends of $P_i$, so that
$P=(\ldots,a_2,\ldots,a_2',$ $
\ldots,a_3,\ldots,a_3',\ldots,a_{10},\ldots,a_{10}',\ldots)$. 

\begin{claim}\label{cl:distinctAtts}
For  $i\in \{2,\ldots,9\}$, $a_i
\neq a'_{i+1}$. 
\end{claim}

\begin{proof} Otherwise, $a_i = a_i' = a_{i+1} = a_{i+1}'$, implying that $B_i$ and $B_{i+1}$ constitute a digon disjoint from $H$, which is
impossible. 
\end{proof}

For $i,j\in\{2,3,\dots,10\}$ with $i<j$, set $K_{ij} = (\cup_{k=i}^j B_k ) \cup a_i P a_j'$.

\begin{claim}
For $i,j \in \{2,\ldots,10\}$ with $i<j$, $K_{ij}$
 is $2$-connected.
\end{claim}

\begin{proof}
Let $R_i$ be an $H$-avoiding $ua_i$-path in $B_i$, and $R_j$ an
$H$-avoiding $ua_j'$-path in $B_j$. Then $C_{ij}:= R_i \cup R_j \cup
a_i P a_j' \subseteq K_{ij}$ is a cycle containing $u$ and $a_i P
a_j'$.

For $x \in B_k$, $i \le k \le j$, $x\notin H$, \wording{for any $H$-node $w \neq u$}, $G$ has $3$ internally disjoint $xw$-paths; at least two \wording{of
these leave $B_k$} in $a_k P a_k'$, and so no cut vertex of $K_{ij}$
separates $x$ from $C_{ij}$.
\end{proof}

Since $b'$ is not crossed in $D$, $D[K_{i,i+2}]$ is
clean and is contained in $F\cup \partial{F}$. 
\wording{There is a unique face $F_i$ of $D[K_{i,i+2}]$ so that $F_i\not\subseteq F$;  since $K_{i,i+2}$ is 2-connected, $F_i$ is bounded by a cycle $C_i$.}
As $D[K_{i,i+2}] \subseteq F \cup \partial{F}$, $\partial{F} \subseteq
F_i \cup \partial{F_i}$. As $D[u] \in \partial{F} \cap D[K_{i,i+2}]$, 
$D[u] \in \partial{F_i}$. Likewise $D[a_i P a_{i+2}']\subseteq \partial F_i$.

\wording{Thus, $u \in C_i$ and} $a_i P a_{i+2}' \subseteq C_i$.  Therefore, $C_i\cap H$ is $u$ and $a_iPa'_{i+2}$, from which we deduce that there is a $C_i$-bridge $M_i$ so that $H\subseteq C_i\cup M_i$.  Observe that  $B_{i+1}$ is a $C_i$-bridge
different from $M_i$.

\wordingrem{(Text moved and modified.)}\wording{For $i=2,5,8$, let $e_i$ be an edge of $B_{i+1}$ incident with $u$,
and let $D_i$ be a 1-drawing of $G-e_i$.}

\begin{claim}\label{cl:BOD}  For $i\in \{2,5,8\}$, $C_i$ has BOD in $G$ \wording{and $D_i[C_i]$ is not clean}.
\end{claim}

\begin{proof}  \wordingrem{(Text removed.)}
At most one of $D_2[C_i]$, $i\in \{2,5,8\}$ is crossed, so for at
least one $i \in \{5,8\}$, $D_e[C_i]$ is clean. It follows that $C_i$ has BOD in $G-e$.

By Claim \ref{cl:distinctAtts}, $a_3 \neq a_i$, whence $B_3
\subseteq M_i$, and $B_3 - e   \subseteq M_i - e$. Furthermore, $u \in
H$, so $u\in \att(M_i - e)$. Thus $\att_{G-e}(M_i - e) = \att_G(M_i)$
and $M_i - e$ is a $C_i$-bridge in $G-e$. We conclude that the overlap diagrams for $C_i$ in $G-e$ and $G$ are isomorphic and, therefore, $C_i$ has BOD in $G$.

We now show that all three $C_j$, $j\in\{2,5,8\}$, have BOD in $G$.
\wordingrem{(Text removed.)}
\wording{If $D_{i}[C_i]$ is} clean,
\wording{then $D_{i}[C_i \cup M_i]$ is} a $1$-drawing of $C_i \cup M_i$,
implying via 
Corollary~\ref{co:TutteTwo} 
that $\crn(G) \le
  1$, a contradiction.  \wording{So $D_{i}[C_i]$ is} not clean, and, therefore, for $j\in
\{2,5,8\}\setminus\{i\}$, \wording{$D_{i}[C_j]$ is} clean.  Thus, $C_j$ has \wording{BOD in $G-e_i$}, and, following the argument above for $C_i$, we deduce that $C_j$ has BOD in $G$. \end{proof}

\wordingrem{(Text moved above and text moved to statement of Claim 3.)}

\begin{claim}\label{cl:sides}
For $i\in \{2,5,8\}$, one face of $D_i[C_i]$ contains all $H$-nodes, other than
(possibly) $u$.
\end{claim}


\begin{proof}
Let $e'_i$ be the edge of $H$ so that $D_i[e'_i]$ crosses $D_i[a_iba'_{i+2}]$ and let $b'_i$ be the $H$-branch containing $e'_i$.
If $n = 3$, let $R$ be a hexagon
in $H$ containing $b$ and $b'_i$. For $n\ge 4$, both $b$ and $b'_i$ are in the rim $R$ of  $H$.

Since $b$ and $b'_i$ are disjoint, for
$n \ge 3$, $R - (\oo{b}\cup\oo{b'})$ has two components, each with at
least two nodes of $H$. Either of these with $\le n$ nodes has all its
nodes adjacent by spokes to the other component. Obviously, there is
at least one such.

Observe that if $A$ is any path in $R - (\oo{b} \cup \oo{b'_i})$ such that $D_i[A]$ has a vertex in each face of $D_i[C_i]$, then $u\in V(A)$ and the two paths $P,P'$ in $A$ having $u$ as an end are such that $D_i[P]$ and $D_i[P']$ are in different faces of $D_i[C_i]$.

Let $K$ be a component of $R - (\oo{b} \cup \oo{b'_i})$ not containing
$u$ and let $L$ be the other. Then $D_i[K]$ is in the closure of a
face $F_i$ of $D_i[C_i]$. We claim that $D_i[L] \subseteq F_i \cup
\{u\}$. 

Any $H$-node $w$ in $L$ that is joined by a spoke to an $H$-node $w'$ in $K$ has $D_i[w]\subseteq F_i\cup D_i[u]$, since otherwise $D_i[ww']$ crosses $D_i[C_i]$.  

If there is an $H$-node $w$ in $L$ that is not adjacent by a spoke to any vertex in $K$, then $w$ is adjacent by a spoke to another $H$-node $w'$ in $L$ and, moreover, $w$ and $w'$ are the first and last nodes of $L$.  As $D_i[ww']$ is disjoint from $D_i[C_i]$, we deduce that there is a face $F$ of $D_i[C_i]$ so that  $D_i[w]$ and $D_i[w']$ are both in $F\cup D_i[u]$. Therefore, $D_i[L]$ is contained in that face.  As at least one $H$-node in $L$ is adjacent by a spoke to an $H$-node in $K$, we conclude that $D_i[L]\subseteq F_i\cup D_i[u]$.
\end{proof}

Let $F_i$ be the face of $D_i[C_i]$ containing all the $H$-nodes and let $F'_i$ be the other face of $D_i[C_i]$.

\begin{claim}
For $i\in\{2,5,8\}$, the crossing in $D_i$ is not in $\oo{a_{i+1},b,a_{i+1}'}$.
\end{claim}

\begin{proof}
Suppose by way of contradiction that $e'_i$ is an edge of $G-e_i$ so that $D_i[e'_i]$ crosses $\oo{a_{i+1},b,a'_{i+1}}$.  Clearly, $a_{i+1} \neq a_{i+1}'$.   Since $H-\oo b$ is 2-connected, there is a cycle $C'\subseteq H$ containing $e'_i$.  Let $P$ be an $H$-avoiding $a_{i+1}a'_{i+1}$-path in $B_{i+1}$ and let $C$ be the cycle $P\cup [a_{i+1},b,a'_{i+1}]$.  Then $C$ and $C'$ are graph-theoretically disjoint and $D_i[C]\cap D_i[C']$ contains the crossing of $D_i$.  But then $D_i[C]$ and $D_i[C']$ must cross a second time, a contradiction. 
\end{proof}

\begin{claim}\label{cl:onlyoverlap}
The only $C_i$-bridge that overlaps $B_{i+1}$ is $M_i$.
\end{claim}

\begin{proof}
Let $B$ be a $C_i$-bridge different from $M_i$ overlapping
$B_{i+1}$. Then $\att(B) \subseteq [a_i b a_{i+2}'] \cup \{u\}$.
As $H$ is smooth, $u\in \att(B)$.   We claim both $B_{i+1}$ and $B$ overlap $M_i$. 

By Claim \ref{cl:distinctAtts}, $a_i\neq a_{i+1}'$, so $B_{i+1}$ either has an
attachment in $\oo{a_i,a_{i+2}'}$ or it has both $a_i$ and $a_{i+2}'$
as attachments. In either case, $B_{i+1}$ overlaps $M_i$ (which has
attachments at $u,a_i, a_{i+2}'$).

Likewise $B$ either has two attachments in $\cc{a_i, a_{i+2}'}$ or at
least one attachment in $\oo{a_{i+1},a_{i+1}'} \subseteq \oo{a_i,a_{i+2}'}$,
so B overlaps $M_i$.  But now $B_{i+1}$, $B_i$, and $M_i$ make a triangle in $OD(C_i)$, \wording{contradicting Claim \ref{cl:BOD}}.
\end{proof}

\ignore{\begin{claim}
The ends of $e_i$ are not incident with the same face of $D_i[C_i\cup
  M_i]$. \marginpar{\tiny DOESN'T seem to be used anywhere}
\end{claim}

\begin{proof}
Suppose they are on the same face of $D_i[C_i\cup M_i]$. Let $b'$ be
the $H$-branch crossing $C_i$ in $D_i$. 
Let $x$ be the $H$-node incident with $b'$ so that the crossing is in
$xu$. Then $H- a_i P a_{i+2}' - x b' u   $ contains a claw $Y$ with points
\marginpar{\tiny Points: Nail tips of claws} $a_i, a_{i+2}'$, and
$u$. Note that $Y \cup C_i$ is a subdivision of $K_4$ and $D_i[C_i\cup
Y]$ is clean. Hence $D[C_i]$ bounds a face of $D_i[C_i\cup Y]$ and
this is the only face incident with $u$ and $a_{i+1} P
a_{i+1}'$. Hence
$D_i[B_{i+1} - e_i]$ is in this face.

Since no $C_i$-bridge other than $M_i$ overlaps $B_{i+1}$, and
$D_i[G-e_i]$ embeds all the other $C_i$-bridges, $u$ and $w$ are on
the same face of $D_i$. This yields $\ucr(G) \le 1$ \blitza.
\end{proof}}

Let $b'$ be the $H$-branch that crosses $C_i$ in $D_i$ and let $x$ be the $H$-node so that the crossing is in $[x,b',u]$.

\begin{claim}\label{cl:overlapbip1}
Let $L$ be the graph $[D_i[G-e_i]\cap (\hbox{\rm cl}(F_i'))]^\times
\cup B_{i+1}$. Then the $C_i$-bridge containing $[\times, b', u]$ overlaps
$B_{i+1}$ in $L$.
\end{claim}

\begin{proof}
If $L$ embeds in the plane with $C_i$ bounding a face, then \wordingrem{(useless text removed)}
 this embedding combines with $D_i$ restricted to the closure of $F$ to yield a 1-drawing of $G$, which is impossible.  As each individual $C_i$-bridge $B$ in $L$ has $C_i\cup B$ planar, there are overlapping $C_i$-bridges in $L$.

By definition, $L$ is planar with all $C_i$-bridges other
than $B_{i+1}$ on the same side of $C_i$.  
Therefore $B_{i+1}$ overlaps some other $C_i$-bridge in $L$.
By Claim~\ref{cl:onlyoverlap}, this is not any $C_i$-bridge \wording{other
than $D_i[M_i]^\times \cap D_i[L]$, that is}, the one containing $[\times,b',u]$.
\end{proof}

By Claim~\ref{cl:sides}, $[a_i, b ,a_{i+2}'] - \times$ has a component $A$
containing $\att(B_{i+1}) - u$. Let $z$ be the one of $ a_i$ and $ a_{i+2}$ that is an end of $A$ and let $Q$ be the minimal subpath of $A$ containing
all of $z, a_{i+1}, a_{i+1}'$. By Claim~\ref{cl:overlapbip1}, $M_i$
has an attachment $w_i \in \co{zQ}$ and an $H$-avoiding path $Q_i$
from $w_i$ to a vertex $x_i \in \oo{\times, b', u}$. Notice that, if $j \in
\{2,5,8\} \setminus \{i\}$, then $Q_i \cap C_j = \emptyset$.

There are at most two $H$-branches (or subpaths thereof) incident
with $u$ that can cross $b$. Thus for some $i,j \in \{2,5,8\}$, $b_i'
= b_j'$.  Choose the labelling so that $x_i$ is no further in $b_i'$
from $u$ than $x_j$ is. Since $x b_j' u$ contains $x_i$, $D_j[x_i]
\subseteq F_j'$ but $D_j[w_i] \subseteq F_j$. Since $Q_i \cap C_j =
\emptyset$, $D_j[Q_i]$ crosses $C_j$, the final contradiction.
\end{cproof}

The other steps in the argument are to show that a smooth subdivision $H$ of $V_{2n}$ in $G$ has few bridges with attachments in the interiors of distinct $H$-branches.  There are two parts to this:  either the branches do or do not have a node in common.  We first deal with the latter case.

\bigskip{\bf Step 2:} {\em $H$-bridges joining interiors of disjoint $H$-branches.}

\bigskip

}\begin{lemma}\label{lm:disjointbranches}
Let $G\in \mc{3}$, $\hvng$, $n\ge 3$, $H$ smooth and suppose $G$ has no subdivision of $V_{2(n+1)}$.  If $b_1, b_2$ are disjoint
$H$-branches, then there are at most $164n+9$ $H$-bridges having
attachments in both $\oo{b_1}$ and $\oo{b_2}$. 
\end{lemma}\printFullDetails{

\begin{cproof}
Suppose there is a set $\bbb$ of $164n+10$ $H$-bridges having attachments in
both $\oo{b_1}$ and $\oo{b_2}$. Let $B_0 \in \bbb$ and let $e\in
B_0$. In $D_e$, at most $4$ faces are incident with $\oo{b_1}$,
so there is a set $\bbb'$ consisting of  $41n+3$ elements of $\bbb\setminus\{B_0\}$ in
the same face of $D_e[H]$. By Lemma~\ref{lm:orderingLemma}, there is 
a unique ordering $(B_1,\ldots,B_{41n+3})$ of the elements of $\bbb'$
so they appear in this order in both $\oo{b_1}$ and $\oo{b_2}$. It follows that 
$B_2,\ldots,B_{41n+2}$ have all attachments in $\oo{b_1}\cup \oo{b_2}$. By
Lemmas~\ref{lm:orderingLemma} and~\ref{lm:fromvertextobranch}, $B_i$
and $B_{i+41}$ are totally disjoint.  So there are 
$n+1$ totally disjoint $\oo{b_1}\oo{b_2}$-paths with their ends having
the same relative orders on both.

We aim to use these disjoint paths to find a subdivision of $V_{2(n+1)}$ in $G$.  We need the following new notion.

}\begin{definition}  Let $e=uw$ and $f=xy$ be edges in a graph $G$.  Two cycles $C$ and $C'$ in $G$ are {\em $ef$-twisting\/}\index{twisting}\index{$ef$-twisting} if $C=(u,e,w,\dots,x,f,y,\dots)$ and $C'=(u,e,w,\dots,y,f,x,\dots)$, i.e., $C$ and $C'$ traverse the edges $e$ and $f$ in opposite ways.
\end{definition}\printFullDetails{

We note that $V_6$ has edge-twisting cycles:  if
 $e=uw$ and $f=xy$ \wording{are disjoint edges} in $V_6$, with $u,x$ not
adjacent, then the $4$-cycle $(u,w,x,y,u)$ and the $6$-cycle
$(u,w,z,y,x,z',u)$ are $ef$-twisting.

Next suppose $n \ge 4$. There are three possibilities for $b_1$ and
$b_2$. 
\label{pi:134}
\begin{description}

\item{{\bf Case 1}: {\sl Both $b_1$ and $b_2$ are in $R$.}} We may
  assume without loss of generality (recall that $b_1$ and $b_2$ are
  not adjacent)
 that $b_1 = r_0$, $b_2 = r_i$, $2 \le i
 \le n$. Set $H' = R \cup s_0 \cup s_1 \cup s_2$, so $H'\cong V_6$. Then $b_1$ and
 $b_2$ are in disjoint $H'$-branches and so $H'$, and therefore $H$,
 contains $b_1 b_2$-twisting cycles.
 
\item{{\bf Case 2}: {\sl One is in $R$, the other is a spoke.}}
We may assume without loss of generality that $b_1 = r_0$, $b_2 = s_i,
i\notin \{0,1\}$. Set $H'= R \cup s_0 \cup s_1 \cup s_i$. Then $b_1$
and $b_2$ are in disjoint $H'$-branches, so $H'$, and therefore $H$,
 contains $b_1 b_2$-twisting cycles.

\item{{\bf Case 3}: {\sl Both $b_1$ and $b_2$ are spokes.}}
We may assume without any loss of generality that $b_1 = s_0, b_2 =
s_i$. Then there exists $j \in \{ 0,\ldots,
n-1\}\setminus\{0,i\}$. Set $H'=R\cup s_0  \cup s_i \cup
s_j$.  Then $b_1$ and $b_2$ are in disjoint $H'$-branches and so $H'$,
and therefore $H$, contains $b_1 b_2$-twisting cycles.
\end{description}

\label{pi:135}

Choose the cycle $C$ in the twisting pair in $H$ for $b_1$ and $b_2$ so that $C$ traverses $b_1$
and $b_2$ in order so that the ends $u_i, w_i$ of the $n+1$ disjoint paths
occur in $C$ as $u_1, u_2, \ldots, u_{n+1}, \ldots,$ $w_1, \ldots,
w_{n+1}$. Then $C$ and these paths are a subdivision of $V_{2(n+1)}$ in $G$, contradicting the assumption that $G$ has no subdivision of $V_{2(n+1)}$.
\end{cproof}
\label{pi:136}

\wordingrem{(Text moved from below.)}\wording{Next is the third and final consideration.}

\bigskip\wording{{\bf Step 3:}  {\em $H$-bridges joining interiors of $H$-branches having a common node.}}

}\begin{lemma}\label{lm:consecutivebranches}
Let $G\in \mc{3}$, $\hvng$, $n\ge 3$, and let $b_1, b_2$ be adjacent
$H$-branches. Then at most $2$ $H$-bridges have attachments in both
$\oo{b_1}$ and $\oo{b_2}$. 
\end{lemma}\printFullDetails{

\begin{cproof}
By way of contradiction, suppose there is a set $\{B_1,B_2,B_3\}$ of
$3$ such 
$H$-bridges. For each $i\in\{1,2,3\}$, let $e_i\in B_i$. There is
precisely
one \wording{face $F_i$, of a 1-drawing $D_i$ of $G-e_i$, that is}   incident  with both $\oo{b_1}$ and
$\oo{b_2}$.  Thus, for each $B_j$, $j\neq i$, $D_i[B_j] \subseteq
F_i$. Clearly for $\{j,k\} = \{1,2,3\}\setminus\{i\}$, $B_j$ and $B_k$
do not overlap on $F_i$. In particular, their attachments in $b_1$ and
$b_2$ are in the same order as we traverse them from their common end
$u$.
\label{pi:137}
Thus we may assume $B_1, B_2, B_3$ appear in this order from $u$ on
both $b_1$ and $b_2$. 

\wording{Notice that  
$\att(B_3)\ne \att(B_2)$.  Therefore,}  there is a cycle $C
\subseteq B_2 \cup b_1 \cup b_2$ consisting of a
$\oo{b_1}\oo{b_2}$-path in $B_2$ and a subpath of $b_1\cup b_2$
containing $u$, such that $C$ does not contain some attachment $w$ of
$B_3$. \wording{Reselect $e_3\in B_3$ to be incident with $w$.  Let $M_C$ be the $C$-bridge so that $H\subseteq C\cup M_C$.}

\label{pi:138}
\wording{Then $w\in \Nuc(M_C)$, so $B_3\subseteq M_C$.  Furthermore, if  $e_3$ is incident with an
attachment $x$ of $M_C$, then $x$ is contained in $R$.  In particular, it is incident with another edge of $M_C$.}   Thus, $M_C - e_3$ is a
$C$-bridge in $G-e_3$ \wording{having the same attachments as $M_C$ has in $G$.  Because $C$ is $H$-close, $D_1[C]$ is clean; furthermore,  $D_1[C\cup M_C]$ is a 1-drawing of $C\cup M_C$.  Since $D_3[C]$ is also clean,} $C$ has BOD in
$G-e_3$ and hence in $G$. Corollary~\ref{co:TutteTwo} implies \wording{the contradiction that $\ucr(G)
\le 1$.}
\end{cproof}}



\wordingrem{(Other version of Lemma 16.13 removed.)}

We end this section with the asserted finiteness of 3-connected \2cc\ graphs with no subdivision of $V_{2n+2}$.

\label{pi:142}
}\begin{theorem}\label{th:allbounded} Suppose $G\in \mc{3}$ and there is an $n\ge3$ so that $G$ has a subdivision of $V_{2n}$, but no subdivision of $V_{2(n+1)}$.
Then $|V(G)|  = O(n^3)$. 
\end{theorem}\printFullDetails{

\begin{cproof}
By Lemma  \ref{lm:smooth},
$G$ has a smooth subdivision $H$ of $V_{2n}$.
\label{pi:seeAlso146}
We may assume no $H$-bridge contains a tripod, as otherwise
$|V(G)| \le 14$ by Lemma~\ref{lm:bridgeIsTree}.
\label{pi:145end} 

We first claim that a vertex $u$ of $H$ that is not an $H$-node is an attachment of some $H$-bridge $B$ not having all its attachments in the same $H$-branch.   Since $u$ has degree 2 in $H$ and degree greater than 2 in $G$, $u$ is an attachment of some $H$-bridge.  Because $H$ is smooth, an $H$-bridge that has all its attachments in the same $H$-branch is an edge in a digon.  If all the $H$-bridges attaching at $u$ are such edges, then $u$ has only two neighbours and $G$ is not 3-connected, a contradiction.

Thus, every vertex of $G$ is either an $H$-node or is in some $H$-bridge that does not have all its attachments in the same $H$-branch.
We bound the number of these  \label{pi:142again}
$H$-bridges as follows.  

We claim that,  \wordingrem{(text removed)}
for \wording{any three $H$-nodes $u,v,w$}, at most two $H$-bridges have all three of $u,v,w$ as attachments.  To see this, suppose three nontrivial $H$-bridges
$B_i, i=1,2,3$, all have all of $u,v,w$ as attachments. Each $B_i$ contains a
claw $Y_i$ having $u,v,w$ as talons.  Then $Y_1 \cup Y_2 \cup Y_3 \cup H$
contains a subdivision of $K_{3,4}$, in which case \wording{2-criticality implies} $G$ is $K_{3,4}$.  Thus,
\label{pi:144}
at most two $H$-bridges
have attachments in any three nodes. So there are at most 
$2\binom{2n}3$
nontrivial $H$-bridges with only node
attachments. 

Every other $H$-bridge of concern has an attachment in the interior of some $H$-branch and at some vertex of $H$ not in that $H$-branch.
Lemma~\ref{lm:fromvertextobranch}
implies that there are at most $(2n) (3n) 41$ $H$-bridges with an
attachment in an $H$-node and in an open $H$-branch.

\label{pi:145}
 Lemma~\ref{lm:disjointbranches} implies there are at most 
$(\binom{3n}2-6n)(164n +9)$ 
$H$-bridges having attachments in the
interiors of disjoint $H$-branches.

\wording{Lemma~\ref{lm:consecutivebranches} implies there are at most 2} $H$-bridges with attachments on two given adjacent $H$-branches and so there are \wording{at most $6n(2)$ $H$-bridges} with attachments on two adjacent $H$-branches.  

Every $H$-bridge has at most  $88$
vertices, and every vertex of $G$ is either an $H$-node or in one of these enumerated $H$-bridges.  Therefore,  
\begin{equation}
\nonumber
|V(G)| \le 88 \left\{ 
2 \binom{2n} 3 + 2n\cdot 3n \cdot 41 + 6n(2)+\left[\binom{3n}2-6n\right]\biggl [164 n + 9\biggr]
\right\}\,.
\end{equation}

\vskip -.4truein
\end{cproof}
}

\chapter{Summary}\printFullDetails{

\bigskip
This short section provides a single theorem and some remarks summarizing the current state of knowledge about \2cc\ graphs.

}\begin{theorem}[Classification of 2-crossing-critical graphs] Let $G$ be a 2-crossing-critical graph. 
\begin{enumerate}\item\label{it:degTwo} Then $G$ has minimum degree at least two and is a subdivision of a \2cc\ graph with minimum degree at least three. 

Thus, we henceforth assume $G$ has minimum degree at least three.

\item If $G$ is 3-connected and contains a subdivision of $V_{10}$, then $G\in\tileS$ (Definition \ref{df:t(s)}).   That is, $G$ is a twisted circular sequence of tiles, each tile being one of the 42 elements of $\mathcal S$ (Definition \ref{df:theTiles}). 

\item If $G$ is 3-connected and does not have a subdivision of $V_{10}$, then $G$ has at most three million vertices (so there are only finitely many such examples). Each of these examples either 
\begin{itemize}\item has a subdivision of $V_8$ or 
\item is either one of the four graphs described in Theorem \ref{th:3CutBothNonPlanar} or obtained from a 2-crossing-critical peripherally-4-connected graph with at most ten vertices by replacing each vertex $v$ having precisely three neighbors with one of at most twenty patches, each patch having at most six vertices (so $G$ has at most sixty vertices).
\end{itemize}
\item If $G$ is not 3-connected, then either 
\begin{itemize} \item $G$ is one of 13 examples that are not 2-connected, or
\item $G$ is 2-connected, has two nonplanar cleavage units, and is one of 36 graphs, or
\item   $G$ is 2-connected, has one nonplanar cleavage unit, and is obtained from a 3-connected 2-crossing-critical graph by replacing digons with digonal paths.
\end{itemize}
\end{enumerate}
\end{theorem}\printFullDetails{

We conclude with some remarks on what remains to be done to find all 2-crossing-critical graphs.

\begin{remark}  In Section \ref{V8free}, we provided a method for finding all 3-connected, 2-crossing-critical graphs not containing a subdivision of $V_8$.  It would be desirable for this program to be completed.  \end{remark}

\begin{remark}  The remaining unclassified 3-connected, 2-crossing-critical graphs have a subdivision of $V_8$ but not of $V_{10}$.  The works of Urrutia \cite{isabel} and Austin \cite{bethann} have found many of these, but more work is needed to find a complete set.  It may be helpful to note that we have found all such examples that do not have a representativity 2 embedding in the projective plane.  The known instances are all quite small, so it is reasonable to expect that each of these has at most 60 vertices or so.\end{remark}

}
\bigskip \centerline{\bf ACKNOWLEDGEMENTS}

Initial impetus to this project came through Shengjun Pan, who described mechanisms for proving a version of Theorem \ref{th:classification} (for $G$ containing a subdivision of $V_{2n}$, with $2n$ likely somewhat larger than 10).

We are grateful to CIMAT for hosting us on multiple occasions for work on this project.  In particular, we appreciate the support of Jose Carlos G\'omez Larra\~naga, then director of CIMAT.


\backmatter
\bibliographystyle{amsplain}

\begin{thebibliography}{99}
\bibitem{danThesis} D.~Archdeacon, A Kuratowski theorem for the projective plane, Ph.D.\ thesis, The Ohio State University, 1980.

\bibitem{danJGT} D.~Archdeacon, A Kuratowski theorem for the projective plane, J.~Graph Theory  5  (1981), no.\ 3, 243--246.

\bibitem{bethann} \wording{E. Austin, 2-crossing critical graphs with a $V_8$-minor, MMath thesis, U. Waterloo, 2011, http://uwspace.uwaterloo.ca/handle/10012/6464.}

\bibitem{barnette}  D.~W.~Barnette, Generating projective plane polyhedral maps,
J.~Combin.\ Theory Ser.\ B 51 (1991), no.\ 2, 277--291. 


\bibitem{bb}  L.~Beaudou and D.~Bokal, 
On the sharpness of some results relating cuts and crossing numbers,  \wording{Electron.\ J. Combin.\ {\bf17} (2010), no. 1, Research Paper 96, 8 pp}.

\bibitem{bc} R.~E.~Bixby and W.~H.~Cunningham, Matroids, graphs, and $3$-connectivity, in Graph theory and related \wording{topics,  91--103}, Academic Press, New York-London, 1979.

\bibitem{bkq}  G.~S.~Bloom, J.~W.~Kennedy, and L.~V.~Quintas, 
On crossing numbers and linguistic structures. Graph theory (\L ag—w, 1981), 14--22,
Lecture Notes in Math., 1018, Springer, Berlin, 1983.

\bibitem{avgdeg2} D. Bokal, Infinite families of crossing-critical graphs with prescribed average degree and crossing number, J.~Graph Theory 65 (2010), 139--162.

\bibitem{BM} J.~M.~Boyer and W.~J.~Myrvold, On the cutting edge: simplified $O(n)$ planarity by edge addition, J. Graph Algorithms Appl. {\bf8} (2004), \wording{no.\ 3, 241--273} (electronic).

\bibitem{chimani}  M.\ Chimani, C.\ Gutwenger, and P.\ Mutzel, On the minimum cut of planarizations,  6th Czech-Slovak International Symposium on Combinatorics, Graph Theory, Algorithms and Applications,  177--184, Electron.\ Notes Discrete Math. {\bf 28}, Elsevier, Amsterdam, 2007.

\bibitem{DMP} G. Demoucron, Y. Malgrange, and R. Pertuiset, Graphes planaires:  reconnaisance et construction de representation planaires topologiques, Rev.\ Franc.\ Rech.\ Oper.\ {\bf 8} (1964), 33--34.

\bibitem{diestel} R.~Diestel, Graph theory, (3rd ed.), Graduate Texts in Mathematics, 173, Springer-Verlag, Berlin, 2005.

\bibitem{DOTV} G. Ding, B. Oporowski, R. Thomas, and D. Vertigan, Large nonplanar graphs and an application to crossing-critical graphs, \wording{J.\ Combin.\ Theory Ser.\ B. {\bf101} (2011), no.\ 2, 111--121.}

\bibitem{G} A.~Gibbons, Algorithmic Graph Theory, Cambridge University Press, New York, 1985.

\bibitem{ghw} H.H. Glover, J.P. Huneke, and C.S. Wang, 103 graphs that are irreducible for the projective plane. J.~Combin.\ Theory Ser.\ B {\bf27} (1979), no.\ 3, 332--370.

\bibitem{dwh} D.~W.~Hall, A note on primitive skew curves,  Bull.\ Amer.\ Math.\ Soc.\ {\bf 49}  (1943), 935--936.

\bibitem{HKS55}{F.~Harary, P.~C.~Kainen, and A.~J.~Schwenk, Toroidal graphs with arbitrarily high crossing numbers, Nanta Math.\ {\bf6} (1973), 58--67.}

\bibitem{petr} P. Hlin\v en\'y, Crossing-number critical graphs have bounded path-width, J.~Combin.\ Theory Ser.~B {\bf88}  (2003),  no. 2, 347--367.

\bibitem {kelmans} A. K. Kelmans, $3$--connected graphs without essential $3$--cuts and triangles, Dokl.~Akad.~Nauk SSSR {\bf 288} (1986), no.~3, 531--535.

\bibitem{kochol} M.~Kochol, Construction of crossing-critical graphs, Discrete Math.\ {\bf 66} (1987), 311--313.

\bibitem{ls} J. Lea\~nos and G. Salazar,  On the additivity of crossing numbers of graphs,
J.~Knot Theory Ramifications  {\bf17} (2008), no.\ 9, 1043--1050. 

\bibitem{mr}  J. Maharry and N. Robertson, The structure of graphs not topologically containing the Wagner graph, preprint, October 2013.

\bibitem{M1} B. McKay, Isomorph-free exhaustive generation,
J. Algorithms {\bf26} (1998), no.\ 2, 306Ð324.

\bibitem{M2} B. McKay, {\tt nauty} available via his homepage \wording{http://cs.anu.edu.au/$\sim$bdm/}.
 
\bibitem{mohar}  B.\ Mohar,  Obstructions for the disk and the cylinder embedding extension problems, Comb.\ Probab.\ Comput. {\bf 3} (1994), no.~3,  375-406. 

\bibitem{mt} B.~Mohar and C.~Thomassen, Graphs on surfaces, Johns Hopkins Studies in the Mathematical Sciences, Johns Hopkins University Press, Baltimore, MD, 2001.

\bibitem{tiles1} B. Pinontoan and R. B. Richter, Crossing number of sequences of
graphs I: general tiles, {Australas.\ J.\ Combin.}\ {\bf 30} (2004), 197--206.

\bibitem{tiles2} B. Pinontoan and R. B. Richter, Crossing number of sequences of
graphs II: planar tiles, {J.\ Graph Theory}  {\bf 42} (2003), no.\ 4, 332--341.

\bibitem{rbr} B.~Richter, 
Cubic graphs with crossing number two,
J.~Graph Theory {\bf12} (1988), no.~3, 363--374.

\bibitem{survey} R. B. Richter and G. Salazar, Crossing numbers, in Selected Topics in Topological Graph Theory, L. Beineke and R. Wilson, eds, Oxford University Press, 2009, 133--150.


\bibitem{avgdeg} G. Salazar, Infinite families of crossing-critical graphs with given average degree,  Discrete Math.\ {\bf 271}  (2003),  no. 1-3, 343--350.

\bibitem{siran} J. \v Sir\'a\v n, Additivity of the crossing number of graphs with connectivity 2, 
 Period.\ Math.\ Hungar.\ {\bf 15}  (1984),  no.\ 4, 301--305.
 
\bibitem{siran2} J. \v Sir\'a\v n,
Infinite families of crossing-critical graphs with a given crossing number,
Discrete Math.\ {\bf48} (1984), no.~1, 129--132. 

\bibitem{TutCinG} W.~T.~Tutte, Connectivity in Graphs, Mathematical Expositions, No.~15  Oxford University Press, London, 1966

\bibitem{TutEncy} W.~T.~Tutte, Graph Theory, Encyclopedia of Mathematics and its Applications, 21, Cambridge University Press, Cambridge, 2001.

\bibitem{isabel} I. Urrutia-Schroeder, MMath essay, U.~Waterloo, 2011.

\bibitem{vitray} R.~P.~Vitray, The $2$- and $3$-representative projective planar embeddings,  J.~Combin.\ Theory Ser.\ B  54  (1992),  no.\ 1, 1--12.

\bibitem{OSUtalk} R.~P.~Vitray, { Graphs containing graphs of
crossing number 2}, presentation at AMS Summer Conference, Ohio State
University, August 1990.
\end{thebibliography}

\printindex

\end{document}